\documentclass{amsart}
\usepackage{amssymb}
\usepackage{graphicx}
\usepackage{xcolor} % A package to add color.

\usepackage{hyperref}

\usepackage{tensor}

\newcommand{\LieBr}[2]{[ #1, #2 ]}
\newcommand{\LieGrp}{\mathfrak{G}}
\newcommand{\LieAlg}{\mathfrak{g}}
\newcommand{\covD}{\bfD}
\newcommand{\df}{\mathrm{df}}
\newcommand{\cf}{\mathrm{cf}}

\newcommand{\SH}{\dot{S}}

\newcommand{\covDlow}{\underline{\covD}}
\newcommand{\Alow}{\underline{A}}
\newcommand{\Flow}{\underline{F}}
\newcommand{\wlow}{\underline{w}}
\newcommand{\Ulow}{\underline{U}}

\newcommand{\Aini}{\overline{A}}
\newcommand{\Bini}{\overline{B}}
\newcommand{\Eini}{\overline{E}}
\newcommand{\Fini}{\overline{F}}
\newcommand{\Vini}{\overline{V}}
\newcommand{\calIini}{\overline{\calI}}

\newcommand{\sbr}{\overline{s}}

\newcommand{\subr}{\underline{s}}

\newcommand{\calAlow}{\underline{\calA}}
\newcommand{\calElow}{\underline{\calE}}

\newcommand{\curl}{\rd \times}

\newcommand{\linear}{\mathrm{linear}}
\newcommand{\quadratic}{\mathrm{quadratic}}
\newcommand{\cubic}{\mathrm{cubic}}
\newcommand{\forcing}{\mathrm{forcing}}

\newcommand{\Atemp}{A^{\dagger}}
\newcommand{\AtempPrime}{(A')^{\dagger}}
\newcommand{\Ftemp}{F^{\dagger}}

\makeatletter
\newtheorem*{rep@theorem}{\rep@title}
\newcommand{\newreptheorem}[2]{%
\newenvironment{rep#1}[1]{%
 \def\rep@title{#2 \ref{##1}}%
 \begin{rep@theorem}}%
 {\end{rep@theorem}}}
\makeatother
 
\theoremstyle{plain}
\newtheorem*{MainTheorem}{Main Theorem}
\newtheorem*{corrPrinciple}{Correspondence Principle}
\newtheorem{theorem}{Theorem}[section]
\newtheorem{corollary}[theorem]{Corollary}
\newtheorem{lemma}[theorem]{Lemma}
\newtheorem{proposition}[theorem]{Proposition}

\newreptheorem{proposition}{Proposition}

\theoremstyle{definition}
\newtheorem{definition}[theorem]{Definition}
\newtheorem{example}[theorem]{Example}
\theoremstyle{remark}
\newtheorem{remark}[theorem]{Remark}
\numberwithin{equation}{section}

\definecolor{green}{rgb}{0,0.8,0} % Redefines the color green.
\newcommand{\nrm}[1]{\Vert#1\Vert}
\newcommand{\abs}[1]{\vert#1\vert}

\newcommand{\set}[1]{\{#1\}}

\newcommand{\tr}{\textrm{tr}}

\newcommand{\aeq}{\sim}

\newcommand{\lap}{\triangle}

\newcommand{\ud}{\mathrm{d}}
\newcommand{\rd}{\partial}
\newcommand{\nb}{\nabla}

\newcommand{\bb}{\Big}

\newcommand{\alp}{\alpha}

\newcommand{\gmm}{\gamma}

\newcommand{\dlt}{\delta}

\newcommand{\eps}{\epsilon}

\newcommand{\lmb}{\lambda}

\newcommand{\sgm}{\sigma}

\newcommand{\tht}{\theta}

\newcommand{\omg}{\omega}
\newcommand{\Omg}{\Omega}

\newcommand{\bfa}{{\bf a}}
\newcommand{\bfb}{{\bf b}}
\newcommand{\bfc}{{\bf c}}

\newcommand{\bfB}{{\bf B}}

\newcommand{\bfD}{{\bf D}}
\newcommand{\bfE}{{\bf E}}

\newcommand{\bbR}{\mathbb R}

\newcommand{\calA}{\mathcal{A}}
\newcommand{\calB}{\mathcal{B}}
\newcommand{\calC}{\mathcal{C}}
\newcommand{\calD}{\mathcal{D}}
\newcommand{\calE}{\mathcal{E}}
\newcommand{\calF}{\mathcal{F}}

\newcommand{\calH}{\mathcal{H}}
\newcommand{\calI}{\mathcal{I}}

\newcommand{\calL}{\mathcal{L}}
\newcommand{\calM}{\mathcal{M}}
\newcommand{\calN}{\mathcal{N}}
\newcommand{\calO}{\mathcal{O}}
\newcommand{\calP}{\mathcal{P}}

\newcommand{\calS}{\mathcal{S}}
\newcommand{\calT}{\mathcal{T}}

\newcommand{\calW}{\mathcal{W}}
\newcommand{\calX}{\mathcal{X}}

\newcommand{\pfstep}[1]{\vspace{0.1in} {\it #1.}}
\begin{document}

\title[]{Gauge choice for the Yang-Mills equations using the Yang-Mills heat flow and local well-posedness in $H^{1}$.}%: Title of the article
\author{Sung-Jin Oh}%
\address{Department of Mathematics, Princeton University, Princeton, New Jersey, 08544}%
\email{sungoh@math.princeton.edu}%

%\thanks{}%
%\subjclass{}%
%\keywords{}%

%\date{\today}%
%\dedicatory{}%
%\commby{}%
% ----------------------------------------------------------------
\begin{abstract}
In this work, we introduce a novel approach to the problem of gauge choice for the Yang-Mills equations on the Minkowski space $\bbR^{1+3}$, which uses the Yang-Mills heat flow in a crucial way. As this approach does not possess the drawbacks of the previous approaches, it is expected to be more robust and easily adaptable to other settings.

As the first application, we give an alternative proof of the local well-posedness of the Yang-Mills equations for initial data $(\Aini_{i}, \Eini_{i}) \in (\dot{H}^{1}_{x} \cap L^{3}_{x}) \times L^{2}_{x}$, which is a classical result of S. Klainerman and M. Machedon \cite{Klainerman:1995hz} that had been proved using a different method (local Coulomb gauges). The new proof does not involve localization in space-time, which had been the key drawback of the previous method. Based on the results proved in this paper, a new proof of finite energy global well-posedness of the Yang-Mills equations, also using the Yang-Mills heat flow, is established in the companion article \cite{Oh:2012fk}. 
\end{abstract}
\maketitle
% ----------------------------------------------------------------

\tableofcontents

\section{Introduction}
In this paper, we address the problem of gauge choice for the Yang-Mills equations on the Minkowski space $\bbR^{1+3}$, with a non-abelian structure group $\LieGrp$, in the context of low regularity well-posedness of the associated Cauchy problem. The traditional gauge choices in such setting were the \emph{(local) Coulomb gauge} $\rd^{\ell} A_{\ell} = 0$ \cite{Klainerman:1995hz} or the \emph{temporal gauge} $A_{0} = 0$ \cite{Tao:2000vba}. Each, however, had a shortcoming of its own (to be discussed below), because of which not much has been known about low regularity solutions to the Yang-Mills equations with \emph{large} data\footnote{We however remark that better results are available in the case of \emph{small} initial data. See \cite{Tao:2000vba} and also the discussion below.}. In fact, the best result (in terms of the regularity condition on the initial data) along this direction so far has been the local well-posedness of the Yang-Mills equations for data $(\Aini_{i}, \Eini_{i}) \in \dot{H}^{1}_{x} \times L^{2}_{x}$ by Klainerman-Machedon \cite{Klainerman:1995hz}, whereas the scaling property of the Yang-Mills equations dictates that the optimal regularity condition should be $(\Aini_{i}, \Eini_{i}) \in \dot{H}^{1/2}_{x} \times \dot{H}^{-1/2}_{x}$ (for the notations, we refer the reader to \S \ref{subsec:intro:bg}). 

In this work, we propose a novel approach to this problem using the celebrated \emph{Yang-Mills heat flow}, which does not possess the drawbacks of the previous gauge choices. As such, this approach is expected to be more robust and have many applications, including that of establishing large data low regularity well-posedness of the Yang-Mills equations. As the first demonstration of the potential of this approach, we give a new proof of the aforementioned local well-posedness result of Klainerman-Machedon \cite{Klainerman:1995hz}. In the companion paper \cite{Oh:2012fk}, we demonstrate that the main result of \cite{Klainerman:1995hz}, namely the \emph{finite energy global well-posedness} of the Yang-Mills equations on $\bbR^{1+3}$, can be proved using the new approach as well.

\subsection{The Yang-Mills equation on $\bbR^{1+3}$} \label{subsec:intro:bg}
%This paper is concerned with the classical Yang-Mills equations on $\bbR^{1+3}$, which are a non-linear generalization of the more familiar Maxwell equations. These equations were introduced in the setting of Quantum Field Theory (in the quantized version) in order to explain the fundamental forces of nature other than electromagnetism and gravity. In General Relativity, the Yang-Mills equations, in particular the classical equations which we study here, are also of interest in view of the fact that they furnish a nice, simpler model equation for the more formidable Einstein equations, which governs the evolution of the space-time. (See the recent exciting work \cite{Klainerman:2012ve} for Einstein equations which was essentially based on this idea.)
%
We will work on the Minkowski space $\bbR^{1+3}$. All tensorial indices will be raised and lowered by using the Minkowski metric, which we assume to be of signature $(-+++)$. Moreover, we will adopt the Einstein summation convention of summing up repeated upper and lower indices. Greek indices, such as $\mu, \nu, \lmb$, will run over $x^{0}, x^{1}, x^{2}, x^{3}$, whereas latin indices, such as $i, j, k, \ell$, will run \emph{only} over the spatial indices $x^{1}, x^{2}, x^{3}$. We will often write $t$ instead of $x^{0}$. A $k$-fold application of a derivative will be denoted $\rd^{(k)}$, in contrast to $\rd^{k}$ which will always mean the partial derivative in the direction $x^{k}$, with its index raised. 

Let $\LieGrp$ be a Lie group and $\LieAlg$ be its Lie algebra. We will assume that $\LieGrp$ admits a bi-invariant inner product $( \cdot, \cdot) : \LieAlg \times \LieAlg \to [0, \infty)$, i.e., an inner product invariant under the adjoint map\footnote{A sufficient condition for a bi-invariant inner product to exist is that $\LieGrp$ is a product of an abelian and a semi-simple Lie groups.}.  For simplicity, we will furthermore take $\LieGrp$ to be a matrix group. A model example that one should keep in mind is the Lie group $\LieGrp = \mathrm{SU}(n)$ of complex unitary matrices, in which case $\LieAlg = \mathfrak{su}(n)$ is the set of complex traceless anti-hermitian matrices and $(A, B) = \tr (A B^{\star})$.

Below, we will present some geometric concepts we will need in a pragmatic, condensed fashion; for a more thorough treatment, we refer the reader to the standard references \cite{Bleeker:2005uj}, \cite{Kobayashi:1963uh}, \cite{Kobayashi:1969ub}. 

Let us consider a $\LieAlg$-valued 1-form $A_{\mu}$ on $\bbR^{1+3}$, which we call a \emph{connection 1-form}. For any $\LieAlg$-valued tensor $B$ on $\bbR^{1+3}$, we define the associated \emph{covariant derivative} $\covD = {}^{(A)}\covD$ by
\begin{equation*}
	\covD_{\mu} B := \rd_{\mu} B + \LieBr{A_{\mu}}{B}
\end{equation*}
where $\rd_{\mu}$ refers to the ordinary directional derivative on $\bbR^{1+3}$. Due to the bi-invariance of $(\cdot, \cdot)$, we have the following Leibniz's rule for $\LieAlg$-valued tensors $B, C$:
\begin{equation*}
	\rd_{\mu} (B, C) = (\covD_{\mu} B, C) + (B, \covD_{\mu} C)
\end{equation*}

The commutator of two covariant derivatives gives rise to a $\mathfrak{g}$-valued $2$-form $F_{\mu \nu} = F[A]_{\mu \nu}$, which we call the \emph{curvature 2-form} associated to $A$, as follows :
\begin{equation*}
	\covD_{\mu} \covD_{\nu} B - \covD_{\nu} \covD_{\mu} B = \LieBr{F_{\mu \nu}}{B}.
\end{equation*}

It is easy to compute that
\begin{equation*}
	F_{\mu \nu} = \rd_{\mu} A_{\nu} - \rd_{\nu} A_{\mu} + \LieBr{A_{\mu}}{A_{\nu}}.
\end{equation*}

From the way $F_{\mu \nu}$ arises from $A_{\mu}$, the following identity, called the \emph{Bianchi identity}, always holds.
\begin{equation} \label{eq:bianchi} \tag{Bianchi}
\covD_{\mu} F_{\nu \lmb} + \covD_{\lmb} F_{\mu \nu} + \covD_{\nu} F_{\lmb \mu} = 0.
\end{equation}

The \emph{Yang-Mills equations on $\bbR^{1+3}$} are the following additional first order equations for $F_{\mu \nu}$.
\begin{equation} \label{eq:hyperbolicYM} \tag{YM}
	\covD^{\mu} F_{\nu \mu} = 0,
\end{equation}
where we utilize the Einstein convention of summing repeated lower and upper indices.

%Note the similarity of \eqref{eq:bianchi} and \eqref{eq:hyperbolicYM} with the Maxwell equations. In fact, the Maxwell equations are a special case of the Yang-Mills equations in the case $\LieGrp = \mathrm{SU}(1)$.

An important feature of the Yang-Mills equations is its \emph{gauge structure}. Given a smooth $\LieGrp$-valued function $U$ on $\bbR^{1+3}$, we let $U$ act on $A, \covD, F$ as a \emph{gauge transform} according to the following rules.
\begin{equation*}
\begin{aligned}
	\widetilde{A}_{\mu} = U A_{\mu} U^{-1} - \rd_{\mu} U U^{-1}, \qquad
	\widetilde{\covD}_{\mu} =  U \covD_{\mu} U^{-1}, \qquad
	\widetilde{F}_{\mu \nu} =  U F_{\mu \nu} U^{-1}.
\end{aligned}
\end{equation*} 

We say that a $\LieAlg$-valued tensor $B$ is \emph{gauge covariant}, or \emph{covariant under gauge transforms}, if it transforms in the fashion $\widetilde{B} = U B U^{-1}$. Given a gauge covariant $\LieAlg$-valued tensor $B$, its covariant derivative $\covD_{\mu} B$ is also gauge covariant as the following formula shows.
\begin{equation*}
	\widetilde{\covD}_{\mu} \widetilde{B} = U \covD_{\mu} B U^{-1}.
\end{equation*}

The Yang-Mills equations are evidently covariant under gauge transforms. It has the implication that a solution to \eqref{eq:hyperbolicYM} makes sense only as a class of gauge equivalent connection 1-forms. Accordingly, we make the following definition.

\begin{definition} 
A \emph{classical solution} to \eqref{eq:hyperbolicYM} is a class of smooth connection 1-forms $A$ satisfying \eqref{eq:hyperbolicYM}, which are related to each other by smooth gauge transforms. A \emph{generalized solution} to \eqref{eq:hyperbolicYM} is defined to be a class of gauge equivalent connection 1-forms $A$ for which there exists a sufficiently smooth representative $A$ (say $\rd_{t,x} A \in C_{t} L^{2}_{x}, A \in C_{t} L^{3}_{x}$) which satisfies \eqref{eq:hyperbolicYM} in the sense of distributions.
\end{definition}

A choice of a particular representative is called a \emph{gauge choice}. A gauge is usually chosen by imposing a condition, called a \emph{gauge condition}, on the representative. Classical examples of gauge conditions include the \emph{temporal gauge condition} $A_{0} = 0$ and the \emph{Coulomb gauge condition} $\rd^{\ell} A_{\ell} = 0$.

In this work, and also in the companion article \cite{Oh:2012fk}, we will study the Cauchy problem associated to \eqref{eq:hyperbolicYM}. The initial data set, which consists of $(\Aini_{i}, \Eini_{i})$ for $i=1,2,3$ with $\Aini_{i} = A_{i}(t=0)$ (magnetic potential) and $\Eini_{i} = F_{0i}(t=0)$ (electric field), has to satisfy the \emph{constraint equation}:
\begin{equation} \label{eq:YMconstraint}
	\rd^{\ell} \Eini_{\ell} + \LieBr{\Aini^{\ell}}{\Eini_{\ell}} = 0,
\end{equation}
where $\ell$ is summed only over $\ell = 1,2,3$. We remark that this is the $\nu = 0$ component of \eqref{eq:hyperbolicYM}.

The Yang-Mills equations possess a positive definite conserved quantity $\bfE[t]$, defined by:
\begin{equation*}
	\bfE(t) := \frac{1}{2} \int_{\bbR^{3}} \sum_{\ell=1,2,3} (F_{0\ell}(t,x), F_{0\ell}(t,x)) + \sum_{k, \ell=1,2,3, k < \ell} (F_{k \ell}(t,x), F_{k \ell}(t,x))  \, \ud x
\end{equation*}
We call $\bfE(t)$ the \emph{conserved energy} of the field at time $t$.

Note that the Yang-Mills equations remain invariant under the scaling
\begin{equation} \label{eq:intro:scaling} 
	x^{\alp} \to \lmb x^{\alp}, \quad A \to \lmb^{-1} A, \quad F \to \lmb^{-2} F.
\end{equation}

The norms $\nrm{\rd_{x} \Aini_{i}}_{L^{2}_{x}}$, $\nrm{\Eini_{i}}_{L^{2}_{x}}$, as well as the conserved energy $\bfE(t)$, of the rescaled field become $\lmb^{-1}$ of that of the original field, which allows us to assume smallness of these quantities by scaling. This reflects the \emph{sub-criticality} of these quantities compared to the Yang-Mills equations.

\subsection{Statement of the Main Theorem}
To state the Main Theorem of this paper, we first define the class of initial data sets that will be considered.

\begin{definition}[Admissible $H^{1}$ initial data set] \label{def:admID}
We say that a pair $(\Aini_{i}, \Eini_{i})$ of 1-forms on $\bbR^{3}$ is an \emph{admissible $H^{1}$ initial data} set for the Yang-Mills equations if the following conditions hold:
\begin{enumerate}
\item $\Aini_{i} \in \dot{H}_{x}^{1} \cap L^{3}_{x}$ and $\Eini_{i} \in L^{2}$,
\item The \emph{constraint equation}
\begin{equation*}
	\rd^{i} \Eini_{i} + \LieBr{\Aini^{i}}{\Eini_{i}} = 0,
\end{equation*}
	holds in the distributional sense.
\end{enumerate}
\end{definition}

Next, we extend the notion of a solution to \eqref{eq:hyperbolicYM} by taking the closure of the set of classical solutions in the appropriate topology. 

\begin{definition}[Admissible solutions] \label{def:admSol}
Let $I \subset \bbR$. We say that a generalized solution $A_{\mu}$ to the Yang-Mills equations \eqref{eq:hyperbolicYM} defined on $I \times \bbR^{3}$ is \emph{admissible} [in the temporal gauge $A_{0} = 0$] if 
\begin{equation*}
	A_{\mu} \in C_{t}(I, \dot{H}^{1}_{x} \cap L^{3}_{x}), \quad \rd_{t} A_{\mu} \in C_{t}(I, L^{2}_{x})
\end{equation*}
and $A_{\mu}$ can be approximated by representatives of classical solutions [in the temporal gauge $A_{0} = 0$] in the above topology.
\end{definition}

Our main theorem is a local well-posedness result for admissible $H^{1}$ initial data, within the class of admissible solutions in the temporal gauge.
\begin{MainTheorem} [$H^{1}$ local well-posedness of the Yang-Mills equations, temporal gauge] \label{thm:mainThm}
Let $(\Aini_{i}, \Eini_{i})$ be an admissible $H^{1}$ initial data set, and define $\calIini := \nrm{\Aini}_{\dot{H}^{1}_{x}} + \nrm{\Eini}_{L^{2}_{x}}$. Consider the initial value problem (IVP) for \eqref{eq:hyperbolicYM} with $(\Aini_{i}, \Eini_{i})$ as the initial data.
\begin{enumerate}
\item There exists $T^{\star} = T^{\star}(\calIini)> 0$, which is non-increasing in $\calIini$, such that a unique admissible solution $A_{\mu} = A_{\mu} (t,x)$ to the IVP in the temporal gauge $A_{0} = 0$ exists on $(-T^{\star}, T^{\star})$. Furthermore, the following estimates hold.
\begin{equation} \label{eq:mainThm:0}
	\sup_{i} \nrm{\rd_{t,x} A_{i}}_{C_{t} ((-T^{\star}, T^{\star}),L^{2}_{x})} \leq C_{\calIini} \calIini.
\end{equation}
\begin{equation} \label{eq:mainThm:1}
	\sup_{i} \nrm{A_{i}}_{C_{t} ((-T^{\star}, T^{\star}), L^{3}_{x})} \leq \sup_{i} \nrm{\Aini_{i}}_{L^{3}_{x}} + (T^{\star})^{1/2} C_{\calIini} \calIini.
\end{equation}
\item Let $(\Aini'_{i}, \Eini'_{i})$ be another admissible $H^{1}$ initial data set such that $\nrm{\Aini'}_{\dot{H}^{1}_{x}} + \nrm{\Eini'}_{L^{2}_{x}} \leq \calIini$, and let $A'_{\mu}$ be the corresponding solution given by (1). Then the following estimates for the difference hold.
\begin{equation} \label{eq:mainThm:2}
	\sup_{i} \nrm{\rd_{t,x} A_{i}- \rd_{t,x} A'_{i}}_{C_{t} ((-T^{\star}, T^{\star}),L^{2}_{x})} \leq C_{\calIini} (\sup_{i} \nrm{\Aini_{i} - \Aini'_{i}}_{\dot{H}^{1}_{x}} +\sup_{i} \nrm{\Eini_{i} - \Eini_{i}'}_{L^{2}_{x}} ).
\end{equation}
\begin{equation} \label{eq:mainThm:3}
\begin{aligned}
	\sup_{i} \nrm{A_{i}- A'_{i}}_{C_{t} ((-T^{\star}, T^{\star}), L^{3}_{x})} \leq &
	\sup_{i} \nrm{\Aini_{i} - \Aini'_{i}}_{L^{3}_{x}} \\ &
	+ (T^{\star})^{1/2} C_{\calIini} (\sup_{i} \nrm{\Aini_{i} - \Aini'_{i}}_{\dot{H}^{1}_{x}} +\sup_{i} \nrm{\Eini_{i} - \Eini_{i}'}_{L^{2}_{x}} ).
\end{aligned}
\end{equation}

\item Finally, the following version of \emph{persistence of regularity} holds: if $\rd_{x} \Aini_{i}, \Eini_{i} \in H^{m}_{x}$ for an integer $m \geq 0$, then the corresponding solution given by (1) satisfies
\begin{equation*}
	\rd_{t,x} A_{i} \in C^{k_{1}}_{t} ((-T^{\star}, T^{\star}), H^{k_{2}}_{x})
\end{equation*}
for every pair $(k_{1}, k_{2})$ of nonnegative integers such that $k_{1} + k_{2} \leq m$.
\end{enumerate}
\end{MainTheorem}

%\begin{remark}
%The estimate \eqref{eq:mainThm:4} is an example of \emph{persistence of higher regularity}, which usually is a part of a local-wellposedness theory. As usual, persistence of even higher regularity can be proved by differentiating the equations and repeating the proof of the Main Theorem, or relying on classical results for local well-posed at the level of $H^{2}$. As this is standard, we omit the exact statement and the proof.
%\end{remark}

%\begin{remark} 
%
%\end{remark}

We emphasize that the temporal gauge in the statement of the Main Theorem plays a rather minor role. It has the advantage of being easy to impose. Moreover, classical local well-posedness results for smoother data (essentially due to Segal \cite{Segal:1979hg} and Eardley-Moncrief \cite{Eardley:1982fb}) are available, which are useful in the proof of the Main Theorem (Theorem \ref{thm:mainThm:H2lwp}). However, most of our analysis in this paper takes place under different gauge conditions defined with the help of the \emph{Yang-Mills heat flow}, to be introduced below.

The Main Theorem is, in fact, a classical result of Klainerman-Machedon \cite{Klainerman:1995hz}, which has been the best result so far concerning low regularity local well-posedness of \eqref{eq:hyperbolicYM} for large data. In both this paper and \cite{Klainerman:1995hz}, it is essential to choose an appropriate gauge to reveal the \emph{null structure} of the quadratic nonlinearities of the wave equations. It is known that such structure is present in the Coulomb gauge \cite{Klainerman:1994jb}. Unfortunately, in the case of a non-abelian structural group $\LieGrp$, it may not be possible in general to impose the Coulomb gauge condition on an arbitrary initial data. In \cite{Klainerman:1995hz}, this issue is avoided by working in a so-called \emph{local Coulomb gauge}, which is the Coulomb gauge condition imposed on a small domain of dependence. Due to the presence of the constraint equation \eqref{eq:YMconstraint},  delicate boundary conditions had to be imposed along the lateral boundary (a cone in the case of the Minkowski space $\bbR^{1+3}$). Because of this, it had been difficult to use this method along with global Fourier-analytic techniques such as hyperbolic-Sobolev spaces (see \cite{DAncona:2012ke}), and thus it has not been extended to initial data with lower regularity than $H^{1}_{\mathrm{loc}} \times L^{2}_{\mathrm{loc}}$.

Let us also mention an alternative approach to local well-posedness by Tao \cite{Tao:2000vba}, who worked entirely in the temporal gauge $A_{0} = 0$ to prove local well-posedness for initial data with even lower regularity than $H^{1}$. The main idea is that \eqref{eq:hyperbolicYM} in the temporal gauge can be cast as a coupled system of wave and transport equations (using the Hodge decomposition of $A_{i}$), where the wave equations possess a null structure similar to that in the Coulomb gauge. Although the temporal gauge has the advantage of being easy to impose (globally in space), the statements proved by this method are unfortunately restricted to \emph{small} data due to, among other things, the presence of too many time derivatives in the transport equation for $\rd^{\ell} A_{\ell}$.  

In this paper, we introduce a new approach to the problem of gauge choice which does not have the drawbacks of the methods outlined above. In particular, it does not involve localization in space-time and works well for large initial data. Moreover, the most dangerous quadratic nonlinearities of the wave equations are seen to possess a null structure, which allows us to give a new proof of the Main Theorem. For these reasons, we expect the present approach to be more robust and applicable to other problems as well, such as large data low regularity well-posedness of the Yang-Mills equations and other non-abelian gauge theories.

In the companion paper \cite{Oh:2012fk}, we prove global well-posedness of \eqref{eq:hyperbolicYM} (in the temporal gauge) for this class of initial data, using the positive definite conserved energy of \eqref{eq:hyperbolicYM}. The proof involves many techniques developed in the this paper (including the Main Theorem). On the other hand, the present paper does not depend on the results proved \cite{Oh:2012fk}\footnote{Except for Kato's inequality used in the proof of Lemma \ref{lem:mainThm:regApprox}, whose proof is standard and can be found in other sources, such as \cite[Proof of Corollary 3.3]{Smith:2011ef}, as well.}.

In simple, heuristic terms, the main idea of the novel approach is to `smooth out' the problem in a `geometric fashion'. For the problem under consideration, the `smoothed out' problem is much easier; indeed, recall the classical works \cite{Segal:1979hg}, \cite{Eardley:1982fb} in which local well-posedness for (possibly large) initial data with higher degree of smoothness was established by working directly in the temporal gauge. The difficulty of the original problem manifests in our approach when estimating the difference between the solutions to the original and `smooth out' problems. Here, we need to exploit the `special structure' inherent to the Yang-Mills equations. That this is possible by using a smoothing procedure based on the associated geometric flow, the \emph{Yang-Mills heat flow} in this case, is the main thesis of this work. 

The present work advances a relatively new idea in the field of hyperbolic PDEs, which is to use a geometric parabolic equation to better understand a hyperbolic equation.  
To the author's knowledge, this was first used in the work of Klainerman-Rodnianski \cite{MR2221254}\footnote{We remark that a preprint of \cite{MR2221254} was posted on the arXiv in 2003 (arXiv:math/0309463), predating the other references discussed below.}, in which the linear heat equation on a compact 2-manifold was used to develop an invariant form of Littlewood-Paley theory for arbitrary tensors on the manifold. This was applied in \cite{MR2125732} and \cite{MR2221255} to study the causal geometry of solutions to the Einstein's equations under very weak hypotheses.

More recently, this idea was carried much further by Tao, who proposed using a nonlinear geometric heat flow to deal with the problem of gauge choice in the context of energy critical wave maps. This approach, called the \emph{caloric gauge}, was used in \cite{Tao:2008wn} to study the long term behavior of large energy wave maps on $\bbR^{1+2}$. It has also played an important role in the recent study of the related energy critical Schr\"odinger map problem; see \cite{Bejenaru:2011wy}, \cite{Smith:2011ef}, \cite{Smith:2010ui}, \cite{Smith:2011ty}, \cite{Dodson:2012uj} and \cite{Dodson:2013vh}. 

The basic idea of the caloric gauge is as follows. The associated heat flow (the \emph{harmonic map flow}) starting from a wave or a Schr\"odinger map on a fixed time slice converges (under appropriate conditions) to a single point (same for every time slice) as the heat parameter goes to $\infty$. For this trivial map at infinity, the canonical choice of gauge is clear; parallel-transporting this gauge choice back along the harmonic map flow, we obtain a (in some sense, canonical) gauge choice for the original map, which has been named the \emph{caloric gauge} by Tao. 

We remark that in comparison to the works involving the caloric gauge, the analytic side of this work is simpler. One reason is that we are working with a sub-critical problem, and hence the function spaces used for treating the hyperbolic equations involved are by far less intricate. Another is that, as indicated earlier, our method depends only on the short time smoothing property of the associated heat flow, and as such, does not require understanding the long time behavior of the heat flow as in the other works. 

\subsection{The Yang-Mills heat flow} 
Before we give an overview of our proof of the Main Theorem, let us introduce the \emph{Yang-Mills heat flow}, which is a crucial ingredient of the new approach.

Consider an one parameter family of spatial 1-forms $A_{i}(s)$ on $\bbR^{3}$, parameterized by $s \in [0,s_{0}]$. We say that $A_{i}(s)$ is a \emph{Yang-Mills heat flow} if it satisfies
\begin{equation} \label{eq:YMHF} \tag{YMHF}
	\rd_{s} A_{i} = \covD^{\ell} F_{\ell i}, \quad i=1,2,3.
\end{equation}

The Yang-Mills heat flow is the gradient flow for the \emph{Yang-Mills energy} (or the \emph{magnetic energy}) on $\bbR^{3}$, which is defined as
\begin{equation*}
	\bfB[A_{i}] := \frac{1}{2} \sum_{1 \leq i < j \leq 3} \int (F_{ij}, F_{ij}) \, \ud x.
\end{equation*}

First introduced by Donaldson \cite{Donaldson:1985vh}, the Yang-Mills heat flow has been a subject of active research on its own. For more on this heat flow on a 3-dimensional manifold, we refer the reader to \cite{Rade:1992tu}, \cite{Charalambous:2010vt} and etc.

As indicated earlier, our intention is to use \eqref{eq:YMHF} to smooth out \eqref{eq:hyperbolicYM}. We must take care, however, since \eqref{eq:YMHF} turns out to be \emph{not} strictly parabolic for $A_{i}$, as the highest order terms of the right-hand side of \eqref{eq:YMHF} has a non-trivial kernel. This phenomenon ultimately originates from the covariance of the term $\covD^{\ell} F_{\ell i}$, and can be compensated if gauge transforms which depends on $s$ are used. However, \eqref{eq:YMHF}, as it stands, is not covariant under such gauge transforms (being covariant only under $s$-independent gauge transforms). Therefore, for the purpose of recovering strict parabolicity of the Yang-Mills heat flow, it is useful to reformulate the flow in a fully covariant form.

Along with $A_{i}$, let us add a new connection component $A_{s}$, and consider $A_{a}$ $(a=x^{1}, x^{2}, x^{3}, s)$, which is a connection 1-form on the product manifold $\bbR^{3} \times [0,s_{0}]$. Corresponding to $A_{s}$, we also define the \emph{covariant derivative} along the $\rd_{s}$ direction
\begin{equation*}
	\covD_{s} := \rd_{s} + \LieBr{A_{s}}{\cdot}.
\end{equation*}

A connection 1-form on $\bbR^{3} \times [0, s_{0}]$ is said to be a \emph{covariant Yang-Mills heat flow} if it satisfies
\begin{equation} \label{eq:cYMHF} \tag{cYMHF}
	F_{si} = \covD^{\ell} F_{\ell i}, \quad i=1,2,3,
\end{equation}
where $F_{si}$ is the commutator between $\covD_{s}$ and $\covD_{i}$, given by the formula
\begin{equation} \label{eq:intro:Fsi}
	F_{si} = \rd_{s} A_{i} -\rd_{i} A_{s} + \LieBr{A_{s}}{A_{i}}.
\end{equation}

As the system \eqref{eq:cYMHF} is underdetermined for $A_{a}$, we need an additional gauge condition (typically for $A_{s}$) in order to solve for $A_{a}$. Choosing $A_{s} = 0$, we recover \eqref{eq:YMHF}. On the other hand, if we choose $A_{s} = \rd^{\ell} A_{\ell}$, then \eqref{eq:cYMHF} becomes strictly parabolic\footnote{Indeed, up to the top order terms, it is easy to verify that the system looks like $\rd_{s} A_{i} = \lap A_{i} + \hbox{(lower order terms)}$.}. This may be viewed as a geometric formulation of the `compensation-by-gauge-transform' procedure hinted earlier.

In what follows, the first gauge condition $A_{s} = 0$ will be called the \emph{caloric gauge}, following the usage of the term in \cite{Tao:2004tm}. On the other hand, the second gauge condition $A_{s} = \rd^{\ell} A_{\ell}$ will be dubbed the \emph{DeTurck gauge}, as the idea of compensating for a non-trivial kernel by a suitable one parameter family of gauge transforms, which lies at the heart of the procedure outlined above, goes under the name \emph{DeTurck's trick}\footnote{It has been first introduced by DeTurck in the context of the Ricci flow in \cite{DeTurck:1983ts}, and applied in the Yang-Mills heat flow context by Donaldson \cite{Donaldson:1985vh}.}.

\subsection{Overview of the arguments}  \label{subsec:overview}
Perhaps due to the fact that we deal simultaneously with two nonlinear PDEs, namely \eqref{eq:hyperbolicYM} and \eqref{eq:YMHF}, the argument of this paper is rather lengthy. To help the reader grasp the main ideas, we would like to present an overview of the paper, with the ambition to indicate each of the major difficulties, as well as their resolutions, without getting into too much technical details. For a shorter, more leisurely overview, we refer the reader to the introduction of \cite{Oh:2012fk}.
 
In this overview, instead of the full local well-posedness statement, we will focus on the simpler problem of deriving a local-in-time {\it a priori} bound of a solution to \eqref{eq:hyperbolicYM} in the temporal gauge.
%\footnote{Once \eqref{eq:intro:overview:0} is proved, a similar argument may be used to prove a \emph{difference estimate} (i.e., an estimate for the difference between two solutions with nearby initial data sets) and persistence of higher regularity. Complemented with a higher regularity local well-posedness statement for \eqref{eq:hyperbolicYM} in the temporal gauge (See Theorem \ref{thm:mainThm:H2lwp}), this will then lead to the Main Theorem.}. 
In other words, under the assumption that a (suitably smooth and decaying) solution $\Atemp_{\mu}$ to \eqref{eq:hyperbolicYM} in the temporal gauge exists on $I \times \bbR^{3}$, where $I := (-T_{0}, T_{0}) \subset \bbR$, we aim to prove
\begin{equation} \label{eq:intro:overview:0}
	\nrm{\rd_{t,x} \Atemp_{\mu}}_{C_{t}(I, L^{2}_{x})} \leq C_{0} \calIini,
\end{equation}
where $\calIini := \sum_{i=1,2,3} \nrm{(\Aini_{i}, \Eini_{i})}_{\dot{H}^{1}_{x} \times L^{2}_{x}}$ measures the size of the initial data, for $T_{0}$ sufficiently small compared to $\calIini$. 

\pfstep{1. Scaling and set-up of the bootstrap}
Observe that, thanks to the scaling \eqref{eq:intro:scaling} and the sub-criticality of $\calIini$, it suffices to prove \eqref{eq:intro:overview:0} for $T_{0} = 1$, assuming $\calIini$ is small. We will use a bootstrap argument to establish \eqref{eq:intro:overview:0}. More precisely, under the \emph{bootstrap assumption} that
\begin{equation} \label{eq:intro:overview:1}
	\nrm{\rd_{t,x} \Atemp_{\mu}}_{C_{t}((-T, T), L^{2}_{x})} \leq 2 C_{0} \calIini
\end{equation}
holds for $0 < T \leq 1$, we will retrieve \eqref{eq:intro:overview:0} for $I = (-T, T)$ provided that $\calIini$ is sufficiently small (independent of $T$). Then, by a standard continuity argument, \eqref{eq:intro:overview:0} will follow for $I = (-1, 1)$.

\pfstep{2. Geometric smoothing of $\Atemp_{\mu}$ by the (dynamic) Yang-Mills heat flow}
As discussed earlier, the main idea of our approach is to smooth out $\Atemp_{\mu}$ by (essentially) using the covariant Yang-Mills heat flow. Let us append a new variable $s$ and extend $\Atemp_{\mu} = \Atemp_{\mu}(t,x)$ to a connection 1-form $A_{\bfa} = A_{\bfa}(t,x,s)$ ($\bfa = x^{0}, x^{1}, x^{2}, x^{3},s $) by solving
\begin{equation} \label{eq:dYMHF} \tag{dYMHF}
	F_{s \mu} = \covD^{\ell} F_{\ell \mu}, \quad \mu = 0,1,2,3.
\end{equation}
with $A_{\mu}(s=0) = \Atemp_{\mu}$. Note that this system is nothing but \eqref{eq:cYMHF} with the extra equation $F_{s0} = \covD^{\ell} F_{\ell 0}$. We will refer to this as the \emph{dynamic Yang-Mills heat flow}. 

As we would like to utilize the smoothing property of \eqref{eq:dYMHF}, we will impose the DeTurck gauge condition $A_{s} = \rd^{\ell} A_{\ell}$. Then \eqref{eq:dYMHF} essentially\footnote{More precisely, \eqref{eq:cYMHF} becomes strictly parabolic, and can be solved by Picard iteration. On the other hand, we can solve for the extra variable $A_{0}$ using $F_{s0} = \covD^{\ell} F_{\ell 0}$ {\it a posteriori}, by a process which involves solving only linear equations. For more details, we refer the reader to Section \ref{sec:pfOfIdEst}, {\it Proof of Theorem \ref{thm:idEst}, Step 1.}} becomes a strictly parabolic system, and thus can be solved (via Picard iteration) on $(-T, T) \times \bbR^{3} \times [0, 1]$ provided that $\sup_{t \in (-T, T)} \nrm{\rd_{x} A_{i}(t, s=0)}_{L^{2}_{x}}$ is small enough; see Sections \ref{sec:covYMHF} and \ref{sec:pfOfIdEst}. The latter condition can be ensured by taking $\calIini$ sufficiently small, thanks to the bootstrap assumption \eqref{eq:intro:overview:1}.

\pfstep{3. The hyperbolic-parabolic-Yang-Mills system and the caloric-temporal gauge}
As a result, we have obtained a connection 1-form $A_{\bfa} = A_{\bfa} (t,x,s)$ on $(-T, T) \times \bbR^{3} \times [0,1]$, which satisfies the following system of equations:
\begin{equation} \label{eq:HPYM} \tag{HPYM}
\left\{
\begin{aligned}
	F_{s \mu} &= \covD^{\ell} F_{\ell \mu} \hspace{.25in} \hbox{ on } \hspace{.1in} I \times \bbR^{3} \times [0,1], \\
	\covD^{\mu} F_{\mu \nu} &= 0 \hspace{.5in} \hbox{ along } I \times \bbR^{3} \times \set{0},
\end{aligned}
\right.
\end{equation}
as well as the DeTurck gauge condition $A_{s} = \rd^{\ell} A_{\ell}$. The system (without the gauge condition) just introduced will be called the \emph{hyperbolic-parabolic-Yang-Mills} or, in short, \eqref{eq:HPYM}. It is covariant under gauge transforms of the form $U = U(t,x,s)$, which act on each variable in the following fashion:
\begin{equation*}
\begin{aligned}
	\widetilde{A}_{\bfa} = U A_{\bfa} U^{-1} - \rd_{\bfa} U U^{-1}, \qquad
	\widetilde{\covD}_{\bfa} =  U \covD_{\bfa} U^{-1}, \qquad
	\widetilde{F}_{\bfa \bfb} =  U F_{\bfa \bfb} U^{-1}.
\end{aligned}
\end{equation*} 
where $\bfa, \bfb = x^{0}, x^{1}, x^{2}, x^{3}, s$.

We will work with \eqref{eq:HPYM} in place of \eqref{eq:hyperbolicYM}. Accordingly, instead of $\Atemp_{\mu}$, we will work with new variables $\Alow_{\mu} := A_{\mu}(s=1)$ and $\rd_{s} A_{\mu}(s)$ $(0 < s < 1)$. The former should be viewed as a smoothed-out version of $\Atemp_{\mu}$, whereas the latter measures the difference between $\Alow_{\mu}$ and $\Atemp_{\mu}$.

In analyzing \eqref{eq:dYMHF}, we indicated that the DeTurck gauge $A_{s} = \rd^{\ell} A_{\ell}$ is employed. This choice was advantageous in the sense that the equations for $A_{\mu}$ were parabolic in this gauge. However, completely different considerations are needed for estimating the evolution in $t$. Here, the gauge condition we will use is
\begin{equation*}
\left\{
\begin{aligned}
	&A_{s} = 0 \quad \hbox{ on } I \times \bbR^{3} \times (0,1), \\
	&\Alow_{0} = 0 \quad \hbox{ on } I \times \bbR^{3} \times \set{1}.
\end{aligned}
\right.
\end{equation*}
which we dub the \emph{caloric-temporal} gauge. 

Let us briefly motivate our choice of gauge. For $\rd_{s} A_{\mu}$ on $(-T, T) \times \bbR^{3} \times (0, 1)$, let us begin by considering the following identity, which is nothing but a rearrangement of the formula \eqref{eq:intro:Fsi}.
\begin{equation} \label{eq:intro:covHodge}
	\rd_{s} A_{i} = F_{si} + \covD_{i} A_{s}.
\end{equation}

A simple computation (see Appendix \ref{sec:HPYM}) shows that $F_{si}$ is covariant-divergence-free, i.e., $\covD^{\ell} F_{s \ell} = 0$. In view of this fact, the identity \eqref{eq:intro:covHodge} may be viewed (heuristically) as a \emph{covariant Hodge decomposition} of $\rd_{s} A_{i}$, where $F_{si}$ is the covariant-divergence-free part and $\covD_{i} A_{s}$, being a pure covariant-gradient term, may be regarded as the `covariant-curl-free part\footnote{Although its covariant curl does not strictly vanish.}'. Recall that the Coulomb gauge condition, which had a plenty of good properties as discussed earlier, is equivalent to having vanishing curl-free part. Proceeding in analogy, we are motivated to set $A_{s} = 0$ on $(-T, T) \times \bbR^{3} \times (0,1)$, which is exactly the caloric gauge condition we introduced earlier.

The second gauge condition, $\Alow_{0} = 0$, is motivated from the fact that $\Alow_{\mu}$ is expected to be \emph{smooth}. More precisely, hinted by the works \cite{Segal:1979hg}, \cite{Eardley:1982fb}, we expect that the increased degree of smoothness of $\Alow_{i}(t=0)$ will render a delicate choice of gauge (such as the Coulomb gauge) unnecessary, and that an easy choice (such as the temporal gauge $\Alow_{0} = 0$) will suffice. 

\pfstep{4. Gauge transform into the caloric-temporal gauge and the initial data estimates}
With these heuristic motivations in mind, let us come back to the problem of establishing the a priori estimate \eqref{eq:intro:overview:0}. In order to proceed, we must perform a gauge transformation on $A_{\bfa}$, which currently is in the DeTurck gauge, into the caloric-temporal gauge. An inspection of the formula for gauge transformation shows that the desired gauge transform $U$ can be found by solving the following hierarchy of ODEs:
\begin{equation} \label{eq:intro:overview:ODE4U}
\left\{
\begin{aligned}
	&\rd_{t} \underline{U} = \underline{U} \Alow_{0} 	\quad \hbox{ on } I \times \bbR^{3} \times \set{1} \\
	&\rd_{s} U = U A_{s}						 \quad \hbox{ on } I \times \bbR^{3} \times (0,1)
\end{aligned}
\right.
\end{equation}
where $U(s=1) = \underline{U}$. We will choose the initial value for \eqref{eq:intro:overview:ODE4U} to be $U(t=0, s=1) = \mathrm{Id}$\footnote{At first sight, one may think that a more natural choice of the initial value is $U(t=0, s=0) = \mathrm{Id}$, as it keeps the initial data set $\Aini_{i}, \Eini_{i}$ unchanged. However, it turns out that the gauge transform arising this way is not bounded on $H^{m}_{x}$ for $m > 1$. As such, it cannot retain the smoothing estimates for \eqref{eq:dYMHF} in the DeTurck gauge, and thus inappropriate for our purposes. The choice $U(t=0, s=1) = \mathrm{Id}$, on the other hand, avoids this issue at the cost of introducing a non-trivial gauge transform (which we call $V$) at $t=0, s=0$. See Lemma \ref{lem:est4gt2caloric} for the relevant estimates.}. Then, combined with smoothing estimates for \eqref{eq:dYMHF} in the DeTurck gauge, we arrive at a gauge transformed solution (which we still call $A_{\bfa}$) to \eqref{eq:HPYM} in the caloric-temporal gauge, which satisfies the following \emph{initial data estimates}\footnote{We remind the reader that $F_{s\mu} = \rd_{s}A_{\mu}$, thanks to the caloric-temporal gauge condition.}
\begin{equation} \label{eq:intro:overview:idEst}
\left\{
\begin{aligned}
	\sup_{0 < s < 1} s^{(m+1)/2} \nrm{\rd_{x}^{(m-1)} \rd_{t,x} F_{si}(t=0, s)}_{L^{2}_{x}}
	 \leq  C_{m} \calIini, \\
	\bb( \int_{0}^{1} s^{m+1} \nrm{\rd_{x}^{(m-1)} \rd_{t,x} F_{si}(t=0, s)}_{L^{2}_{x}}^{2} \, \frac{\ud s}{s} \bb)^{1/2}
	 \leq  C_{m} \calIini, \\
	\nrm{\rd_{x}^{(k-1)} \rd_{t,x} \Alow_{i}(t=0)}_{L^{2}_{x}} 
	 \leq  C_{k} \calIini,
\end{aligned}
\right.
\end{equation}
up to some integers $m_{0}, k_{0} > 1$, i.e., $1 \leq m \leq m_{0}$, $1 \leq k \leq k_{0}$. Moreover, we obtain estimates for the gauge transform $V := U(t=0, s=0)$ as well. The weights of $s$ are dictated by scaling  (see \S \ref{subsec:assocWght} for a more detailed explanation). 

The result described in this step is essentially the content of Theorem \ref{thm:idEst}, which is stated in Section \ref{sec:mainThm} and proved in Section \ref{sec:pfOfIdEst}.  

\pfstep{5. Equations of motion of \eqref{eq:HPYM}}
The next step is to propagate the bounds \eqref{eq:intro:overview:idEst} to all $t \in (-T, T)$ by analyzing a system of coupled hyperbolic and parabolic equations derived from \eqref{eq:HPYM}\footnote{The system is, roughly speaking, parabolic in the $s$-direction and hyperbolic in the $t$-direction. Moreover, all the equations we present are covariant.}. To present this system, let us begin by introducing the notion of the \emph{Yang-Mills tension field}. For a solution $A_{\bfa}$ to \eqref{eq:HPYM} on $(-T, T) \times \bbR^{3} \times [0,1]$, we define its Yang-Mills tension field $w_{\nu}(s)$ at $s \in [0,1]$ by
\begin{equation*} 
	w_{\nu}(s) := \covD^{\mu} F_{\nu \mu}(s).
\end{equation*}

The Yang-Mills tension field $w_{\nu}(s)$ measures the extent to which $A_{\mu}(s)$ fails to satisfy the Yang-Mills equations \eqref{eq:hyperbolicYM}. With $w_{\nu}$ in hand, we may now state the \emph{equations of motion} of \eqref{eq:HPYM}, which are central to the analysis of the $t$-evolution of $A_{\bfa}$.
\begin{align}
	 \covD^\mu \covD_\mu F_{s \nu} 
 	= & 2 \LieBr{\tensor{F}{_s^\mu}}{F_{\nu \mu}} - 2 \LieBr{F^{\mu \ell}}{\covD_\mu F_{\nu \ell} + \covD_\ell F_{\nu \mu}} 
		- \covD^\ell \covD_\ell w_\nu + \covD_\nu \covD^\ell w_\ell - 2 \LieBr{\tensor{F}{_\nu^\ell}}{w_\ell}, \label{eq:hyperbolic4F} \\
	\covDlow^{\mu}\Flow_{\nu \mu} = & \wlow_{\nu}, \label{eq:hyperbolic4Alow} \\
	\covD_s w_\nu 
	= & \covD^\ell \covD_\ell w_\nu + 2 \LieBr{\tensor{F}{_\nu^\ell}}{w_\ell} + 2 \LieBr{F^{\mu \ell}}{\covD_{\mu} F_{\nu \ell} + \covD_{\ell} F_{\nu \mu}}, 
	\label{eq:covParabolic4w} \\
 	\covD_s F_{\bfa \bfb} =& \covD^\ell \covD_\ell F_{\bfa \bfb} - 2\LieBr{\tensor{F}{_\bfa^\ell}}{F_{\bfb \ell}}. \label{eq:covParabolic4Fab} 
\end{align}

The underlines of \eqref{eq:hyperbolic4Alow} signify that each variable is restricted to $\set{s=1}$, and the indices $\bfa, \bfb$ run over $x^{0}, x^{1}, x^{2}, x^{3}, s$. Furthermore, $w_{\nu} \equiv 0$ at $s=0$, for all $\nu = 0,1,2,3$. The derivation of these equations can be found in Appendix \ref{sec:HPYM}.

The equations \eqref{eq:hyperbolic4F} and \eqref{eq:hyperbolic4Alow} are the main hyperbolic equations of the system, used to estimate $F_{si}$ and $\Alow_{i}$, respectively. Both equations possess terms involving $w_{\mu}$ on the right-hand side. The Yang-Mills tension field $w_{\mu}$, in turn, is estimated by studying the parabolic equation \eqref{eq:covParabolic4w}. An important point regarding \eqref{eq:covParabolic4w} is that its data at $s=0$ is \emph{zero}, thanks to the fact that $A_{\mu}(s=0)$ satisfies \eqref{eq:hyperbolicYM}. 
 
Next, the equation \eqref{eq:covParabolic4Fab} says that each curvature component satisfies a covariant parabolic equation. In view of proving the Main Theorem, of particular interest are the equations
\begin{align} 
	\covD_s F_{si} - \covD^\ell \covD_\ell F_{si} = & - 2 \LieBr{\tensor{F}{_s^\ell}}{F_{i \ell}}, \label{eq:covParabolic4Fsi}\\
	\covD_s F_{s0} - \covD^\ell \covD_\ell F_{s0} = & - 2 \LieBr{\tensor{F}{_s^\ell}}{F_{0 \ell}}. \label{eq:covParabolic4w0}
\end{align}

Thanks to the smoothing property of \eqref{eq:covParabolic4Fsi}, we may (at least heuristically) always exchange derivatives of $F_{si}$ for an appropriate power of $s$. The second equation \eqref{eq:covParabolic4w0} will be used to derive estimates for $F_{s0}$, which, combined with the caloric-temporal gauge condition, leads to the corresponding estimates for $A_{0}$. As $F_{s0} = - w_{0}$, note that the data for \eqref{eq:covParabolic4w0} at $s=0$ is zero as well. This has the implication that $A_{0}$ is, in general, obeys more favorable estimates than $A_{i}$.

\pfstep{6. Analysis of the time evolution}
We are now ready to present the key ideas for analyzing the hyperbolic equations of \eqref{eq:HPYM}, namely \eqref{eq:hyperbolic4F} and \eqref{eq:hyperbolic4Alow}; this will be the content of Theorem \ref{thm:dynEst}, stated in Section \ref{sec:mainThm} and proved in Sections \ref{sec:redOfDynEst} -- \ref{sec:wave}. 

In order to treat \eqref{eq:hyperbolic4F}, we need to uncover the aforementioned null structure of the most dangerous quadratic nonlinearity. It turns out that, for the problem under consideration, all quadratic nonlinearities can be treated just by Strichartz and Sobolev inequalities, except for the single term
\begin{equation*}
	2 \LieBr{A^{\ell} - \Alow^{\ell}}{\rd_{\ell} F_{si}}.
\end{equation*}

In \cite{Klainerman:1994jb}, it was demonstrated that such a term can be written as a linear combination of null forms, provided that $A_{i}-\Alow_{i}$ satisfied the Coulomb condition $\rd^{\ell} (A_{\ell}-\Alow_{\ell}) = 0$. Of course, this assumption is not true in our case; nevertheless, combining $\rd_{s} A_{i} = F_{si}$ (from the caloric condition $A_{s} = 0$) and the identity $\covD^{\ell} F_{s \ell} = 0$, we see that the covariant Coulomb condition $\covD^{\ell} (\rd_{s}A_{\ell}) = 0$ is satisfied for each $\rd_{s} A_{i}(s)$, $s \in (0,1)$. This turns out to be sufficient for carrying out an argument similar to \cite{Klainerman:1994jb}. We refer the reader to \S \ref{subsec:FsWave} for more details.

On the other hand, the key point regarding \eqref{eq:hyperbolic4Alow}, which is nothing but the Yang-Mills equations in the temporal gauge with the source $\wlow_{\mu}$, is that its data at $t=0$ is smooth. Therefore, we will basically emulate the classical analysis of \eqref{eq:hyperbolicYM} in the temporal gauge for initial data with higher degree of smoothness. See \S \ref{subsec:AlowWave} for more details.

Provided that $\calIini$ is sufficiently small, the analysis sketched above leads to estimates for $F_{si}(s)$ and $\Alow_{i}$, such as
\begin{equation} \label{eq:intro:overview:dynEst}
\left\{
\begin{aligned}
	\sup_{t \in I} \sup_{0 < s < 1} s^{(m+1)/2} \nrm{\rd_{x}^{(m-1)} \rd_{t,x} F_{si}(t, s)}_{L^{2}_{x}}
	 \leq  C_{m} \calIini, \\
	\sup_{t \in I} \bb( \int_{0}^{1} s^{(m+1)} \nrm{\rd_{x}^{(m-1)} \rd_{t,x} F_{si}(t, s)}_{L^{2}_{x}}^{2} \, \frac{\ud s}{s} \bb)^{1/2}
	 \leq  C_{m} \calIini,\\
	\nrm{\rd_{x}^{(k-1)} \rd_{t,x} \Alow_{i}}_{C_{t} (I, L^{2}_{x})} 
	 \leq  C_{k} \calIini,
\end{aligned}
\right.
\end{equation}
for $1 \leq m \leq m_{0}$, $1 \leq k \leq k_{0}$.  

\pfstep{7. Returning to $\Atemp_{\mu}$}
The last step is to translate estimates for $\rd_{s} A_{i}$ and $\Alow_{i}$, such as \eqref{eq:intro:overview:dynEst}, to those for $\Atemp_{\mu}$ so that \eqref{eq:intro:overview:0} is retrieved. One immediate issue is that the naive approach of integrating the estimates \eqref{eq:intro:overview:0} in $s$ fails to bound $\nrm{\rd_{t,x} A_{\mu}(s=0)}_{C_{t} (I, L^{2}_{x})}$ by a logarithm. In order to remedy this issue, we take the (weakly-parabolic) equation
\begin{equation*}
	\rd_{s} A_{i} = \lap A_{i} - \rd^{\ell} \rd_{i} A_{\ell} + (\hbox{lower order terms}).
\end{equation*}
differentiate by $\rd_{t,x}$, multiply by $\rd_{t,x} A_{i}$ and then integrate the highest order terms by parts over $\bbR^{3} \times [0, 1]$. This trick, combined with the $L^{2}_{\ud s/s}$-type estimates of \eqref{eq:intro:overview:dynEst}, overcome the logarithmic divergence\footnote{It turns out that such a trick is already needed at the stage of deriving estimates such as \eqref{eq:intro:overview:dynEst}; see Proposition \ref{prop:est4ai}.}.

Another issue is that the estimates derived so far, being in the caloric-temporal gauge, are not in the temporal gauge along $s=0$. Therefore, we are required to control the gauge transform back to the temporal gauge along $s=0$, for which appropriate estimates for $A_{0}(s=0)$ in the caloric-temporal gauge are needed; see Lemma \ref{lem:est4gt2temporal}. These are obtained ultimately as a consequence of the analysis of the hyperbolic equations of \eqref{eq:HPYM}; see Proposition \ref{prop:est4a0}. 

\subsection{Outline of the paper} 
After establishing notations and conventions in Section \ref{sec:notations}, we gather some preliminary results concerning the linear wave and the heat (or parabolic) equations in Section \ref{sec:prelim}. In particular, for the parabolic equation, we develop what we call the \emph{abstract parabolic theory}, which allows us to handle various parabolic equations with a unified approach.  

We embark on the proof of the Main Theorem in Section \ref{sec:mainThm}, where the Main Theorem is reduced to two smaller statements, namely Theorems \ref{thm:idEst} and \ref{thm:dynEst}, both of which are concerned with \eqref{eq:HPYM}. Theorem \ref{thm:idEst}  roughly addresses Points 2 and 4 in \S \ref{subsec:overview}. More precisely, it starts with a solution $\Atemp_{\mu}$ to \eqref{eq:hyperbolicYM} in the temporal gauge, and asserts the existence of its extension $A_{\bfa}$ as a solution to \eqref{eq:HPYM} in the caloric-temporal gauge, along with appropriate initial data estimates at $t=0$. The DeTurck gauge is used in an essential way in the proof. On the other hand, Theorem \ref{thm:dynEst} presents the result of a local-in-time analysis of \eqref{eq:HPYM} in the caloric-temporal gauge, corresponding to Point 6 in \S \ref{subsec:overview}.

The aim of the next two sections is to prove Theorem \ref{thm:idEst}. In Section \ref{sec:covYMHF}, we study the covariant Yang-Mills heat flow \eqref{eq:cYMHF}. In particular, various smoothing estimates for the connection 1-form $A_{i}$ will be derived in the DeTurck gauge (in \S \ref{subsec:covYMHF}), and estimates for the gauge transform from the DeTurck gauge to the caloric gauge is presented as well (in \S \ref{subsec:covYMHFgt}). As a byproduct of the analysis, we obtain a proof of local existence of a solution to \eqref{eq:YMHF} (Theorem \ref{thm:lwp4YMHF}), which is used in \cite{Oh:2012fk}. We remark that the proof is independent of the original one in \cite{Rade:1992tu}. Next, based on the results proved in Section \ref{sec:covYMHF}, a proof of Theorem \ref{thm:idEst} is given in Section \ref{sec:pfOfIdEst}. 
 
The remainder of the paper is devoted to a proof of Theorem \ref{thm:dynEst}. We begin in Section \ref{sec:redOfDynEst} by reducing Theorem \ref{thm:dynEst} to several smaller statements, namely Propositions \ref{prop:est4a0} - \ref{prop:cont4FA} and Theorems \ref{thm:AlowWave} and \ref{thm:FsWave}, where the latter two theorems concern estimates for the hyperbolic equations \eqref{eq:hyperbolic4Alow} and \eqref{eq:hyperbolic4F}, respectively. Section \ref{sec:pEst4HPYM} is where we derive estimates for solutions to various parabolic equations, and forms the `parabolic' heart of the paper. It is here that our efforts for developing the abstract parabolic theory amply pays off. Equipped with the results from the previous section, we prove Propositions \ref{prop:est4a0} - \ref{prop:cont4FA} in Section \ref{sec:pfOfProps}. The following section, namely Section \ref{sec:wave}, is where we finally study the wave equations for $\Alow_{i}$ and $F_{si}$. Combined with the parabolic estimates from Section \ref{sec:pEst4HPYM}, we establish Theorems \ref{thm:AlowWave} and \ref{thm:FsWave}. 

In Appendix \ref{sec:HPYM}, we give a derivation of various covariant equations from \eqref{eq:HPYM}. Then finally, in Appendix \ref{sec:gt}, we prove estimates for gauge transforms that are deferred in the main body of the paper.

\subsection*{Acknowledgements}
The author is deeply indebted to his Ph.D. advisor Sergiu Klainerman, without whose support and constructive criticisms this work would not have been possible. He would also like to thank l'ENS d'Ulm for hospitality, where a major part of this work was done. The author was supported by the Samsung Scholarship.

\section{Notations and Conventions} \label{sec:notations}
%\subsection{Geometric notations and conventions}
%Throughout the paper, we will use the index notation to denote tensorial quantities, such as $A_{\bfa}, F_{\bfa \bfb}, w_{\mu}$ and so on. We remark, however, that the tensorial property will not be used seriously as we always work in the coordinate system $(x^{0},x^{1}, x^{2}, x^{3}, x^{4}) =(t, x^{1}, x^{2}, x^{3}, s)$. We will use boldfaced latin indices (e.g. $\bfa, \bfb, \ldots$) for all coordinates $(t,x^{1}, x^{2}, x^{3}, s)$, greek indices (e.g. $\mu, \nu, \ldots$) for the spatio-temporal coordinates $(t,x^{1}, x^{2}, x^{3})$ and plain latin indices (e.g. $i,j,k, \ell, \ldots$) for the spatial coordinates $(x^{1},x^{2}, x^{3}$).  
%
%We will raise and lower indices using the Minkowski metric, and adopt the summation convention of Einstein, according to which all repeated upper and lower indices are to be summed. (e.g. $\rd^{\ell} A_{\ell}$ for $\sum_{\ell=1,2,3} \rd^{\ell} A_{\ell}$.)
%
\subsection{Schematic notations and conventions}
%\comment{Discuss the $\calO$ notation, and the convention of writing $A, F_{s}, w$ for $A_{i}, F_{si}, w_{i}$ respectively.}
We will often omit the spatial index $i$; that is, we will write $A, F_{s}, w$ as the shorthands for $A_{i}, F_{si}, w_{i}$, respectively, and so on. A norm of such an expression, such as $\nrm{A}$, is to be understood as the maximum over $i=1,2,3$. (i.e., $\nrm{A} = \sup_{i} \nrm{A_{i}}$ and etc.) 

We will use the notation $\calO(\phi_{1}, \ldots, \phi_{k})$ to denote a $k$-linear expression in the \emph{values} of $\phi_{1}, \ldots, \phi_{k}$. For example, when $\phi_{i}$ and the expression itself are scalar-valued, then $\calO(\phi_{1}, \ldots, \phi_{k}) = C \phi_{1} \phi_{2} \cdots \phi_{k}$ for some constant $C$. In many cases, however, each $\phi_{i}$ and the expression $\calO(\phi_{1}, \ldots, \phi_{k})$ will actually be matrix-valued. In such case, $\calO(\phi_{1}, \ldots, \phi_{k})$ will be a matrix, whose each entry is a $k$-linear functional of the matrices $\phi_{i}$. 

When stating various estimates, we adopt the standard convention of denoting by the same letter $C$ positive constants which are different, possibly line to line. Dependence of $C$ on other parameters will be made explicit by subscripts. Furthermore, we will adopt the convention that \emph{$C$ always depends on each of its parameters in a non-decreasing manner, in its respective range}, unless otherwise specified. For example, $C_{\calE, \calF}$, where $\calE, \calF$ range over positive real numbers, is a positive, non-decreasing function of both $\calE$ and  $\calF$.

\subsection{Notations and conventions for the estimates for differences}
In the paper, along with estimating a single solution $A_{\bfa}$ of the Yang-Mills equation, we will also be estimating the difference of two nearby solutions. We will refer to various variables arising from the other solution by putting a prime, e.g. $A'_{\bfa}$, $F'_{s\mu}$, $w'_{i}$ and etc, and the corresponding differences will be written with a $\dlt$, i.e., $\dlt A_{\bfa} := A_{\bfa} - A'_{\bfa}$, $\dlt F_{s \mu} := F_{s \mu} - F'_{s \mu}$ and $\dlt w_{i} = w_{i} - w'_{i}$ and etc.  

We will also use equations for differences, which are derived by taking the difference of the equations for the original variables. In writing such equations schematically using the $\calO$-notation, we will not distinguish between primed and unprimed variables. For example, the expression $\calO(A, \rd_{x} (\dlt A))$ refers to a sum of bilinear expressions, of which the first factor could be any of $A_{i}, A'_{i}$ ($i=1,2,3$), and the second is one of $\rd_{i} (\dlt A_{j})$ ($i, j = 1,2,3$).

%\comment{Discuss the formal Leibniz's rule for $\dlt$}
For this purpose, the following rule, which we call the \emph{formal Leibniz's rule for $\dlt$}, is quite useful :
\begin{equation*}
	\dlt \calO(\phi_{1}, \phi_{2}, \ldots, \phi_{k}) = \calO( \dlt \phi_{1}, \phi_{2}, \ldots, \phi_{k}) + \calO( \phi_{1}, \dlt \phi_{2}, \ldots, \phi_{k}) + \cdots + \calO(\phi_{1}, \phi_{2} , \ldots, \dlt \phi_{k}).
\end{equation*}

\subsection{Small parameters} The following small parameters will be used in this paper.
\begin{equation*}
	0 < \dlt_{H} \ll \dlt_{E} \ll \dlt_{P} \ll \dlt_{A} \ll 1.
\end{equation*}

In many parts of the argument, we will need an auxiliary small parameter, which may be fixed within that part; for such parameters, we will reserve the letter $\eps$, and its variants thereof.

%\subsection{Index of norms} For the convenience of the reader, a list of norms used in the paper, along with an indication of where they are defined, is compiled below.
%\comment{Add table!}

\section{Preliminaries} \label{sec:prelim}
The aim of this section is to gather basic inequalities and preliminary results concerning the linear wave and parabolic equations, which will be rudimentary for our analysis to follow. We also develop what we call the \emph{abstract parabolic theory}, which is a book-keeping scheme allowing for a unified approach to the diversity of parabolic equations to arise below. In the end, a short discussion is given on the notion of the \emph{associated $s$-weights}, which is a useful heuristic for figuring out the appropriate weight of $s$ in various instances.

\subsection{Basic inequalities}
We collect here some basic inequalities that will be frequently used throughout the paper. Let us begin with some inequalities involving Sobolev norms for $\phi \in \calS_{x}$, where $\calS_{x}$ refers to the space of Schwartz functions on $\bbR^{3}$. 

\begin{lemma}[Inequalities for Sobolev norms] \label{lem:prelim:sob}
Let $\phi \in \calS_{x}$. The following statements hold.
\begin{itemize}
\item {\bf (Sobolev inequality)} For $1 \leq q \leq r$, $k \geq 0$ such that $\frac{3}{q} = \frac{3}{r} - k$, we have
\begin{equation} \label{eq:prelim:sob:1}
	\nrm{\phi}_{L^{q}_{x}} \leq C \nrm{\phi}_{\dot{W}^{k,r}_{x}},
\end{equation} 
where $\dot{W}^{k,r}_{x}$ is the $L^{r}$-based homogeneous Sobolev norm of order $k$.

\item {\bf (Interpolation inequality)} For $1 < q < \infty$, $k_{1} \leq k_{0} \leq k_{2}$, $0 < \tht_{1}, \tht_{2} < 1$ such that $\tht_{1} + \tht_{2} = 1$ and $k_{0} = \tht_{1} k_{1} + \tht_{2} k_{2}$, we have
\begin{equation} \label{eq:prelim:sob:2}
	\nrm{\phi}_{\dot{W}^{k_{0}, q}_{x}} \leq C \nrm{\phi}_{\dot{W}^{k_{1}, q}_{x}}^{\tht_{1}} \nrm{\phi}_{\dot{W}^{k_{2}, q}_{x}}^{\tht_{2}}.
\end{equation}

\item {\bf (Gagliardo-Nirenberg inequality)} For $1 \leq q_{1}, q_{2}, r \leq \infty$ and $0 < \tht_{1}, \tht_{2} < 1$ such that $\tht_{1} + \tht_{2} = 1$ and $\frac{3}{r} = \tht_{1} \cdot \frac{3}{q_{1}} + \tht_{2} (\frac{3}{q_{2}} - 1)$, we have
\begin{equation} \label{eq:prelim:sob:3}
	\nrm{\phi}_{L^{r}_{x}} \leq C \nrm{\phi}_{L^{q_{1}}_{x}}^{\tht_{1}} \nrm{\rd_{x} \phi}_{L^{q_{2}}_{x}}^{\tht_{2}}.
\end{equation}
\end{itemize}
\end{lemma}
\begin{proof} 
These inequalities are standard; we refer the reader to \cite{MR2424078}. \qedhere
\end{proof}

\begin{lemma}[Product estimates for homogeneous Sobolev norms] \label{lem:homSob}
For a triple $(\gmm_{0}, \gmm_{1}, \gmm_{2})$ of real numbers satisfying
\begin{equation} \label{eq:homSob:hyp}
	\gmm_{0} + \gmm_{1} + \gmm_{2} = 3/2, \quad \gmm_{0} + \gmm_{1} + \gmm_{2} > \max( \gmm_{0}, \gmm_{1}, \gmm_{2} )
\end{equation}
there exists $C > 0$ such that the following product estimate holds for $\phi_{1}, \phi_{2} \in \calS_{x}$:
\begin{equation} \label{eq:homSob:0}
	\nrm{\phi_{1} \phi_{2}}_{\dot{H}^{-\gmm_{0}}_{x}} \leq C \nrm{\phi_{1}}_{\dot{H}^{\gmm_{1}}_{x}} \nrm{\phi_{2}}_{\dot{H}^{\gmm_{2}}_{x}}.
\end{equation}
\end{lemma}
\begin{proof} 
It is a standard result (see, for example, \cite{DAncona:2012ke}) that given a triple $(\gmm_{0}, \gmm_{1}, \gmm_{2})$ of real numbers satisfying \eqref{eq:homSob:hyp}, the following \emph{inhomogeneous} Sobolev product estimate holds for $\phi_{1}, \phi_{2} \in \calS_{x}$:
\begin{equation*}
	\nrm{\phi_{1} \phi_{2}}_{H^{-\gmm_{0}}_{x}} \leq C \nrm{\phi_{1}}_{H^{\gmm_{1}}_{x}} \nrm{\phi_{2}}_{H^{\gmm_{2}}_{x}}.
\end{equation*}

Thanks to the condition $\gmm_{0} + \gmm_{1} + \gmm_{2} = 3/2$, the above estimate implies the homogeneous estimate \eqref{eq:homSob:0} by scaling. \qedhere
\end{proof}

Next, we state Gronwall's inequality, which will be useful in several places below.
\begin{lemma}[Gronwall's inequality] \label{lem:prelim:Gronwall}
Let $(s_{0}, s_{1}) \subset \bbR$ be an interval, $D \geq 0$, and $f(s)$, $r(s)$ non-negative measurable functions on $J$. Suppose that for all $s \in (s_{0}, s_{1})$, the following inequality holds:
\begin{equation*}
	\sup_{\sbr \in (s_{0}, s]} f(\sbr) \leq \int_{s_{0}}^{s} r(\sbr) f(\sbr) \, \ud \sbr  + D.
\end{equation*}

Then for all $s \in (s_{0}, s_{1})$, we have
\begin{equation*}
	\sup_{\sbr \in (s_{0}, s]} f(\sbr) \leq D \exp \bb( \int_{s_{0}}^{s} r(\sbr) \, \ud \sbr \bb).
\end{equation*}
\end{lemma}
\begin{proof} 
See \cite[Lemma 3.3]{MR2455195}. \qedhere
\end{proof}

In the course of the paper, we will often perform integration-by-parts arguments which require the functions involved to be sufficiently smooth and decaying sufficiently fast towards the spatial infinity (the latter assumption is used to show that the boundary term which may arise vanishes at infinity). Usually, this issue is usually dealt with by working with Schwartz functions to justify the arguments and then passing to the appropriate limit in the end. In our case, however, the nature of the Yang-Mills equation does not allow us to do so (in particular due to the elliptic constraint equation \eqref{eq:YMconstraint}). Instead, we formulate the notion of \emph{regular} functions, which is weaker than the Schwartz assumption but nevertheless strong enough for our purposes.
 
\begin{definition}[Regular functions] \label{def:regFtns}
Let $I \subset \bbR$, $J \subset [0, \infty)$ be intervals.
\begin{enumerate}
\item A function $\phi = \phi(x)$ defined on $\bbR^{3}$ is \emph{regular} if $\phi \in H^{\infty}_{x} := \cap_{m=0}^{\infty} H^{m}_{x}$.
\item A function $\phi = \phi(t, x)$ defined on $I \times \bbR^{3}$ is \emph{regular} if $\phi \in C^{\infty}_{t}(I, H^{\infty}_{x}) := \cap_{k,m=0}^{\infty} C^{k}_{t}(I, H^{m}_{x})$.
\item A function $\psi = \psi(t, x,s)$ defined on $I \times \bbR^{3} \times J$ is \emph{regular} if $\phi \in C^{\infty}_{t,s}(I \times J, H^{\infty}_{x}):= \cap_{k,m=0}^{\infty} C^{k}_{t,s}(I \times J, H^{m}_{x})$.
\end{enumerate}
\end{definition}

In particular, a regular function is always smooth on its domain. Moreover, Lemmas \ref{lem:prelim:sob} and \ref{lem:homSob} still hold for regular functions, by approximation.

\subsection{Estimates for the linear wave equation and the space $\SH^{k}$} \label{subsec:wavePrelim}
%Here we state, without proofs, the estimates for solutions to an inhomogeneous wave equation which we will use, and define the space $\SH^{k}$.
%
%For the purpose of this section, let $\psi, \varphi$ be smooth solutions with a suitable decay towards the spatial infinity to the inhomogeneous wave equations
%\begin{equation*}
%	\Box \psi = \calN, \qquad \Box \varphi = \calM,
%\end{equation*}
%on $[-T,T] \times \bbR^3$ with initial data $(\psi(0), \rd_t \psi(0)) = (\psi_0, \widetilde{\psi}_0)$, $(\varphi(0), \rd_t \varphi(0)) = (\varphi_0, \widetilde{\varphi}_0)$.
We summarize the estimates for solutions to an inhomogeneous wave equation that will be used in the following proposition.
\begin{proposition}[Wave estimates] \label{prop:prelim:est4wave}
Let $\psi, \varphi$ be smooth solutions with a suitable decay towards the spatial infinity (say $\psi, \varphi \in C^{\infty}_{t} \calS_{x}$) to the inhomogeneous wave equations
\begin{equation*}
	\Box \psi = \calN, \qquad \Box \varphi = \calM,
\end{equation*}
on $(-T,T) \times \bbR^3$. The following estimates hold.
\begin{itemize}
\item {\bf ($L^{\infty}_{t} L^{2}_{x}$ estimate)}
\begin{equation} \label{eq:prelim:est4wave:energy}
	\nrm{\rd_{t,x} \psi}_{L^\infty_t L^2_x ((-T, T) \times \bbR^3)} \leq C \left( \nrm{(\psi, \rd_{0} \psi)(t=0)}_{\dot{H}^1_x \times L^2_x(\bbR^{3})} + \nrm{\calN}_{L^1_t L^2_x ((-T, T) \times \bbR^3)} \right) 
\end{equation}
\item {\bf ($L^{4}_{t,x}$-Strichartz estimate)}
\begin{equation} \label{eq:prelim:est4wave:Strichartz}
	\nrm{\rd_{t,x} \psi}_{L^4_{t,x}((-T, T) \times \bbR^{3})} \leq C \left( \nrm{(\psi, \rd_{0} \psi)(t=0)}_{\dot{H}^{3/2}_x \times \dot{H}^{1/2}_x (\bbR^{3})} + \nrm{\calN}_{L^1_t \dot{H}^{1/2}_x ((-T, T) \times \bbR^3)} \right).
\end{equation}
\item {\bf (Null form estimate)} For $Q_{ij} (\psi, \phi) := \rd_{i} \psi \rd_{j} \phi - \rd_{j} \psi \rd_{i} \phi$, we have
\begin{equation} \label{eq:prelim:est4wave:nullform}
\begin{aligned}
\nrm{Q_{ij}(\psi, \phi)}_{L^2_{t,x}((-T, T) \times \bbR^{3}} \leq & \, C \left( \nrm{(\psi, \rd_{0} \psi)(t=0)}_{\dot{H}^2_x \times \dot{H}^{1}_x (\bbR^{3})} + \nrm{\calN}_{L^1_t \dot{H}^{1}_x ((-T, T) \times \bbR^3)} \right) \\
& \times C \left( \nrm{(\phi, \rd_{0} \phi)(t=0)}_{\dot{H}^1_x \times L^2_x (\bbR^{3)}} + \nrm{\calM}_{L^1_t L^2_x ((-T, T) \times \bbR^3)} \right).
\end{aligned}
\end{equation}
\end{itemize}
\end{proposition}
\begin{proof} 
This is a standard material. For the $L^{\infty}_{t} L^{2}_{x}$ and the Strichartz estimates, we refer the reader to \cite[Chapter III]{MR2455195}. For the null form estimate, see the original article \cite{Klainerman:ei}. \qedhere
\end{proof}

%\comment{Introduce the space $\SH^{k}$.}
Motivated by Proposition \ref{prop:prelim:est4wave}, let us define the norms\footnote{We remark that $\nrm{\cdot}_{\SH^{1}}$ is a norm after restricted to regular functions, by Sobolev.}  $\SH^{k}$ which will be used as a convenient device for controlling the wave-like behavior of certain dynamic variables. Let $\psi$ be a smooth function on $I \times \bbR^{3}$ ($I \subset \bbR$) which decays sufficiently towards the spatial infinity. We start with the norm $\SH^{1}$, which we define by
\begin{equation}
	\nrm{\psi}_{\SH^{1}(I)} := \nrm{\rd_{t,x} \psi}_{L^{\infty}_{t} L^{2}_{x}} + \abs{I}^{1/2} \nrm{\Box \psi}_{L^{2}_{t,x}}.
\end{equation}

The norms $\SH^{k}$ for $k=2,3,\cdots$ are then defined by taking spatial derivatives, i.e.
\begin{equation}
	\nrm{\psi}_{\SH^{k}(I)} := \nrm{\rd_{x}^{(k-1)} \psi}_{\SH^{1}(I)},
\end{equation}
and we furthermore define $\SH^{k}$ for $k \geq 1$ a real number by using fractional derivatives. Note the interpolation property
\begin{equation}
	\nrm{\psi}_{\SH^{k+\tht}(I)} \leq C_{\tht} \nrm{\psi}_{\SH^{k}(I)}^{1-\tht} \nrm{\psi}_{\SH^{k+1}(I)}^{\tht}, \quad 0 < \tht < 1.
\end{equation}

The following estimates concerning the $\SH^{k}$-norms follow immediately from Proposition \ref{prop:prelim:est4wave} and the fact that regular functions can be approximated by functions in $C^{\infty}_{t} \calS_{x}$ with respect to each of the norms involved.

\begin{proposition} \label{prop:prelim:est4SH}
Let $k \geq 1$ be an integer and $\psi, \phi \in C^{\infty}_{t} ((-T, T), H^{\infty}_{x})$. Then the following estimates hold.
\begin{itemize}
\item {\bf ($L^{\infty}_{t} L^{2}_{x}$ estimate)} 
\begin{equation} \label{eq:prelim:est4SH:energy}
	\nrm{\rd_{x}^{(k-1)} \rd_{t,x} \psi}_{L^{\infty}_{t} L^{2}_{x}((-T, T) \times \bbR^{3})} \leq \nrm{\psi}_{\SH^{k}(-T, T)}.
\end{equation}
\item {\bf ($L^{4}_{t,x}$-Strichartz estimate)} 
\begin{equation} \label{eq:prelim:est4SH:Strichartz}
	\nrm{\rd_{x}^{(k-1)} \rd_{t,x} \psi}_{L^{4}_{t,x}((-T, T) \times \bbR^{3})} \leq C \nrm{\psi}_{\SH^{k+1/2}(-T, T)}.
\end{equation}
\item {\bf (Null form estimate)} 
\begin{equation} \label{eq:prelim:est4SH:nullform}
	\nrm{Q_{ij}(\psi, \phi)}_{L^{2}_{t,x}((-T, T) \times \bbR^{3})} \leq C \nrm{\psi}_{\SH^{2}(-T, T)} \nrm{\phi}_{\SH^{1}(-T, T)}.
\end{equation}
\end{itemize}
\end{proposition}

On the other hand, in order to control the $\SH^{k}$ norm of $\psi$, all one has to do is to estimate the d'Alembertian of $\psi$ along with the initial data. This is the content of the following proposition, which is sometimes referred to as the \emph{energy estimate} in the literature.
\begin{proposition}[Energy estimate] \label{prop:prelim:energyEst4SH}
Let $k \geq 1$ be an integer and $\psi \in C^{\infty}_{t}((-T, T), H^{\infty}_{x}(\bbR^{3}))$. Then the following estimate holds.
\begin{equation*} 
	\nrm{\psi}_{\SH^{k}(-T,T)} \leq C \bb( \nrm{(\psi, \rd_{0} \psi)(t=0)}_{\dot{H}^{k}_{x} \times \dot{H}^{k-1}_{x} (\bbR^{3})} + T^{1/2} \nrm{\Box \psi}_{L^{2}_{t,x}((-T,T)\times \bbR^{3})} \bb).
\end{equation*}
\end{proposition}
\begin{proof} 
After a standard approximation procedure, this is an immediate consequence of \eqref{eq:prelim:est4wave:energy}. \qedhere 
\end{proof}

\subsection{Parabolic-normalized (or p-normalized) norms}
The purpose of the rest of this section is to develop a theory of parabolic equations suited to our needs later on.
 
Given a function $\phi$ on $\bbR^{3}$, we consider the operation of \emph{scaling by $\lmb >0$}, defined by
\begin{equation*}
	\phi \to \phi_{\lmb} (x) := \phi(x/\lmb).
\end{equation*}

We say that a norm $\nrm{\cdot}_{X}$ is \emph{homogeneous} if it is covariant with respect to scaling, i.e., there exists a real number $\ell$ such that
\begin{equation*} 
	\nrm{\phi_{\lmb}}_{X} = \lmb^{\ell} \nrm{\phi}_{X}.
\end{equation*}

The number $\ell$ is called the \emph{degree of homogeneity} of the norm $\nrm{\cdot}_{X}$. 

Let $\phi$ be a solution to the heat equation $\rd_{s} \phi - \lap \phi = 0$ on $\bbR^{3} \times [0,\infty)$. Note that this equation `respects' the scaling $\phi_{\lmb}(x,s) := \phi(x/\lmb, s/\lmb^{2})$, in the sense that any scaled solution to the linear heat equation remains a solution. Moreover, one has \emph{smoothing estimates} of the form $\nrm{\rd_{x}^{(k)} \phi(s)}_{L^{p}_{x}} \leq s^{-\frac{k}{2} + ( \frac{3}{2p} - \frac{3}{2q} )}\nrm{\phi(0)}_{L^{q}_{x}}$ (for $q \leq p$, $k \geq 0$) which are invariant under this scaling. The norms $\nrm{\rd_{x} \cdot}_{L^{p}_{x}}$ and $\nrm{\cdot}_{L^{q}_{x}}$ are homogeneous, and the above estimate can be rewritten as
\begin{equation*}
	s^{-\ell_{1}/2} \nrm{\rd_{x} \phi(s)}_{L^{p}_{x}} \leq s^{-\ell_{2}/2} \nrm{\phi(0)}_{L^{q}_{x}}
\end{equation*}
where $\ell_{1}$, $\ell_{2}$ are the degrees of homogeneity of the norms $\nrm{\rd_{x} \cdot}_{L^{p}_{x}}$ and $\nrm{\cdot}_{L^{q}_{x}}$, respectively.

Motivated by this example, we will define the notion of \emph{parabolic-normalized}, or \emph{p-normalized}, norms and derivatives. These are designed to facilitate the analysis of parabolic equations by capturing their scaling property. 

Consider a homogeneous norm $\nrm{\cdot}_{X}$ of degree $2\ell$, which is well-defined for smooth functions $\phi$ on $\bbR^{3}$. (i.e., for every smooth $\phi$, $\nrm{\phi}_{X}$ is defined uniquely as either a non-negative real number or $\infty$.) We will define its \emph{p-normalized}  analogue $\nrm{\cdot}_{\calX(s)}$ for each $s > 0$ by
\begin{equation*} 
	\nrm{\cdot}_{\calX(s)} := s^{-\ell} \nrm{\cdot}_{X}.
\end{equation*}

We will also define the p-normalization of space-time norms. As we will be concerned with functions restricted to a time interval,  we will adjust the notion of homogeneity of norms as follows. For $I \subset \bbR$, consider a \emph{family} of norms $X(I)$ defined for functions $\phi$ defined on $I \times \bbR^{3}$. For $\lmb >0$, consider the scaling $\phi_{\lmb}(t,x) := \phi(t/\lmb, x/\lmb)$. We will say that $X(I)$ is \emph{homogeneous of degree $\ell$} if 
\begin{equation*} 
	\nrm{\phi_{\lmb}}_{X(I)} = \lmb^{\ell} \nrm{\phi}_{X(\lmb I)}.
\end{equation*}

As before, we define its \emph{p-normalized} analogue $\nrm{\cdot}_{\calX(I, s)}$ as $\nrm{\cdot}_{\calX(I, s)} := s^{-\ell} \nrm{\cdot}_{X(T)}$.

Let us furthermore define the \emph{parabolic-normalized derivative} $\nb_{\mu}(s)$ by $s^{1/2} \rd_{\mu}$. Accordingly, for $k > 0$ we define the \emph{homogeneous $k$-th derivative norm} $\nrm{\cdot}_{\dot{\calX}^{k}(s)}$ by
\begin{equation*} 
	\nrm{\cdot}_{\dot{\calX}^{k}(s)} := \nrm{\nb_{x}^{(k)}(s) \cdot}_{\calX(s)}.
\end{equation*}

We will also define the \emph{parabolic-normalized covariant derivative} $\calD_{\mu}(s) := s^{1/2} \covD_{\mu}$.

We will adopt the convention $\dot{\calX}^{0} := \calX$. For $m > 0$ an integer, we define \emph{inhomogeneous $m$-th derivative norm} $\nrm{\cdot}_{\calX^{k}(s)}$ by
\begin{equation*} 
	\nrm{\cdot}_{\dot{\calX}^{m}(s)} := \sum_{k=0}^{m} \nrm{\cdot}_{\dot{\calX}^{k}(s)}.
\end{equation*}

We will often omit the $s$-dependence of $\calX(s)$, $\dot{\calX}(s)$ and $\nb_{\mu}(s)$ by simply writing $\calX$, $\dot{\calX}$ and $\nb_{\mu}$, where the value of $s$ should be clear from the context.

\begin{example} A few examples of homogeneous norms and their p-normalized versions are in order. We will also take this opportunity to fix the notations for the p-normalized norms which will be used in the rest of the paper.
\begin{enumerate}
\item $X = L^{p}_{x}$, in which case the degree of homogeneity is $2\ell = 3/p$. We will define $\calX = \calL^{p}_{x}$ and $\dot{\calX}^{k} := \dot{\calW}^{k,p}_{x}$ as follows. 
\begin{equation*}
	\nrm{\cdot}_{\calL^{p}_{x}(s)} := s^{-3/(2p)} \nrm{\cdot}_{L^{p}_{x}}, \qquad
	\nrm{\cdot}_{\dot{\calW}^{k,p}_{x}(s)} := s^{k/2 - 3/(2p)} \nrm{\cdot}_{\dot{W}^{k, p}_{x}}.
\end{equation*}

The norm $\calX^{m} := \calW^{m,p}_{x}$ will be defined as the sum of $\dot{\calW}^{k,p}$ norms for $k=0, \ldots, m$. In the case $p =2$, we will use the notation $\dot{\calX}^{k} := \dot{\calH}^{k}_{x}$ and $\calX^{k} := \calH^{k}_{x}$.

%\item For $k \geq 1$, it will often be desirable to consider the slight variant $\widehat{H}^{k}_{x}$ of $H^{k}_{x}$ which omits the $L^{2}_{x}$-norm in its definition, i.e., $\nrm{\cdot}_{\widehat{H}^{k}_{x}} := \nrm{\rd_{x} \cdot}_{H^{k-1}_{x}}$. Its p-normalized version will be denoted by $\widehat{\calH}^{k}_{x}$, whose precise definition is as follows.
%	\begin{equation*} 
%		\nrm{\cdot}_{\widehat{\calH}^{k}_{x}} := \nrm{\nb_{x} \cdot}_{\calH^{k-1}_{x}}.
%	\end{equation*}
	
\item Consider a time interval $I \subset \bbR$. For $X = L^{q}_{t} L^{p}_{x} (I \times \bbR^{3})$, in which case $2\ell = 1/q + 3/p$, we will write 
\begin{align*}
	\nrm{\cdot}_{\calL^{q}_{t} \calL^{p}_{x}(I,s)} := & s^{-1/(2q) - 3/(2p)} \nrm{\cdot}_{L^{q}_{t} L^{p}_{x}}, \\
	\nrm{\cdot}_{\calL^{q}_{t} \dot{\calW}^{k, p}_{x}(I,s)} := & s^{k/2 -1/(2q) - 3/(2p)} \nrm{\cdot}_{L^{q}_{t} \dot{W}^{k, p}_{x}}.
\end{align*}
The norms $\calL^{q}_{t} \calW^{k,p}_{x}(I,s)$, $\calL^{q}_{t} \dot{\calH}^{k}_{x}(I,s)$ and $\calL^{q}_{t} \calH^{k}_{x}(I,s)$ are defined accordingly.

%\item In the special case of $q=p=2$, we will use the notation $U := L^{2}_{t,x}$. Its p-normalized analogues will be denoted $\calU := \calL^{2}_{t,x}$, $\dot{\calU}^{k} = \calL^{2}_{t} \dot{\calH}^{k}_{x}$ and $\calU^{k} = \calL^{2}_{t} \calH^{k}_{x}$.

\item Finally, the norm $\dot{S}^{1}$, defined in Section \ref{subsec:wavePrelim} for the purpose of hyperbolic estimates, is homogeneous of degree $2\ell = 1/2$, which is the same as $L^{\infty}_{x} \dot{H}^{1}_{x}$ (i.e., the energy). For p-normalized version of $\dot{S}^{1}$, we will use a set of notations slightly deviating from the rest in order to keep consistency with the intuition that $\nrm{\phi}_{\dot{S}^{1}}$ is at the level of $L^{\infty}_{x} \dot{H}^{1}_{x}$. Indeed, for $m, k \geq 1$ and $m$ an integer, we will write
\begin{equation*} 
	\nrm{\phi}_{\dot{\calS}^{k}} := s^{(k-1)/2 - 1/4} \nrm{\rd_{x}^{(k-1)} \phi}_{\dot{S}^{1}_{x}}, \quad
	\nrm{\phi}_{\widehat{\calS}^{m}} := \sum_{k=1}^{m} \nrm{\phi}_{\dot{\calS}^{k}}.
\end{equation*}
\end{enumerate}
\end{example}

For $f = f(s)$ a measurable function defined on an $s$-interval $J \subset (0, \infty)$, we define its p-normalized Lebesgue norm $\nrm{f}_{\calL^{p}_{s}(J)}$ by $\nrm{f}_{\calL^{p}_{s}(J)}^{p} := \int_{J} \abs{f(s)}^{p} \frac{\ud s}{s}$ for $1 \leq p < \infty$, and $\nrm{f}_{\calL^{\infty}_{s}(J)} := \nrm{f}_{L^{\infty}_{s}(J)}$. Given $\ell \geq 0$, we will define the weighted norm $\nrm{f}_{\calL^{\ell, p}_{s}}(J)$ by
\begin{equation*} 
	\nrm{f}_{\calL^{\ell,p}_{s} (J)} := \nrm{s^{\ell} f(s)}_{\calL^{p}_{s} (J)}.
\end{equation*}

Let us consider the case $J = (0, s_{0})$ or $J = (0, s_{0}]$ for some $s_{0} > 0$. For $\ell > 0$ and $1 \leq p \leq \infty$, note the obvious computation $\nrm{s^{\ell}}_{\calL^{p}_{s}(0, s_{0})} = C_{\ell, p} s_{0}^{\ell}$. Combining this with the H\"older inequality
\begin{equation*} 
	\nrm{f g}_{\calL^{p}_{s}} \leq \nrm{f}_{\calL^{p_{1}}_{s}} \nrm{g}_{\calL^{p_{2}}_{s}} \hbox{ for $\frac{1}{p} = \frac{1}{p_{1}} + \frac{1}{p_{2}}$},
\end{equation*}
we obtain the following simple lemma.

\begin{lemma} [H\"older for $\calL^{\ell,p}_{s}$] \label{lem:absP:Holder4Ls}
Let $\ell, \ell_{1}, \ell_{2} \geq 0$, $1 \leq p, p_{1}, p_{2} \leq \infty$ and $f, g$ functions on $J = (0,s_{0})$ (or $J = (0, s_{0}]$) such that $\nrm{f}_{\calL^{\ell_{1}, p_{1}}_{s}}, \nrm{g}_{\calL^{\ell_{2}, p_{2}}_{s}} < \infty$. Then we have
\begin{equation*} 
	\nrm{fg}_{\calL^{\ell,p}_{s}(J)} \leq C s_{0}^{\ell-\ell_{1}-\ell_{2}} \nrm{f}_{\calL^{\ell_{1},p_{1}}_{s}(J)} \nrm{g}_{\calL^{\ell_{2},p_{2}}_{s}(J)}
\end{equation*}
provided that either $\ell = \ell_{1} + \ell_{2}$ and $\frac{1}{p} = \frac{1}{p_{1}} + \frac{1}{p_{2}}$, or $\ell > \ell_{1} + \ell_{2}$ and $\frac{1}{p} \geq \frac{1}{p_{1}} + \frac{1}{p_{2}}$. In the former case, $C = 1$, while in the latter case, $C$ depends on $\ell-\ell_{1}-\ell_{2}$ and $\frac{1}{p} - \frac{1}{p_{1}} - \frac{1}{p_{2}}$.
\end{lemma}

We will often used the mixed norm $\nrm{\psi}_{\calL^{\ell, p}_{s} \calX (J)} := \nrm{\nrm{\psi(s)}_{\calX(s)}}_{\calL^{\ell,p}_{s}(J)}$ for $\psi = \psi(x, s)$ such that $s \to \nrm{\psi(s)}_{\calX(s)}$ is measurable. The norms $\calL^{\ell,p}_{s} \dot{\calX}^{k} (J)$ and $\calL^{\ell,p}_{s} \calX^{k} (J)$ are defined analogously.

\subsection{Abstract parabolic theory} \label{subsec:prelim:absPth}
Let $J \subset (0,\infty)$ be an $s$-interval. Given a homogeneous norm $X$ and $k \geq 1$ an integer, let us define the (semi-)norm $\calP^{\ell} \dot{\calX}^{k}(J)$ for a smooth function $\psi$ by
\begin{equation*} 
	\nrm{\psi}_{\calP^{\ell} \dot{\calX}^{k}(J)} := \nrm{\psi}_{\calL^{\ell, \infty}_{s} \dot{\calX}^{k-1}(J)} + \nrm{\psi}_{\calL^{\ell, 2}_{s} \dot{\calX}^{k}(J)}.
\end{equation*}

For $m_{0} < m_{1}$, we will also define the (semi-)norm $\calP^{\ell} \calX^{m_{1}}_{m_{0}} (J)$ by
\begin{equation*} 
	\nrm{\psi}_{\calP^{\ell} \calX_{m_{0}}^{m_{1}} (J)} := \sum_{k=m_{0}+1}^{m_{1}} \nrm{\psi}_{\calP^{\ell} \dot{\calX}^{k}(J)}.
\end{equation*}

We will omit $m_{0}$ when $m_{0} = 0$, i.e., $\calX^{m} := \calX^{m}_{0}$.

We remark that despite the notation $\calP^{\ell} \dot{\calX}^{k}$, this norms controls both the $\dot{\calX}^{k-1}$ as well as the $\dot{\calX}^{k}$ norm of $\psi$.  Note furthermore that $\nrm{\psi}_{\calP^{\ell} \calX_{m_{0}}^{m_{1}}}$ controls the derivatives of $\psi$ of order from $m_{0}$ to $m_{1}$.

\begin{definition} \label{def:absP:pEst}
Let $X$ be a homogeneous norm of degree $2 \ell_{0}$. We say that $X$ satisfies the \emph{parabolic energy estimate} if there exists $C_{X} > 0$ such that for all $\ell \in \bbR$, $[s_{1}, s_{2}] \subset (0, \infty)$ and $\psi$ smooth\footnote{The assumption of smoothness is here only for convenience; we remark that it is not essential in the sense that, by an approximation argument, both \eqref{eq:absP:pEst:1} and \eqref{eq:absP:pEst:2} may be extended to functions $\psi$ which are not smooth.} and satisfying $\nrm{\psi}_{\calP^{\ell} \dot{\calX}^{1}(s_{1}, s_{2}]} < \infty$, the following estimate holds.
\begin{equation} \label{eq:absP:pEst:1}
\begin{aligned}
	\nrm{\psi}_{\calP^{\ell} \dot{\calX}^{1} (s_{1}, s_{2}]} 
	\leq &C_{X} s_{1}^{\ell} \nrm{\psi(s_{1})}_{\calX(s_{1})} + C_{X} (\ell - \ell_{0}) \nrm{\psi}_{\calL^{\ell,2}_{s} \calX (s_{1}, s_{2}]} \\
	& + C_{X} \nrm{(\rd_{s} - \lap) \psi}_{\calL^{\ell + 1, 1}_{s} \calX(s_{1}, s_{2}]}.
\end{aligned}
\end{equation}

The norm $X$ satisfies the \emph{parabolic smoothing estimate} if there exists $C_{X} > 0$ such that for all $\ell \in \bbR $, $[s_{1}, s_{2}] \subset (0, \infty)$ and $\psi$ smooth and satisfying $\nrm{\psi}_{\calP^{\ell} \calX^{2} (s_{1}, s_{2}]} < \infty$ , the following estimate holds:
\begin{equation} \label{eq:absP:pEst:2}
\begin{aligned}
	\nrm{\psi}_{\calP^{\ell} \dot{\calX}^{2} (s_{1}, s_{2}]}
	 \leq & C_{X} s_{1}^{\ell} \nrm{\psi(s_{1})}_{\dot{\calX}^{1}(s_{1})} + C_{X} (\ell +1/2 - \ell_{0}) \nrm{\psi}_{\calL^{\ell,2}_{s} \dot{\calX}^{1}(s_{1}, s_{2}]} \\
	& + C_{X} \nrm{(\rd_{s} - \lap) \psi}_{\calL^{\ell + 1, 2}_{s} \calX (s_{1}, s_{2}]}.
\end{aligned}
\end{equation}
\end{definition}

%Note that the second term on the right hand sides of \eqref{eq:absP:pEst:1} or \eqref{eq:absP:pEst:2} vanishes if $\ell = 1/4$ or $\ell=3/4$, respectively. This definition is modeled after precisely those estimates that can be proved by an integration by parts. See Lemma \ref{lem:absP:EII}

For the purpose of application, we will consider vector-valued solutions $\psi$ to an inhomogeneous heat equation. The norms $X$, $\calX$, $\calP^{\ell} \calX$, etc. of a vector-valued function $\psi$ are defined in the obvious manner of taking the supremum of the respective norm of all components of $\psi$. 

\begin{theorem}[Abstract parabolic theory] \label{thm:absP:absPth}
Let $X$ be a homogeneous norm of degree $2\ell_{0}$, and $\psi$ be a vector-valued smooth solution to $\rd_{s} \psi - \lap \psi = \calN$ on $[0, s_{0}]$. The function $\psi$ is defined on $\bbR^{3} \times [0, s_{0}]$ or $I \times \bbR^{3} \times [0, s_{0}]$ depending on whether $X$ is for functions on the space or the space-time, respectively. 
\begin{enumerate}
\item Let $X$ satisfy the parabolic energy and smoothing estimates \eqref{eq:absP:pEst:1}, \eqref{eq:absP:pEst:2}, and $\psi$ satisfy $\nrm{\psi}_{\calP^{\ell_{0}} \calX^{2}(0, s_{0}]} < \infty$. Let $1 \leq p < \infty$, $\eps > 0$, $D > 0$ and $C(s)$ a function defined on $(0, s_{0}]$ which satisfies
\begin{equation*}
\int_{0}^{s_{0}} C(s)^{p} \frac{\ud s}{s} < \infty
\end{equation*}
for some $1 \leq p < \infty$, and
\begin{equation} \label{eq:absP:apriori:1}
	\nrm{\calN}_{\calL^{\ell_{0}+1, 1}_{s} \calX (0,\subr]} + \nrm{\calN}_{\calL^{\ell_{0}+1, 2}_{s} \calX (0,\subr]} 
	\leq \nrm{C(s) \psi}_{\calL^{\ell_{0}, p}_{s} \calX^{1}(0,\subr]} + \eps \nrm{\psi}_{\calP^{\ell_{0}} \calX^{2}(0, s_{0}]} + D,
\end{equation}
for every $\subr \in (0, s_{0}]$.

Then there exists a constant $\dlt_{A} = \dlt_{A}(C_{X}, \int_{0}^{s_{0}} C(s)^{p} \, \frac{\ud s}{s}, p) > 0$ such that if 
$0 < \eps < \dlt_{A}$,
then the following a priori estimate holds.
\begin{equation} \label{eq:absP:apriori:2}
	\nrm{\psi}_{\calP^{\ell_{0}} \calX^{2} (0, s_{0}]} \leq C e^{C \int_{0}^{s_{0}} C(s)^{p} \, \frac{\ud s}{s}}(\nrm{\psi(s=0)}_{X} + D),
\end{equation}
where $C$ depends only on $C_{X}$ and $p$.

\item Suppose that $X$ satisfies the the parabolic smoothing estimate \eqref{eq:absP:pEst:2}, and that for some $\ell \geq \ell_{0} - 1/2$ and $0 \leq m_{0} \leq m_{1}$ (where $m_{0}$, $m_{1}$ are integers) we have $\nrm{\psi}_{\calP^{\ell} \calX^{m_{1}+2}_{m_{0}}(0, s_{0}]} < \infty$. Suppose furthermore that for $m_{0} \leq m \leq m_{1}$, there exists $\eps > 0$ and a non-negative non-decreasing function $\calB_{m}(\cdot)$ such that
\begin{equation} \label{eq:absP:smth:1}
	\nrm{\calN}_{\calL^{\ell+1, 2}_{s} \dot{\calX}^{m} (0,s_{0}]} 
	\leq \eps \nrm{\psi}_{\calP^{\ell+1} \dot{\calX}^{m+2}(0, s_{0}]} + \calB_{m}(\nrm{\psi}_{\calP^{\ell} \calX_{m_{0}}^{m+1}(0, s_{0}]}).
\end{equation}

Then for $0 < \eps < 1/(2C_{X})$, the following \emph{smoothing estimate} holds:
\begin{equation} \label{eq:absP:smth:2}
	\nrm{\psi}_{\calP^{\ell} \calX_{m_{0}}^{m_{1}+2}(0, s_{0}]} \leq C
\end{equation}
where $C$ is determined from $C_{X}$, $\calB_{m_{0}}, \ldots, \calB_{m_{1}}$ and $\nrm{\psi}_{\calP^{\ell} \dot{\calX}^{m_{0}+1} (0, s_{0}]}$. 

More precisely, consider the non-decreasing function $\widetilde{\calB_{m}}(r) := (2C_{X}(\ell-\ell_{0}+1/2) + 1) r + 2 C_{X} \calB_{m}(r)$. Then $C$ in \eqref{eq:absP:smth:2} is given by the composition
\begin{equation} \label{eq:absP:smth:3}
	C = \widetilde{\calB}_{m_{1}} \circ \widetilde{\calB}_{m_{1}-1}\circ \cdots \circ \widetilde{\calB}_{m_{0}}(\nrm{\psi}_{\calP^{\ell} \dot{\calX}^{m_{0}+1}(0, s_{0}]}).
\end{equation}
\end{enumerate}
\end{theorem}
\begin{proof} 
\pfstep{Step 1: Proof of (1)}
Without loss of generality, assume that $C_{X} \geq 2$. Thanks to the hypothesis on $\psi$ and \eqref{eq:absP:apriori:1}, we can apply the parabolic energy estimate \eqref{eq:absP:pEst:1} to obtain
\begin{equation} \label{eq:absP:apriori:pf:1}
\begin{aligned}
	\nrm{\psi}_{\calP^{\ell_{0}} \dot{\calX}^{1}(0, \subr]} 
	\leq & C_{X} \nrm{\psi(0)}_{X}
	 + C_{X} (\nrm{C(s)\psi}_{\calL^{\ell_{0},p}_{s} \calX^{1}(0,\subr]} + \eps \nrm{\psi}_{\calP^{\ell_{0}} \calX^{2}(0, s_{0}]} + D).
\end{aligned}
\end{equation}
where we have used the fact that $\liminf_{s_{1} \to 0} s_{1}^{\ell_{0}} \nrm{\psi(s_{1})}_{\calX(s_{1})} = \nrm{\psi(0)}_{X}$. Using again the hypothesis on $\psi$ and \eqref{eq:absP:apriori:1}, we can apply the parabolic smoothing estimate \eqref{eq:absP:pEst:2} and get
\begin{equation} \label{eq:absP:apriori:pf:2}
\begin{aligned}
	\nrm{\psi}_{\calP^{\ell_{0}} \dot{\calX}^{2}(0, \subr]} 
	\leq & \frac{C_{X}}{2} \nrm{\psi}_{\calL^{\ell_{0},2}_{s} \dot{\calX}^{1}(0, \subr]}
	 + C_{X} (\nrm{C(s) \psi}_{\calL^{\ell_{0},p}_{s} \calX^{1}(0,\subr]} + \eps \nrm{\psi}_{\calP^{\ell_{0}} \calX^{2}(0, s_{0}]} + D),
\end{aligned}
\end{equation}
where we used $\liminf_{s_{1} \to 0} s_{1}^{\ell_{0}} \nrm{\psi(s_{1})}_{\dot{\calX}^{1}(s_{1})} = 0$, which holds as $\nrm{\psi}_{\calL^{\ell_{0}, 2}_{s} \dot{\calX}^{1}} < \infty$. Using \eqref{eq:absP:apriori:pf:1} to bound the first term on the right-hand of \eqref{eq:absP:apriori:pf:2}, we arrive at
\begin{equation*}
	\nrm{\psi}_{\calP^{\ell_{0}} \calX^{2}(0, \subr]} 
	\leq C_{X}^{2} \nrm{\psi(0)}_{X} + C_{X}(1+C_{X}) (\nrm{C(s)\psi}_{\calL^{\ell_{0},p}_{s} \calX^{1}(0,\subr]} + \eps \nrm{\psi}_{\calP^{\ell_{0}} \calX^{2}(0, s_{0}]} + D),
\end{equation*}
for every $0 < \subr \leq s_{0}$. 

We will apply Gronwall's inequality to deal with the term involving $C(s) \psi$. For convenience, let us make the definition
\begin{equation*}
	D' = C_{X}^{2} \nrm{\psi(0)}_{X} + C_{X}(1+C_{X}) (\eps \nrm{\psi}_{\calP^{\ell_{0}} \calX^{2}(0, s_{0}]} + D).
\end{equation*}

Recalling the definition of $\nrm{\psi}_{\calP^{\ell_{0}} \calX^{2}(0, \subr]}$ and unravelling the definition of $\calL^{\ell_{0}, p}_{s} \calX^{1}$, we see in particular that
\begin{equation*}
	\sup_{0 < s \leq \subr} s^{\ell_{0}} \nrm{\psi(s)}_{\calX^{1}} 
	\leq C_{X} (1+C_{X}) \bb( \int_{0}^{\subr} C(s)^{p} (s^{\ell_{0}} \nrm{\psi(s)}_{\calX^{1}})^{p} \frac{\ud s}{s} \bb)^{1/p} + D',
\end{equation*}
for every $0 < \subr \leq s_{0}$. Taking the $p$-th power, using Gronwall's inequality and then taking the $p$-th root back, we arrive at
\begin{equation*}
	\sup_{0 < s \leq \subr} s^{\ell_{0}} \nrm{\psi(s)}_{\calX^{1}} \leq 2^{1/p} D' \exp \bb( \frac{2 C_{X}^{p} (1+C_{X})^{p}}{p} \int_{0}^{\subr} C(s)^{p} \frac{\ud s}{s} \bb) .
\end{equation*}

Iterating this bound into $\nrm{C(s) \psi}_{\calL^{\ell_{0}, p}_{s} \calX^{1}(0, s_{0}]}$ and expanding $D'$ out, we obtain
\begin{equation*}
	\nrm{\psi}_{\calP^{\ell_{0}} \calX^{2}(0, s_{0}]} \leq C_{0} \bb( C_{X}^{2} \nrm{\psi(0)}_{X} + C_{X}(1+C_{X}) (\eps \nrm{\psi}_{\calP^{\ell_{0}} \calX^{2}(0, s_{0}]} + D) \bb).
\end{equation*}
where $C_{0} = \exp \bb( \frac{2 C_{X}^{p} (1+C_{X})^{p}}{p} \int_{0}^{s_{0}} C(s)^{p} \frac{\ud s}{s} \bb)$.

Let us define $\dlt_{A} := (2 C_{0} C_{X} (1+C_{X}))^{-1}$. Then from the hypothesis $0 < \eps < \dlt_{A}$, we can absorb the term $C_{0} C_{X}(1+C_{X}) \eps \nrm{\psi}_{\calP^{\ell_{0}} \calX^{2}(0, s_{0}]}$ into the left-hand side. The desired estimate \eqref{eq:absP:apriori:2} follows.

\pfstep{Step 2: Proof of (2)} 
In this step, we will always work on the whole $s$-interval $(0, s_{0}]$. 

We claim that under the assumptions of (2), the following inequality holds for $m_{0} \leq m \leq m_{1}$:
\begin{equation} \label{eq:absP:smth:pf:1}
	\nrm{\psi}_{\calP^{\ell} \calX^{m+2}_{m_{0}}} \leq \widetilde{\calB}_{m}(\nrm{\psi}_{\calP^{\ell} \calX^{m+1}_{m_{0}}}).
\end{equation}

Assuming the claim, we can start from $\nrm{\psi}_{\calP^{\ell} \calX^{m_{0}+1}_{m_{0}}} = \nrm{\psi}_{\calP^{\ell} \dot{\calX}^{m_{0}+1}}$ and iterate \eqref{eq:absP:smth:pf:1} for $m= m_{0}, m_{0} + 1, \ldots, m_{1}$ (using the fact that each $\widetilde{\calB}_{m}$ is non-decreasing) to conclude the proof.

To prove the claim, we use the hypothesis on $\psi$ and \eqref{eq:absP:smth:1} to apply the parabolic smoothing estimate \eqref{eq:absP:pEst:2}, which gives
\begin{equation*}
	\nrm{\psi}_{\calP^{\ell} \dot{\calX}^{m+2}} \leq C_{X} (\ell - \ell_{0} +1/2) \nrm{\psi}_{\calL^{\ell, 2}_{s} \dot{\calX}^{m+1}}+ C_{X} (\eps \nrm{\psi}_{\calP^{\ell} \dot{\calX}^{m+2}} + \calB_{m}(\nrm{\psi}_{\calP^{\ell} \calX^{m+1}_{m_{0}}})),
\end{equation*}
where we have used $\liminf_{s_{1} \to 0} s_{1}^{\ell} \nrm{\psi(s_{1})}_{\dot{\calX}^{m+1}(s_{1})} = 0$, which holds as $\nrm{\psi}_{\calL^{\ell, 2}_{s} \dot{\calX}^{m+1}} < \infty$. 

Using the smallness of $\eps >0$, we can absorb $C_{X} \eps \nrm{\psi}_{\calP^{\ell} \dot{\calX}^{m+2}}$ into the left-hand side. Then adding $\nrm{\psi}_{\calL^{\ell,2}_{s} \calX^{m+1}_{m_{0}}}$ to both sides, we easily obtain
\begin{equation*}
	\nrm{\psi}_{\calP^{\ell} \calX^{m+2}_{m_{0}}} \leq (2 C_{X} (\ell-\ell_{0} +1/2) +1) \nrm{\psi}_{\calP^{\ell} \calX^{m+1}_{m_{0}}} + 2 C_{X} \calB_{m}(\nrm{\psi}_{\calP^{\ell} \calX^{m+1}_{m_{0}}}).
\end{equation*}

Recalling the definition of $\widetilde{\calB}_{m}$, this is exactly \eqref{eq:absP:smth:pf:1}. \qedhere
\end{proof}

The following proposition allows us to use Theorem \ref{thm:absP:absPth} in the situations of interest in our work.

\begin{proposition} \label{prop:absP:application}
The following statements hold.
\begin{enumerate}
\item Let $\psi \in C^{\infty}_{s}(J, H^{\infty}_{x})$ (resp. $\psi \in C^{\infty}_{t,x}(I  \times J, H^{\infty}_{x})$), where $J$ is a finite interval. Then for $X = L^{2}_{x}$ (resp. $X = L^{2}_{t,x}$ or $\dot{S}^{1}$), we have
\begin{equation} \label{eq:absP:application:1}
	\nrm{\psi}_{\calL^{\ell,p}_{s} \dot{\calX}^{k}(J)} < \infty
\end{equation}
if either $1 \leq p \leq \infty$ and $\ell - \ell_{0} + k/2 > 0$, or $p=\infty$ and $\ell-\ell_{0} + k/2 = 0$.

\item Furthermore, the norms $L^{2}_{x}$, $L^{2}_{t,x}$ and $\dot{S}^{1}$ satisfy the parabolic energy and smoothing estimates \eqref{eq:absP:pEst:1}, \eqref{eq:absP:pEst:2}. 
\end{enumerate}
\end{proposition}
\begin{proof} 
By definition, we have
\begin{equation*}
	\nrm{\psi}_{\calL^{\ell, p} \dot{\calX}^{k}} = \nrm{s^{\ell-\ell_{0} +k/2} \nrm{\rd_{x}^{(k)} \psi(s)}_{X}}_{\calL^{p}_{s}}.
\end{equation*}

Since $\sup_{s \in J} \nrm{\rd_{x}^{(k)} \psi(s)}_{X} < \infty$ for each $X$ under consideration when $\psi$ is regular, the first statement follows.

To prove the second statement, let us begin by proving that the norm $L^{2}_{x}$ satisfies the parabolic energy estimate \eqref{eq:absP:pEst:1}. In this case, $\ell_{0} = 3/4$. Let $\ell \in \bbR$, $[s_{1}, s_{2}] \subset (0,\infty)$ and $\psi$ a smooth (complex-valued) function such that $\nrm{\psi}_{\calP^{\ell} \dot{\calH}^{1}_{x}} < \infty$. We may assume that $\nrm{(\rd_{s} - \lap)\psi}_{\calL^{\ell+1,1}_{s} \calL^{2}_{x}} < \infty$, as the other case is trivial. Multiplying the equation $(\rd_{s} - \lap) \psi$ by $s^{2(\ell - \ell_{0})} \overline{\psi}$ and integrating by parts over $[s_{1}, \subr]$ (where $s_{1} \leq \subr \leq s_{2}$), we obtain
\begin{equation} \label{eq:parabolicEII1}
\begin{aligned} 
\frac 1 2  \subr^{2(\ell-\ell_{0})} & \int \abs{\psi (\subr)}^2 \, \ud x + \int_{s_1}^{\subr} \int s^{2(\ell-\ell_{0})+1} \abs{\rd_x \psi}^2 \, \ud x \, \frac{\ud s}{s}  \\
= & \frac 1 2 s_1^{2(\ell-\ell_{0})} \int \abs{\psi(s_1)}^2 \, \ud x
 + (\ell - \ell_{0})  \int_{s_1}^{\subr} \int s^{2(\ell-\ell_{0})} \abs{\psi}^2 \, \ud x \frac{\ud s}{s} \\
& + \int_{s_1}^{\subr} \int s^{2(\ell-\ell_{0}) + 1} (\rd_{s} - \lap)\psi \cdot \overline{\psi} \, \ud x \, \frac{\ud s}{s}.
\end{aligned}
\end{equation}

Taking the supremum over $s_{1} \leq \subr \leq s_{2}$ and rewriting in terms of p-normalized norms, we obtain
\begin{equation*}
\begin{aligned}
	\frac{1}{2} \nrm{\psi}_{\calL^{\ell, \infty}_{s} \calL^{2}_{x}(s_{1}, s_{2}]}^{2} + \nrm{\psi}_{\calL^{\ell, 2}_{s} \dot{\calH}^{1}_{x} (s_{1}, s_{2}]}^{2}
	\leq & \frac{1}{2} s_{1}^{2\ell} \nrm{\psi(s_{1})}_{\calL^{2}_{x}(s_{1})}^{2} + (\ell-\ell_{0}) \nrm{\psi}_{\calL^{\ell, 2}_{s} \calL^{2}_{x} (s_{1}, s_{2}]}^{2} \\
	& + \nrm{(\rd_{s} - \lap)\psi \cdot \overline{\psi}}_{\calL^{2\ell + 1, 1}_{s} \calL^{1}_{x}(s_{1}, s_{2}]}.
\end{aligned}
\end{equation*}

By H\"older and Lemma \ref{lem:absP:Holder4Ls}, we can estimate the last term by $\nrm{(\rd_{s} - \lap)\psi}_{\calL^{\ell+1,1}_{s} \calL^{2}_{x} (s_{1}, s_{2}]}^{2} + (1/4) \nrm{\psi}_{\calL^{\ell, \infty}_{s} \calL^{2}_{x}(s_{1}, s_{2}]}^{2}$, where the latter can be absorbed into the left-hand side. Taking the square root of both sides, we obtain \eqref{eq:absP:pEst:1} for $L^{2}_{x}$.

Next, let us prove that the norm $L^{2}_{x}$ satisfies the parabolic smoothing estimate \eqref{eq:absP:pEst:2}. Let $\ell \in \bbR$, $[s_{1}, s_{2}] \subset (0,\infty)$ and $\psi$ a smooth (complex-valued) function such that $\nrm{\psi}_{\calP^{\ell} \dot{\calH}^{2}_{x}} < \infty$. As before, we assume that $\nrm{(\rd_{s} - \lap)\psi}_{\calL^{\ell+1,2}_{s} \calL^{2}_{x}} < \infty$. Multiplying the equation $(\rd_{s} - \lap) \psi$ by $s^{2(\ell - \ell_{0})+1} \lap \overline{\psi}$ and integrating by parts over $[s_{1}, \subr]$ (where $s_{1} \leq \subr \leq s_{2}$), we obtain
\begin{equation} \label{eq:parabolicEII2}
\begin{aligned} 
\frac 1 2  \subr^{2(\ell-\ell_{0})} & \int \abs{\rd_{x} \psi (\subr)}^2 \, \ud x + \int_{s_1}^{\subr} \int s^{2(\ell-\ell_{0})+2} \abs{\lap \psi}^2 \, \ud x \, \frac{\ud s}{s}  \\
= & \frac 1 2 s_1^{2(\ell-\ell_{0})+1} \int \abs{\rd_{x} \psi(s_1)}^2 \,\ud x
 + (\ell - \ell_{0} + \frac{1}{2})  \int_{s_1}^{\subr} \int s^{2(\ell-\ell_{0})+1} \abs{\rd_{x} \psi}^2 \, \ud x \frac{\ud s}{s} \\
& + \int_{s_1}^{\subr} \int s^{2(\ell-\ell_{0}) + 2} (\rd_{s} - \lap)\psi \cdot \lap \overline{\psi} \, \ud x \, \frac{\ud s}{s}.
\end{aligned}
\end{equation}

By a further integration by parts, the second term on the left-hand side is equal to $\nrm{\psi}_{\calL^{\ell, 2}_{s} \dot{\calH}^{2}_{x} (s_{1}, s_{2}]}^{2}$. Taking the supremum over $s_{1} \leq \subr \leq s_{2}$ and rewriting in terms of p-normalized norms, we obtain
\begin{equation*}
\begin{aligned}
	\frac{1}{2} \nrm{\psi}_{\calL^{\ell, \infty}_{s} \dot{\calH}^{1}_{x}(s_{1}, s_{2}]}^{2} + \nrm{\psi}_{\calL^{\ell, 2}_{s} \dot{\calH}^{2}_{x} (s_{1}, s_{2}]}^{2}
	\leq & \frac{1}{2} s_{1}^{2\ell} \nrm{\psi(s_{1})}_{\dot{\calH}^{1}_{x}(s_{1})}^{2} + (\ell-\ell_{0} + \frac{1}{2}) \nrm{\psi}_{\calL^{\ell, 2}_{s} \dot{\calH}^{1}_{x} (s_{1}, s_{2}]}^{2} \\
	& + \nrm{(\rd_{s} - \lap)\psi \cdot \nb^{k} \nb_{k} \overline{\psi}}_{\calL^{2\ell + 1, 1}_{s} \calL^{1}_{x}(s_{1}, s_{2}]}.
\end{aligned}
\end{equation*}

By Cauchy-Schwarz and Lemma \ref{lem:absP:Holder4Ls}, we can estimate the last term by
\begin{equation*}
(1/2)\nrm{(\rd_{s} - \lap)\psi}_{\calL^{\ell+1,2}_{s} \calL^{2}_{x} (s_{1}, s_{2}]}^{2} + (1/2) \nrm{\psi}_{\calL^{\ell, 2}_{s} \dot{\calH}^{2}_{x}(s_{1}, s_{2}]}^{2},
\end{equation*}
where the latter can be absorbed into the left-hand side. Taking the square root of both sides, we obtain \eqref{eq:absP:pEst:2} for $L^{2}_{x}$.

For the norm $L^{2}_{t,x}$, in which case $\ell_{0} = 1$, it simply suffices to repeat the above proof with the new value of $\ell_{0}$, and integrate further in time. 

Finally, in the case of the norm $\dot{S}^{1}$, for which $\ell_{0} = 1/4$, we begin by observing that
\begin{equation*}
	\nrm{\psi}_{\calL^{\ell,p}_{s} \dot{\calS}^{1}}:= \nrm{\nb_{t,x} \psi}_{\calL^{\ell, p}_{s} \calL^{\infty}_{t} \calL^{2}_{x}} + \abs{I}^{1/2} \nrm{s^{3/4} \Box \psi}_{\calL^{\ell, p}_{s} \calL^{2}_{t,x}}
	\aeq \nrm{s^{1/2} \rd_{t,x} \psi(t=0)}_{\calL^{\ell, p}_{s} \calL^{2}_{x}} + \abs{I}^{1/2} \nrm{s^{3/4} \Box \psi}_{\calL^{\ell, p}_{s} \calL^{2}_{t,x}}
\end{equation*}
for every $\ell \geq 0$ and $1 \leq p \leq \infty$, where $A \aeq B$ means that $A$, $B$ are comparable, i.e., there exist $C >0$ such that $A \leq C B, B \leq C A$. One direction is trivial, whereas the other follows from the energy estimate. Using furthermore the fact that $\rd_{t,x}, \Box$ commute with $(\rd_{s} - \lap)$, this case follows from the last two cases. \qedhere
\end{proof}

%\begin{remark} 
%From the proof of \eqref{eq:absP:application:1}, it is evident that the following variant is also true: 
%\begin{quote}
%Let $\ell \in \bbR$, $1 \leq p \leq \infty$ and $k \geq 1$. For a smooth function $\psi = \psi(t,x,s)$ such that $\rd_{t,x} \psi$ is regular, we have
%$\nrm{\psi}_{\calL^{\ell,p}_{s} \dot{\calS}^{k}} < \infty$ if either $1 \leq p \leq \infty$ and $\ell - 3/4 + k/2 > 0$, or $p = \infty$ and $\ell - 3/4 + k/2 = 0$.
%\end{quote}
%\end{remark}

\begin{remark} 
A point that the reader should keep in mind is that, despite the heavy notations and abstract concepts developed in this subsection, the analytic heart of the `abstract parabolic theory' is simply the standard $L^{2}$-energy integral estimates for the linear heat equation, as we have seen in Proposition \ref{prop:absP:application}. The efforts that we had put in this subsection will pay off in various parts below (in particular Sections \ref{sec:covYMHF} and \ref{sec:pEst4HPYM}), as it will allow us to treat the diverse parabolic equations which arise in a unified, economical way.
\end{remark}

\subsection{Correspondence principle for p-normalized norms}
In this subsection, we develop a systematic method of obtaining linear and multi-linear estimates in terms of p-normalized norms, which will be very useful to us later. The idea is to start with an estimate involving the norms of functions independent of the $s$-variable, and arrive at the corresponding estimate for $s$-dependent functions in terms of the corresponding p-normalized norms by putting appropriate weights of $s$.

Throughout this subsection, we will denote by $J \subset (0, \infty)$ an $s$-interval, $\phi_{i}=\phi_{i}(x)$ a smooth function independent of $s$, and $\psi_{i} = \psi_{i}(s,x)$ a smooth function of both $s \in J$ and $x$. All norms below will be assumed {\it a priori} to be finite. In application, $\phi_{i}$ may be usually taken to be in $H^{\infty}_{x}$, and either $\psi_{i} \in C^{\infty}_{s}(J, H^{\infty}_{x})$. The discussion to follow holds also for functions which depend additionally on $t$.

It is rather cumbersome to give a precise formulation of the Correspondence Principle. We will instead adopt a more pragmatic approach and be satisfied with the following `cookbook-recipe' type statement.

\begin{corrPrinciple}
Suppose that we are given an estimate (i.e., an inequality) in terms of the norms $X_{i}$ of functions $\phi_{i} = \phi_{i}(x)$, all of which are homogeneous. Suppose furthermore that the estimate is scale-invariant, in the sense that the both sides transform the same under scaling. 

Starting from the usual estimate, make the following substitutions on both sides: $\phi_{i} \to \psi_{i}(s)$, $\rd_{x} \to \nb_{x}(s)$, $X_{i} \to \calX_{i}(s)$. Then the resulting estimate still holds, with the same constant, for every $s \in J$.
\end{corrPrinciple}

In other words, given an $s$-independent, scale-invariant estimate which involve only homogeneous norms, we obtain its p-normalized analogue by replacing the norms and the derivatives by their respective p-normalizations. The `proof' of this principle is very simple: For each fixed $s$, the substitution procedure above amounts to applying the usual estimate to $\psi_{i}(s)$ and multiplying each side by an appropriate weight of $s$. The point is that the same weight works for both sides, thanks to scale-invariance of the estimate that we started with.

\begin{example} Some examples are in order to clarify the use of the principle. We remark that all the estimates below will be used freely in what follows.
\begin{enumerate}
\item {\bf (Sobolev)} We begin with the Sobolev inequality \eqref{eq:prelim:sob:1} from Lemma \ref{lem:prelim:sob}. Applying the Correspondence Principle, for every $1 < q \leq r$, $k \geq 0$ such that $ \frac{3}{q} = \frac{3}{r} - k$, we obtain
\begin{equation*}
	\nrm{\psi(s)}_{\calL^{r}_{x}(s)} \leq C \nrm{\psi(s)}_{\calL^{\ell,p}_{s} \dot{\calW}^{k,q}_{x}(s)},
\end{equation*}
for every $s \in J$.

\item {\bf (Interpolation)} Recall the interpolation inequality \eqref{eq:prelim:sob:2} from Lemma \ref{lem:prelim:sob}. Applying the Correspondence Principle, for $1 < q < \infty$, $k_{1} \leq k_{0} \leq k_{2}$, $0 < \tht_{1}, \tht_{2} < 1$ such that $\tht_{1} + \tht_{2} = 1$ and $k_{0} = \tht_{1} k_{1} + \tht_{2} k_{2}$, we obtain 
\begin{equation*}
	\nrm{\psi(s)}_{\dot{\calW}^{k_{0}, q}_{x}(s)} \leq C \nrm{\psi(s)}^{\tht_{1}}_{\dot{\calW}^{k_{1}, q}_{x}(s)} \nrm{\psi(s)}^{\tht_{2}}_{\dot{\calW}^{k_{2}, q}_{x}(s)},
\end{equation*}
for every $s \in J$.

\item {\bf (Gagliardo-Nirenberg)}
Let us apply the Correspondence Principle to the Gagliardo-Nirenberg inequality \eqref{eq:prelim:sob:3} from Lemma \ref{lem:prelim:sob}. Then for $1 \leq q_{1}, q_{2}, r \leq \infty$, $0 < \tht_{1}, \tht_{2} < 1$ such that $\frac{3}{r} = \tht_{1} \cdot \frac{3}{q_{1}} + \tht_{2} ( \frac{3}{q_{2}} - 1 )$, we obtain
\begin{equation*}
	\nrm{\psi(s)}_{\calL^{r}_{x}(s)} \leq C \nrm{\psi(s)}_{\calL^{q_{1}}_{x}(s)}^{\tht_{1}} \nrm{\nb_{x} \psi(s)}_{\calL^{q_{2}}_{x}}^{\tht_{2}}
\end{equation*}
for every $s \in J$.

\item {\bf (H\"older)} Let us start with $\nrm{\phi_{1} \phi_{2}}_{L^{r}_{x}} \leq \nrm{\phi_{1}}_{L^{q_{1}}_{x}} \nrm{\phi_{2}}_{L^{q_{2}}_{x}}$, where $1 \leq q_{1}, q_{2}, r \leq \infty$ and $\frac{1}{r} = \frac{1}{q_{1}} + \frac{1}{q_{2}}$. Applying the Correspondence Principle, for every $s \in J$, we obtain 
\begin{equation*}
	\nrm{\psi_{1} \psi_{2} (s)}_{\calL^{r}_{x}(s)} \leq \nrm{\psi_{1}(s)}_{\calL^{q_{1}}_{x}(s)} \nrm{\psi_{2}(s)}_{\calL^{q_{2}}_{x}(s)}.
\end{equation*}

	All the estimates above extend to functions on $I \times \bbR^{3}$ with $I \subset \bbR$ in the obvious way. In this case, we have the following analogue of the H\"older inequality:
	\begin{equation*} 
		\nrm{\psi(s)}_{\calL^{q_{1}}_{t} \calL^{p}_{x}(s)} \leq s^{- \left( \frac{1}{2 q_{1}}-\frac{1}{2 q_{2}} \right)} \abs{I}^{\frac{1}{q_{1}}-\frac{1}{q_{2}}} \nrm{\psi(s)}_{\calL^{q_{2}}_{t} \calL^{p}_{x}(s)} \hbox{ for $q_{1} \leq q_{2}$.}
	\end{equation*}
\end{enumerate}
\end{example}

The following consequence of the Gagliardo-Nirenberg and Sobolev inequalities is useful enough to be separated as a lemma on its own. It provides a substitute for the incorrect $\dot{H}^{3/2}_{x} \subset L^{\infty}_{x}$ Sobolev embedding, and has the benefit of being scale-invariant. We will refer to this simply as \emph{Gagliardo-Nirenberg} for $p$-normalized norms.

\begin{lemma}[Gagliardo-Nirenberg] \label{lem:absP:algEst}
For every $s \in J$, the following estimate holds.
\begin{equation} \label{eq:absP:algEst}
\begin{aligned}
	\nrm{\nb^{(k)}_{x} \psi(s)}_{\dot{\calH}^{3/2}_{x} \cap \calL^{\infty}_{x}(s)} 
	\leq & C_{k} \nrm{\psi(s)}_{\dot{\calH}^{k+2}_{x}(s)}^{1/2} \nrm{\psi(s)}_{\dot{\calH}^{k+1}_{x}(s)}^{1/2} \\
	\leq & \frac{C_{k}}{2} (\nrm{\psi(s)}_{\dot{\calH}^{k+2}_{x}(s)} + \nrm{\psi(s)}_{\dot{\calH}^{k+1}_{x}(s)}).
\end{aligned}
\end{equation}
\end{lemma}
\begin{proof} 
Without loss of generality, assume $k =0$. To prove the first inequality, by Gagliardo-Nirenberg, interpolation and the Correspondence Principle, it suffices to prove $\nrm{\phi}_{L^{6}_{x}} \leq C \nrm{\phi}_{\dot{H}^{1}_{x}}$ and $\nrm{\rd_{x} \phi}_{L^{6}_{x}} \leq C \nrm{\phi}_{\dot{H}^{2}_{x}}$; the latter two are simple consequences of Sobolev. Next, the second inequality follows from the first by Cauchy-Schwarz. \qedhere
\end{proof}

We remark that in practice, the Correspondence Principle, after multiplying by an appropriate weight of s and integrating over $J$, will often be used in conjunction with H\"older's inequality for the spaces $\calL^{\ell, p}_{s}$ (Lemma \ref{lem:absP:Holder4Ls}).

Finally, recall that the notation $\calO(\psi_{1}, \psi_{2}, \ldots, \psi_{k})$ refers to a linear combination of expressions in \emph{the values of} the arguments $\psi_{1}, \psi_{2}, \ldots, \psi_{k}$, where they could in general be vector-valued. It therefore follows immediately that any multi-linear estimate for the usual product $\nrm{\psi_{1} \cdot \psi_{2} \cdots \psi_{k}}$ for scalar-valued functions $\psi_{1}, \psi_{2}, \ldots, \psi_{k}$ implies the corresponding estimate for $\nrm{\calO(\psi_{1}, \psi_{2}, \ldots, \psi_{k}}$, where $\psi_{1}, \psi_{2}, \ldots, \psi_{k}$ may now be vector-valued, at the cost of some absolute constant depending on $\calO$. This remark will be used repeatedly in the sequel.

\subsection{Associated $s$-weights for variables of \eqref{eq:HPYM}} \label{subsec:assocWght}
Let us consider the system \eqref{eq:HPYM}, introduced in \S \ref{subsec:overview}. To each variable of \eqref{eq:HPYM}, there is associated a power of $s$ which represents the expected size of the variable in a dimensionless norm (say $L^{\infty}_{t,s,x}$); we call this the \emph{associated $s$-weight} of the variable. The notion of associated $s$-weights provides a useful heuristic which will make keeping track of these weights easier in the rest of the paper. 

The associated $s$-weights for the `spatial variables' $A=A_{i}, F=F_{ij}, F_{s}=F_{si}$ are derived directly from scaling considerations, and thus are easy to determine. Indeed, as we expect that $\nrm{\rd_{x} A_{i}}_{L^{2}_{x}}$ should stay bounded for every $t, s$, using the scaling heuristics $\rd_{x} \aeq s^{-1/2}$ and $L^{2}_{x} \aeq s^{3/4}$, it follows that $A_{i} \aeq s^{-1/4}$. The worst term in $F_{ij}$ is at the level of $\rd_{x} A$, so $F_{ij} \aeq s^{-3/4}$, and similarly $F_{si} \aeq s^{-5/4}$.

The associated $s$-weights for $w_{\nu}$ is $s^{-1}$, which is actually better than that which comes from scaling considerations (which is $s^{-5/4}$). To see why, observe that $w_{\nu}$ satisfies a parabolic equation $(\rd_{s} - \lap) w_{\nu} = {}^{(w_{\nu})} \calN$ with zero data at $s=0$.\footnote{We remind the reader, that this is a consequence of the original Yang-Mills equations $\covD^{\mu} F_{\nu \mu} =0$ at $s=0$.} Duhamel's principle then tells us that $w_{\nu} \aeq s {}^{(w_{\nu})} \calN$. Looking at the equation \eqref{eq:covParabolic4w}, we see that ${}^{(w_{\nu})} \calN \aeq s^{-2}$, from which we conclude $w_{\nu} \aeq s^{-1}$. Note that as $w_{0} = - F_{s0}$, this shows that the `temporal variables' $A_{0}, F_{s0}$ behave better than their `spatial' counterparts.

We summarize the associated $s$-weights for important variables as follows.
\begin{center}
\begin{tabular}{c c c}
$A_{i} \aeq s^{-1/4}$ & $A_{0} \aeq s^{0}$ & $F_{\mu \nu} \aeq s^{-3/4}$ \\
$F_{si} \aeq s^{-5/4}$ & $F_{s0} \aeq s^{-1}$ & $w_{\mu} \aeq s^{-1}$. 
\end{tabular}
\end{center}

Accordingly, when we control the sizes of these variables, they will be weighted by the inverse of their respective associated weights.

As we always work on a finite $s$-interval $J$ such that $J \subset [0,1]$, extra powers of $s$ compared to the inverse of the associated $s$-weight should be considered favorable when estimating. For example, it is easier to estimate $\nrm{A_{i}}_{\calL^{1/4+\ell,\infty}_{s} \dot{\calH}^{1}_{x}}$ when $\ell > 0$ than $\ell = 0$. (Compare Lemma \ref{lem:fundEst4A} with Proposition \ref{prop:est4ai}.) Informally, when it suffices to control a variable with more power of $s$, say $s^{\ell}$, compared to the associated $s$-weight, we will say that there is an \emph{extra $s$-weight of $s^{\ell}$}. Thanks to the sub-critical nature of the problem, such extra weights will be abundant, and this will simplify the analysis in many places.

It is also useful to keep in mind the following heuristics.
\begin{equation*}
\rd_{t,x}, \covD_{t,x} \aeq s^{-1/2}, \quad 
\rd_{s}, \covD_{s} \aeq s^{-1}, \quad
L^{q}_{t} L^{r}_{x} \aeq s^{1/(2q)+3/(2r)}\\
\end{equation*}

\section{Reduction of the Main Theorem to Theorems \ref{thm:idEst} and \ref{thm:dynEst}} \label{sec:mainThm}
%\comment{The following Main Theorem should probably go to Introduction.
%\begin{definition}[Admissible $H^{1}$ Initial Data] \label{def:admID}
%We say that a pair $(\Aini_{i}, \Eini_{i})$ of 1-forms on $\bbR^{3}$ is an \emph{admissible $H^{1}$ initial data} for the Yang-Mills equations if the following conditions hold:
%\begin{enumerate}
%\item $\Aini_{i} \in \dot{H}_{x}^{1} \cap L^{3}_{x}$ and $\Eini_{i} \in L^{2}$,
%\item The \emph{constraint equation}
%\begin{equation} \label{eq:YMconstraint}
%	\rd^{i} \Eini_{i} + \LieBr{\Aini^{i}}{\Eini_{i}} = 0,
%\end{equation}
%	is satisfies, in the distributional sense.
%\end{enumerate}
%\end{definition}
%
%Our main theorem is a local well-posedness result for such initial data.
%\begin{MainTheorem} [$H^{1}$ Local Well-Posedness of Yang-Mills] 
%Let $(\Aini_{i}, \Eini_{i})$ be an admissible $H^{1}$ initial data set. 
%\begin{enumerate}
%\item There exists a unique admissible solution to the Yang-Mills equation \eqref{eq:hyperbolicYM} with the given initial data on $(-T^{\star}, T^{\star})$ in the temporal gauge $A_{0} = 0$, with a lower bound on $T^{\star} > 0$ depending only on $\calIini := \nrm{\Aini_{i}}_{\dot{H}^{1}_{x}} + \nrm{\Eini_{i}}_{L^{2}_{x}}$. 
%\end{enumerate}
%\end{MainTheorem}
%}

In the first subsection, we state and prove some preliminary results that we will need in this section. These include a $H^{2}$ local well-posedness statement for the Yang-Mills equations in the temporal gauge, an approximation lemma for initial data sets and a gauge transform lemma. Next, we will state Theorems \ref{thm:idEst} (Estimates for the initial data in the caloric-temporal gauge) and \ref{thm:dynEst} (Estimates for $t$-evolution in the caloric-temporal gauge), and show that the proof of the Main Theorem is reduced to that of Theorems \ref{thm:idEst} and \ref{thm:dynEst} by a simple bootstrap argument involving a gauge transformation. 

%Then in the next subsection, we claim the existence of appropriate bootstrapped quantities ($\calF$ and $\calAlow$) and prove the Main Proposition, modulo the precise definition and technical estimates for $\calF$ and $\calAlow$. The proof of the latter will occupy the rest of the paper.

\subsection{Preliminary results}
We will begin this subsection by making a number of important definitions. Let us define the notion of \emph{regular} solutions, which are smooth solutions with appropriate decay towards the spatial infinity.

\begin{definition}[Regular solutions]  \label{def:mainThm:reg4YM}
We say that a representative $A_{\mu} : I \times \bbR^{3} \to \LieAlg$ of a classical solution to \eqref{eq:hyperbolicYM} is \emph{regular} if $A_{\mu} \in C^{\infty}_{t}(I, H^{\infty}_{x})$. Furthermore, we say that a smooth solution $A_{\bfa} : I \times \bbR^{3} \times J \to \LieAlg$ to \eqref{eq:HPYM} is \emph{regular} if $A_{\bfa} \in C^{\infty}_{t,s}(I \times J, H^{\infty}_{x})$.
\end{definition}

In relation to regular solutions, we also define the notion of a \emph{regular gauge transform}, which is basically that which keeps the `regularity' of the connection 1-form.
\begin{definition}[Regular gauge transform] \label{def:reg4gt}
We say that a gauge transform $U$ on $I \times \bbR^{3} \times J$ is a \emph{regular gauge transform} if $U - \mathrm{Id}, \, U^{-1} - \mathrm{Id} \in C^{\infty}_{t,s}(I \times J, H^{\infty}_{x})$.
A gauge transform $U$ defined on $I \times \bbR^{3}$ is a \emph{regular gauge transform} if it is a regular gauge transform viewed as an $s$-independent gauge transform on $I \times \bbR^{3} \times J$ for some $J \subset [0, \infty)$.
\end{definition}

We remark that a regular solution (whether to \eqref{eq:hyperbolicYM} or \eqref{eq:HPYM}) remains regular under a regular gauge transform.

Let us also give the definition of \emph{regular initial data sets} for \eqref{eq:hyperbolicYM}.
\begin{definition}[Regular initial data sets] \label{def:reg4id}
We say that an initial data set $(\Aini_{i}, \Eini_{i})$ to \eqref{eq:hyperbolicYM} is \emph{regular} if, in addition to satisfying the constraint equation \eqref{eq:YMconstraint}, $\Aini_{i}, \Eini_{i} \in H^{\infty}_{x}$.
\end{definition}

Next, let us present some results needed to prove the Main Theorem. The first result we present is a local well-posedness result for initial data with higher regularity. For this purpose, we have an $H^{2}$ local well-posedness theorem, which is essentially due to Eardley-Moncrief \cite{Eardley:1982fb}. However, as we do not assume anything on the $L^{2}_{x}$ norm of the initial data $\Aini_{i}$ (in particular, it does not need to belong to $L^{2}_{x}$), we need a minor variant of the theorem proved in \cite{Eardley:1982fb}.

In order to state the theorem, let us define the space $\widehat{H}^{2}_{x}$ to be the closure of $\calS_{x}(\bbR^{3})$ with respect to the partially homogeneous Sobolev norm $\nrm{\phi}_{\widehat{H}^{2}_{x}} := \nrm{\rd_{x} \phi}_{H^{1}_{x}}$. The point, of course, is that this norm\footnote{That $\nrm{\cdot}_{\widehat{H}^{2}_{x}}$ is indeed a norm when restricted to $\widehat{H}^{2}_{x}$ follows from Sobolev.} does not contain the $L^{2}_{x}$ norm.

\begin{theorem}[$H^{2}$ local well-posedness of Yang-Mills] \label{thm:mainThm:H2lwp}
Let $(\Aini_{i}, \Eini_{i})$ be an initial data set satisfying \eqref{eq:YMconstraint} such that $\rd_{x} \Aini_{i}, \Eini_{i} \in H^{1}_{x}$.
\begin{enumerate}
\item There exists $T = T(\nrm{(\Aini, \Eini)}_{\widehat{H}^{2}_{x} \times H^{1}_{x}}) > 0$, which is non-increasing in $\nrm{(\Aini, \Eini)}_{\widehat{H}^{2}_{x} \times H^{1}_{x}}$, such that a unique solution $A_{\mu}$ to \eqref{eq:hyperbolicYM} in the temporal gauge satisfying
\begin{equation}
	A_{i} \in C_{t} ((-T, T), \widehat{H}^{2}_{x}) \cap C^{1}_{t}((-T, T), H^{1}_{x})
\end{equation}
exists on $(-T, T) \times \bbR^{3}$.

\item Furthermore, \emph{persistence of higher regularity} holds, in the following sense: If $\rd_{x} \Aini, \Eini \in H^{m}_{x}$  (for an integer $m \geq 1$), then the solution $A_{i}$ obtained in Part (1) satisfies $\rd_{t,x} A_{i} \in C_{t}^{k_{1}} ((-T, T), H^{k_{2}}_{x})$ for non-negative integers $k_{1}, k_{2}$ such that $k_{1} + k_{2} \leq m$.

In particular, if $(\Aini_{i}, \Eini_{i})$ is a regular initial data set, then the corresponding solution $A_{\mu}$ is a regular solution to \eqref{eq:hyperbolicYM} in the temporal gauge.

\item Finally, we have the following \emph{continuation criterion:} If $\sup_{t \in (-T', T')} \nrm{\rd_{t,x} A}_{H^{1}_{x}} < \infty$, then the solution given by Part (1) can be extended past $(-T', T')$, while retaining the properties stated in Parts (1) and (2).
\end{enumerate}
\end{theorem}

\begin{proof} 
It is not difficult to see that the iteration scheme introduced in Klainerman-Machedon \cite[Proposition 3.1]{Klainerman:1995hz} goes through with the above norm, from which Parts (1) -- (3) follow. A cheaper way of proving Theorem \ref{thm:mainThm:H2lwp} is to note that $\nrm{\Aini_{i}}_{H^{2}_{x}(B)} \leq C \nrm{\Aini_{i}}_{\widehat{H}^{2}_{x}(\bbR^{3})}$, $\nrm{\Eini_{i}}_{H^{1}_{x}(B)} \leq \nrm{\Eini_{i}}_{H^{1}_{x}(\bbR^{3})}$ uniformly for all unit balls in $\bbR^{3}$. This allows us to apply the localized local well-posedness statement Proposition 3.1 of \cite{Klainerman:1995hz} to each ball, and glue these local solutions to form a global solution via a domain of dependence argument.
\end{proof}

Next, we prove a technical lemma, which shows that an arbitrary admissible $H^{1}$ initial data set can be approximated by a sequence of regular initial data sets.
\begin{lemma}[Approximation lemma] \label{lem:mainThm:regApprox}
Any admissible $H^{1}$ initial data set $(\Aini_{i}, \Eini_{i}) \in (\dot{H}^{1}_{x} \cap L^{3}_{x}) \times L^{2}_{x}$ can be approximated by a sequence of regular initial data sets $(\Aini_{(n) i}, \Eini_{(n) i})$ satisfying the constraint equation \eqref{eq:YMconstraint}. More precisely, the initial data sets $(\Aini_{(n) i}, \Eini_{(n) i})$ may be taken to satisfy the following properties.
\begin{enumerate}
\item $\Aini_{(n)}$ is smooth, compactly supported, and
\item $\Eini_{(n)} \in H^{\infty}_{x}$.
\end{enumerate}
\end{lemma}

\begin{proof} 
This proof can essentially be read off from \cite[Proposition 1.2]{Klainerman:1995hz}. We reproduce it below for the convenience of the reader.

Choose compactly supported, smooth sequences $\Aini_{(n) i}, \Fini_{(n) i}$ such that $\Aini_{(n) i} \to \Aini_{i}$ in $\dot{H}^{1}_{x} \cap L^{3}_{x}$ and $\Fini_{(n) i} \to \Eini_{i}$ in $L^{2}_{x}$. Let us denote the covariant derivative associated to $\Aini_{(n)}$ by $\covD_{(n)}$. Using the fact that $(\Aini_{i}, \Eini_{i})$ satisfies the constraint equation \eqref{eq:YMconstraint} in the distributional sense and the $\dot{H}^{1}_{x} \subset L^{6}_{x}$ Sobolev, we see that for any test function $\varphi$,
\begin{align*}
	\abs{\int & (\covD_{(n)}^{\ell} \, \Fini_{(n) \ell}, \varphi) \, \ud x} 
	= \abs{\int (\covD_{(n)}^{\ell} \, \Fini_{(n) \ell} - \covD^{\ell} \, \Eini_{\ell}, \varphi) \, \ud x} \\
	= & \abs{\int - (\Fini_{(n) \ell} - \Eini_{\ell}, \rd^{\ell} \varphi)
	 +(\LieBr{\Aini_{(n)}^{\ell} - \Aini^{\ell}}{\Fini_{(n) \ell}} + \LieBr{\Aini^{\ell}}{\Fini_{(n) \ell} - \Eini_{\ell}}, \varphi) \, \ud x}  \\
	\leq & \bb( \nrm{{\Fini_{(n)}} - \Eini}_{L^{2}_{x}} + \nrm{{\Aini_{(n)}} - \Aini}_{L^{3}_{x}} \nrm{{\Fini_{(n)}}}_{L^{2}_{x}} + \nrm{\Aini}_{L^{3}_{x}} \nrm{{\Fini_{(n)}} - \Eini}_{L^{2}_{x}}  \bb) \nrm{\varphi}_{\dot{H}^{1}_{x}}.
\end{align*}

In view of the $L^{3}_{x}, L^{2}_{x}$ convergence of $\Aini_{(n)}, \Fini_{(n)}$ to $\Aini, \Eini$, respectively, it follows that 
\begin{equation*}
\covD_{(n)}^{\ell} \Fini_{(n) \ell} \in \dot{H}_{x}^{-1} \hbox{ for each $n$}, \quad \nrm{\covD_{(n)}^{\ell} \, \Fini_{(n) \ell}}_{\dot{H}_{x}^{-1}} \to 0 \quad \hbox{ as }n \to \infty,
\end{equation*}
where $\dot{H}_{x}^{-1}$ is the dual space of $\dot{H}_{x}^{1}$ (defined to be the closure of Schwartz functions on $\bbR^{3}$ under the $\dot{H}_{x}^{1}$-norm).

Let us now define $\Eini_{(n) i} := \Fini_{(n) i} + \covD_{(n) i} \phi_{(n)}$, where the $\LieAlg$-valued function $\phi_{(n)}$ is constructed by solving the elliptic equation
\begin{equation} \label{eq:regApprox:pf:1}
	\covD_{(n)}^{\ell} \, \covD_{(n) \ell} \phi_{(n)} = - \covD_{(n)}^{\ell} \, \Fini_{(n) \ell},
\end{equation}
imposing a suitable decay condition at infinity; we want, in particular, to have $\phi_{(n)} \in \dot{H}^{1}_{x} \cap L^{6}_{x}$. This ensures that $(\Aini_{(n) i}, \Eini_{(n) i})$ satisfies the constraint equation. Furthermore, in view of the fact that $\Aini_{(n)}, \Fini_{(n)}$ are smooth and compactly supported, it is clear that $\covD_{(n)} \phi_{(n)}$ belongs to any $H^{k}_{x}$ for $k \geq 0$, and hence so does $\Eini_{(n)}$. Therefore, in order to prove the lemma, it is only left to prove $\covD_{(n)} \phi_{(n)} \to 0$ in $L^{2}_{x}$.

Multiplying \eqref{eq:regApprox:pf:1} by $\phi_{(n)}$ and integrating by parts, we obtain
\begin{equation} \label{eq:regApprox:pf:2}
	\int \abs{\covD_{(n)} \phi_{(n)}}^{2} \, \ud x \leq \nrm{\covD_{(n)}^{\ell} \Fini_{(n) \ell}}_{\dot{H}_{x}^{-1}} \nrm{\phi_{(n)}}_{\dot{H}_{x}^{1}}.
\end{equation}

On the other hand, expanding out $\covD_{(n)}$, we have
\begin{equation} \label{eq:regApprox:pf:3}
	\nrm{\phi_{(n)}}_{\dot{H}_{x}^{1}} \leq \nrm{\covD_{(n)} \phi_{(n)}}_{L^{2}_{x}} + \nrm{\Aini_{(n)}}_{L^{3}_{x}} \nrm{\phi_{(n)}}_{L^{6}_{x}}.
\end{equation}

Recall Kato's inequality (for a proof, see \cite[Lemma 4.2]{Oh:2012fk}), which shows that $\abs{\rd_{i} \abs{\phi_{(n)}}} \leq \abs{ \covD_{(n) i} \phi_{(n)}}$ in the distributional sense. Combining this with the $\dot{H}^{1}_{x} \subset L^{6}_{x}$ Sobolev inequality for $\abs{\phi_{(n)}}$, we get
\begin{equation} \label{eq:regApprox:pf:4}
	\nrm{\phi_{(n)}}_{L^{6}_{x}} \leq C \nrm{\covD_{(n)} \phi_{(n)}}_{L^{2}_{x}}.
\end{equation}

Combining \eqref{eq:regApprox:pf:2} - \eqref{eq:regApprox:pf:4} and canceling a factor of $\nrm{\covD_{(n)} \phi_{(n)}}_{L^{2}_{x}}$, we arrive at
\begin{equation*}
	\nrm{\covD_{(n)} \phi_{(n)}}_{L^{2}_{x}} 
	\leq \nrm{\covD_{(n)}^{\ell} (\Fini_{(n) \ell})}_{\dot{H}_{x}^{-1}} (1 + C \nrm{\Aini_{(n)}}_{L^{3}_{x}})
	\to 0,
\end{equation*}
as desired. \qedhere
\end{proof}

Given a time interval $I \subset \bbR$, we claim the existence of norms $\calA_{0}(I)$ and $\dlt \calA_{0}(I)$ for $A_{0}$ and $\dlt A_{0}$ on $I$, respectively, for which the following lemma holds. The significance of these norms will be that they can be used to estimate the gauge transform back to the original temporal gauge.

\begin{lemma}[Estimates for gauge transform to temporal gauge] \label{lem:est4gt2temporal}
%\comment{Add regularity : State in terms of regular objects.}
Consider the following ODE on $(-T, T) \times \bbR^{3}$:
\begin{equation} \label{eq:est4gt2temporal:0}
\left \{
\begin{aligned}
	& \rd_{t} V = V A_{0} \\
	& V(t=0) = \Vini,
\end{aligned}
\right.
\end{equation}
where we assume that $A_{0}$ is smooth and $\calA_{0}(-T, T) < \infty$.

\begin{enumerate}
\item Suppose that $\Vini = \Vini(x)$ is a smooth $\LieGrp$-valued function on $\set{t=0} \times \bbR^{3}$ such that 
\begin{equation*}
	\Vini, \Vini^{-1} \in L^{\infty}_{x}, \quad \rd_{x} \Vini, \rd_{x} \Vini^{-1} \in L^{3}_{x}, \quad \rd_{x}^{(2)} \Vini, \rd_{x}^{(2)} \Vini^{-1} \in L^{2}_{x} .
\end{equation*} 

Then there exists a unique solution $V$ to the ODE, which obeys the following estimates.\begin{equation} \label{eq:est4gt2temporal:V}
\left\{
\begin{aligned}
	\nrm{V}_{L^{\infty}_{t} L^{\infty}_{x} (-T, T)} 
	\leq & C_{\calA_{0}(-T, T)} \cdot \nrm{\Vini}_{L^{\infty}_{x}}, \\
	\nrm{\rd_{t,x} V}_{L^{\infty}_{t} L^{3}_{x} (-T, T)} 
	\leq & C_{\calA_{0}(-T, T)} \cdot (\nrm{\rd_{x} \Vini}_{L^{3}_{x}} +\calA_{0}(-T, T) \nrm{\Vini}_{L^{\infty}_{x}}), \\
	\nrm{\rd_{x} \rd_{t,x} V}_{L^{\infty}_{t} L^{2}_{x} (-T, T)} 
	\leq & C_{\calA_{0}(-T, T)} \cdot (\nrm{\rd_{x}^{(2)} \Vini}_{L^{2}_{x}} +\calA_{0}(-T, T) (\nrm{\Vini}_{L^{\infty}_{x}} + \nrm{\rd_{x} \Vini}_{L^{3}_{x}}) ). 
\end{aligned}
\right.
\end{equation}

\item Let $A'_{0}$ be a smooth connection coefficient with $\calA'_{0}(-T, T) < \infty$, and $\Vini'$ a $\LieGrp$-valued smooth function on $\set{t=0} \times \bbR^{3}$ also satisfying the hypotheses of (1). Let $V'$ be the solution to the ODE \eqref{eq:est4gt2temporal:0} with $A_{0}$ and $\Vini$ replaced by $A'_{0}$, $\Vini'$, respectively. Then the difference $\dlt V := V - V'$ satisfies the following estimates.
\begin{equation} \label{eq:est4gt2temporal:dltV}
\begin{aligned}
	&\nrm{\dlt V}_{L^{\infty}_{t} L^{\infty}_{x}(-T, T)}  + \nrm{\rd_{t,x} \dlt V}_{L^{\infty}_{t} L^{3}_{x}(-T, T)} + \nrm{\rd_{x} \rd_{t,x} \dlt V(t)}_{L^{\infty}_{t} L^{2}_{x}(-T, T)}  \\
	& \quad \leq  C_{\calA_{0}(-T, T)} \cdot (\nrm{\dlt \Vini}_{L^{\infty}_{x}} + \nrm{\rd_{x} \dlt \Vini}_{L^{3}_{x}} + \nrm{\rd_{x}^{(2)} \dlt \Vini}_{L^{2}_{x}}) \\
	& \phantom{\quad \leq}	+ C_{\calA_{0}(-T, T)} \cdot (\nrm{\Vini}_{L^{\infty}_{x}} + \nrm{\rd_{x} \Vini}_{L^{3}_{x}} + \nrm{\rd_{x}^{(2)} \Vini}_{L^{2}_{x}}) \, \dlt \calA_{0}(-T, T)
\end{aligned}
\end{equation}

\item Finally, all of the above statement remain true with $V$, $\dlt V$, $\Vini$, $\dlt \Vini$ replaced by $V^{-1}$, $\dlt V^{-1}$, $\Vini^{-1}$ and $\dlt \Vini^{-1}$, respectively.
\end{enumerate}

\end{lemma}

The precise definitions of $\calA_{0}, \dlt \calA_{0}$ will be given in \S \ref{subsec:defOfNorms}, whereas we defer the proof of Lemma \ref{lem:est4gt2temporal} to Appendix \ref{sec:gt}.

Next, we prove a simple lemma which will be used to estimate the $L^{3}_{x}$ norm of our solution.
\begin{lemma} \label{lem:est4L3nrm}
Let $\psi = \psi(t,x)$ be a function defined on $(-T, T) \times \bbR^{3}$ such that $\psi(0) \in L^{3}_{x}$ and $\rd_{t,x} \psi \in C_{t} L^{2}_{x}$. Then $\psi \in C_{t} L^{3}_{x}$ and the following estimate holds.
\begin{equation} \label{eq:est4L3nrm:0}
	\sup_{t \in (-T, T)} \nrm{\psi(t)}_{L^{3}_{x}} \leq \nrm{\psi(0)}_{L^{3}_{x}} + C T^{1/2} \nrm{\rd_{t,x} \psi}_{L^{\infty}_{t} L^{2}_{x}}.
\end{equation}
\end{lemma}
\begin{proof} 
By a standard approximation procedure, it suffices to consider $\psi = \psi(t,x)$ defined on $(-T, T) \times \bbR^{3}$ which is smooth in time and Schwartz in space. For $t \in (-T, T)$, we estimate via H\"older, Sobolev and the fundamental theorem of calculus as follows:
\begin{align*}
	\nrm{\psi(t) - \psi(0)}_{L^{3}_{x}} 
	\leq & \nrm{\psi(t) - \psi(0)}_{L^{2}_{x}}^{1/2} \nrm{\psi(t)-\psi(0)}_{L^{6}_{x}}^{1/2} \\
	\leq & C (\int_{0}^{t} \nrm{\rd_{t} \psi(t')}_{L^{2}_{x}} \, \ud t')^{1/2} (\nrm{\rd_{x} \psi(t)}_{L^{2}_{x}} + \nrm{\rd_{x} \psi(0)}_{L^{2}_{x}} )^{1/2} \\
	\leq & C T^{1/2} \nrm{\rd_{t} \psi}_{L^{\infty}_{t} L^{2}_{x}}^{1/2} \nrm{\rd_{x} \psi}_{L^{\infty}_{t} L^{2}_{x}}^{1/2} 
	\leq C T^{1/2} \nrm{\rd_{t,x} \psi}_{L^{\infty}_{t} L^{2}_{x}}.
\end{align*}

By the triangle inequality, \eqref{eq:est4L3nrm:0} follows. \qedhere
\end{proof}

\subsection{Reduction of the Main Theorem}
For $A_{\bfa}, A'_{\bfa}$ regular solutions to \eqref{eq:HPYM} (defined in \S \ref{subsec:overview}) on $I \times \bbR^{3} \times [0,1]$, we claim the existence of norms $\calI$ and $\dlt \calI$ which measure the sizes of $A_{\bfa}$ and $\dlt A_{\bfa}$, respectively, at $t=0$ (i.e., the size of the initial data), such that the theorems below hold. The precise definitions will be given in Section \ref{sec:pfOfIdEst}. 

\begin{theorem}[Estimates for initial data in the caloric-temporal gauge] \label{thm:idEst}
Let $0 < T \leq 1$, and $\Atemp_{\mu}$ a regular solution to the Yang-Mills equation in the temporal gauge $\Atemp_{0} = 0$ on $(-T, T) \times \bbR^{3}$ with the initial data $(\Aini_{i}, \Eini_{i})$ at $t=0$. Define $\calIini := \nrm{\Aini}_{\dot{H}^{1}_{x}} + \nrm{\Eini}_{L^{2}_{x}}$. Suppose that
\begin{equation} \label{eq:idEst:hypothesis}
	\sup_{t \in (-T, T)} \sup_{i} \nrm{\Atemp_{i}(t)}_{\dot{H}^{1}_{x}} < \dlt_{P},
\end{equation}
where $\dlt_{P}$ is a small constant to be introduced in Proposition \ref{prop:YMHF4A:lwp4deT}.  Then the following statements hold.
\begin{enumerate}
\item There exists a regular gauge transform $V = V(t,x)$ on $(-T, T) \times \bbR^{3}$ and a regular solution $A_{\bfa}$ to \eqref{eq:HPYM} on $(-T, T) \times \bbR^{3} \times [0,1]$ such that
\begin{equation} \label{eq:idEst:0}
	A_{\mu}(s=0) = V (\Atemp_{\mu}) V^{-1} - \rd_{\mu} V V^{-1}.
\end{equation}

\item Furthermore, the solution $A_{\bfa}$ satisfies the caloric-temporal gauge condition, i.e., $A_{s} = 0$ everywhere and $\Alow_{0} = 0$.

\item Let $\AtempPrime_{\mu}$ be another regular solution to the Yang-Mills equation in the temporal gauge with the initial data $(\Aini'_{i}, \Eini'_{i})$ satisfying $\nrm{(\Aini, \Eini)}_{\dot{H}^{1}_{x} \times L^{2}_{x}} \leq \calIini$ and \eqref{eq:idEst:hypothesis}. Let $A'_{\bfa}$ be the solution to \eqref{eq:HPYM} in the caloric-temporal gauge obtained from $\AtempPrime_{\mu}$ as in Parts (1) and (2).  Then the following initial data estimates hold:
\begin{equation} \label{eq:idEst:1}
	\calI \leq C_{\calIini} \cdot \calIini, \quad \dlt \calI \leq C_{\calIini} \cdot \dlt \calIini,
\end{equation}
where $\dlt \calIini := \nrm{\dlt \Aini}_{\dot{H}^{1}_{x}} + \nrm{\dlt \Eini}_{L^{2}_{x}}$. 
We remark that $\calI, \dlt \calI$ are defined in Section \ref{sec:pfOfIdEst}.
%We furthermore have
%\begin{equation} \label{eq:idEst:2}
%	\nrm{\Alow(t=0)}_{L^{3}_{x}} \leq D^{1}(\calIini) + \nrm{\Aini}_{L^{3}_{x}}, \quad \nrm{\dlt \Alow(t=0)}_{L^{3}_{x}} \leq D^{0}(\calIini) \, \dlt \calIini + \nrm{\dlt \Aini}_{L^{3}_{x}}.
%\end{equation}

\item Let $V'$ be the gauge transform obtained from $\AtempPrime_{i}$ as in Part (1), and let us write $\Vini := V(t=0), \Vini' := V'(t=0)$. For the latter two gauge transforms, the following estimates hold:
\begin{equation} \label{eq:idEst:3}
	\nrm{\Vini}_{L^{\infty}_{x}} \leq C_{\calIini}, \quad  \nrm{\rd_{x} \Vini}_{L^{3}_{x}} + \nrm{\rd_{x}^{(2)} \Vini}_{L^{2}_{x}} \leq C_{\calIini} \cdot \calIini,
\end{equation}
\begin{equation} \label{eq:idEst:4}
	\nrm{\dlt \Vini}_{L^{\infty}_{x}} + \nrm{\rd_{x} (\dlt \Vini)}_{L^{3}_{x}} + \nrm{\rd_{x}^{(2)} (\dlt \Vini)}_{L^{2}_{x}} \leq C_{\calIini} \cdot \dlt \calIini.
\end{equation}
\end{enumerate}

The same estimates with $\Vini$ and $\dlt \Vini$ replaced by $\Vini^{-1}$ and $\dlt \Vini^{-1}$, respectively, also hold.

\end{theorem}

\begin{theorem}[Estimates for $t$-evolution in the caloric-temporal gauge] \label{thm:dynEst}
Let $0 < T \leq 1$, and $A_{\bfa}$ a regular solution to the hyperbolic-parabolic Yang-Mills system \eqref{eq:HPYM} on $(-T, T) \times \bbR^{3} \times [0,1]$ in the caloric-temporal gauge. Then there exists $\dlt_{H} >0$ such that if
\begin{equation} \label{eq:dynEst:hypothesis}
	\calI < \dlt_{H},
\end{equation}
then we have
\begin{equation} \label{eq:dynEst:0}
	\sup_{0 \leq s \leq 1} \nrm{\rd_{t,x} A_{i}(s)}_{C_{t}((-T, T), L^{2}_{x})} 	
	+  \calA_{0}(-T, T)
	\leq C \calI,
\end{equation}
where $\calA_{0}(-T, T)$ is defined in \S \ref{subsec:defOfNorms}.
%\begin{equation} \label{eq:dynEst:1}
%	\sup_{0 \leq s \leq 1} \nrm{A_{i}(s)}_{C_{t}((-T, T), L^{3}_{x})} \leq C \calI + C \sup_{i} \nrm{\Alow_{i}(t=0)}_{L^{3}_{x}}.
%\end{equation}

Also, if $A'_{\bfa}$ is an additional solution to \eqref{eq:HPYM} on $(-T, T) \times \bbR^{3} \times [0,1]$ in the caloric-temporal gauge which also satisfies \eqref{eq:dynEst:hypothesis}, then we have
\begin{equation} \label{eq:dynEst:2}
	\sup_{0 \leq s \leq 1} \nrm{(\rd_{t,x} A_{i} - \rd_{t,x} A'_{i})(s)}_{C_{t}((-T, T), L^{2}_{x})} 
	+  \dlt \calA_{0}(-T, T)
	\leq C_{\calI} \cdot \dlt \calI,
\end{equation}
where $\dlt \calA_{0}(-T, T)$ is defined in \S \ref{subsec:defOfNorms}.
%\begin{equation} \label{eq:dynEst:3}
%	\sup_{0 \leq s \leq 1} \nrm{(A_{i}-A'_{i})(s)}_{C_{t}((-T, T), L^{3}_{x})} \leq C_{\calI} \, \dlt \calI + C \sup_{i} \nrm{(\Alow_{i}- \Alow'_{i})(t=0)}_{L^{3}_{x}}.
%\end{equation}

%Provided that the size of the initial data, measured by $\calI$, is small, we have estimates for $A_{i}$ and $A_{0}$, uniform in $0 \leq s \leq 1$ in terms of $\calI$. We also have estimates for the differences.
\end{theorem}

Our goal is to prove the Main Theorem, assuming Theorems \ref{thm:idEst} and \ref{thm:dynEst}.
\begin{proof} [Proof of the Main Theorem]
%\comment{Rewrite the bootstrap part! The point is to ensure that $\Atemp_{i}$ exists and the initial data analysis can be applied by the bootstrap procedure.}

In view of Lemma \ref{lem:mainThm:regApprox} (approximation lemma) and the fact that we are aiming to prove the difference estimates \eqref{eq:mainThm:2} and \eqref{eq:mainThm:3}, we will first consider initial data sets $(\Aini_{i}, \Eini_{i})$ which are regular in the sense of Definition \ref{def:reg4id}. Also, for the purpose of stating the estimates for differences, we will consider an additional regular initial data set $(\Aini'_{i}, \Eini'_{i})$. The corresponding solution will be also marked by a prime. The statements in this proof concerning a solution $A$ should be understood as being applicable to both $A$ and $A'$. 

Observe that $\calIini$ does not contain the $L^{3}_{x}$ norm of $\Aini$, and has the scaling property.
\begin{equation*}
	\calIini \to \lmb^{-1/2} \calIini
\end{equation*}
under the scaling of the Yang-Mills equation \eqref{eq:intro:scaling}. This allows us to treat the `local-in-time, large-data' case on an equal footing as the `unit-time, small-data' case. More precisely, we will assume by scaling that $\calIini$ is sufficiently small, and prove that the solution to the Yang-Mills equation exists on the time interval $(-1, 1)$. Unravelling the scaling at the end, the Main Theorem will follow. We remark that the length of the time interval of existence obtained by this method will be of size $\aeq \nrm{(\Aini, \Eini)}_{\dot{H}^{1}_{x} \times L^{2}_{x}}^{-2}$. 

Using Theorem \ref{thm:mainThm:H2lwp}, we obtain a unique solution $\Atemp_{\mu}$ to the hyperbolic Yang-Mills equation \eqref{eq:hyperbolicYM} under the temporal gauge condition $\Atemp_{0} = 0$. We remark that this solution is \emph{regular} by persistence of regularity. Let us denote by $T_{\star}$ the largest number $T > 0$ such that the solution $\Atemp_{\mu}$ exists smoothly on $(-T, T) \times \bbR^{3}$, and furthermore satisfies the following estimates for some $B>0$ and $C_{\calIini} > 0$:
\begin{equation} \label{eq:mainThm:pf:0}
\left\{
\begin{aligned}
	\nrm{\rd_{t,x} \Atemp_{i}}_{C_{t}((-T, T), L^{2}_{x})} \leq & B \calIini, \\
	\nrm{\rd_{t,x} \Atemp_{i} - \rd_{t,x} \AtempPrime_{i}}_{C_{t}((-T, T), L^{2}_{x})} \leq & C_{\calIini, B} \cdot \dlt \calIini, \\
	\nrm{\Atemp_{i} - \AtempPrime_{i}}_{C_{t}((-T, T), L^{3}_{x})} \leq & C_{\calIini, B} \cdot \dlt \calIini + \nrm{\Aini - \Aini'}_{L^{3}_{x}}.
\end{aligned}
\right.
\end{equation}

The goal is to show that $T_{\star} \geq 1$, provided that $\calIini > 0$ is small enough. 

We will proceed by a bootstrap argument. In view of the continuity of the norms involved, the inequalities \eqref{eq:mainThm:pf:0} are satisfied for $T > 0$ sufficiently small if $B \geq 2$ and $C_{\calIini, B} \geq 2$. Next, we claim that if we assume
\begin{equation} \label{eq:mainThm:pf:1}
\begin{aligned}
	\nrm{\rd_{t,x} \Atemp_{i}}_{C_{t}((-T, T), L^{2}_{x})} \leq & 2B \calIini.
\end{aligned}
\end{equation}
then we can recover \eqref{eq:mainThm:pf:0} by assuming $\calIini$ to be small enough and $T \leq 1$. 

Assuming the claim holds, let us first complete the proof of the Main Theorem. Indeed, suppose that \eqref{eq:mainThm:pf:0} holds for some $0 \leq T < 1$. Applying the difference estimate in \eqref{eq:mainThm:pf:0} to infinitesimal translations of $\Aini, \Eini$ and using the translation invariance of the Yang-Mills equation, we obtain
\begin{equation*}
 \nrm{\rd_{x} \rd_{t,x} \Atemp_{i}}_{C_{t}((-T, T), L^{2}_{x})}  < \infty.
\end{equation*}

This, in turn, allows us to apply Theorem \ref{thm:mainThm:H2lwp} ($H^{2}$ local well-posedness) to ensure that the solution $\Atemp_{i}$ extends uniquely as a regular solution to a larger time interval $(-T-\eps, T+\eps)$ for some $\eps > 0$. Taking $\eps > 0$ smaller if necessary, we can also ensure that the bootstrap assumption \eqref{eq:mainThm:pf:1} holds and $T + \eps \leq 1$. This, along with the claim, allows us to set up a continuity argument to show that a regular solution $\Atemp_{i}$ exists uniquely on the time interval $(-1, 1)$ and furthermore satisfies \eqref{eq:mainThm:pf:0} with $T =1$. From \eqref{eq:mainThm:pf:0}, the estimates \eqref{eq:mainThm:0} - \eqref{eq:mainThm:3} follow immediately, for regular initial data sets. Then by Lemma \ref{lem:mainThm:regApprox} and the difference estimates \eqref{eq:mainThm:2} and \eqref{eq:mainThm:3}, these results are extended to admissible initial data sets and solutions, which completes the proof of the Main Theorem. We remark that Part (3) of the Main Theorem follows from the persistence of regularity statement in Theorem \ref{thm:mainThm:H2lwp}.

Let us now prove the claim. Assuming $2 B \calIini < \dlt_{P}$, we can apply Theorem \ref{thm:idEst}. This provides us with a regular gauge transform $V$ and a regular solution $A_{\bfa}$ to \eqref{eq:HPYM} satisfying the caloric-temporal gauge condition, along with the following estimates at $t=0$:
\begin{equation*}
\begin{aligned}
	\nrm{\Vini}_{L^{\infty}_{x}} \leq C_{\calIini}, \quad 
	\nrm{\rd_{x} \Vini}_{L^{3}_{x}}  + \nrm{\rd_{x}^{(2)} \Vini}_{L^{2}_{x}} \leq C_{\calIini} \cdot \calIini, \\
	\calI \leq C_{\calIini} \cdot \calIini, \quad 
	\nrm{\Alow_{i}(t=0)}_{L^{3}_{x}} \leq C_{\calIini} \cdot \calIini + \nrm{\Aini}_{L^{3}_{x}}.
\end{aligned}\end{equation*}

The same estimates as the first two hold with $\Vini$ replaced by $\Vini^{-1}$. We remark that all the constants stated above are independent of $B > 0$. Applying Theorem \ref{thm:dynEst} with $\calIini$ small enough (so that $\calI  \leq C_{\calIini} \cdot \calIini$ is also small), we have
\begin{equation*}
	\sup_{0 \leq s \leq 1} \nrm{\rd_{t,x} A_{i}(s)}_{C_{t}((-T, T),  L^{2}_{x})} + \calA_{0}(-T, T)
	\leq C \calI \leq C_{\calIini} \cdot \calIini.
\end{equation*}
%\begin{equation*}
%\sup_{0 \leq s \leq 1} \nrm{A_{i}(s)}_{C_{t}((-T, T), L^{3}_{x})} \leq C \calI + C_{\calI} \, \sup_{i} \nrm{\Alow_{i}(t=0)}_{L^{3}_{x}}.
%\end{equation*}

Note that $V$ is a solution to the ODE \eqref{eq:est4gt2temporal:0}, which is unique by the standard ODE theory. Furthermore, in view of the estimates we have for $\calA_{0}(-T, T)$ and $\Vini$ in terms of $\calIini$, we may invoke Lemma \ref{lem:est4gt2temporal} to estimate the gauge transform $V$ in terms of $\calIini$. The same procedure can be used to obtain estimates for $V^{-1}$ in terms of $\calIini$. Then using the previous bound for $\rd_{t,x} A_{i}(s=0)$ and the gauge transform formula
\begin{equation*}
	\Atemp_{i} = V^{-1} A_{i}(s=0) V - \rd_{i} (V^{-1}) V,
\end{equation*}  
we obtain 
\begin{equation*}
\nrm{\rd_{t,x} \Atemp_{i}}_{C_{t}((-T, T), L^{2}_{x})} \leq C_{\calIini} \cdot \calIini.
\end{equation*}

Applying Lemma \ref{lem:est4L3nrm} and the initial data estimate for the $L^{3}_{x}$ norm of $\Alow_{i}$, we also obtain 
\begin{equation*}
\nrm{\Atemp_{i}}_{C_{t}((-T, T), L^{3}_{x})} <  C_{\calIini} \cdot \calIini + C_{\calIini} \nrm{\Aini}_{L^{3}_{x}}.
\end{equation*}

Furthermore, applying a similar procedure to the difference, we obtain
\begin{equation*}
	\nrm{\rd_{t,x} \Atemp_{i} - \rd_{t,x} \AtempPrime_{i}}_{C_{t}((-T, T), L^{2}_{x})} \leq C_{\calIini} \cdot \dlt\calIini,
\end{equation*}
\begin{equation*}
	\nrm{\Atemp_{i} - \AtempPrime_{i}}_{C_{t}((-T, T), L^{3}_{x})} \leq C_{\calIini} \cdot \dlt \calIini + \nrm{\Aini - \Aini'}_{L^{3}_{x}}.
\end{equation*}

Therefore, taking $B > 0$ sufficiently large (while keeping $2 B \calIini < \dlt_{P}$), we recover \eqref{eq:mainThm:pf:0}. \qedhere

%\comment{The point is to control the $\dot{H}^{1}_{x}$ and the $\dot{H}^{2}_{x}$ norms of $\Atemp_{i}$ via a bootstrap argument, showing that $\Atemp_{i}$ can be extended in $t$ and the corresponding Yang-Mills heat flow also exists until $s=1$ at every $t$. We will need to estimate the gauge transform.}
\end{proof}

The rest of this paper will be devoted to proofs of Theorems \ref{thm:idEst} and \ref{thm:dynEst}.

\section{Analysis of the covariant Yang-Mills heat flow}  \label{sec:covYMHF}
In this section, which serves as a preliminary to the proof of Theorem \ref{thm:idEst} to be given in Section \ref{sec:pfOfIdEst}, we will study the \emph{covariant Yang-Mills heat flow} \eqref{eq:cYMHF}, i.e.
\begin{equation*}
	F_{si} = \covD^{\ell} F_{\ell i},
\end{equation*}
which is the spatial part of \eqref{eq:dYMHF}. 

The original Yang-Mills heat flow \eqref{eq:YMHF} corresponds to the special case $A_{s} = 0$. Compared to \eqref{eq:YMHF}, which is covariant only under gauge transforms independent of $s$, the group of gauge transforms for \eqref{eq:cYMHF} is enlarged to those which may depend on $s$; at the level of the equation, this amounts to the extra freedom of choosing $A_{s}$. In \S \ref{subsec:covYMHF}, we will see that this additional gauge freedom may be used in our favor to obtain a genuinely semi-linear parabolic system of equations. Being parabolic, this system possesses a smoothing property, which lies at the heart of our proof of Theorem \ref{thm:idEst} in Section \ref{sec:pfOfIdEst}. The system is connected to \eqref{eq:YMHF} by a gauge transform $U$ solving a certain ODE, for which we will derive various estimates in \S \ref{subsec:covYMHFgt}. In \S \ref{subsec:covYMHFtech}, we will analyze a covariant parabolic equation satisfied by $B_{i} = F_{\nu i}$ for $\nu = 0,1,2,3$. As a byproduct of the results in \S \ref{subsec:covYMHF} - \ref{subsec:covYMHFtech}, we will obtain a proof of the following local existence result for \eqref{eq:YMHF} in \S \ref{subsec:pfOfLwp4YMHF}, which is different from the original one given by \cite{Rade:1992tu}.

\begin{theorem}[Local existence for \eqref{eq:YMHF} with $\dot{H}^{1}_{x}$ initial data
\footnote{In order to complete this theorem to a full local well-posedness result, we need to supplement it with a uniqueness statement. We omit such a statement here, as it will not be needed in the sequel. We refer the interested reader to \cite{Oh:2012fk}, where a proof of uniqueness in the class of regular solutions will be given.}
] \label{thm:lwp4YMHF}
Consider the initial value problem (IVP) for \eqref{eq:YMHF} with initial data $\Aini_{i} \in \dot{H}^{1}_{x}$ at $s=0$. Then the following statements hold.
\begin{enumerate}
\item There exists a number $s^{\star} = s^{\star} (\nrm{\Aini}_{\dot{H}^{1}_{x}}) > 0$, which is non-increasing in $\nrm{\Aini}_{\dot{H}^{1}_{x}}$, such that there exists a solution $A_{i} \in C_{s} ([0,s^{\star}],  \dot{H}^{1}_{x})$ to the IVP satisfying
\begin{equation} \label{eq:lwp4YMHF:1}
	\sup_{s \in [0,s^{\star}]} \nrm{A(s)}_{\dot{H}^{1}_{x}} \leq C  \nrm{\Aini}_{\dot{H}^{1}_{x}}.
\end{equation}

\item Let $\Aini'_{i} \in \dot{H}^{1}_{x}$ be another initial data set such that $\nrm{\Aini'}_{\dot{H}^{1}_{x}} \leq \nrm{\Aini}_{\dot{H}^{1}_{x}}$, and $A'$ the corresponding solution to the IVP on $[0, s^{\star}]$ given in (1). Then the following estimate for the difference $\dlt A := A - A'$ holds. 
\begin{equation} \label{eq:lwp4YMHF:2}
	\sup_{s \in [0,s^{\star}]} \nrm{\rd_{x}( \dlt A) (s)}_{\dot{H}^{1}_{x}} \leq C  \nrm{\dlt \Aini}_{\dot{H}^{1}_{x}}.
\end{equation}

\item If $\Aini_{i}(t)$ ($t \in I$) is a one parameter family of initial data such that $\Aini_{i} \in C^{\infty}_{t}(I, H^{\infty}_{x})$, then $A_{i} = A_{i}(t,x,s)$ given by (1) for each $t \in I$ satisfies $A_{i} \in C^{\infty}_{t,s}(I \times [0,s^{\star}], H^{\infty}_{x})$.
\end{enumerate}

\end{theorem}

This theorem itself is not needed for the rest of this paper, but will be used in \cite{Oh:2012fk}.

\subsection{Estimates for covariant Yang-Mills heat flow in the DeTurck gauge} \label{subsec:covYMHF}
Here, we will study \eqref{eq:cYMHF} in the DeTurck gauge $A_{s} = \rd^{\ell} A_{\ell}$, which makes \eqref{eq:cYMHF} a system of (semi-linear) strictly parabolic equations for $A_{i}$. Some parts of the standard theory for semi-linear parabolic equations, such as local-wellposedness and smoothing, will be sketched for later use. 

Let us begin by deriving the system of equations that we will study. Writing out \eqref{eq:cYMHF} in terms of $A_{i}, A_{s}$, we obtain
\begin{equation} \label{eq:YMHF4A:general}
	\rd_s A_i = \lap A_i + 2 \LieBr{A^\ell}{\rd_\ell A_i} - \LieBr{A^\ell}{\rd_i A_\ell} + \LieBr{A^\ell}{\LieBr{A_\ell}{A_i}} + \rd_i (A_s - \rd^\ell A_\ell) + \LieBr{A_i}{A_s - \rd^\ell A_\ell}.
\end{equation}

%This is just the usual Yang-Mills heat equation \eqref{eq:YMHF} if we set $A_{s} = 0$. In that case, \eqref{eq:YMHF4A:general} \emph{almost} constitutes a determined system of semi-linear heat equations, except for the term $\rd_i \rd^\ell A_\ell$. To cancel this troublesome term, we are motivated to set $A_s = \rd^\ell A_\ell$. The following proposition is an easy consequence of the computation \eqref{eq:YMHF4A:general}.
%\begin{proposition}
%	Suppose that $A_i, A_s$ are smooth $\LieAlg$-valued functions on $\bbR^3 \times [0,1]$ satisfying the following system of equations.
%	\begin{equation} \label{eq:YMHF4A:deT1}
%	\rd_s A_i - \lap A_i = 2 \LieBr{A^\ell}{\rd_\ell A_i} - \LieBr{A^\ell}{\rd_i A_\ell}  + \LieBr{A^\ell}{\LieBr{A_\ell}{A_i}},
%	\end{equation}
%	\begin{equation} \label{eq:YMHF4A:deT2}
%	A_s = \rd^\ell A_\ell.
%	\end{equation}
%	Then, viewing $(A_i, A_s)$ as a $\LieAlg$-valued connection 1-form on $\bbR^3 \times [0,1]$, the equation $F_{si} = \covD^\ell F_{\ell i}$ holds.
%\end{proposition}

Then using the DeTurck gauge condition $A_{s} = \rd^{\ell} A_{\ell}$, we obtain
\begin{equation} \label{eq:YMHF4A:deT1}
	\rd_s A_i - \lap A_i = 2 \LieBr{A^\ell}{\rd_\ell A_i} - \LieBr{A^\ell}{\rd_i A_\ell}  + \LieBr{A^\ell}{\LieBr{A_\ell}{A_i}},
\end{equation}

The study of this system, which is now strictly parabolic, will be the main subject of this subsection. For convenience, let us denote the right-hand side of the above equation by ${}^{(A_{i})} \calN$, so that $(\rd_{s} - \lap)A_{i} = {}^{(A_{i})} \calN$. Note that schematically,
\begin{equation*}
	{}^{(A_{i})} \calN = \calO(A, \rd_{x} A) + \calO(A, A, A).
\end{equation*}

We will study the initial value problem for \eqref{eq:YMHF4A:deT1} with the initial data $A_{i}(s=0) = \Aini_{i}$. We will also consider the difference of two nearby solutions. Given solutions $A, A'$ to \eqref{eq:YMHF4A:deT1} with initial data $\Aini, \Aini'$ respectively, we will denote the difference between the solutions and the initial data by $\dlt A_{i} := A_{i} - A'_{i}$ and $\dlt \Aini_{i} = \Aini_{i} - \Aini'_{i}$, respectively.
%Recall that we work with $\Aini_i$ smooth and decaying sufficiently fast at the spatial infinity.

Our first result is the local well-posedness for initial data $\Aini_{i} \in \dot{H}^{1}_{x}$. We refer the reader back to \S \ref{subsec:prelim:absPth} for the definition of $\calP^{3/4} \calH^{m}_{x}$. 

\begin{proposition}[$\dot{H}^{1}_{x}$ local well-posedness of \eqref{eq:cYMHF} in the DeTurck gauge] \label{prop:YMHF4A:lwp4deT} 
The following statements hold.
\begin{enumerate}
\item There exists a number $\dlt_{P} > 0$ such that for any initial data $\Aini_{i} \in \dot{H}^{1}_{x}$ with
\begin{equation} \label{eq:YMHF4A:lwp4deT:hypothesis}
	\nrm{\Aini}_{\dot{H}^{1}} \leq \dlt_{P},
\end{equation}
there exists a unique solution $A_{i} = A_{i}(x,s) \in C_{s}([0,1], \dot{H}^{1}_{x}) \cap \calL^{2}_{s} ((0,1], \dot{H}^{2}_{x})$ to the equation \eqref{eq:YMHF4A:deT1} on $s \in [0,1]$, which satisfies 
\begin{equation} \label{eq:YMHF4A:lwp4deT:est4A}
	\nrm{\rd_{x} A}_{\calP^{3/4} \calH^{2}_{x}(0,1]} \leq C \nrm{\Aini}_{\dot{H}^{1}_{x}}.
\end{equation}
\item For $A_{i}, A'_{i}$ solutions to \eqref{eq:YMHF4A:deT1} with initial data $\Aini_{i}, \Aini'_{i}$ satisfying \eqref{eq:YMHF4A:lwp4deT:hypothesis} given by (1), respectively, the following estimate hold for the difference $\dlt A_{i}$
\begin{equation} \label{eq:YMHF4A:lwp4deT:est4dltA}
	\nrm{\rd_{x} (\dlt A)}_{\calP^{3/4} \calH^{2}_{x}(0,1]} \leq C_{\nrm{\Aini}_{\dot{H}^{1}_{x}}, \nrm{\Aini'}_{\dot{H}^{1}_{x}}}  \nrm{\dlt \Aini}_{\dot{H}^{1}_{x}}.
\end{equation}
\item We have \emph{persistence of regularity} and \emph{smooth dependence on the initial data}. In particular, let $\Aini_{i}(t)$ ($t \in I$ for some open interval $I$) be a one parameter family of initial data such that $\Aini_{i} \in C^{\infty}_{t} (I, H^{\infty}_{x})$. Then the solution $A_{i}$ given by (1) satisfies $A_{i} \in C^{\infty}_{t,x}(I \times [0,1], H^{\infty}_{x})$.
\end{enumerate}
\end{proposition}

\begin{proof} 
This is a standard result. We will only present the proof of the {\it a priori} estimate \eqref{eq:YMHF4A:lwp4deT:est4A}, which means that we will assume the existence of a solution $A_{i}$ to \eqref{eq:YMHF4A:deT1} with initial data $\Aini_{i}$, which we may assume furthermore to be in $H^{\infty}_{x}$. As usual, a small variant of the arguments for the proof of the {\it a priori} estimate leads to estimates needed to run a Picard iteration argument (in a subset of a suitable Banach space), from which existence and uniqueness follows. Similar arguments applied to the equation for the difference $\dlt A_{i}$ and the differentiated equation for $\rd_{x}^{(m)} A_{i}$ will prove \eqref{eq:YMHF4A:lwp4deT:est4dltA} and persistence of regularity, respectively. For a parametrized family of data, differentiation with respect to the parameter yields smooth dependence on the initial data. We leave these standard details to the interested reader.

In order to derive the {\it a priori} estimate for $\rd_{x} A_{i}$, let us differentiate the equation. We then obtain
\begin{equation*} 
	\rd_{s} (\rd_{x} A_{i}) - \lap (\rd_{x} A_{i}) = s^{-1} \nb_{x} \calO(A, \nb_{x} A) + s^{-1/2} \calO(A, A, \nb_{x} A) =: {}^{(\rd_{x} A_{i})} \calN.
\end{equation*}

Let us work on a subinterval $(0, \subr] \subset (0,1)$, assuming the bootstrap assumption $\nrm{\rd_{x} A}_{\calP^{3/4} \calH^{2}_{x}(0, \subr]} \leq 10 \eps$, where $\eps = \nrm{\overline{A}}_{\dot{H}^{1}_{x}}$ is the size of the initial data. As $A \in H^{\infty}_{x}$, we have 
\begin{equation*}
\limsup_{\subr \to 0} \nrm{\rd_{x} A}_{\calP^{3/4} \calH^{2}_{x}(0, \subr]} 
= \limsup_{\subr \to 0} \nrm{\rd_{x} A}_{\calP^{3/4} \calL^{2}_{x}(0, \subr]} = \nrm{\rd_{x} \overline{A}}_{L^{2}_{x}},
\end{equation*}
so the assumption holds for $\subr > 0$ small enough. 

Note the obvious inequalities 
\begin{equation*} 
	\nrm{\rd_{x} (\phi_{1} \rd_{x} \phi_{2})}_{L^{2}_{x}} \leq C \nrm{\phi_{1}}_{\dot{H}^{3/2}_{x} \cap L^{\infty}_{x}} \nrm{\phi_{2}}_{\dot{H}^{2}_{x}}, \quad
	\nrm{\phi_{1} \phi_{2} \rd_{x} \phi_{3}}_{L^{2}_{x}}  \leq C \nrm{\phi_{1}}_{\dot{H}^{1}_{x}} \nrm{\phi_{2}}_{\dot{H}^{1}_{x}} \nrm{\phi_{3}}_{\dot{H}^{2}_{x}},
\end{equation*}
which follow from H\"older and Sobolev. Using the Correspondence Principle, H\"older for $\calL^{\ell, p}_{s}$ (Lemma \ref{lem:absP:Holder4Ls}) and Gagliardo-Nirenberg (Lemma \ref{lem:absP:algEst}), we obtain
\begin{equation} \label{eq:YMHF4A:lwp4deT:pf:1}
\begin{aligned}
\nrm{s^{-1} \nb_{x} \calO(\psi_{1}, \nb_{x} \psi_{2})}_{\calL^{3/4+1,p}_{s} \calL^{2}_{x}} 
= & \nrm{\nb_{x} \calO(\psi_{1}, \nb_{x} \psi_{2})}_{\calL^{3/4,p}_{s} \calL^{2}_{x}} \\
\leq & C \nrm{\psi_{1}}_{\calL^{1/4,\infty}_{s} (\dot{\calH}^{3/2}_{x} \cap \calL^{\infty}_{x})} \nrm{\psi_{2}}_{\calL^{1/4,2}_{s} \dot{\calH}^{2}_{x}} \\
\leq & C \nrm{\rd_{x} \psi_{1}}_{\calL^{3/4,\infty}_{s} \calL^{2}_{x}}^{1/2} \nrm{\rd_{x} \psi_{1}}_{\calL^{3/4,\infty}_{s} \dot{\calH}^{1}_{x}}^{1/2} \nrm{\rd_{x} \psi_{2}}_{\calL^{3/4,2}_{s} \dot{\calH}^{1}_{x}}, 
\end{aligned}
\end{equation}
\begin{equation} \label{eq:YMHF4A:lwp4deT:pf:2}
\begin{aligned}
\nrm{s^{-1/2} \calO(\psi_{1}, \psi_{2}, \nb_{x} \psi_{3})}_{\calL^{3/4+1,p}_{s} \calL^{2}_{x}} 
= & \nrm{\calO(\psi_{1}, \psi_{2}, \nb_{x} \psi_{3})}_{\calL^{5/4,p}_{s} \calL^{2}_{x}} \\
\leq & C \nrm{\psi_{1}}_{\calL^{1/4,\infty}_{s} \dot{\calH}^{1}_{x}} \nrm{\psi_{2}}_{\calL^{1/4,\infty}_{s} \dot{\calH}^{1}_{x}} \nrm{\psi_{3}}_{\calL^{1/4,2}_{s} \dot{\calH}^{2}_{x} } \\
\leq & C \nrm{\rd_{x} \psi_{1}}_{\calL^{3/4,\infty}_{s} \calL^{2}_{x}} \nrm{\rd_{x} \psi_{2}}_{\calL^{3/4,\infty}_{s} \calL^{2}_{x}} \nrm{\rd_{x} \psi_{3}}_{\calL^{3/4,2}_{s} \dot{\calH}^{1}_{x} },
\end{aligned}
\end{equation}
for both $p=1,2$. Applying these inequalities with $\psi_{j} = A$, note that each factor on the right-hand sides of the above two inequalities is controlled by $\nrm{\rd_{x} A}_{\calP^{3/4} \calH^{2}_{x}}$. Using the bootstrap assumption, we have for $p=1,2$
\begin{equation} \label{eq:YMHF4A:lwp4deT:pf:3}
	\sup_{i} \nrm{{}^{(\rd_{x} A_{i})} \calN}_{\calL^{3/4+1,p}_{s} \calL^{2}_{x}(0, \subr]} \leq C (\eps + \eps^{2}) \nrm{\rd_{x} A}_{\calP^{3/4} \calH^{2}_{x}(0, \subr]}.
\end{equation}

Taking $\eps > 0$ small enough and applying Theorem \ref{thm:absP:absPth}, we beat the bootstrap assumption, i.e., $\nrm{\rd_{x} A}_{\calP^{3/4} \calH^{2}_{x}} \leq 5 \eps$. By a standard bootstrap argument, we conclude that \eqref{eq:YMHF4A:lwp4deT:est4A} holds on $[0,1]$. \qedhere
\end{proof}

An important property of the parabolic PDE is that it is \emph{infinitely} and \emph{immediately smoothing}. Quantitavely, this means that smoother norms of the solution $A$ becomes controllable for $s > 0$ in terms of a rougher norm of the initial data $\Aini$. We will see many manifestations of this property throughout the paper. Here, we give a version of the infinite, immediate smoothing for the solutions to the equation \eqref{eq:YMHF4A:deT1}.

\begin{proposition}[Smoothing estimates] \label{prop:YMHF4A:smth4deT}
Let $\Aini_{i}, \Aini'_{i}$ be $\dot{H}^{1}_{x}$ initial data satisfying \eqref{eq:YMHF4A:lwp4deT:hypothesis}, and $A_{i}, A'_{i}$ the corresponding unique solutions to \eqref{eq:YMHF4A:deT1} given by Proposition \ref{prop:YMHF4A:lwp4deT}, respectively.
\begin{enumerate}
\item For integers $m \geq 1$, the following estimate for $A_{i}$ (and of course also for $A'_{i}$) holds.
	\begin{equation} \label{eq:YMHF4A:smth4A} 
		\nrm{\rd_{x} A}_{\calP^{3/4} \dot{\calH}^{m+2}_{x}(0,1]}
		 \leq C_{m, \nrm{\Aini}_{\dot{H}^{1}_{x}}}  \nrm{\Aini}_{\dot{H}^{1}_{x}}.
	\end{equation}
\item Furthermore, for integers $m \geq 1$, the following estimate for the difference also holds.
	\begin{equation} \label{eq:YMHF4A:smth4dltA}
		\nrm{\rd_{x} (\dlt A)}_{\calP^{3/4} \dot{\calH}^{m+2}_{x}(0,1]} 
		\leq C_{m, \nrm{\Aini}_{\dot{H}^{1}_{x}}, \nrm{\Aini}_{\dot{H}^{1}_{x}}}  \nrm{\dlt \Aini}_{\dot{H}^{1}_{x}}.
	\end{equation}
\end{enumerate}
\end{proposition}
\begin{proof} 
We now work on the whole interval $s \in (0, 1]$. We will prove only the non-difference estimate \eqref{eq:YMHF4A:smth4A}, as the difference analogue \eqref{eq:YMHF4A:smth4dltA} can be proved in a similar manner.

By approximation, it suffices to consider $A_{i} \in C^{\infty}_{s}([0,1], H^{\infty}_{x})$. Using Leibniz's rule and the estimates \eqref{eq:YMHF4A:lwp4deT:pf:1}, \eqref{eq:YMHF4A:lwp4deT:pf:2}, for $m \geq 1$, we immediately obtain
\begin{align*}
\sup_{i} \nrm{{}^{(\rd_{x} A_{i})} \calN}_{\calL^{3/4+1}_{s} \dot{\calH}^{m}_{x}} 
\leq & C \nrm{\rd_{x} A}_{\calL^{3/4,\infty}_{s} \dot{\calH}^{m}_{x}}^{1/2} \nrm{\rd_{x} A}_{\calL^{3/4,\infty}_{s} \dot{\calH}^{m+1}_{x}}^{1/2} \nrm{\rd_{x} A}_{\calL^{3/4,2}_{s} \dot{\calH}^{1}_{x}} \\
&+ C \nrm{\rd_{x} A}_{\calL^{3/4,\infty}_{s} \calH^{m}_{x}} \nrm{\rd_{x} A}_{\calL^{3/4,2}_{s} \calH^{m+1}_{x}} 
+ C \nrm{\rd_{x} A}_{\calL^{3/4,\infty}_{s} \calH^{m}_{x}}^{2} \nrm{\rd_{x} A}_{\calL^{3/4,2}_{s} \calH^{m+1}_{x}}.
\end{align*}

The second and third term on the right hand side can be controlled by $C \nrm{\rd_{x} A}_{\calP^{3/4} \calH^{m+1}_{x}}^{2}$ and $C \nrm{\rd_{x} A}_{\calP^{3/4} \calH^{m+1}_{x}}^{3}$, respectively. On the other hand, the first term is problematic as $\nrm{\rd_{x} A}_{\calL^{3/4,\infty}_{s} \dot{\calH}^{m+1}_{x}}$ is not controlled by $\nrm{\rd_{x} A}_{\calP^{3/4} \calH^{m+1}_{x}}$. In this case, we apply Cauchy-Schwarz and estimate it by 
\begin{equation*}
\eps^{-1} C \nrm{\rd_{x} A}_{\calP^{3/4} \calH^{m+1}_{x}}^{3} + \eps \nrm{\rd_{x} A}_{\calP^{3/4} \dot{\calH}^{m+2}_{x}}.
\end{equation*}

Therefore, we obtain \eqref{eq:absP:smth:1} for all $m \geq 1$ with $\psi = \rd_{x} A$, $X=L^{2}_{x}$, $\ell = 3/4$ and $\calB_{m}(r) = C r^{2} + C r^{3}$. Applying the second part of Theorem \ref{thm:absP:absPth}, the smoothing estimate \eqref{eq:YMHF4A:smth4A} easily follows.
\end{proof}

The following statement is crucial for estimating the gauge transform into the caloric gauge $A_s =0$. It is a corollary of the proof of Proposition \ref{prop:YMHF4A:lwp4deT}.

\begin{corollary} \label{cor:YMHF4A:bnd4lapAs}
Let $\Aini_{i}, \Aini'_{i}$ be $\dot{H}^{1}_{x}$ initial data satisfying \eqref{eq:YMHF4A:lwp4deT:hypothesis}, and $A_{i}, A'_{i}$ the corresponding unique solutions to \eqref{eq:YMHF4A:deT1} given by Proposition \ref{prop:YMHF4A:lwp4deT}, respectively. Consider also $A_{s}, A'_{s}$ given by the equations
\begin{equation*}
	A_{s} = \rd^{\ell} A_{\ell}, \qquad A_{s}' = \rd^{\ell} A'_{\ell}.
\end{equation*}

Then the following estimate holds for $\lap A_{s}$.
\begin{equation}\label{eq:YMHF4A:bnd4lapAs}
	\sup_{0 < s \leq 1} \nrm{\int_s^1 \lap A_s (s') \ud s'}_{L^2_x} \leq C_{\nrm{\Aini}_{\dot{H}^{1}_{x}}}   \nrm{\Aini}_{\dot{H}^{1}_{x}}.
\end{equation}

Furthermore, the following estimate holds for $\lap (\dlt A_{s})$.
\begin{equation}\label{eq:YMHF4A:bnd4lapDltAs}
	\sup_{0 < s \leq 1} \nrm{\int_s^1 \lap (\dlt A_s )(s') \ud s'}_{L^2_x} \leq C_{\nrm{\Aini}_{\dot{H}^{1}_{x}}, \nrm{\Aini'}_{\dot{H}^{1}_{x}}} \nrm{\dlt \Aini}_{\dot{H}^{1}_{x}},
\end{equation}
\end{corollary}

\begin{proof}
Here, we will only give a proof of \eqref{eq:YMHF4A:bnd4lapAs}, the proof of \eqref{eq:YMHF4A:bnd4lapDltAs} being similar. Again it suffices to consider $A_{i}$ such that $\rd_{x} A_{i}$ (and therefore $A_{s}$ also) is regular.

Taking $\rd^\ell$ of the equations $\rd_s A_\ell - \lap A_\ell = {}^{(A_{\ell})} \calN$, we get a parabolic equation for $A_s$ of the form $\rd_s A_s - \lap A_s = \sum_{\ell} \rd_{\ell} {}^{(A_{\ell})} \calN$. Integrating this equation from $s$ to $1$, we obtain the following identity.
\begin{equation*}
	\int_s^1 \lap A_s (s') \ud s' = A_s (1) - A_s (s) - \sum_{\ell} \int_s^1  \rd_{\ell} {}^{(A_{\ell})} \calN (s')  \ud s'.
\end{equation*}

Let us take the $L^{2}_{x}$ norm of both sides and take the supremum over $s \in (0,1]$. The first two terms on the right-hand side are acceptable, in view of the fact that $A_{s} = \calO(\rd_{x} A)$. Using Minkowski, it is not difficult to see that in order to estimate the contribution of the last term, it suffices to establish 
\begin{equation} \label{eq:bnd4Ns}
	\sup_{i} \nrm{{}^{(\rd_{x} A_{i})} \calN}_{\calL^{3/4+1,1}_{s} \calL^{2}_{x}(0,1]} \leq C_{\nrm{\Aini}_{\dot{H}^{1}_{x}}}  \nrm{\Aini}_{\dot{H}^{1}_{x}}.
\end{equation}

This immediately follows by combining \eqref{eq:YMHF4A:lwp4deT:est4A} with \eqref{eq:YMHF4A:lwp4deT:pf:3}, recalling that $\eps = \nrm{\Aini}_{\dot{H}^{1}_{x}}$.  \qedhere
\end{proof}

\subsection{Estimates for gauge transform to the caloric gauge $A_{s} = 0$} \label{subsec:covYMHFgt}
In the previous subsection, we analyzed \eqref{eq:cYMHF} under the DeTurck gauge condition $A_{s} = \rd^{\ell} A_{\ell}$, which led to a nice system of semi-linear parabolic equations. In this subsection, we present estimates for gauge transforms for the solution to \eqref{eq:cYMHF} in the DeTurck gauge into the caloric gauge $A_{s}= 0$.

\begin{lemma} \label{lem:est4gt2caloric}
Fix $s_{0} \in [0,1]$. Let $\Aini_{i}, \Aini'_{i}$ be $\dot{H}^{1}_{x}$ initial data sets satisfying \eqref{eq:YMHF4A:lwp4deT:hypothesis}, and $A_{i}, A'_{i}$ the corresponding unique solutions to \eqref{eq:cYMHF} in the DeTurck gauge given by Proposition \ref{prop:YMHF4A:lwp4deT}, respectively. Let us consider the following ODE on $\bbR^{3} \times [0,1]$.
\begin{equation*}
\left \{
\begin{aligned}
	\rd_{s} U &= U A_{s} \\
	U(s=s_{0}) &= \mathrm{Id},
\end{aligned}
\right.
\end{equation*}
where we remind the reader that $A_{s} = \rd^{\ell} A_{\ell}$. Then the following statements hold.

\begin{enumerate}
\item There exists a unique solution $U$ such that $U(x,s) \in \LieGrp$ for all $s \in [0,1]$, and $U$ obeys the following estimates for $m=2$.
\begin{equation} \label{eq:est4gt2caloric:U:1}
	\nrm{U}_{L^{\infty}_{s} L^{\infty}_{x} [0,1]} 
	\leq C_{\nrm{\Aini}_{\dot{H}^{1}_{x}}}, \quad
	\nrm{\rd_{x} U}_{L^{\infty}_{s} L^{3}_{x} [0,1]} 
	\leq C_{\nrm{\Aini}_{\dot{H}^{1}_{x}}} \nrm{\Aini}_{\dot{H}^{1}_{x}}.
\end{equation}
\begin{equation} \label{eq:est4gt2caloric:U:2}
	\nrm{s^{(m-2)/2} \rd_{x}^{(m)} U}_{L^{\infty}_{s} L^{2}_{x}[0,1]} 
	\leq C_{m, \nrm{\Aini}_{\dot{H}^{1}_{x}}} \nrm{\Aini}_{\dot{H}^{1}_{x}}.
\end{equation}

\item Let $U'$ be the solution to the same ODE with $A_{s}$ replaced by $A'_{s}$, which possesses the identical initial data $U'(s=1) = \mathrm{Id}$. The difference $\dlt U := U - U'$ satisfies the following estimates for $m=2$.
\begin{equation} \label{eq:est4gt2caloric:dltU:1}
	\nrm{\dlt U}_{L^{\infty}_{s} L^{\infty}_{x} [0,1]} + \nrm{\rd_{x} (\dlt U)}_{L^{\infty}_{s} L^{3}_{x} [0,1]} 
	\leq C_{\nrm{\Aini}_{\dot{H}^{1}_{x}}} \nrm{\dlt \Aini}_{\dot{H}^{1}_{x}}.
\end{equation}
\begin{equation} \label{eq:est4gt2caloric:dltU:2}
	\nrm{s^{(m-2)/2} \rd_{x}^{(m)} (\dlt U)}_{L^{\infty}_{s} L^{2}_{x}[0,1]} 
	\leq C_{m, \nrm{\Aini}_{\dot{H}^{1}_{x}}} \nrm{\dlt \Aini}_{\dot{H}^{1}_{x}}.
\end{equation}

\item Furthermore, if $s_{0} = 1$, then \eqref{eq:est4gt2caloric:U:2} and \eqref{eq:est4gt2caloric:dltU:2} hold for all integers $m \geq 3$ as well. 
%\footnote{In fact, an inspection of the proof in Appendix \ref{sec:gt} shows that one only needs the assumption $s_{0} > 0$ .}

\item Finally, all of the above statements with $U$, $\dlt U$ replaced by $U^{-1}$, $\dlt U^{-1}$, respectively, also hold.
\end{enumerate}
\end{lemma}

We defer the proof of Lemma \ref{lem:est4gt2caloric} to Appendix \ref{sec:gt}.

\begin{remark}
We remark that Lemma \ref{lem:est4gt2caloric} will be for us the analogue of Uhlenbeck's lemma\footnote{Which states, roughly speaking, that there exists a gauge transform (with good regularity properties) which transforms a given connection 1-form $A_{i}$ into the Coulomb gauge, provided that the $L^{3/2}_{x}$ norm of $F_{ij}$ is small.} \cite{Uhlenbeck:1982vna}, on which the work \cite{Klainerman:1995hz} crucially rely, in the following sense: 
Heuristically, an application of this lemma with $s_{0} = 1$, combined with the smoothing estimates of Proposition \ref{prop:YMHF4A:smth4deT}, amounts to transforming a given initial data set to another whose curl-free part is `smoother'. On the other hand, Uhlenbeck's lemma sets the curl-free part to be exactly zero.

In the simpler case of an abelian gauge gauge theory (i.e., Maxwell's equations), this heuristic can be demonstrated in a more concrete manner as follows: In this case, the connection component $A_s$ will exist all the way to $s \to \infty$, and will converge to zero in a suitable sense. Note furthermore that $\rd^{\ell} F_{s\ell} = \rd_{s} (\rd^{\ell }A_{\ell}) - \lap A_{s} = 0$. Therefore, this lemma, if applied with `$s_{0}=\infty$', transforms the initial data to one such that the curl-free part is zero, i.e., one satisfying the Coulomb gauge condition.
\end{remark}

%\begin{remark} 
%As remarked in the Introduction, the procedure of solving the system with $A_{s} = \rd^{\ell} A_{\ell}$ and applying a gauge transform back to $A_{s} =0$ nothing but the well-known \emph{de Turck trick} in disguise. Ours differs from the original one, though, in that we prescribe the data for the ODE for the gauge transform at $s=1$ instead of $s=0$, in order to retain the regularity obtained by analyzing the semi-linear parabolic equations. See \cite{?} for examples of the use of the de Turck trick in its original form.
%\end{remark}

\subsection{Linear covariant parabolic equation for $F_{0i}$} \label{subsec:covYMHFtech}
In this subsection, we will prove a technical well-posedness proposition for a certain covariant parabolic equation, which is satisfied by $B_{i} = F_{\nu i}$. (See Appendix \ref{sec:HPYM})
\begin{proposition} \label{prop:linCovHeat}
Fix $m \geq 2$, and let $A_{i}, A_{s}$ be connection coefficients such that $A_{i}, A_{s} \in C^{\infty}_{s}([0,1], H^{\infty}_{x})$. 

Consider the following initial value problem for the linear parabolic equation
\begin{equation} \label{eq:linCovHeat:ivp}
\left\{
\begin{aligned}
	\covD_{s} B_{i} - \covD^{\ell} \covD_{\ell} B_{i} =& 2 \LieBr{F_{i\ell}}{B^{\ell}}, \\
	B_{i}(s=0) =& \Bini_{i},
\end{aligned}
\right.
\end{equation}
where the initial data $\Bini_{i} \in H^{\infty}_{x}$. Then the following statements hold.

\begin{enumerate}
\item There exists a unique regular solution $B_{i} = B_{i}(x, s)$ on $[0, 1]$ to the problem \eqref{eq:linCovHeat:ivp}, i.e., $B_{i} \in C^{\infty}_{s}([0,1], H^{\infty}_{x})$. 
\item Assume furthermore that we have the following bounds for $A_{i}$ and $A_{s}$:
\begin{equation} \label{eq:linCovHeat:hypothesis}
	\sup_{i} \nrm{\nb_{x} A_{i}}_{\calL^{1/4,\infty}_{s} \calH^{m-1}_{x}(0,1]} + \nrm{A_{s}}_{\calL^{3/4,\infty}_{s} \calH^{m-1}_{x} (0,1]} \leq \calC < \infty.
\end{equation}

Then the solution $B_{i}$ obtained in (1) satisfies the following estimate.
\begin{equation} \label{eq:linCovHeat:est4B}
	\sup_{i} \nrm{B_{i}}_{\calP^{3/4} \calH_{x}^{m}(0, 1]} \leq C_{\calC} \sup_{i} \nrm{\Bini_{i}}_{L^{2}_{x}}.
\end{equation}

\item Let $A_{i} = A_{i}(t,x,s)$, $A_{s} = A_{s} (t,x,s)$ be a family of coefficients, parametrized by $t \in I$, such that $A_{i}, A_{s} \in C^{\infty}_{t,s}(I \times [0,1], H^{\infty}_{x})$. Consider the corresponding one parameter family of IVPs \eqref{eq:linCovHeat:ivp}, where the initial data sets $\Bini_{i}(t)$ are also parametrized by $t \in I$, in such a way that $\Bini_{i} = \Bini_{i}(t,x)$ belongs to $C^{\infty}_{t} (I, H^{\infty}_{x})$

Then the solution $B_{i} = B_{i}(t, x, s)$ on $I \times \bbR^{3} \times [0,1]$ (obtained by applying (1) to each $t$) belongs to $C^{\infty}_{t,s}(I \times [0,1], H^{\infty}_{x})$.
\end{enumerate}
\end{proposition}

\begin{proof} 
%As a first step, let us derive an appropriate bound for $F_{i\ell}$. We claim that the following estimate holds:
%\begin{equation} \label{eq:linCovHeat:pf:1}
%	\nrm{F_{i\ell}}_{\calL^{3/4,\infty}_{s} \widehat{\calH}^{m-1}_{x}} + \nrm{\calO(A, A)}_{\calL^{3/4,\infty}_{s} \widehat{\calH}^{m-1}_{x}} \leq C(\calC + \calC^{2}).
%\end{equation}
%
%Indeed, recalling the definition of $F_{ij}$, we estimate for $1 \leq k \leq m-1$
%\begin{equation*}
%	\nrm{F_{i\ell}}_{\calL^{3/4,\infty}_{s} \dot{\calH}^{k}_{x}} \leq 2 \nrm{\rd_{x} A}_{\calL^{3/4,\infty}_{s} \dot{\calH}^{k}_{x}} + \nrm{\calO(A,A)}_{\calL^{3/4,\infty}_{s} \dot{\calH}^{k-1}_{x}}.
%\end{equation*}
%
%The first term on the right-hand side can be bounded by $2 \calC$ by hypothesis. For the second term, we use the inequality $\nrm{\calO(\phi, \phi)}_{\dot{H}^{k}_{x}} \leq C \nrm{\phi}_{\dot{H}^{3/2}_{x} \cap L^{\infty}_{x}}\nrm{\phi}_{\dot{H}^{k}_{x}}$, the Correspondence Principle, Lemma \ref{lem:absP:Holder4Ls} and Lemma \ref{lem:absP:algEst}, obtaining
%\begin{equation*}
%	\nrm{\calO(A, A)}_{\calL^{3/4,\infty}_{s} \dot{\calH}^{k}_{x}} \leq C \nrm{A}_{\calL^{3/4,\infty}_{s} \calL^{\infty}_{x}} \nrm{A}_{\calL^{3/4,\infty}_{s} \dot{\calH}^{k}_{x}} \leq C \calC^{2},
%\end{equation*}
%since $m \geq 3$ and $k \geq 1$. Now summing over $k$, we obtain \eqref{eq:linCovHeat:pf:1}.

%Next, let us prove global existence and uniqueness, as well as \eqref{eq:linCovHeat:est4B}. 
As in the proof of Proposition \ref{prop:YMHF4A:lwp4deT}, we will present only the proof of the estimate \eqref{eq:linCovHeat:est4B} of Part (2) under the assumption that a regular solution $B$ already exists. The actual existence, uniqueness, persistence of regularity and stability required to justify Parts (1), (3) follow from a standard Picard iteration argument, which can be set up by a slightly modifying of the argument below. We leave the details of the procedure to the interested reader.

Let us begin by rewriting the equation \eqref{eq:linCovHeat:ivp} so that it is manifestly a semi-linear equation for the vector-valued unknown $B$:
\begin{equation*} 
\begin{aligned}
	\rd_{s} B_{i} - \lap B_{i} 
	= & 2 \LieBr{A^{\ell}}{\rd_{\ell} B_{i}} + \LieBr{\rd^{\ell} A_{\ell} - A_{s}}{B_{i}} + \LieBr{A^{\ell}}{\LieBr{A_{\ell}}{B_{i}}} + 2 \LieBr{F_{i\ell}}{B^{\ell}} \\
	= & s^{-1/2} \calO(A, \nb_{x} B) + s^{-1/2} \calO(\nb_{x} A, B) + \calO(A_{s}, B) + \calO(A, A, B).
\end{aligned}
\end{equation*}

Note the following inequalities, which follow easily from H\"older and Sobolev.
\begin{equation} \label{eq:linCovHeat:pf:0}
\left\{
\begin{aligned}
& \nrm{\phi_{1} \rd_{x} \phi_{2}}_{L^{2}_{x}} + \nrm{\rd_{x} \phi_{1} \phi_{2}}_{L^{2}_{x}} 
\leq C \nrm{\phi_{1}}_{\dot{H}^{3/2}_{x} \cap L^{\infty}_{x}} \nrm{\phi_{2}}_{\dot{H}^{1}_{x}}, \\
& \nrm{\phi_{1} \phi_{2}}_{L^{2}_{x}} 
\leq C \nrm{\phi_{1}}_{\dot{H}^{1/2}_{x}} \nrm{\phi_{2}}_{\dot{H}^{1}_{x}}, \quad
\nrm{\phi_{1} \phi_{2} \phi_{3}}_{L^{2}_{x}} 
\leq C \nrm{\phi_{1}}_{\dot{H}^{1}_{x}} \nrm{\phi_{2}}_{\dot{H}^{1}_{x}} \nrm{\phi_{3}}_{\dot{H}^{1}_{x}}.
\end{aligned}
\right.
\end{equation}

Fix $(0, \subr] \subset (0, 1]$. Applying the Correspondence Principle, H\"older for $\calL^{\ell, p}_{s}$ (Lemma \ref{lem:absP:Holder4Ls}), Gagliardo-Nirenberg (Lemma \ref{lem:absP:algEst}) and interpolation, we obtain the following set of inequalities on $(0, \subr]$.
\begin{equation} \label{eq:linCovHeat:pf:1}
\left\{
\begin{aligned}
&\nrm{s^{-1/2} \calO(\psi_{1}, \nb_{x} \psi_{2})}_{\calL^{3/4+1,q}_{s} \calL^{2}_{x}}
+\nrm{s^{-1/2} \calO(\nb_{x} \psi_{1}, \psi_{2})}_{\calL^{3/4+1,q}_{s} \calL^{2}_{x}} \\
& \qquad \leq C \nrm{\nb_{x} \psi_{1}}_{\calL^{1/4,\infty}_{s} \calH^{1}_{x}} \nrm{s^{1/4-\eps'} \, \psi_{2}}_{\calL^{3/4,2}_{s} \dot{\calH}^{1}_{x}},\\
&\nrm{\calO(\psi_{0}, \psi_{2})}_{\calL^{3/4+1,q}_{s} \calL^{2}_{x}}
\leq C \nrm{\psi_{0}}_{\calL^{3/4,\infty}_{s} \calH^{1}_{x}} \nrm{s^{1/4-\eps'} \, \psi_{2}}_{\calL^{3/4,2}_{s} \dot{\calH}^{1}_{x}}, \\
& \nrm{\calO(\psi_{1}, \psi_{2}, \psi_{3})}_{\calL^{3/4+1,q}_{s} \calL^{2}_{x}}
\leq C \nrm{\psi_{1}}_{\calL^{1/4,\infty}_{s} \dot{\calH}^{1}_{x}} \nrm{s^{1/4-\eps'} \ \psi_{2}}_{\calL^{3/4,2}_{s} \dot{\calH}^{1}_{x}} \nrm{\psi_{3}}_{\calL^{1/4,\infty}_{s} \dot{\calH}^{1}_{x}}.
\end{aligned}
\right.
\end{equation}
where $1 \leq q \leq 2$ and $\eps' > 0$ is small enough. Using \eqref{eq:linCovHeat:pf:1} with $q=1,2$, $\psi_{0} = A_{s}$, $\psi_{1} = A$, $\psi_{2} = B$, and $\psi_{3} = A$, we obtain \eqref{eq:absP:apriori:1} with $\psi = B$, $X=L^{2}_{x}$, $\ell = 3/4$, $\eps = D = 0$, $p=2$ and $C(s) = C (\calC + \calC^{2}) s^{1/4-\eps'}$. Since $\nrm{B}_{\calP^{3/4} \calH^{m}_{x}} < \infty$ (as $B$ is regular) and $C(s)^{2}$ is integrable on $(0, 1]$, we can apply the first part of Theorem \ref{thm:absP:absPth} to conclude that 
\begin{equation*} 
	\nrm{B}_{\calP^{3/4} \calH^{2}_{x}(0, 1]} \leq C_{\calC} \nrm{\Bini}_{L^{2}_{x}}.
\end{equation*}
%In particular, the unique smooth solution $B_{i}$ exists on the whole interval $(0, 1]$.

Finally, in the case $m \geq 3$, let us prove the smoothing estimate \eqref{eq:linCovHeat:est4B}. We use Leibniz's rule and \eqref{eq:linCovHeat:pf:1} with $q=2$ (and ignoring all extra weights of $s$) to estimate $\nrm{(\rd_{s} - \lap) B_{i}}_{\calL^{3/4+1,2}_{s} \dot{\calH}^{k}_{x}}$ by
\begin{equation*}
\begin{aligned}
	& C (\nrm{\nb_{x} A}_{\calL^{1/4,\infty}_{s} \calH^{k+1}_{x}} + \nrm{A_{s}}_{\calL^{3/4,\infty}_{s} \calH^{k+1}_{x}} + \nrm{\nb_{x} A}_{\calL^{1/4,\infty}_{s} \calH^{k}_{x}}^{2}) \nrm{B}_{\calL^{3/4,2}_{s} \calH^{k+1}_{x}} \\
	& \quad \leq C ( \calC + \calC^{2}) \nrm{B}_{\calP^{3/4} \calH^{k+1}_{x}}.
\end{aligned}\end{equation*}
for $0 \leq k \leq m-2$. Using the second part of Theorem \ref{thm:absP:absPth}, \eqref{eq:linCovHeat:est4B} follows. \qedhere
\end{proof}

By almost the same proof, the following slight variant of Proposition \ref{prop:linCovHeat} immediately follows.
\begin{proposition} \label{prop:linCovHeat:Diff}
Fix $m \geq 2$, and assume that $(A_{i}, A_{s})$, $(A'_{i}, A_{s}')$ are smooth connection coefficients such that $A_{i}, A'_{i}, A_{s}, A_{s}' \in C^{\infty}_{s}([0,1], H^{\infty}_{x})$ and \eqref{eq:linCovHeat:hypothesis} are satisfied. Assume furthermore
\begin{equation} \label{eq:linCovHeat:Diff:hypothesis}
	\sup_{i} \nrm{\nb_{x} (\dlt A_{i})}_{\calL^{1/4,\infty}_{s} \calH^{m-1}_{x}(0,1]} + \nrm{\dlt A_{s}}_{\calL^{3/4,\infty}_{s} \calH^{m-1}_{x}(0,1]} \leq \dlt \calC < \infty.
\end{equation}

Consider initial data $B_{i}, B'_{i} \in H^{\infty}_{x}$ and the corresponding solutions $B_{i}, B'_{i}$ on $[0,1]$ to the initial value problems
\begin{equation*}
\left\{
\begin{aligned}
\covD_{s} B_{i} - \covD^{\ell} \covD_{\ell} B_{i} =& 2\LieBr{F_{i\ell}}{B^{\ell}}, \\
B_{i}(s=0) =& \Bini_{i},
\end{aligned}
\right. 
\hbox{ and }
\left\{
\begin{aligned}
\covD'_{s} B'_{i} - (\covD')^{\ell} \covD'_{\ell} B_{i} =& 2\LieBr{F'_{i\ell}}{B'^{\ell}}, \\
B'_{i}(s=0) =& \Bini'_{i},
\end{aligned}
\right.
\end{equation*}
given by Proposition \ref{prop:linCovHeat}. The the following estimate holds.
\begin{equation} \label{eq:linCovHeat:est4dltB}
	\sup_{i} \nrm{\dlt B_{i}}_{\calP^{3/4} \calH^{m}_{x}(0,1]} \leq C_{\calC, \nrm{\Bini}_{L^{2}_{x}}, \nrm{\Bini'}_{L^{2}_{x}}} (\sup_{i} \nrm{\dlt \Bini_{i}}_{L^{2}_{x}} + \dlt \calC).
\end{equation}
\end{proposition}

\subsection{Proof of Theorem \ref{thm:lwp4YMHF}} \label{subsec:pfOfLwp4YMHF}
Combining the results in \S \ref{subsec:covYMHF} - \ref{subsec:covYMHFtech}, we can give a proof of Theorem \ref{thm:lwp4YMHF}. This theorem is needed in \cite{Oh:2012fk}, but not for the rest of this paper, so the reader is free to skip this subsection insofar as only the Main Theorem is concerned.

\begin{proof} 
The idea is to first use scaling to make the initial data small, and then solve \eqref{eq:cYMHF} in the DeTurck gauge by the theory we developed in \S \ref{subsec:covYMHF}. Then we will apply Lemma \ref{lem:est4gt2caloric} with $s_{0} = 0$ to obtain a solution to the \eqref{eq:YMHF} (which is \eqref{eq:cYMHF} in the caloric gauge).

We will give a detailed proof of (1) for smooth initial data such that $\Aini_{i} \in H^{\infty}_{x}$; then a similar argument leads to (2) for the same class of initial data, at which point we can recover the full statements of (1) and (2) (for general initial data sets in $\dot{H}^{1}_{x}$) by approximation. 

Note that \eqref{eq:YMHF} is invariant under the scaling $A(x,s) \to \lmb^{-1} A(x/\lmb, s / \lmb^{2})$; using this scaling, we may enforce $\nrm{\Aini}_{\dot{H}^{1}_{x}} \leq \dlt_{P}$. We are then in a position to apply Proposition \ref{prop:YMHF4A:lwp4deT}, from which we obtain a smooth solution $\widetilde{A}_{i} \in C_{t} ([0,1], \dot{H}^{1}_{x})$ to the IVP for \eqref{eq:cYMHF} in the DeTurck gauge, i.e.
\begin{equation} \label{eq:lwp4YMHF:pf:0}
	\left\{
\begin{aligned}
	&\widetilde{F}_{si} = \widetilde{\covD}^{\ell} \widetilde{F}_{\ell i}, \,\, \widetilde{A}_{s} = \rd^{\ell} \widetilde{A}_{\ell} \quad \hbox{ on } \bbR^{3} \times [0,1], \\
	&\widetilde{A}_{i}(s=0) = \Aini_{i},
\end{aligned}	
\right.
\end{equation}
which, by \eqref{eq:YMHF4A:lwp4deT:est4A}, obeys
\begin{equation} \label{eq:lwp4YMHF:pf:1}
	\sup_{s \in [0,1]} \nrm{\widetilde{A}(s)}_{\dot{H}^{1}_{x}} \leq C \nrm{\Aini}_{\dot{H}^{1}_{x}}.
\end{equation}

Next, consider a gauge transform $U = U(x,s)$ which solves the ODE
\begin{equation} \label{eq:lwp4YMHF:pf:2}
	\left\{
\begin{aligned}
	&\rd_{s} U = U \widetilde{A}_{s}, \quad \hbox{ on } \bbR^{3} \times [0,1] \\
	&U(s=0) = \mathrm{Id}.
\end{aligned}
	\right.
\end{equation}

Note that $U$ is regular, as $\widetilde{A}_{s} \in C^{\infty}_{s}([0,1], H^{\infty}_{x})$, and furthermore satisfies the following estimates on $\bbR^{3} \times [0,1]$ thanks to Lemma \ref{lem:est4gt2caloric} :
\begin{equation} \label{eq:lwp4YMHF:pf:3}
	\nrm{U}_{L^{\infty}_{x,s}} \leq C_{\nrm{\Aini}_{\dot{H}^{1}_{x}}}, \quad
	\nrm{\rd_{x} U}_{L^{\infty}_{s} L^{3}_{x}} + \nrm{\rd_{x}^{(2)} U}_{L^{\infty}_{s} L^{2}_{x}} 
	\leq C_{\nrm{\Aini}_{\dot{H}^{1}_{x}}}\nrm{\Aini}_{\dot{H}^{1}_{x}}.
\end{equation}

The identical estimates hold with $U$ replaced by $U^{-1}$ as well.

Let $A_{i} := U \widetilde{A}_{i} U^{-1} - \rd_{i} U U^{-1}$ be the connection 1-form obtained by gauge transforming $(\widetilde{A}_{i}, \widetilde{A}_{s})$ by $U$. We remark that $A_{s} = U \widetilde{A}_{s} U^{-1} - \rd_{s} U U^{-1} = 0$ thanks to the above ODE, and therefore $A_{i}$ solves \eqref{eq:YMHF}, whereas $A_{i}(s=0) = \widetilde{A}_{i}(s=0) = \Aini_{i}$ as $U(s=0) = \mathrm{Id}$. Therefore, we conclude that $A_{i}$ is a (smooth) solution to the IVP for \eqref{eq:YMHF} with the prescribed initial data. Furthermore, from \eqref{eq:lwp4YMHF:pf:1}, \eqref{eq:lwp4YMHF:pf:3} and the gauge transform formula for $A_{i}$, we see that $A_{i} \in C^{\infty}_{s} ([0,1], H^{\infty}_{x})$ and also that \eqref{eq:lwp4YMHF:pf:1} holds with $\widetilde{A}_{i}$ replaced by $A_{i}$. Scaling back, we obtain (1). 

Finally, note that (3) in Theorem \ref{thm:lwp4YMHF} follows easily from (3) in Proposition \ref{prop:YMHF4A:lwp4deT}.
\qedhere
\end{proof}

\begin{remark} 
The idea of the proof of Theorem \ref{thm:lwp4YMHF} outlined above is not new, and is in fact nothing but the standard \emph{DeTurck trick} in disguise, first introduced by D. DeTurck \cite{DeTurck:1983ts} for the Ricci flow and introduced in the context of the Yang-Mills heat flow by S. Donaldson in \cite{Donaldson:1985vh}. A similar procedure will be used in the next section in our proof of Theorem \ref{thm:idEst}, but with an extra twist of choosing $s_{0} =1$ instead of $s_{0}=0$ in Lemma \ref{lem:est4gt2caloric}. This allows us to keep the smoothing estimates  through the gauge transform back to $A_{s} = 0$ (which is not the case for the original DeTurck trick), at the expense of introducing a non-trivial gauge transform for the initial data at $t=0, s=0$.
\end{remark}

\section{Proof of Theorem \ref{thm:idEst} : Estimates for the initial data} \label{sec:pfOfIdEst}
The goal of this section is to prove Theorem \ref{thm:idEst}, using the preliminary results established in the previous section.

We begin by giving the precise definitions of $\calI$ and $\dlt \calI$, which had been alluded in Section \ref{sec:mainThm}.
Let $A_{\bfa}, A_{\bfa}'$ be regular solutions to \eqref{eq:HPYM} (which, we remind the reader, was introduced in \S \ref{subsec:overview}) on $I \times \bbR^{3} \times [0,1]$. We define the norms $\calI$ and $\dlt \calI$ for $F_{si}, \Alow_{i}$ at $t=0$ by
\begin{equation*}
\calI := \sum_{k=1}^{10} \bb[ \nrm{\nb_{t,x} F_{s}(t=0)}_{\calL^{5/4,\infty}_{s} \dot{\calH}^{k-1}_{x}} + \nrm{\nb_{t,x} F_{s}(t=0)}_{\calL^{5/4,2}_{s} \dot{\calH}^{k-1}_{x}} \bb] + \sum_{k=1}^{31}  \nrm{\rd_{t,x} \Alow(t=0)}_{\dot{H}^{k-1}_{x}}.
\end{equation*}
\begin{equation*}
\dlt \calI := \sum_{k=1}^{10} \bb[ \nrm{\nb_{t,x} (\dlt F_{s})(t=0)}_{\calL^{5/4,\infty}_{s} \dot{\calH}^{k-1}_{x}} + \nrm{\nb_{t,x} (\dlt F_{s})(t=0)}_{\calL^{5/4,2}_{s} \dot{\calH}^{k-1}_{x}} \bb] + \sum_{k=1}^{31}  \nrm{\rd_{t,x} (\dlt \Alow)(t=0)}_{\dot{H}^{k-1}_{x}}.
\end{equation*}
where we remind the reader the conventions $\nrm{F_{s}} = \sup_{i} \nrm{F_{si}}$ and $\nrm{\Alow} = \sup_{i} \nrm{\Alow_{i}}$.

\begin{proof} [Proof of Theorem \ref{thm:idEst}]
Throughout the proof, let us use the notation $I = (-T, T)$. The proof will proceed in a number of steps.

\pfstep{Step 1 : Solve \eqref{eq:dYMHF} in the DeTurck gauge}
The first step is to exhibit a regular solution to \eqref{eq:HPYM}, by solving \eqref{eq:dYMHF} under the DeTurck gauge condition $A_{s} = \rd^{\ell} A_{\ell}$.
 For the economy of notation, we will denote the solution by $A_{\bfa}$ in this proof; however, the reader should keep in mind that it is \emph{not} the $A_{\bfa}$ in the statement of the theorem, since we are in a different gauge. 

We note the reader that in this step and the next, we will mostly be interested in obtaining \emph{qualitative} statements, such as smoothness of various quantities, etc. These statements will typically depend on smooth norms of $\Atemp_{i}$. 

We begin by solving \eqref{eq:cYMHF}, i.e., $F_{si} = \covD^{\ell} F_{\ell i}$ for every $t$, with the initial data $A_{i}(t, s=0) = \Atemp_{i}(t)$. 
Let us impose the DeTurck gauge condition $A_{s} = \rd^{\ell} A_{\ell}$. Recall that this makes \eqref{eq:cYMHF} a system of genuine semi-linear heat equations, which can be solved on the unit $s$-interval $[0,1]$ provided that the initial data $\Atemp_{i}$ is small in a suitable sense. Indeed,  for $t \in I$, by Proposition \ref{prop:YMHF4A:lwp4deT} and the hypothesis $\nrm{\Atemp(t)}_{\dot{H}^{1}_{x}} < \dlt_{P}$, there exists a unique smooth solution $A_{i} (t)= A_{i} (t, s, x)$ to the above system on $0 \leq s \leq 1$. Furthermore, by Part (3) of Proposition \ref{prop:YMHF4A:lwp4deT}, $A_{i}, A_{s} \in C^{\infty}_{t,s} (I \times [0,1], H^{\infty}_{x})$.

Our next task is to show that there exists $A_{0}$ which satisfies the remaining equation $F_{s0} = \covD^{\ell} F_{\ell 0}$ of \eqref{eq:dYMHF}. To begin with, let us solve the linear covariant parabolic equation (with smooth coefficients) :
\begin{equation*}
\left\{
\begin{aligned}
\covD_{s} B_{i} &= \covD^{\ell} \covD_{\ell} B_{i} + 2 \LieBr{\tensor{F}{_{i}^{\ell}}}{B_{\ell}}, \\
B_{i} (t)\vert_{s=0} & = \Ftemp_{i0}(t).
\end{aligned}
\right.
\end{equation*}

It is easy to check that the hypotheses of Proposition \ref{prop:linCovHeat} are satisfied, by using the estimates in Proposition \ref{prop:YMHF4A:smth4deT}. Therefore, a unique regular solution $B_{i}$ to this equation exists on $0 \leq s \leq 1$. 

The idea is that $B_{i}$ should be $F_{i0}$ in the end, as it solves exactly the equation that $F_{i0}$ is supposed to solve. With this in mind, let us extend $A_{0}$ by formally setting $F_{s0} = \covD^{\ell} B_{\ell}$, which leads to the ODE
\begin{equation*}
\left\{
\begin{aligned}
\rd_{s} A_{0} &= \rd_{0} A_{s} + \LieBr{A_{0}}{A_{s}} + \covD^{\ell} B_{\ell} \\
A_{0} \vert_{s=0} & = 0.
\end{aligned}
\right.
\end{equation*}

This is a linear ODE with smooth coefficients, as $A_{s}$ is regular. Therefore, by the standard ODE theory, there exists a unique smooth solution $A_{0}$ on $0 \leq s \leq 1$. Furthermore, in view of the regularity of $A_{s}$, $A_{i}$ and $B_{i}$, we see that $A_{0} \in C^{\infty}_{t,s} (I \times [0,1], H^{\infty}_{x})$. 
%The connection coefficients $A_{i}$, $A_{s} = \rd^{\ell} A_{\ell}$ and $A_{0}$ that we have obtained so far constitutes a candidate for the solution to \eqref{eq:HPYM}. 

In what follows, we will prove that the connection 1-form $A_{\bfa}$ formed by $A_{0}, A_{i}, A_{s} = \partial^{\ell} A_{\ell}$ is the desired solution to \eqref{eq:HPYM}. It only remains to check that the connection 1-form $A_{\bfa}$ is indeed a solution to the equation $F_{s \mu} = \covD^{\ell} F_{\ell \mu}$, along with the condition $A_{s} = \rd^{\ell} A_{\ell}$. For this purpose, it suffices to verify that $B_{i} = F_{i0}$, where $F_{i0}$ is the curvature 2-form given by
\begin{equation*}
F_{i0} := \rd_{i} A_{0} - \rd_{0} A_{i} + \LieBr{A_{i}}{A_{0}}.
\end{equation*}

From the regularity of $A_{i}$ and $A_{0}$, it follows that $F_{i0} \in C^{\infty}_{t,s} (I \times [0,1], H^{\infty}_{x})$. Then the following lemma shows that $B_{i} = F_{i0}$ indeed holds.
\begin{lemma} 
Let $B_{i}, A_{0}$ be defined as above, and define $F_{i0} = \rd_{i} A_{0} - \rd_{0} A_{i} + \LieBr{A_{i}}{A_{0}}$. Then we have $B_{i} = F_{i0}$ on $0 \leq s \leq 1$.
\begin{proof} 
Let us begin by computing $\covD_{s} F_{i0}$ :
\begin{align*}
	\covD_{s} F_{i0}
	&= \covD_{i} F_{s0} - \covD_{0} F_{si} \\
	&= \covD_{i} \covD^{\ell} F_{\ell 0} - \covD_{0} \covD^{\ell} F_{\ell i} - \covD_{i} \covD^{\ell} (F_{\ell 0} - B_{\ell}) \\
	& =\covD^{\ell} \covD_{\ell} F_{i0} + 2 \LieBr{\tensor{F}{_{i}^{\ell}}}{F_{\ell 0}} - \covD_{i} \covD^{\ell} (F_{\ell 0} - B_{\ell}) .
\end{align*}
Subtracting the equation
\begin{equation*}
	\covD_{s} B_{i} = \covD^{\ell} \covD_{\ell} B_{i} + 2 \LieBr{\tensor{F}{_{i}^{\ell}}}{B_{\ell}},
\end{equation*}
from the previous equation, we obtain
\begin{equation*}
	\covD_{s} (\dlt F_{i0}) = \covD^{\ell} \covD_{\ell} (\dlt F_{i0}) - \covD^{\ell} \covD_{i} (\dlt F_{\ell 0}) + 2 \LieBr{\tensor{F}{_{i}^{\ell}}}{\dlt F_{\ell 0}},
\end{equation*}
where $\dlt F_{i0} :=  F_{i0} - B_{i}$. By construction, note that $\dlt F_{i0} = 0$ at $s=0$. Furthermore, $\dlt F_{i0} \in C^{\infty}_{t,s} (I \times [0,1], H^{\infty}_{x})$. 

To proceed further, let us fix $t \in I$ and $0 \leq s \leq 1$. Taking the bi-invariant inner product of the last equation with $\dlt F_{i0}$, summing over $i$, and integrating by parts over $\bbR^{3}$, we see that\footnote{In order to justify the integration by parts carried out on the third line, it suffices to note that $\covD_{\ell} \dlt F_{i0}, \dlt F_{i0} \in L^{2}_{x}$ for every fixed $t \in I$ and $0 \leq s \leq 1$.}
\begin{align*}
	\frac{1}{2} \rd_{s} & \bb( \sum_{i} \int \big( \dlt F_{i0}, \dlt F_{i0} \big)(s) \, \ud x \bb)  \\
	= &\sum_{i, \ell} \int \big( \covD_{\ell} \covD_{\ell} (\dlt F_{i0}), \dlt F_{i0} \big)(s) - \big( \covD_{\ell} \covD_{i} (\dlt F_{\ell 0}), \dlt F_{i0} \big)(s) + 2 \big( \LieBr{F_{i\ell}}{\dlt F_{\ell 0}}, \dlt F_{i 0} \big)(s) \, \ud x \\
%	= & \sum_{i, \ell} \int - \big( \covD_{\ell} (\dlt F_{i0}) , \covD_{\ell} (\dlt F_{i0}) \big)(s) + \big( \covD_{i} (\dlt F_{\ell 0}), \covD_{\ell} (\dlt F_{i0}) \big)(s) + \big( \LieBr{F_{i\ell}}{\dlt F_{\ell 0}}, \dlt F_{i 0} \big)(s) \, \ud x \\
	= & \sum_{i, \ell} \int - \frac{1}{2} (\covD_{\ell} (\dlt F_{i0}) - \covD_{i} (\dlt F_{\ell 0}), \covD_{\ell} (\dlt F_{i0}) - \covD_{i} (\dlt F_{\ell 0}) \big)(s) + 2 \big( \LieBr{F_{i\ell}}{\dlt F_{\ell 0}}, \dlt F_{i 0} \big)(s) \, \ud x. 
\end{align*}

The first term on the last line has a favorable sign, and can be thrown away. The remaining term is easily bounded by $C \bb( \sup_{i, \ell} \nrm{F_{i \ell}(s)}_{L^{\infty}_{x}} \bb) \bb( \sum_{i} \int (\dlt F_{i0}, \dlt F_{i0})(s) \, \ud x \bb)$. Since $F_{i \ell}$ is uniformly bounded on $I \times \bbR^{3} \times [0,1]$, we may apply Gronwall's inequality and conclude that $ \sum_{i} \int (\dlt F_{i0}, \dlt F_{i0})(s) \, \ud x = 0$ on $0 \leq s \leq 1$. This shows that $\dlt F_{i0} = 0$ on $0 \leq s \leq 1$, as desired. \qedhere
\end{proof}
\end{lemma}

In sum, we conclude that $A_{\bfa}$ is a regular solution to \eqref{eq:HPYM}, which satisfies the DeTurck gauge condition $A_{s} = \rd^{\ell} A_{\ell}$.
%Furthermore, uniqueness of $A_{\bfa}$ is an easy consequence of the uniqueness statement of Proposition \ref{prop:YMHF4A:lwp4deT} and the preceding lemma.

\pfstep{Step 2 : Construction of a gauge transform to the caloric-temporal gauge}

In Step 1, we have constructed a regular solution $A_{\bfa}$ to \eqref{eq:HPYM} in the DeTurck gauge $A_{s} = \rd^{\ell} A_{\ell}$. The next step is to construct a suitable smooth gauge transform $U = U(t,x,s)$ to impose the caloric-temporal gauge condition, i.e., $A_{s} = 0$ and $\Alow_{0} = 0$. 

We begin by briefly going over the gauge structure of \eqref{eq:HPYM}. As before, the gauge transform corresponding to a smooth $\LieGrp$-valued function $U(t,x,s)$ on $I \times \bbR^{3} \times [0,1]$ is given by the formulae
\begin{equation*}
	A_{\bfa} \to U A_{\bfa} U^{-1} - \rd_{\bfa} U U^{-1} = : \widetilde{A}_{\bfa}, \quad F_{\bfa \bfb} \to U F_{\bfa \bfb} U^{-1} = : \widetilde{F}_{\bfa \bfb}.
\end{equation*}
where $\bfa = (t, x, s)$. As a consequence, any gauge transformed connection 1-form $\widetilde{A}_{\bfa}$ will still solve \eqref{eq:HPYM}.

From the above discussion, it follows that in order to impose $\widetilde{\Alow}_{0} = \widetilde{A}_{0} (s=1)=0$ and $\widetilde{A}_{s} =0$, the gauge transform $U(t,s,x)$ must satisfy
\begin{equation*}
\left\{
\begin{aligned}
\rd_{0} U &= U \Alow_{0} \hbox{ along $s=1$,} \\
\rd_{s} U &= U A_{s} \hbox{ everywhere}.
\end{aligned}
\right.
\end{equation*}

We will solve this system by starting from the identity gauge transform at $(t=0, s=1)$. More precisely, let us fix $x \in \bbR^{3}$, and first solve the ODE
\begin{equation*}
\left\{
\begin{aligned}
\rd_{0} \big( \Ulow(t) \big) &= \Ulow(t) \Alow_{0}(t), \\
\Ulow(t=0) &= \mathrm{Id}.
\end{aligned}
\right.
\end{equation*}
along $I \times \set{s=1}$, where $\Alow_{0}(t) := A_{0}(t, s=1)$. Then using these as the initial data, we solve
\begin{equation*}
\left\{
\begin{aligned}
\rd_{s} U(t,s) &= U(t,s) A_{s}(t,s), \\
U(t, s=1) &=\Ulow(t).
\end{aligned}
\right.
\end{equation*}

Since both $A_{0}$ and $A_{s} = \rd^{\ell} A_{\ell}$ belong to $C^{\infty}_{t,s} (I \times [0,1], H^{\infty}_{x})$, it is clear, again by the standard ODE theory, that there exists a unique smooth solution $U(t,x,s)$ satisfying the above ODEs. Furthermore, using arguments as in the proof of Proposition \ref{prop:est4gt}, we can readily prove that $U$ is a regular gauge transform (in the sense of Definition \ref{def:reg4gt}).
%\footnote{These are all easy to establish if we assume \eqref{eq:idEst:pf:step1:0} and \eqref{eq:idEst:pf:step1:1}, since by differentiating the transport equations and proceeding by an induction argument, everything reduces to estimating an ODE of the form $\rd_{\tau} \phi(\tau, x) = C(\tau, x) \phi(\tau, x) + \calB(\tau, x)$, where $C(\tau, x)$ is uniformly bounded and an $L^{p}_{x}$ norm of $\calB$ is uniformly bounded in $\tau$. Then the $L^{p}_{x}$ norm of $\phi(\tau, x)$ itself can be uniformly bounded in $\tau$ by Gronwall. On the other hand, we refer the reader to the Appendix for the proof of the quantitative version of these estimates, which is of course more involved.}

By construction, the gauge transformed connection 1-form $\widetilde{A}_{\bfa}$ satisfies the following equations.
\begin{equation*}
\left\{
\begin{aligned}
\widetilde{F}_{s \mu} &= \widetilde{\covD}^{\ell} \widetilde{F}_{\ell \mu}, \\
\widetilde{A}_{\mu} & = U A_{\mu} U^{-1} - \rd_{\mu} U U^{-1}.
\end{aligned}
\right.
\end{equation*}
Furthermore, it satisfies the caloric-temporal gauge condition, i.e., $\widetilde{\Alow}_{0} = 0$ and $\widetilde{A}_{s} = 0$. As $A_{\bfa}$ is a regular solution to \eqref{eq:HPYM} and $U$ is a regular gauge transform, it readily follows that $\widetilde{A}_{\mu}$ is a regular solution of \eqref{eq:HPYM} as well.

\pfstep{Step 3 : Quantitative initial data estimates for non-differences, in the caloric-temporal gauge}

In this step, we will prove the non-difference estimates among the initial data estimates \eqref{eq:idEst:1} -- \eqref{eq:idEst:4}. From this point on, we must be \emph{quantitative}, which means that we must make sure that all estimates here depends only on $\calIini$. 

Before we begin, a word of caution on the notation. We will keep the same notation as the previous steps, which means that the $A_{\bfa}, F_{si}$ in the statement of the theorem are \emph{not} the same as $A_{\bfa}, F_{si}$ below, but are rather $\widetilde{A}_{\bfa}, \widetilde{F}_{si}$. Accordingly, the gauge transform $V$ in the statement of the theorem is given by $V(t,x) = U(t, x, s=0)$ (which is regular).

Recalling the definition of $\calI$, in order to prove $\calI \leq C_{\calIini} \cdot \calIini$, we must estimate $\nrm{\nb_{t,x} F_{si}}_{\calL^{5/4,p}_{s} \calH^{9}_{x}}$ for $p=2, \infty$ and $\nrm{\rd_{t,x} \Alow_{i}}_{H^{30}_{x}}$ in terms of $C_{\calIini} \cdot \calIini$. On the other hand, to prove \eqref{eq:idEst:3}, we need to show
\begin{equation} \label{eq:idEst:pf:est4U}
	\nrm{U(t=0, s=0)}_{L^{\infty}_{x}} \leq C_{\calIini}, \quad 
	\nrm{\rd_{x} U(t=0, s=0)}_{L^{3}_{x}} + \nrm{\rd_{x}^{(2)} U(t=0, s=0)}_{L^{2}_{x}} \leq C_{\calIini} \cdot \calIini,
\end{equation}
and the analogous estimates for $U^{-1}$.

By Proposition \ref{prop:YMHF4A:smth4deT}, we have the following estimates at $t=0$ for $m \geq 1$:
\begin{equation} \label{eq:calTempGauge:est4Apre}
\sup_{i} \nrm{A_{i}}_{\calL^{1/4,\infty}_{s} \dot{\calH}^{m}_{x}} + \sup_{i} \nrm{A_{i}}_{\calL^{1/4,2}_{s} \dot{\calH}^{m+1}_{x}} \leq C_{m, \calIini} \cdot \calIini.
\end{equation}

%By interpolation, we obtain the following $L^{\infty}_{x}$ estimates for $m \geq 0$ as a consequence.
%\begin{equation} \label{eq:calTempGauge:est4Apre:Linfty}
%\sup_{i} \sup_{0 \leq s \leq 1} s^{m/2 + 1/4} \nrm{\rd_{x}^{(m)} A_{i}(s)}_{L^{\infty}_{x}} \leq D^{1}_{m}(\calIini).
%\end{equation}

Note that the gauge transform $U$ is equal to the identity transform at $t=0, s=1$. Consequently, we have $\widetilde{A}_{i}(t=0, s=1) = A_{i}(t=0, s=1)$. Therefore $\nrm{\rd_{x} \widetilde{\Alow}}_{H^{30}_{x}} \leq C_{\calIini} \cdot \calIini$ follows immediately from \eqref{eq:calTempGauge:est4Apre}.

Next, we estimate the gauge transform $U$ at $t=0$. Using Lemma \ref{lem:est4gt2caloric}, it follows that the gauge transform $U$ at $t=0$ satisfy the following estimates for $m \geq 2$.
\begin{equation} \label{eq:calTempGauge:est4U}
%\left\{
%\begin{aligned}
\nrm{U}_{L^{\infty}_{s} L^{\infty}_{x}} \leq C_{\calIini}, \quad
\nrm{\rd_{x} U}_{L^{\infty}_{s} L^{3}_{x}}  \leq C_{\calIini} \cdot \calIini, \quad
\nrm{s^{(m-2)/2} \rd_{x}^{(m)} U}_{L^{\infty}_{s} L^{2}_{x}} \leq C_{m, \calIini} \cdot \calIini.
%\end{aligned}
%\right.
\end{equation}
%\nrm{U}_{\calL^{0,\infty}_{s} \calL^{\infty}_{x}} + \nrm{\nb_{t, x} U}_{\calL^{0,\infty}_{s} \calL^{3}_{x}} \leq& C_{\calIini},  \\
%\nrm{U}_{\calL^{0,\infty}_{s} \dot{\calH}^{m}_{x}} \leq& C_{\calIuni} \cdot \calIini. \label{eq:calTempGauge:est4U:2}
%\end{align}

The analogous estimates hold for $U^{-1}$. Then \eqref{eq:idEst:pf:est4U} (or its analogue for $U^{-1}$) follows immediately from \eqref{eq:calTempGauge:est4U} (or its analogue for $U^{-1}$).

The next step is to estimate $\widetilde{F}_{si}$. Let us begin by considering the gauge transformed connection 1-form $\widetilde{A}_{i}$ at $t=0$. We claim that they satisfy the estimates 
\begin{equation} \label{eq:calTempGauge:est4A:L2}
	\sup_{i} \nrm{\widetilde{A}_{i}}_{\calL^{1/4,\infty}_{s} \dot{\calH}^{m}_{x}} \leq C_{m, \calIini} \cdot \calIini,
\end{equation}
for $m \geq 1$. For $m = 1$, we can compute via Leibniz's rule
\begin{align*}
	\rd_{x} \widetilde{A}_{i} 
	= & \rd_{x} ( U A_{i} U^{-1} - \rd_{i} U U^{-1} ) \\
	= & (\rd_{x} U ) A_{i} U^{-1} + U (\rd_{x} A_{i}) U^{-1} + U A_{i} (\rd_{x} U^{-1}) - (\rd_{x} \rd_{i} U ) U^{-1} - \rd_{i} U (\rd_{x} U^{-1}).
\end{align*}

Using H\"older, Sobolev and \eqref{eq:calTempGauge:est4U}, it is not difficult to show 
\begin{equation*}
	\nrm{\rd_{x} \widetilde{A}_{i}(s)}_{L^{2}_{x}} \leq C_{\calIini} \cdot (\nrm{\rd_{x} A_{i}(s)}_{L^{2}_{x}} + \calIini).
\end{equation*}
At this point, from \eqref{eq:calTempGauge:est4Apre}, we obtain \eqref{eq:calTempGauge:est4A:L2} in the case $m=1$. Proceeding similarly, we can also prove \eqref{eq:calTempGauge:est4A:L2} for $m \geq 2$; we omit the details. 
%Then, by interpolation, we obtain the following $L^{\infty}_{x}$ estimates for $m \geq 0$:
%\begin{equation} \label{eq:calTempGauge:est4A:Linfty}
%	\sup_{i} \sup_{0 \leq s \leq 1} s^{m/2 + 1/4} \nrm{\rd_{x}^{(m)} \widetilde{A}_{i}(s)}_{L^{\infty}_{x}} \leq D^{1}_{m} (\calIini).
%\end{equation}

From the previous estimates, it is already possible to estimate $\nrm{\nb_{x} F_{s}}_{\calL^{5/4,p}_{s} \calH^{9}_{x}}$ for $p=\infty$ but not $p=2$. This is essentially due to the unpleasant term $- \rd_{i} U U^{-1}$ in the gauge transformation formula for $A_{i}$. To estimates both terms at the same time, we argue differently as follows, utilizing the gauge covariance of $F_{si}$. We start by recalling the equation for $F_{si}$ in terms of $A_{i}$ and ordinary derivatives. 
\begin{align*}
	F_{si} =& \covD^{\ell} F_{\ell i} 
		=\rd^{\ell} \rd_{\ell} A_{i} - \rd_{i} \rd^{\ell} A_{\ell} + 2 \LieBr{A^{\ell}}{\rd_{\ell} A_{i}} + \LieBr{\rd^{\ell} A_{\ell}}{A_{i}} - \LieBr{A^{\ell}}{\rd_{i} A_{\ell}} + \LieBr{A^{\ell}}{\LieBr{A_{\ell}}{A_{i}}}.
\end{align*}
Using this formula and \eqref{eq:calTempGauge:est4Apre}, it is not difficult to prove the following estimates for $m \geq 0$.
\begin{equation} \label{eq:calTempGauge:est4Fsi}
	\sup_{i} \nrm{\nb_{x}^{(m)} F_{si}}_{\calL^{5/4,\infty}_{s} \calL^{2}_{x}} + \sup_{i} \nrm{\nb_{x}^{(m)} F_{si}}_{\calL^{5/4,2}_{s} \calL^{2}_{x}}
	\leq C_{m, \calIini} \cdot \calIini.
\end{equation}
%Indeed, for $m=0$, we fix $s$ and estimate
%\begin{align*}
%	s^{1/2}\nrm{F_{si}(s)}_{L^{2}}
%	\leq & C s^{1/2}\bb( \nrm{\rd_{x}^{(2)} A(s)}_{L^{2}_{x}} + \nrm{\calO(A, \rd_{x} A)(s)}_{L^{2}_{x}} + \nrm{\calO(A,A,A)(s)}_{L^{2}_{x}} \bb) \\
%	\leq & C s^{1/2}\bb( \nrm{\rd_{x}^{(2)} A(s)}_{L^{2}_{x}} + \nrm{A(s)}_{L^{\infty}_{x}} \nrm{\rd_{x} A(s)}_{L^{2}}  + \nrm{A}_{L^{6}_{x}}^{3} \bb) \\
%	\leq & C s^{1/2} \nrm{\rd_{x}^{(2)} A(s)}_{L^{2}_{x}} + s^{1/4} D^{1}(\calIini) \nrm{\rd_{x} A(s)}_{L^{2}} + s^{1/2} D^{2}(\calIini) \nrm{\rd_{x} A(s)}_{L^{2}_{x}}.
%\end{align*}
%Note that we used the first term on the left-hand side of \eqref{eq:calTempGauge:est4Apre}. Taking either the supremum or square-integrating over $0 \leq s \leq 1$ and appealing to \eqref{eq:calTempGauge:est4Apre}, the case $m=0$ follows. The case of $m \geq 1$ is similar, and thus we omit the proof.

Using interpolation and \eqref{eq:calTempGauge:est4Apre} to control $\nrm{\nb_{x}^{(m)} A_{i}}_{\calL^{1/4,\infty}_{s} \calL^{\infty}_{x}}$, it is a routine procedure to replace the ordinary derivatives in \eqref{eq:calTempGauge:est4Fsi} by covariant derivatives, i.e.
\begin{equation} \label{eq:calTempGauge:est4covDFsi}
	\sup_{i} \nrm{\calD_{x}^{(m)} F_{si}}_{\calL^{5/4,\infty}_{s} \calL^{2}_{x}} + \sup_{i} \nrm{\calD_{x}^{(m)} F_{si}}_{\calL^{5/4,2}_{s} \calL^{2}_{x}}
	\leq C_{m, \calIini} \cdot \calIini.
\end{equation}
for $m \geq 0$, where we remind the reader that $\calD$ is the p-normalized $\covD$, i.e., $\calD_{\mu} := s^{1/2} \covD_{\mu}$. Then using the gauge transform formula $\widetilde{\covD}_{x}^{(m)} \widetilde{F}_{si} = U (\covD^{(m)}_{x} F_{si}) U^{-1}$ and \eqref{eq:calTempGauge:est4U} to estimate $U$ in $\calL^{0,\infty}_{s} \calL^{\infty}_{x}$, we obtain
\begin{equation*}
	\sup_{i} \nrm{\widetilde{\calD}_{x}^{(m)} \widetilde{F}_{si}}_{\calL^{5/4,\infty}_{s} \calL^{2}_{x}} + \sup_{i} \nrm{\widetilde{\calD}_{x}^{(m)} \widetilde{F}_{si}}_{\calL^{5/4,2}_{s} \calL^{2}_{x}}
	\leq C_{m, \calIini} \cdot \calIini.
\end{equation*}
for $m \geq 0$. Using interpolation and \eqref{eq:calTempGauge:est4A:L2} to control $\nrm{\nb_{x}^{(m)} \widetilde{A}_{i}}_{\calL^{1/4,\infty}_{s} \calL^{\infty}_{x}}$, we can replace the covariant derivatives by ordinary derivatives, and finally obtain for $m \geq 0$ the following estimate:
\begin{equation*}
	\sup_{i} \nrm{\nb_{x}^{(m)} \widetilde{F}_{si}}_{\calL^{5/4,\infty}_{s} \calL^{2}_{x}} + \sup_{i} \nrm{\nb_{x}^{(m)} \widetilde{F}_{si}}_{\calL^{5/4,2}_{s} \calL^{2}_{x}}
	\leq C_{m, \calIini} \cdot \calIini.
\end{equation*}

This gives an adequate control on $\nrm{\nb_{x} F_{s}}_{\calL^{5/4,p}_{s} \calH^{9}_{x}}$ for $p=2, \infty$.

In order to estimate the terms involving the time derivative, we must look at $\widetilde{F}_{0i}$. Recall that $\widetilde{F}_{0i}$ satisfies the covariant heat equation
\begin{equation*}
	\rd_{s} \widetilde{F}_{i0} = \widetilde{\covD}^{\ell} \widetilde{\covD}_{\ell} \widetilde{F}_{i0} - 2 \LieBr{\tensor{\widetilde{F}}{_{i}^{\ell}}}{\widetilde{F}_{0\ell}}.
\end{equation*}
Applying Proposition \ref{prop:linCovHeat} with \eqref{eq:calTempGauge:est4A:L2}, we obtain the following estimates on $\widetilde{F}_{0i}$ for $m \geq 0$.
\begin{equation} \label{eq:calTempGauge:est4F0i}
	\sup_{i} \nrm{\widetilde{F}_{0i}}_{\calL^{3/4,\infty}_{s} \dot{\calH}^{m}_{x}} + \sup_{i} \nrm{\widetilde{F}_{0i}}_{\calL^{3/4,2}_{s} \dot{\calH}^{m+1}_{x}}
	\leq C_{m, \calIini} \cdot \calIini.
\end{equation}

Because of our gauge condition $\widetilde{A}_{0}(s=1) = 0$, we have $\rd_{t} \widetilde{A}_{i}(s=1) = \widetilde{F}_{0i}(s=1)$. This immediate proves $\nrm{\rd_{0} \widetilde{\Alow}}_{H^{30}_{x}} \leq C_{\calIini} \cdot \calIini$. 

In order to estimate the time derivative of $\widetilde{F}_{si}$, we need to argue a bit more. Expanding the covariant derivatives in the Bianchi identity $\widetilde{\covD}_{0} \widetilde{F}_{si} + \widetilde{\covD}_{s} \widetilde{F}_{i0} + \widetilde{\covD}_{i} \widetilde{F}_{0s} = 0$, we obtain the equation
\begin{equation} \label{eq:calTempGauge:DtFsi}
	\rd_{0} \widetilde{F}_{si} = \rd_{s} \widetilde{F}_{0i} + \rd_{i} \widetilde{F}_{s0} - \LieBr{\widetilde{A}_{0}}{\widetilde{F}_{si}} + \LieBr{\widetilde{A}_{i}}{\widetilde{F}_{s0}}.
\end{equation}

We use the parabolic equation for $\widetilde{F}_{0i}$ to rewrite term $\rd_{s} \widetilde{F}_{0i}$ in terms of $\widetilde{F}_{si}$, $\widetilde{A}_{i}$ and their spatial derivatives. We estimate $\widetilde{F}_{s0}$ by \eqref{eq:dYMHF}
\begin{equation} \label{eq:calTempGauge:Fs0}
	\widetilde{F}_{s0} 
	= \widetilde{\covD}^{\ell} \widetilde{F}_{\ell 0} 
	= \rd^{\ell} \widetilde{F}_{\ell 0} + \LieBr{\widetilde{A}^{\ell}}{\widetilde{F}_{\ell 0}},
\end{equation}
whereas $\widetilde{A}_{0}$ is estimated by the equation
\begin{equation} \label{eq:calTempGauge:fund4A0}
	\widetilde{A}_{0}(s) = - \int_{s}^{1} \widetilde{F}_{s0}(s') \, \ud s',
\end{equation}
which holds due to the caloric-temporal gauge condition. 

Using \eqref{eq:calTempGauge:DtFsi}, \eqref{eq:calTempGauge:Fs0}, \eqref{eq:calTempGauge:fund4A0} and the previous estimates for $\widetilde{F}_{i0}, \widetilde{F}_{si}$ and $\widetilde{A}_{i}$, it is not difficult to show
\begin{equation*}
	\sup_{i} \nrm{\nb_{0} \widetilde{F}_{si}}_{\calL^{5/4,\infty}_{s} \calL^{2}_{x}} + \sup_{i} \nrm{\nb_{0} \widetilde{F}_{si}}_{\calL^{5/4,2}_{s} \calL^{2}_{x}}
	\leq C_{\calIini} \cdot \calIini.
\end{equation*}

Differentiating \eqref{eq:calTempGauge:DtFsi}, \eqref{eq:calTempGauge:Fs0} and \eqref{eq:calTempGauge:fund4A0} appropriate number of times with respect to $x$, $\nb_{x}^{(m)} \nb_{0} \widetilde{F}_{si}$ can be estimated analogously; we leave the details to the reader. This concludes the proof of $\calI \leq C_{\calIini} \cdot \calIini$.

\pfstep{Step 4: Quantitative initial data estimates for differences}
In the fourth and the final step, we will prove the difference estimates among \eqref{eq:idEst:1} -- \eqref{eq:idEst:4}.

Let $(A, U), (A', U')$ be constructed according to the Steps 1, 2 from $\Atemp_{i}, \AtempPrime_{i}$, respectively. The idea is to adapt Step 3 to differences. In particular, we remark that everything we do here is done on $\set{t=0}$. 

In order to prove $\dlt \calI \leq C_{\calIini} \cdot \dlt \calIini$, we must estimate $\nrm{\nb_{t,x} (\dlt F_{si})}_{\calL^{5/4,p}_{s} \calH^{9}_{x}}$ for $p=2, \infty$ and $\nrm{\rd_{t,x} \Alow_{i}}_{H^{30}_{x}}$ in terms of $C_{\calIini} \cdot \dlt \calIini$. To prove \eqref{eq:idEst:4}, we need to show
\begin{equation} \label{eq:idEst:pf:est4dltU}
\left\{
\begin{aligned}
	& \nrm{\dlt \widetilde{U}(t=0, s=0)}_{L^{\infty}_{x}} 
	+ \nrm{\rd_{x} (\dlt \widetilde{U})(t=0, s=0)}_{L^{3}_{x}} \leq C_{\calIini} \cdot \dlt \calIini, \\
	& \nrm{\rd_{x}^{(2)} (\dlt \widetilde{U})(t=0, s=0)}_{L^{2}_{x}} \leq C_{\calIini} \cdot \dlt \calIini.
\end{aligned}
\right.
\end{equation}
and their analogues for $\dlt U^{-1}$.

From Proposition \ref{prop:YMHF4A:smth4deT}, we have the following analogue of \eqref{eq:calTempGauge:est4Apre} for $\dlt A_{i} := A_{i} - A_{i}'$ at $t=0$, for $m \geq 1$.
\begin{equation} \label{eq:calTempGauge:est4Apre:Diff}
\sup_{i} \nrm{\dlt A_{i}}_{\calL^{1/4,\infty}_{s} \dot{\calH}^{m}_{x}} + \sup_{i} \nrm{\dlt A_{i}}_{\calL^{1/4,2}_{s} \dot{\calH}^{m+1}_{x}} \leq C_{\calIini} \cdot \dlt \calIini.\end{equation} 

%By interpolation, we also have the following $L^{\infty}_{x}$ estimates for $\dlt A_{i}$, for $m \geq 0$.
%\begin{equation} \label{eq:calTempGauge:est4Apre:Diff:Linfty}
%	\sup_{i} \sup_{0 \leq s \leq 1} s^{m/2 + 1/4} \nrm{\rd_{x}^{(m)} (\dlt A_{i})(s)}_{L^{\infty}_{x}} \leq D^{0}_{m}(\calIini) \, \dlt \calIini.
%\end{equation}

As before, this immediately shows $\nrm{\rd_{x} \Alow}_{H^{9}_{x}} \leq C_{\calIini} \cdot \dlt \calIini$, since the gauge transforms $U$, $U'$ are equal to the identity $\mathrm{Id}$ at $t=0, s=1$. Next, from Lemma \ref{lem:est4gt2caloric} and the bounds \eqref{eq:calTempGauge:est4Apre} (for A, A') and \eqref{eq:calTempGauge:est4Apre:Diff}, we have the following analogue of \eqref{eq:calTempGauge:est4U} for $\dlt U = U - U'$, $\dlt U^{-1} = U^{-1} - (U')^{-1}$ at $t=0$, for $m \geq 2$.
\begin{equation} \label{eq:calTempGauge:est4U:Diff}
\nrm{\dlt U}_{L^{\infty}_{s} L^{\infty}_{x}} + \nrm{\rd_{x} (\dlt U)}_{L^{\infty}_{s} L^{3}_{x}}
\leq C_{\calIini} \cdot \dlt \calIini, \quad
\nrm{s^{(m-2)/2} \rd_{x}^{(m)} (\dlt U)}_{L^{\infty}_{s} L^{2}_{x}} \leq C_{m, \calIini} \cdot \dlt \calIini.
\end{equation}

Analogous estimates hold also for $\dlt U^{-1}$. We remark that in contrast to \eqref{eq:calTempGauge:est4U}, the right hand side vanishes as $\dlt \calIini \to 0$. Since $\dlt\Vini$ in the statement of the theorem is exactly $\dlt \Vini = \dlt U(t=0, s=0)$, the estimate \eqref{eq:idEst:pf:est4dltU} follows (for both $\dlt U$ and $\dlt U^{-1}$).

Next, we estimate $\dlt F_{si}$ in the statement of the theorem, which in our case is $\dlt \widetilde{F}_{si} := \widetilde{F}_{si} - \widetilde{F}'_{si}$. The idea is the same as before, namely we want to use the gauge transform property of $F_{si}$, $F_{si}'$. As in Step 3, we begin by establishing some bounds for $\dlt \widetilde{A}_{i}$. Our starting point is the following formula for the difference $\dlt \widetilde{A}_{i}$.
\begin{align*}
\dlt \widetilde{A}_{i} 
= (\dlt U) A_{i} U^{-1} + U' (\dlt A_{i}) U^{-1} + U' A'_{i} (\dlt U^{-1}) - \big( \rd_{i} (\dlt U) \big) U^{-1} - \rd_{i} U' ( \dlt U^{-1}).
\end{align*}

Note that this is nothing but Leibniz's rule for $\dlt$. Differentiating the above formula $m$-times and using the estimates \eqref{eq:calTempGauge:est4Apre}, \eqref{eq:calTempGauge:est4U}, along with their difference analogues \eqref{eq:calTempGauge:est4Apre:Diff}, \eqref{eq:calTempGauge:est4U:Diff}, it is not difficult to prove the analogue of \eqref{eq:calTempGauge:est4A:L2} for $\dlt \widetilde{A}_{i} := \widetilde{A}_{i} - \widetilde{A}'_{i}$, for $m \geq 1$
\begin{equation} \label{eq:calTempGauge:est4A:Diff:L2}
	\sup_{i} \nrm{\dlt \widetilde{A}_{i}}_{\calL^{1/4,\infty}_{s} \dot{\calH}^{m}_{x}} \leq C_{m, \calIini} \cdot \dlt \calIini.
\end{equation}

The next step is to estimate $\dlt F_{si}$. Using Leibniz's rule for $\dlt$, we obtain the following formula for $\dlt F_{si}$:
\begin{equation*}
\dlt F_{si} = \calO(\rd_{x}^{(2)} (\dlt A)) + \calO(\dlt A, \rd_{x} A) + \calO(A, \rd_{x} (\dlt A)) + \calO(\dlt A, A, A).
\end{equation*}

From this formula and the estimates \eqref{eq:calTempGauge:est4Apre}, \eqref{eq:calTempGauge:est4Apre:Diff}, we obtain the following estimate for $\dlt F_{si}$ for $m \geq 0$ as before:
\begin{equation} \label{eq:calTempGauge:est4Fsi:Diff}
	\sup_{i} \nrm{\nb_{x}^{(m)} (\dlt F_{si})}_{\calL^{5/4,\infty}_{s} \calL^{2}_{x}} + \sup_{i} \nrm{\nb_{x}^{(m)} (\dlt F_{si}) }_{\calL^{5/4,2}_{s} \calL^{2}_{x}}
	\leq C_{m, \calIini} \cdot \dlt \calIini.
\end{equation}

Equipped with these estimates, we claim that the following estimates are true, for $m \geq 0$.
\begin{equation}\label{eq:calTempGauge:est4CovDFsi:Diff}
	\sup_{i} \nrm{\dlt (\widetilde{\calD}_{x}^{(m)} \widetilde{F}_{si})}_{\calL^{5/4,\infty}_{s} \calL^{2}_{x}} + \sup_{i} \nrm{\dlt (\widetilde{\calD}_{x}^{(m)} \widetilde{F}_{si})}_{\calL^{5/4,2}_{s} \calL^{2}_{x}}
	\leq C_{m, \calIini} \cdot \dlt \calIini,
\end{equation}
where $\dlt (\widetilde{\calD}_{x}^{(m)} \widetilde{F}_{si}) :=  \widetilde{\calD}_{x}^{(m)} \widetilde{F}_{si} - \widetilde{\calD'}_{x}^{(m)} \widetilde{F'}_{si}$, and accordingly $\dlt (\widetilde{\covD}_{x}^{(m)} \widetilde{F}_{si}) :=  \widetilde{\covD}_{x}^{(m)} \widetilde{F}_{si} - \widetilde{\covD'}_{x}^{(m)} \widetilde{F'}_{si}$.

For $m=0$, this is an easy consequence of the difference formula
\begin{align*}
	\widetilde{F}_{si} - \widetilde{F}'_{si} 
	= & U F_{si} U^{-1} - U' F'_{si}(U')^{-1} \\
	= & (\dlt U) F_{si} U^{-1} + U' (\dlt F_{si} )U^{-1} + U' F'_{si} (\dlt U^{-1}),
\end{align*}
and the estimates \eqref{eq:calTempGauge:est4U}, \eqref{eq:calTempGauge:est4Fsi}, \eqref{eq:calTempGauge:est4U:Diff} and \eqref{eq:calTempGauge:est4Fsi:Diff}. 

For $m=1$, we begin to see covariant derivatives, which we just write out in terms of ordinary derivatives and connection 1-forms $A_{i}$. Then we can easily check , using Leibniz's rule for $\dlt$, that the following difference formula holds.
\begin{align*}
\dlt (\widetilde{\covD}_{x}^{(m)} \widetilde{F}_{si})
=& U \covD_{x} F_{si} U^{-1} - U' \covD'_{x} F'_{si} (U')^{-1} \\
%=& U (\rd_{x} F_{si} )U^{-1} - U' (\rd_{x} F'_{si}) (U')^{-1} + U \LieBr{A_{x}}{F_{si}} U^{-1} - U' \LieBr{A'_{x}}{F'_{si}} (U')^{-1} \\
=& \calO(\dlt U, F_{s}, U^{-1}) + \calO(U, \dlt F_{s}, U^{-1}) + \calO(U, F_{s}, \dlt U^{-1}) \\
& + \calO(\dlt U, A, F_{s}, U^{-1}) + \calO(U, \dlt A, F_{s}, U^{-1}) + \calO(U, A, \dlt F_{s}, U^{-1}) + \calO(U, A, F_{s}, \dlt U^{-1}).
%= & (\dlt U) (\rd_{x} F_{si}) U^{-1} + U' (\rd_{x} (\dlt F_{si})) U^{-1} + U' (\rd_{x} F'_{si}) (\dlt U^{-1}) \\
%   & + (\dlt U) \LieBr{A_{x}}{F_{si}} U^{-1} + U' \LieBr{\dlt A_{x}}{F_{si}} U^{-1} + U' \LieBr{A'_{x}}{\dlt F_{si}} U^{-1} + U' \LieBr{A'_{x}}{F'_{si}} (\dlt U^{-1}).
\end{align*}
	
Using the estimates \eqref{eq:calTempGauge:est4U}, \eqref{eq:calTempGauge:est4Fsi}, \eqref{eq:calTempGauge:est4U:Diff}, \eqref{eq:calTempGauge:est4Fsi:Diff}, as well as \eqref{eq:calTempGauge:est4Apre}, \eqref{eq:calTempGauge:est4Apre:Diff} with interpolation, we obtain \eqref{eq:calTempGauge:est4CovDFsi:Diff} for $m=1$.

The cases $m \geq 2$ can be proved in an analogous fashion, namely first computing the difference formula for $\widetilde{\covD}_{x}^{(m)} \widetilde{F}_{si} - (\widetilde{\covD}'_{x})^{(m)} \widetilde{F}'_{si} $ by writing out all covariant derivatives and applying Leibniz's rule for $\dlt$, and then estimating using \eqref{eq:calTempGauge:est4U}, \eqref{eq:calTempGauge:est4Fsi}, \eqref{eq:calTempGauge:est4U:Diff}, \eqref{eq:calTempGauge:est4Fsi:Diff}, as well as \eqref{eq:calTempGauge:est4Apre}, \eqref{eq:calTempGauge:est4Apre:Diff} with interpolation. We leave the details to the interested reader.

Using \eqref{eq:calTempGauge:est4A:L2} and \eqref{eq:calTempGauge:est4A:Diff:L2}, we can easily substitute the covariant derivatives in \eqref{eq:calTempGauge:est4CovDFsi:Diff} by ordinary derivatives. This proves $\nrm{\nb_{x} (\dlt F_{s})}_{\calL^{5/4,p}_{s} \calH^{9}_{x}} \leq C_{\calIini} \cdot \dlt \calIini$ for $p =2, \infty$.

We are now only left to estimate the norms involving $\rd_{0}$. From Proposition \ref{prop:linCovHeat:Diff}, we have the following analogue of \eqref{eq:calTempGauge:est4F0i} for $\dlt \widetilde{F}_{0i}$ for $m \geq 0$.
\begin{equation} \label{eq:calTempGauge:est4F0i:Diff}
	\sup_{i} \nrm{\dlt \widetilde{F}_{0i}}_{\calL^{3/4,\infty}_{s} \dot{\calH}^{m}_{x}} + \sup_{i} \nrm{\dlt \widetilde{F}_{0i}}_{\calL^{3/4,2}_{s} \dot{\calH}^{m+1}_{x}}
	\leq C_{m, \calIini} \cdot \dlt \calIini.
\end{equation}
As before, since $( \rd_{0} \widetilde{A}_{i}- \rd_{0} \widetilde{A}'_{i} )(t=0, s=1) = (\widetilde{F}_{0i} - \widetilde{F}'_{0i}) (t=0, s=1)$,  this proves $\nrm{\rd_{0}(\dlt \Alow)}_{H^{30}_{x}} \leq C_{\calIini} \cdot \dlt \calIini$.

Finally, we turn to the estimate $\nrm{\nb_{0} (\dlt F_{s})}_{\calL^{5/4,p}_{s} \calH^{9}_{x}} \leq C_{\calIini} \cdot \dlt \calIini$ for $p=2, \infty$. For this purpose, note that we have the following difference versions of \eqref{eq:calTempGauge:DtFsi}, \eqref{eq:calTempGauge:Fs0}, \eqref{eq:calTempGauge:fund4A0}, as follows (in order).
\begin{equation*} 
\rd_{0} (\dlt \widetilde{F}_{si}) = \rd_{s} (\dlt \widetilde{F}_{0i}) + \rd_{i} (\dlt \widetilde{F}_{s0}) - \LieBr{\dlt \widetilde{A}_{0}}{\widetilde{F}_{si}} - \LieBr{\widetilde{A}'_{0}}{\dlt \widetilde{F}_{si}} + \LieBr{\dlt \widetilde{A}_{i}}{\widetilde{F}_{si}} + \LieBr{\widetilde{A}'_{i}}{\dlt \widetilde{F}_{si}},
\end{equation*}
\begin{equation*}
\dlt \widetilde{F}_{s0} = \rd^{\ell} (\dlt \widetilde{F}_{\ell 0} ) + \LieBr{\dlt \widetilde{A}^{\ell}}{\widetilde{F}_{\ell 0}} + \LieBr{(\widetilde{A}')^{\ell}}{\dlt \widetilde{F}_{\ell 0}},
\end{equation*}
\begin{equation*}
\dlt \widetilde{A}_{0} (s) = - \int_{s}^{1} \dlt \widetilde{F}_{s0}(s') \, \ud s'.
\end{equation*}
Taking the appropriate number of derivatives of the above equations and using the previous bounds, the desired estimate follows. We leave the easy details to the reader. \qedhere
\end{proof}

\section{Definition of norms and reduction of Theorem \ref{thm:dynEst}} \label{sec:redOfDynEst}
The purpose of the rest of the paper is to prove Theorem \ref{thm:dynEst}. In this section, which is of preliminary nature, we first introduce the various norms which will be used in the sequel, and reduce Theorem \ref{thm:dynEst} to six smaller statements: Propositions \ref{prop:est4a0}, \ref{prop:est4ai}, \ref{prop:est4Fs0:low} and \ref{prop:cont4FA}, and Theorems \ref{thm:AlowWave} and \ref{thm:FsWave}.

\subsection{Definition of norms} \label{subsec:defOfNorms}
In this subsection, we define the norms $\calA_{0}$, $\calAlow$, $\calF$ and $\calE$, along with their difference analogous.

Let $I \subset \bbR$ be a time interval. All the space-time norms below are taken over $I \times \bbR^{3}$.

The norms $\calA_{0}(I)$ and $\dlt \calA_{0}(I)$, which are used to estimate the gauge transform back to the temporal gauge $A_{0} = 0$ at $s=0$, are defined by
\begin{align*}
	\calA_{0}(I) := & \nrm{A_{0}}_{L^{\infty}_{t} L^{3}_{x}} + \nrm{\rd_{x} A_{0}}_{L^{\infty}_{t} L^{2}_{x}} + \nrm{A_{0}}_{L^{1}_{t} L^{\infty}_{x}} + \nrm{\rd_{x} A_{0}}_{L^{1}_{t} L^{3}_{x}} + \nrm{\rd_{x}^{(2)} A_{0}}_{L^{1}_{t} L^{2}_{x}}, \\
	\dlt \calA_{0}(I) := & \nrm{\dlt A_{0}}_{L^{\infty}_{t} L^{3}_{x}} + \nrm{\rd_{x} (\dlt A_{0})}_{L^{\infty}_{t} L^{2}_{x}} + \nrm{\dlt A_{0}}_{L^{1}_{t} L^{\infty}_{x}} + \nrm{\rd_{x} (\dlt A_{0})}_{L^{1}_{t} L^{3}_{x}} + \nrm{\rd_{x}^{(2)} (\dlt A_{0})}_{L^{1}_{t} L^{2}_{x}},
\end{align*}
where $A_{0}, \dlt A_{0}$ are evaluated at $s=0$.

The norms $\calAlow(I)$ and $\dlt \calAlow(I)$, which control the sizes of $\Alow_{i}$ and $\dlt \Alow_{i}$, respectively, are defined by
\begin{align*}
	\calAlow(I) := & \sup_{i} \nrm{\Alow_{i}}_{L^{\infty}_{t} \dot{H}^{31}_{x}} + \sup_{i} \nrm{\rd_{0} (\curl \Alow)_{i}}_{L^{\infty}_{t} \dot{H}^{29}_{x}} + \sum_{k=1}^{30} \sup_{i} \nrm{\Alow_{i}}_{\dot{S}^{k}},\\
	\dlt \calAlow(I) := &\sup_{i} \nrm{\dlt \Alow_{i}}_{L^{\infty}_{t} \dot{H}^{31}_{x}} + \sup_{i} \nrm{\rd_{0} (\curl (\dlt \Alow))_{i}}_{L^{\infty}_{t} \dot{H}^{29}_{x}} + \sum_{k=1}^{30} \sup_{i} \nrm{\dlt \Alow_{i}}_{\dot{S}^{k}}.
\end{align*}

Here, $(\rd_{x} \times B)_{i} := \eps_{i j k} \rd^{j} B^{k}$, where $\eps_{ijk}$ is the Levi-Civita symbol, i.e., the completely anti-symmetric 3-tensor on $\bbR^{3}$ with $\eps_{123} = 1$.

Next, let us define the norms $\calF(I)$ and $\dlt \calF(I)$, which control the sizes of $F_{si}$ and $\dlt F_{si}$, respectively.
\begin{align*}
	\calF(I) := & \sum_{k=1}^{10} \bb( \sup_{i} \nrm{F_{si}}_{\calL^{5/4, \infty}_{s} \dot{\calS}^{k}(0,1]} + \sup_{i} \nrm{F_{si}}_{\calL^{5/4, 2}_{s} \dot{\calS}^{k}(0,1]} \bb), \\
	\dlt \calF(I) := & \sum_{k=1}^{10} \bb( \sup_{i} \nrm{\dlt F_{si}}_{\calL^{5/4, \infty}_{s} \dot{\calS}^{k}(0,1]} + \sup_{i} \nrm{\dlt F_{si}}_{\calL^{5/4, 2}_{s} \dot{\calS}^{k}(0,1]} \bb).
\end{align*}

We remark that $\calF(I)$ (also $\dlt \calF(I)$) controls far less derivatives compared to $\calAlow(I)$. Nevertheless, it is still possible to close a bootstrap argument on $\calF + \calAlow$, thanks to the fact that $F_{si}$ satisfies a parabolic equation, which gives smoothing effects. The difference between the numbers of controlled derivatives, in turn, allows us to be lenient about the number of derivatives of $\Alow_{i}$ we use when studying the wave equation for $F_{si}$. We refer the reader to Remark \ref{rem:pEst4sptNrms} for a more detailed discussion.

For $t \in I$, we define $\calE(t)$ and $\dlt \calE(t)$, which control the sizes of low derivatives of $F_{s0}(t)$ and $\dlt F_{s0}(t)$, respectively, by
\begin{align*}
	\calE(t) :=& \sum_{m=1}^{3} \bb( \nrm{F_{s0}(t)}_{\calL^{1, \infty}_{s} \dot{\calH}^{m-1}_{x}(0,1]} + \nrm{F_{s0}(t)}_{\calL^{1, 2}_{s} \dot{\calH}^{m}_{x}(0,1]} \bb),\\
	\dlt \calE(t) := & \sum_{m=1}^{3} \bb( \nrm{\dlt F_{s0}(t)}_{\calL^{1, \infty}_{s} \dot{\calH}^{m-1}_{x}(0,1]} + \nrm{\dlt F_{s0}(t)}_{\calL^{1, 2}_{s} \dot{\calH}^{m}_{x}(0,1]} \bb).
\end{align*}

We furthermore define $\calE(I) := \sup_{t \in I} \calE(t)$ and $\dlt \calE(I) := \sup_{t \in I} \dlt \calE(t)$. 

\subsection{Statement of Propositions \ref{prop:est4a0} - \ref{prop:cont4FA} and Theorems \ref{thm:AlowWave}, \ref{thm:FsWave}}
For the economy of notation, we will omit the dependence of the quantities and norms on the time interval $(-T, T)$; in other words, all quantities and space-time norms below should be understood as being defined over the time interval $(-T, T)$ with $0 < T \leq 1$. 

\begin{proposition} [Improved estimates for $A_{0}$] \label{prop:est4a0}
Let $A_{\bfa}$, $A'_{\bfa}$ be regular solutions to \eqref{eq:HPYM} in the caloric-temporal gauge and $0 < T \leq 1$. Then the following estimates hold.
\begin{align} 
	\calA_{0}
	 \leq & C_{\calF, \calAlow} \cdot \calE + C_{\calF, \calAlow} \cdot (\calF + \calAlow)^{2}, \label{eq:est4a0} \\
	\dlt \calA_{0}
	\leq & C_{\calF, \calAlow} \cdot \dlt \calE + C_{\calF, \calAlow} \cdot (\calE + \calF + \calAlow) (\dlt \calF + \dlt \calAlow) . \label{eq:est4dltA0}
\end{align}
\end{proposition}

\begin{proposition} [Improved estimates for $A_{i}$] \label{prop:est4ai}
Let $A_{\bfa}$, $A'_{\bfa}$ be regular solutions to \eqref{eq:HPYM} in the caloric-temporal gauge and $0 < T \leq 1$. Then the following estimates hold.
\begin{align*} 
	\sup_{i} \sup_{0 \leq s \leq 1} \nrm{A_{i}(s)}_{\SH^{1}} \leq & C_{\calF, \calAlow} \cdot (\calF + \calAlow), \\
	\sup_{i} \sup_{0 \leq s \leq 1} \nrm{\dlt A_{i}(s)}_{\SH^{1}} \leq & C_{\calF, \calAlow} \cdot (\dlt \calF + \dlt \calAlow).
\end{align*}
\end{proposition}

\begin{proposition}[Estimates for $\calE$] \label{prop:est4Fs0:low}
Let $A_{\bfa}, A'_{\bfa}$ be regular solutions to \eqref{eq:HPYM} in the caloric-temporal gauge and $0 < T \leq 1$. Suppose furthermore that the smallness assumption
\begin{equation*}
	\calF + \calAlow \leq \dlt_{E},
\end{equation*}
holds for sufficiently small $\dlt_{E} > 0$. Then the following estimates hold.
\begin{align} 
	\calE \leq & C_{\calF, \calAlow} \cdot (\calF + \calAlow)^{2}, \label{eq:a0First:low:0} \\
	\dlt \calElow \leq & C_{\calF, \calAlow} \cdot (\calF+\calAlow) (\dlt \calF + \dlt \calAlow). \label{eq:a0First:low:1}
\end{align}
\end{proposition}

\begin{proposition}[Continuity properties of $\calF, \calAlow$] \label{prop:cont4FA}
Let $A_{\bfa}$, $A'_{\bfa}$ be regular solutions to \eqref{eq:HPYM} in the caloric-temporal gauge on some interval $I_{0} := (-T_{0}, T_{0})$. For $\calF = \calF(I), \calAlow = \calAlow(I)$ ($I \subset I_{0}$) and their difference analogues, the following continuity properties hold.
\begin{itemize}
\item The norms $\calF{(-T, T)}$ and $\calAlow{(-T, T)}$ are continuous as a function of $T$ (where $0 < T < T_{0}$).
\item Similarly, the norms $\dlt \calF{(-T, T)}$ and $\dlt \calAlow{(-T, T)}$ are continuous as a function of $T$.
\item We furthermore have
\begin{align*}
	\limsup_{T \to 0+} \bb( \calF{(-T, T)} + \calAlow{(-T, T)} \bb) \leq C \, \calI, \\
	\limsup_{T \to 0+} \bb( \dlt \calF{(-T, T)} + \dlt \calAlow{(-T, T)} \bb) \leq C \, \dlt \calI.
\end{align*}
\end{itemize}
\end{proposition}

\begin{theorem} [Hyperbolic estimates for $\Alow_{i}$] \label{thm:AlowWave}
Let $A_{\bfa}$, $A'_{\bfa}$ be regular solutions to \eqref{eq:HPYM} in the caloric-temporal gauge and $0 < T \leq 1$. Then the following estimates hold.
\begin{align} 
	\calAlow \leq & C \calI + T \bb( C_{\calF, \calAlow} \cdot \calE + C_{\calE, \calF, \calAlow} \cdot (\calE + \calF + \calAlow)^{2} \bb), \label{eq:AlowWave:1} \\
	\dlt \calAlow \leq & C \dlt \calI + T \bb( C_{\calF, \calAlow} \cdot \dlt \calE + C_{\calE, \calF, \calAlow} \cdot (\calE + \calF + \calAlow) (\dlt \calE + \dlt \calF + \dlt \calAlow) \bb). \label{eq:AlowWave:2}
\end{align}
\end{theorem}

\begin{theorem} [Hyperbolic estimates for $F_{si}$] \label{thm:FsWave}
Let $A_{\bfa}$, $A'_{\bfa}$ be regular solutions to \eqref{eq:HPYM} in the caloric-temporal gauge and $0 < T \leq 1$. Then the following estimates hold.
\begin{align} 
	\calF \leq  & C \calI + T^{1/2} C_{\calE, \calF, \calAlow} \cdot (\calE + \calF + \calAlow)^{2}, \label{eq:FsWave:1} \\
	\dlt \calF \leq & C \dlt \calI + T^{1/2} C_{\calE, \calF, \calAlow} \cdot (\calE + \calF + \calAlow) (\dlt \calE + \dlt \calF + \dlt \calAlow). \label{eq:FsWave:2}
\end{align}
\end{theorem}

A few remarks are in order concerning the above statements. 

The significance of Propositions \ref{prop:est4a0} and \ref{prop:est4ai} is that they allow us to pass from the quantities $\calF$ and $\calAlow$ to the norms of $A_{i}$ and $A_{0}$ on the left-hand side of \eqref{eq:dynEst:0}. Unfortunately, a naive approach to any of these will fail, leading to a logarithmic divergence. The structure of \eqref{eq:HPYM}, therefore, has to be used in a crucial way in order to overcome this. 

%Proposition \ref{prop:est4Fsi:L3} is needed to prove the $L^{3}_{x}$ estimates \eqref{eq:dynEst:1} and \eqref{eq:dynEst:3} of Theorem \ref{thm:dynEst}. Its proof will be given in Subsection \ref{subsec:FsiFirst}.

Proposition \ref{prop:est4Fs0:low}, which will be proved in \S \ref{subsec:pEst4Fs0}, deserves some special remarks. This is a perturbative result for the parabolic equation for $F_{s0}$, meaning that we need some smallness to estimate the nonlinearity. However, the latter fact has the implication that the required smallness cannot come from the size of the time interval, but rather only from the size of the data ($\calF + \calAlow$) or the size of the $s$-interval. It turns out that this feature causes a little complication in the proof of global well-posedness in \cite{Oh:2012fk}. Therefore, in \cite{Oh:2012fk}, we will prove a modified version of Proposition \ref{prop:est4Fs0:low}, using more covariant techniques to analyze the (covariant) parabolic equation for $F_{0i}$, which allows one to get around this issue.

In this work, to opt for simplicity, we have chosen to fix the $s$-interval to be $[0,1]$ and make $\calI$ (therefore $\calF + \calAlow$) small by scaling, exploiting the fact that $\calI$ is subcritical with respect to the scaling of the equation (compare this with the smallness required for Uhlenbeck's lemma, which cannot be obtained by scaling). We remark, however, that it would have been just as fine to keep $\calI$ large and obtain smallness by shrinking the size of the $s$-interval.

The proof of Theorem \ref{thm:dynEst} will be via a bootstrap argument for $\calF + \calAlow$, and Proposition \ref{prop:cont4FA} provides the necessary continuity properties. In fact, Proposition \ref{prop:cont4FA} is a triviality in view of the simplicity of our function spaces and the fact that $A_{\bfa}, A'_{\bfa}$ are regular. On the other hand, Theorems \ref{thm:AlowWave} and \ref{thm:FsWave}, obtained by analyzing the hyperbolic equations for $\Alow_{i}$ and $F_{si}$, respectively, give the main driving force of the bootstrap argument. Observe that these estimates themselves do not require any smallness. This will prove to be quite useful in the proof of global well-posedness in \cite{Oh:2012fk}. 

As we need to use some results derived from the parabolic equations of \eqref{eq:HPYM}, we will defer the proofs of Propositions \ref{prop:est4a0} -  \ref{prop:cont4FA}, along with further discussion, until Section \ref{sec:pfOfProps}. The proofs of Theorems \ref{thm:AlowWave} and \ref{thm:FsWave} will be the subject of Section \ref{sec:wave}.

%\comment{Discuss each statement.}

\subsection{Proof of Theorem \ref{thm:dynEst}}
Assuming the above statements, we are ready to prove Theorem \ref{thm:dynEst}.
\begin{proof} [Proof of Theorem \ref{thm:dynEst}]
Let $A_{\bfa}$, $A'_{\bfa}$ be regular solutions to \eqref{eq:HPYM} in the caloric-temporal gauge, defined on $(-T, T) \times \bbR^{3} \times [0,1]$. As usual, $\calI$ will control the sizes of both $A_{\bfa}$ and $A'_{\bfa}$ at $t=0$, in the manner described in Theorem \ref{thm:idEst}.

Let us prove \eqref{eq:dynEst:0}. We claim that 
\begin{equation} \label{eq:dynEst:pf:0}
	\calF(-T, T) + \calAlow(-T, T) \leq B \calI
\end{equation}
for a large constant $B$ to be determined later, and $\calI < \dlt_{H}$ with $\dlt_{H} >0 $ sufficiently small. By taking $B$ large enough, we obviously have $\calF(-T', T') + \calAlow(-T',T') \leq B \calI$ for $T' >0$ sufficiently small by Proposition \ref{prop:cont4FA}. This provides the starting point of the bootstrap argument.

Next, for $0 < T' \leq T$, let us assume the following \emph{bootstrap assumption}:
\begin{equation} \label{eq:dynEst:pf:1}
	\calF(-T', T') + \calAlow(-T', T') \leq 2 B \calI.
\end{equation}

The goal is to improve this to $\calF(-T', T') + \calAlow(-T', T') \leq B \calI$. 

Taking $2B\calI$ to be sufficiently small, we can apply Proposition \ref{prop:est4Fs0:low} and estimate $\calE \leq C_{\calF, \calAlow} (\calF+\calAlow)^{2}$. (We remark that in order to close the bootstrap, it is important that $\calE$ is at least quadratic in $(\calF + \calAlow)$.) Combining this with Theorems \ref{thm:AlowWave} and \ref{thm:FsWave}, and removing the powers of $T'$ by using the fact that $T' \leq T \leq 1$, we obtain
\begin{equation*}
	\calF(-T', T') + \calAlow(-T', T') \leq C \calI + C_{\calF(-T', T'), \calAlow(-T', T')} (\calF(-T', T') + \calAlow(-T', T'))^{2}.
\end{equation*}

Using the bootstrap assumption \eqref{eq:dynEst:pf:1} and taking $2B\calI$ to be sufficiently small, we can absorb the last term into the left-hand side and obtain
\begin{equation*}
	\calF(-T', T') + \calAlow(-T', T') \leq C \calI.
\end{equation*}

Therefore, taking $B$ sufficiently large, we beat the bootstrap assumption, i.e., $\calF(-T', T') + \calAlow(-T', T') \leq B \calI$. Using this, a standard continuity argument gives \eqref{eq:dynEst:pf:0} as desired.

From \eqref{eq:dynEst:pf:0}, estimate \eqref{eq:dynEst:0} follows immediately by Propositions \ref{prop:est4a0}, \ref{prop:est4ai} and \ref{prop:est4Fs0:low}. 

Next, let us turn to \eqref{eq:dynEst:2}. By essentially repeating the above proof for $\dlt \calF + \dlt \calAlow$, and using the estimate \eqref{eq:dynEst:pf:1} as well, we obtain the following difference analogue of \eqref{eq:dynEst:pf:0}:
\begin{equation} \label{eq:dynEst:pf:2}
	\dlt \calF(-T, T) + \dlt \calAlow(-T, T) \leq C_{\calI} \cdot \dlt \calI.
\end{equation}

From \eqref{eq:dynEst:pf:0} and \eqref{eq:dynEst:pf:2}, estimate \eqref{eq:dynEst:2} follows by Propositions \ref{prop:est4a0}, \ref{prop:est4ai} and \ref{prop:est4Fs0:low}.  \qedhere

%By interpolation, we immediately see that $\nrm{\rd_{t} A_{i}}_{L^{3}_{x}}\leq \calAlow$ and $\nrm{\rd_{t} (\dlt A_{i})}_{L^{3}_{x}}\leq \dlt \calAlow$. Furthermore, by Proposition \ref{prop:est4Fsi:L3}, we also have
%\begin{equation*}
%	\nrm{\rd_{s} A_{i}}_{C_{t} L^{3}_{x}} \leq C s^{-3/4} D^{1}(\calF+\calAlow), \quad \nrm{\rd_{s} (\dlt A_{i})}_{C_{t} L^{3}_{x}} \leq C s^{-3/4} D^{0}(\calF+\calAlow) (\dlt \calF + \dlt \calAlow),
%\end{equation*}
%where we observe that $s^{-3/4}$ is integrable on $((0,1], \ud s)$. Starting from $\nrm{\Alow_{i}}_{L^{3}_{x}}$ and $\nrm{\dlt \Alow_{i}}_{L^{3}_{x}}$ at $(t=0, s=1)$ and integrating using the above bounds, as well as \eqref{eq:dynEst:pf:0} and \eqref{eq:dynEst:pf:1}, we obtain \eqref{eq:dynEst:1} and \eqref{eq:dynEst:3}. \qedhere
%\comment{Simple bootstrap argument.}
\end{proof}

\section{Parabolic equations of the hyperbolic-parabolic Yang-Mills system} \label{sec:pEst4HPYM}
In this section, we analyze the parabolic equations of \eqref{eq:HPYM} for the variables $F_{si}, F_{s0}$ and $w_{i}$. The results of this analysis provide one of the `analytic pillars' of the proof of Theorem \ref{thm:dynEst} that had been outlined in Section \ref{sec:redOfDynEst}, the other `pillar' being the hyperbolic estimates in Section \ref{sec:wave}. Moreover, the hyperbolic estimates in Section \ref{sec:wave} depend heavily on the results of this section as well. 

As this section is a bit long, let us start with a brief outline. Beginning in \S \ref{subsec:pEst4HPYM:prelim} with some preliminaries, we prove in \S \ref{subsec:pEst4HPYM:pEst4Fsi} smoothing estimates for $F_{si}$ (Proposition \ref{prop:pEst4Fsi}), which allow us to control higher derivatives of $\rd_{t,x} F_{si}$ in terms of $\calF$, provided that we control high enough derivatives of $\rd_{t,x} \Alow_{i}$. In \S \ref{subsec:pEst4HPYM:lowEst4Fsi}, we also prove that $F_{si}$ itself (i.e., without any derivative) can be controlled in $L^{\infty}_{t} L^{2}_{x}$ and $L^{4}_{t,x}$ by $\calF + \calAlow$ as well (Proposition \ref{prop:lowEst4Fsi}). Next, in \S \ref{subsec:pEst4Fs0}, we study the parabolic equation for $F_{s0}$. Two main results of this subsection are Propositions \ref{prop:pEst4Fs0:low} and \ref{prop:pEst4Fs0:high}. The former states that low derivatives of $F_{s0}$ (i.e., $\calE$) can be controlled \emph{under the assumption that $\calF + \calAlow$ is small}, whereas the latter says that once $\calE$ is under control, higher derivatives of $F_{s0}$ can be controlled (with out any smallness assumption) as long as high enough derivatives of $\Alow_{i}$ are under control. Finally, in \S \ref{subsec:pEst4HPYM:pEst4wi}, we derive parabolic estimates for $w_{i}$ (Proposition \ref{prop:pEst4wi}). Although these are similar to those proved for $F_{s0}$, it is important (especially in view of the proof of finite energy global well-posedness in \cite{Oh:2012fk}) to note that no smallness of $\calF + \calAlow$ is required in this part.

Throughout the section, we will always work with regular solutions $A_{\bfa}, A'_{\bfa}$ to \eqref{eq:HPYM} on $I \times \bbR^{3} \times [0,1]$, where $I = (-T, T)$.

\subsection{Preliminary estimates} \label{subsec:pEst4HPYM:prelim}
Let us begin with a simple integral inequality.
\begin{lemma} \label{lem:SchurTest}
For $\dlt > 0$ and $1 \leq q \leq p \leq \infty$, the following estimate holds. 
\begin{equation*}
	\nrm{\int_{s}^{1} (s/s')^{\dlt} f(s') \, \frac{\ud s'}{s'}}_{\calL^{p}_{s} (0, 1]} \leq C_{\dlt, p,q} \nrm{f}_{\calL^{q}_{s} (0,1]}.
\end{equation*}
\end{lemma}
\begin{proof} 
This is rather a standard fact about integral operators. By interpolation, it suffices to consider the three cases $(p,q) = (1, 1), (\infty,1)$ and $(\infty, \infty)$. The first case follows by Fubini, using the fact that $\sup_{0 < s' \leq 1} \int_{0}^{1} 1_{[0,\infty)}(s'-s) (s/s')^{\dlt} \, \ud s / s \leq C_{\dlt}$, as $\dlt > 0$. On the other hand, the second and the third cases (i.e., $p=\infty$ and $q=1,\infty$) follow by H\"older, using furthermore the fact that $\sup_{0 < s, s' \leq 1} 1_{[0,\infty)}(s'-s) (s/s')^{\dlt} \leq 1$ and $\sup_{0 < s \leq 1} \int_{0}^{1} 1_{[0,\infty)}(s'-s) (s/s')^{\dlt} \, \ud s' / s' \leq C_{\dlt}$, respectively. \qedhere
\end{proof}

By the caloric-temporal gauge condition, we have $\rd_{s} A_{\mu} = F_{s \mu}$. Therefore, we can control $A_{\mu}$ with estimates for $F_{s\mu}$ and $\Alow_{\mu}$. The following two lemmas make this idea precise.

\begin{lemma} \label{lem:fundEst4A}
Let $X$ be a homogeneous norm of degree $2\ell_{0}$. Suppose furthermore that the caloric gauge condition $A_{s} = 0$ holds. Then for $k, \ell \geq 0$ and $1 \leq q \leq p \leq \infty$ such that $1/4 + k/2 + \ell - \ell_{0} > 0$, the following estimate holds.
\begin{equation*} 
\nrm{A_{i}}_{\calL^{1/4+\ell, p}_{s} \dot{\calX}^{k}(0,1]} 
\leq C (\nrm{F_{si}}_{\calL^{5/4, q}_{s} \dot{\calX}^{k}(0,1]} + \nrm{\Alow_{i}}_{\dot{X}^{k}}).
\end{equation*}
where $C$ depends on $p, q$ and $r(\ell, k, \ell_{0}) := 1/4 + k/2 + \ell - \ell_{0}$.
\end{lemma}
\begin{proof} 
By the caloric gauge condition $A_{s} = 0$, it follows that $\rd_{s} A_{i}= F_{si}$. By the fundamental theorem of calculus, we have
\begin{equation*}
	A_{i}(s) = -\int_{s}^{1} s' F_{si}(s') \, \frac{\ud s'}{s'} + \Alow_{i}.
\end{equation*}

Let us take the $\calL^{1/4+\ell, p}_{s} \dot{\calX}^{k}(0,1]$-norm of both sides. Defining $r(\ell, k, \ell_{0}) = 1/4 + k/2 + \ell - \ell_{0}$, we easily compute
%\begin{equation*}
%	s^{1/4+\ell} s'  \nrm{F_{si}(s')}_{\dot{\calX}^{k}(s)} = s^{(1/4)+\ell+(k/2)-\ell_{0}} s' \nrm{F_{si}(s')}_{\dot{X}^{k}} = (s/s')^{(1/4)+\ell+(k/2)-\ell_{0}} (s')^{\ell} (s')^{5/4} \nrm{F_{si}(s')}_{\dot{\calX}^{k}(s')}.
%\end{equation*}
\begin{align*}
	\nrm{\int_{s}^{1} s' F_{si}(s') \, \frac{\ud s'}{s'}}_{\calL^{1/4+\ell,p}_{s} \dot{\calX}^{k}(0,1]}
	= & \nrm{\int_{s}^{1}  (s/s')^{r(\ell, k, \ell_{0})} (s')^{\ell} (s')^{5/4}\nrm{F_{si}(s')}_{\dot{\calX}^{k}(s')} \, \frac{\ud s'}{s'}}_{\calL^{p}_{s}(0,1]} \\
	\leq & \nrm{\int_{s}^{1}  (s/s')^{r(\ell, k, \ell_{0})} (s')^{5/4}\nrm{F_{si}(s')}_{\dot{\calX}^{k}(s')} \, \frac{\ud s'}{s'}}_{\calL^{p}_{s}(0,1]}.
\end{align*}
where on the second line we used $\ell \geq 0$. Since $r > 0$, we can use Lemma \ref{lem:SchurTest} to estimate the last line by $C_{p, q, r}\nrm{F_{si}}_{\calL^{5/4,q}_{s} \dot{\calX}^{k}(0,1]}$.

On the other hand, $\Alow_{i}$ is independent of $s$, and therefore
\begin{equation*}
	\nrm{\Alow_{i}}_{\calL^{1/4+\ell, p}_{s} \dot{\calX}^{k}(0,1]} = \nrm{s^{r(\ell, k, \ell_{0})}}_{\calL^{p}_{s}(0,1]}\nrm{\Alow_{i}}_{\dot{X}^{k}} \leq C_{p, q, r} \nrm{\Alow_{i}}_{\dot{X}^{k}},
\end{equation*}
where the last inequality holds as $r > 0$. \qedhere
\end{proof}

The following analogous lemma for $A_{0}$, whose proof we omit, can be proved by a similar argument.
\begin{lemma} \label{lem:fundEst4A0}
Let $X$ be a homogeneous norm of degree $2\ell_{0}$. Suppose furthermore that the caloric-temporal gauge condition $A_{s} = 0$, $\Alow_{0} = 0$ holds. Then for $k, \ell \geq 0$ and $1 \leq q  \leq p \leq \infty$ such that $k/2 + \ell - \ell_{0} > 0$, the following estimate holds.
\begin{equation*} 
\nrm{A_{0}}_{\calL^{\ell, p}_{s} \dot{\calX}^{k}(0,1]} 
\leq C \nrm{F_{s0}}_{\calL^{1, q}_{s} \dot{\calX}^{k}(0,1]},
\end{equation*}
where $C$ depends on $p, q$ and $r'(\ell, k, \ell_{0}) := k/2 + \ell - \ell_{0}$.
\end{lemma}
%\begin{proof} 
%We proceed as in the proof of Lemma \ref{lem:fundEst4A}. By the caloric gauge condition $A_{s}=0$, we have $\rd_{s} A_{0} = F_{s0}$. By the temporal gauge condition at $s=1$, we have $\Alow_{0} =0$. Therefore, by the fundamental theorem of calculus,
%\begin{equation*}
%	A_{0}(s) = -\int_{s}^{1} s' F_{s0}(s') \, \frac{\ud s'}{s'}.
%\end{equation*}
%
%Let us take the $\calL^{0+\ell, p}_{s} \dot{\calX}^{k}(0,1]$-norm of both sides. Defining $r'(\ell, k, \ell_{0}) = (k/2)+\ell-\ell_{0}$, we easily compute
%%\begin{equation*}
%%	s^{\ell} s'  \nrm{F_{s0}(s')}_{\dot{\calX}^{k}(s)} = s^{\ell+(k/2)-\ell_{0}} s' \nrm{F_{s0}(s')}_{\dot{X}^{k}} = (s/s')^{\ell+(k/2)-\ell_{0}} (s')^{\ell} s' \nrm{F_{si0(s')}}_{\dot{\calX}^{k}(s')}.
%%\end{equation*}
%\begin{align*}
%	\nrm{\int_{s}^{1} s' F_{s0}(s') \, \frac{\ud s'}{s'}}_{\calL^{0+\ell, p}_{s} \dot{\calX}^{k}(0,1]}
%	= & \nrm{\int_{s}^{1}  (s/s')^{r'(\ell, k, \ell_{0})} (s')^{\ell} s' \nrm{F_{s0}(s')}_{\dot{\calX}^{k}(s')} \, \frac{\ud s'}{s'}}_{\calL^{p}_{s}(0,1]} \\
%	\leq & \nrm{\int_{s}^{1}  (s/s')^{r'(\ell, k, \ell_{0})} s' \nrm{F_{s0}(s')}_{\dot{\calX}^{k}(s')} \, \frac{\ud s'}{s'}}_{\calL^{p}_{s}(0,1]} ,
%\end{align*}
%Since $r > 0$, we can apply Lemma \ref{lem:SchurTest} to estimate the last line by $C_{p,q, r'} \nrm{F_{s0}}_{\calL^{1,q}_{s} \dot{\calX}^{k}(0,1]}$. \qedhere
%\end{proof}

Some of the most frequently used choices of $X$ are $X=\dot{S}^{k}$ for Lemma \ref{lem:fundEst4A}, $X=L^{2}_{t} \dot{H}^{k}_{x}$ for Lemma \ref{lem:fundEst4A0}, and $X = \dot{H}^{k}_{x}, \dot{W}^{k,\infty}_{x}$ for both. Moreover, these lemmas will frequently applied to norms which can be written as a sum of such norms, e.g. $\calL^{\ell, p}_{s} \calH^{m}_{x}$, which is the sum of $\calL^{\ell, p}_{s} \dot{\calH}^{k}_{x}$ norms for $k=0, \ldots, m$.

As an application of the previous lemmas, we end this subsection with estimates for some components of the curvature $2$-form and its covariant derivative.
\begin{lemma} [Bounds for $F_{0i}$] \label{lem:fundEst4F0i}
Suppose that the caloric-temporal gauge condition $A_{s} = 0$, $\Alow_{0}=0$ holds. Then:
\begin{enumerate}
\item The following estimate holds for $2 \leq p \leq \infty$:
\begin{equation} \label{eq:fundEst4F0i:1}
\begin{aligned}
	\nrm{F_{0i}(t)}_{\calL^{3/4,p}_{s} \dot{\calH}^{1/2}_{x}}
	\leq & C_{p} \big( \nrm{\nb_{0} F_{si}(t)}_{\calL^{5/4,2}_{s} \dot{\calH}^{1/2}_{x}} + \nrm{\rd_{0} \Alow_{i}(t)}_{\dot{H}^{1/2}_{x}} + \nrm{F_{s0}(t)}_{\calL^{1,2}_{s} \dot{\calH}^{3/2}_{x}} \\
	& + \nrm{\nb_{x} F_{s0}(t)}_{\calL^{1,2}_{s} \calL^{2}_{x}} (\nrm{\nb_{x} F_{si}(t)}_{\calL^{5/4,2}_{s} \calL^{2}_{x}} + \nrm{\rd_{x} \Alow_{i}(t)}_{L^{2}_{x}}) \big).
\end{aligned}
\end{equation}

\item For any $2 \leq p \leq \infty$ and $k \geq 1$ an integer, we have
\begin{equation} \label{eq:fundEst4F0i:2}
\begin{aligned}
	\nrm{F_{0i}(t)}_{\calL^{3/4,p}_{s} \dot{\calH}^{k}_{x}} 
	\leq & C_{p, k} \big(  \nrm{\nb_{0} F_{si}(t)}_{\calL^{5/4,2}_{s} \dot{\calH}^{k}_{x}} + \nrm{\rd_{0} \Alow_{i}(t)}_{\dot{H}^{k}_{x}} 
	+ \nrm{F_{s0}(t)}_{\calL^{1,2}_{s} \dot{\calH}^{k+1}_{x}} \\ 		
	& \quad + \nrm{\nb_{x} F_{s0}(t)}_{\calL^{1,2}_{s} \calH^{k}_{x}} (\nrm{\nb_{x} F_{si}(t)}_{\calL^{5/4,2}_{s} \calH^{k}_{x}} + \nrm{\rd_{x} \Alow_{i}(t)}_{H^{k}_{x}}) \big).
\end{aligned}
\end{equation}

\item For any $2 \leq p \leq \infty$ and $k \geq 0$ an integer, we have
\begin{equation} \label{eq:fundEst4F0i:3}
\begin{aligned}
	\nrm{F_{0i}}_{\calL^{3/4,p}_{s} \calL^{4}_{t} \dot{\calW}^{k,4}_{x}} 
	\leq & C_{p, k} \big(  \nrm{F_{si}}_{\calL^{5/4,2}_{s} \dot{\calS}^{k+3/2}} + \nrm{\Alow_{i}}_{\dot{S}^{k+3/2}} 
	+ T^{1/4} \sup_{t \in I}\nrm{F_{s0}(t)}_{\calL^{1,2}_{s}  \dot{\calH}^{k+7/4}_{x}} \\ 	
	& \quad + T^{1/4} \sup_{t \in I} \nrm{\nb_{x} F_{s0}(t)}_{\calL^{1,2}_{s} \calH^{k+1}_{x}} (\nrm{F_{si}}_{\calL^{5/4,2}_{s} \widehat{\calS}^{k+1}} + \nrm{\Alow_{i}}_{\widehat{S}^{k+1}}) \big).
\end{aligned}
\end{equation}
\end{enumerate}
\end{lemma}
\begin{proof} 
Let us begin with the identity
\begin{equation*}
	F_{0i} = \rd_{0} A_{i} - \rd_{i} A_{0} + \LieBr{A_{0}}{A_{i}} = s^{-1/2} \nb_{0} A_{i} + s^{-1/2} \nb_{i} A_{0} + \LieBr{A_{0}}{A_{i}},
\end{equation*} 

Applying Lemma \ref{lem:fundEst4A} to $s^{-1/2} \nb_{0} A_{i}$ and Lemma \ref{lem:fundEst4A0} to $s^{-1/2}  \nb_{x} A_{0}$, the estimates \eqref{eq:fundEst4F0i:1} and  \eqref{eq:fundEst4F0i:2} are reduced to the product estimates
\begin{align}
	& \nrm{\LieBr{A_{0}}{A_{i}}(t)}_{\calL^{3/4,p}_{s} \dot{\calH}^{1/2}_{x}} 
	\leq C_{p} \nrm{\nb_{x} F_{s0}(t)}_{\calL^{1,2}_{s} \calL^{2}_{x}} (\nrm{\nb_{x} F_{si}(t)}_{\calL^{5/4,2}_{s} \calL^{2}_{x}} + \nrm{\rd_{x} \Alow_{i}(t)}_{L^{2}_{x}}), \label{eq:fundEst4F0i:pf:1} \\
	& \nrm{\LieBr{A_{0}}{A_{i}}(t)}_{\calL^{3/4,p}_{s} \dot{\calH}^{k}_{x}} 
	\leq C_{p} \nrm{\nb_{x} F_{s0}(t)}_{\calL^{1,2}_{s} \calH^{k}_{x}} (\nrm{\nb_{x} F_{si}(t)}_{\calL^{5/4,2}_{s} \calH^{k}_{x}} + \nrm{\rd_{x} \Alow_{i}(t)}_{H^{k}_{x}}), \label{eq:fundEst4F0i:pf:2} 
\end{align}
respectively. 

%All these estimates can be proved by the same strategy of first proving an $s$-independent product estimate, using the Correspondence Principle and H\"older for $\calL^{\ell, p}_{s}$ (Lemma \ref{lem:absP:Holder4Ls}) to derive an estimate for $\LieBr{A_{0}}{A_{i}}$, and then applying Lemmas \ref{lem:fundEst4A} and \ref{lem:fundEst4A0} to estimate $A_{i}$ and $A_{0}$ in terms of $F_{si}$, $\Alow_{i}$ and $F_{s0}$, respectively. 

Let us start with the product estimate
\begin{equation} \label{eq:fundEst4F0i:pf:4}
	\nrm{\phi_{1} \phi_{2}}_{\dot{H}^{1/2}_{x}} \leq C \nrm{\phi_{1}}_{\dot{H}^{1}_{x}} \nrm{\phi_{2}}_{\dot{H}^{1}_{x}},
\end{equation}
which follows from the product rule for homogeneous Sobolev norms (Lemma \ref{lem:homSob}). Applying the Correspondence Principle and Lemma \ref{lem:absP:Holder4Ls}, we obtain
\begin{equation*}
	\nrm{\LieBr{A_{0}}{A_{i}}(t)}_{\calL^{3/4,p}_{s} \dot{\calH}^{1/2}_{x}} \leq C_{p} \nrm{A_{0}(t)}_{\calL^{0+3/8,\infty}_{s} \dot{\calH}^{1}_{x}} \nrm{A_{i}(t)}_{\calL^{1/4+1/8,p}_{s} \dot{\calH}^{1}_{x}}
\end{equation*}

Note the extra weights of $s^{3/8}$ and $s^{1/8}$ for $A_{0}$ and $A_{i}$, respectively. Applying Lemma \ref{lem:fundEst4A0} to $A_{0}$ and Lemma \ref{lem:fundEst4A} to $A_{i}$, the desired estimate \eqref{eq:fundEst4F0i:pf:1} follows.

The other product estimate \eqref{eq:fundEst4F0i:pf:2} can be proved by a similar argument, this time starting with 
$\nrm{\phi_{1} \phi_{2}}_{\dot{H}^{1}_{x}} \leq C \nrm{\phi_{1}}_{\dot{H}^{5/4}_{x}} \nrm{\phi_{2}}_{\dot{H}^{5/4}_{x}}$,
(which follows again from Lemma \ref{lem:homSob}) instead of \eqref{eq:fundEst4F0i:pf:4}, and using Leibniz's rule to deal with the cases $k \geq 2$. 

Finally, let us turn to \eqref{eq:fundEst4F0i:3}. We use Lemma \ref{lem:fundEst4A} and Strichartz to control $s^{-1/2} \nb_{0} A_{i}$, and Lemma \ref{lem:fundEst4A0}, H\"older in time and Sobolev for $s^{-1/2} \nb_{x} A_{0}$. Then we are left to establish
\begin{equation} \label{eq:fundEst4F0i:pf:3}
	\nrm{\LieBr{A_{0}}{A_{i}}}_{\calL^{3/4,p}_{s} \calL^{4}_{t} \dot{\calW}^{k,4}_{x}}
	\leq C_{p} T^{1/4} \sup_{t \in I} \nrm{\nb_{x} F_{s0}(t)}_{\calL^{1,2}_{s} \calH^{k+1}_{x}} (\nrm{F_{si}}_{\calL^{5/4,2}_{s} \widehat{\calS}^{k+1}} + \nrm{\Alow_{i}}_{\widehat{S}^{k+1}}).
\end{equation}

To prove \eqref{eq:fundEst4F0i:pf:3}, one starts with
$\nrm{\phi_{1} \phi_{2}}_{L^{4}_{t} L^{4}_{x}} \leq C \abs{I}^{1/4} \nrm{\phi_{1}}_{L^{\infty}_{t} \dot{H}^{5/4}_{x}} \nrm{\phi_{2}}_{L^{\infty}_{t} \dot{H}^{1}_{x}}$,
(which follows via H\"older and Sobolev) instead of \eqref{eq:fundEst4F0i:pf:4}. Using Leibniz's rule, the Correspondence Principle and Lemma \ref{lem:absP:Holder4Ls}, we obtain for $k \geq 0$
\begin{equation*}
	\nrm{\LieBr{A_{0}}{A_{i}}}_{\calL^{3/4,p}_{s} \calL^{4}_{t} \dot{\calW}^{k,4}_{x}} 
	\leq C T^{1/4} \sum_{j=0}^{k} \nrm{A_{0}}_{\calL^{0+5/16,\infty}_{s} \calL^{\infty}_{t} \dot{\calH}^{j+5/4}_{x}} \nrm{A_{i}}_{\calL^{1/4+1/16, p}_{s} \calL^{\infty}_{t} \dot{\calH}^{k+1-j}_{x}}
\end{equation*}

%By interpolation and the Correspondence Principle, we also have
%\begin{equation}\label{eq:fundEst4F0i:pf:5}
%\nrm{A_{0}(t)}_{\calL^{0+3/8,\infty}_{s} \calL^{\infty}_{t} \dot{\calW}^{j,12}_{x}} \leq C \big( \sup_{t \in I} \nrm{A_{0}(t)}_{\calL^{0+3/8,\infty}_{s} \dot{\calH}^{j+1}_{x}}\big)^{3/4} \big( \sup_{t \in I} \nrm{A_{0}(t)}_{\calL^{0+3/8,\infty}_{s} \dot{\calH}^{j+2}_{x}} \big)^{1/4}.
%\end{equation}

Now we are in position to apply Lemmas \ref{lem:fundEst4A} and \ref{lem:fundEst4A0} to $A_{i}$ and $A_{0}$, respectively. Using furthermore $\nrm{\nb_{x} F_{si}}_{\calL^{5/4,2}_{s} \calL^{\infty}_{t} \calH^{k}_{x}} \leq \nrm{F_{si}}_{\calL^{5/4,2}_{s} \widehat{\calS}^{k+1}}$, $\nrm{\rd_{x} \Alow_{i}}_{L^{\infty}_{t} H^{k}_{x}} \leq \nrm{\Alow_{i}}_{\widehat{S}^{k+1}}$, \eqref{eq:fundEst4F0i:pf:3} follows. \qedhere
\end{proof}

By the same proof applied to $\dlt F_{0i}$, we obtain the following difference analogue of Lemma \ref{lem:fundEst4F0i}.
\begin{lemma} [Bounds for $\dlt F_{0i}$] \label{lem:fundEst4F0i:Diff}
Suppose that the caloric-temporal gauge condition $A_{s} = 0$, $\Alow_{0}=0$ holds (for both $A$ and $A'$). Then:
\begin{enumerate}
\item The following estimate holds for $2 \leq p \leq \infty$:
\begin{equation} \label{eq:fundEst4F0i:Diff:1}
\begin{aligned}
	\nrm{\dlt F_{0i}(t)}_{\calL^{3/4,p}_{s} \dot{\calH}^{1/2}_{x}}
	\leq & C_{p} ( \nrm{\nb_{0} (\dlt F_{si})(t)}_{\calL^{5/4,2}_{s} \dot{\calH}^{1/2}_{x}} + \nrm{\rd_{0} (\dlt \Alow_{i})(t)}_{\dot{H}^{1/2}_{x}} + \nrm{\dlt F_{s0}(t)}_{\calL^{1,2}_{s} \dot{\calH}^{3/2}_{x}} )\\
	& + C_{\nrm{\nb_{x} F_{si}(t)}_{\calL^{5/4,2}_{s} \calL^{2}_{x}}, \nrm{\rd_{x} \Alow_{i}(t)}_{L^{2}_{x}}} \cdot \nrm{\nb_{x} (\dlt F_{s0})(t)}_{\calL^{1,2}_{s} \calL^{2}_{x}} \\
	& + C_{\nrm{\nb_{x} F_{s0}(t)}_{\calL^{1,2}_{s} \calL^{2}_{x}}} \cdot (\nrm{\nb_{x} (\dlt F_{si}) (t)}_{\calL^{5/4,2}_{s} \calL^{2}_{x}} + \nrm{\rd_{x} (\dlt \Alow_{i})(t)}_{L^{2}_{x}}) 
\end{aligned}
\end{equation}

\item For any $2 \leq p \leq \infty$ and $k \geq 1$, we have
\begin{equation} \label{eq:fundEst4F0i:Diff:2}
\begin{aligned}
	\nrm{\dlt F_{0i}(t)}_{\calL^{3/4,p}_{s} \dot{\calH}^{k}_{x}} 
	\leq & C_{p, k} (  \nrm{\nb_{0} (\dlt F_{si})(t)}_{\calL^{5/4,2}_{s} \dot{\calH}^{k}_{x}} + \nrm{\rd_{0} (\dlt \Alow_{i})(t)}_{\dot{H}^{k}_{x}} )
	+ \nrm{\dlt F_{s0}(t)}_{\calL^{1,2}_{s} \dot{\calH}^{k+1}_{x}} \\ 		
	& + C_{\nrm{\nb_{x} F_{si}(t)}_{\calL^{5/4,2}_{s} \calH^{k}_{x}}, \nrm{\rd_{x} \Alow_{i}(t)}_{H^{k}_{x}}} \cdot \nrm{\nb_{x} (\dlt F_{s0})(t)}_{\calL^{1,2}_{s} \calH^{k}_{x}} \\
	& + C_{\nrm{\nb_{x} F_{s0}(t)}_{\calL^{1,2}_{s} \calH^{k}_{x}}} \cdot (\nrm{\nb_{x} (\dlt F_{si}) (t)}_{\calL^{5/4,2}_{s} \calH^{k}_{x}} + \nrm{\rd_{x} (\dlt \Alow_{i})(t)}_{H^{k}_{x}}).
\end{aligned}
\end{equation}

\item For any $2 \leq p \leq \infty$ and $k \geq 0$, we have
\begin{equation} \label{eq:fundEst4F0i:Diff:3}
\begin{aligned}
	\nrm{\dlt F_{0i}}_{\calL^{3/4,p}_{s} \calL^{4}_{t} \dot{\calW}^{k,4}_{x}} 
	\leq & C_{p, k} (  \nrm{\dlt F_{si}}_{\calL^{5/4,2}_{s} \dot{\calS}^{k+3/2}} + \nrm{\dlt \Alow_{i}}_{\dot{S}^{k+3/2}} 
	+ T^{1/4} \sup_{t \in I}\nrm{\dlt F_{s0}(t)}_{\calL^{1,2}_{s}  \dot{\calH}^{k+7/4}_{x}} ) \\ 	
	& + T^{1/4} C_{\nrm{F_{si}}_{\calL^{5/4,2}_{s} \widehat{\calS}^{k+1}}, \nrm{\Alow_{i}}_{\widehat{S}^{k+1}}} \cdot \sup_{t \in I} \nrm{\nb_{x} F_{s0}(t)}_{\calL^{1,2}_{s} \calH^{k+1}_{x}}  \\
	& + T^{1/4} C_{\sup_{t \in I} \nrm{\nb_{x} F_{s0}(t)}_{\calL^{1,2}_{s} \calH^{k+1}_{x}}} \cdot (\nrm{\dlt F_{si}}_{\calL^{5/4,2}_{s} \widehat{\calS}^{k+1}} + \nrm{\dlt \Alow_{i}}_{\widehat{S}^{k+1}}) .
\end{aligned}
\end{equation}
\end{enumerate}
\end{lemma}
 
Next, we derive estimates for $\covD_{0} F_{ij} + \covD_{i} F_{0j}$.
\begin{lemma} [Bounds for $\covD_{0} F_{ij}$ and $\covD_{i} F_{0j}$] \label{lem:fundEst4DF}
Suppose that the caloric-temporal gauge condition $A_{s} = 0$, $\Alow_{0}=0$ holds.
\begin{enumerate}
\item For any $2 \leq p \leq \infty$ and $k \geq 0$, we have
\begin{equation} \label{eq:fundEst4DF:1}
\begin{aligned}
	& \nrm{\covD_{0} F_{ij}(t)}_{\calL^{5/4,p}_{s} \dot{\calH}^{k}_{x}} + \nrm{\covD_{i} F_{0j}(t)}_{\calL^{5/4,p}_{s} \dot{\calH}^{k}_{x}} \\
	& \qquad \leq  C_{p, k} \bb(  \nrm{\nb_{0} F_{si}(t)}_{\calL^{5/4,2}_{s} \dot{\calH}^{k+1}_{x}} + \nrm{\rd_{0} \Alow_{i}(t)}_{\dot{H}^{k+1}_{x}} 
	+ \nrm{F_{s0}(t)}_{\calL^{1,2}_{s} \dot{\calH}^{k+2}_{x}} \\ 		
	& \phantom{\qquad \leq } + (\nrm{\nb_{t,x} F_{si}(t)}_{\calL^{5/4,2}_{s} \calH^{k+1}_{x}} + \nrm{\rd_{t,x} \Alow_{i}(t)}_{H^{k+1}_{x}} 
	+ \nrm{\nb_{x} F_{s0}(t)}_{\calL^{1,2}_{s} \calH^{k+1}_{x}})^{2} \\
	& \phantom{\qquad \leq } + (\nrm{\nb_{t,x} F_{si}(t)}_{\calL^{5/4,2}_{s} \calH^{k+1}_{x}} + \nrm{\rd_{t,x} \Alow_{i}(t)}_{H^{k+1}_{x}} 
	+ \nrm{\nb_{x} F_{s0}(t)}_{\calL^{1,2}_{s} \calH^{k+1}_{x}})^{3} \bb).
\end{aligned}
\end{equation}

\item For any $2 \leq  p \leq \infty$ and $k \geq 0$, we have
\begin{equation} \label{eq:fundEst4DF:2}
\begin{aligned}
	& \nrm{\covD_{0} F_{ij}}_{\calL^{5/4,p}_{s} \calL^{4}_{t} \dot{\calW}^{k,4}_{x}} + \nrm{\covD_{i} F_{0j}}_{\calL^{5/4,p}_{s} \calL^{4}_{t} \dot{\calW}^{k,4}_{x}} \\
	& \qquad \leq  C_{p, k} \bb(  \nrm{F_{si}}_{\calL^{5/4,2}_{s} \dot{\calS}^{k+5/2}} + \nrm{\Alow_{i}}_{\dot{S}^{k+5/2}_{x}} 
	+ T^{1/4} \sup_{t \in I} \nrm{F_{s0}(t)}_{\calL^{1,2}_{s} \dot{\calH}^{k+11/4}_{x}} \\ 		
	& \phantom{\qquad \leq } + (\nrm{F_{si}}_{\calL^{5/4,2}_{s} \widehat{\calS}^{k+2}} + \nrm{\Alow_{i}}_{\widehat{S}^{k+2}} 
	+ \sup_{t \in I} \nrm{\nb_{x} F_{s0}(t)}_{\calL^{1,2}_{s} \calH^{k+1}_{x}})^{2} \\
	& \phantom{\qquad \leq } + (\nrm{F_{si}}_{\calL^{5/4,2}_{s} \widehat{\calS}^{k+2}} + \nrm{\Alow_{i}}_{\widehat{S}^{k+2}} 
	+ \sup_{t \in I} \nrm{\nb_{x} F_{s0}(t)}_{\calL^{1,2}_{s} \calH^{k+1}_{x}})^{3} \bb).
\end{aligned}
\end{equation}
\end{enumerate}
\end{lemma}

\begin{proof} 
The proof proceeds in a similar manner as Lemma \ref{lem:fundEst4F0i}. We will give a treatment of the contribution of the term $\nrm{\covD_{0} F_{ij}}$, and leave the similar case of $\nrm{\covD_{i} F_{0j}}$ to the reader.

Our starting point is the schematic identity
\begin{equation} \label{eq:fundEst4DF:pf:1}
	\covD_{0} F_{ij} = s^{-1} \calO(\nb_{0} \nb_{x} A) + s^{-1/2} \calO(A_{0}, \nb_{x} A) + s^{-1/2} \calO(A, \nb_{0} A) + \calO(A, A, A_{0}),
\end{equation} 
which can be checked easily by expanding $\covD_{0} F_{ij}$ in terms of $A_{\mu}$. 

The first term on the right-hand side of \eqref{eq:fundEst4DF:pf:1} is acceptable for both \eqref{eq:fundEst4DF:1} and \eqref{eq:fundEst4DF:2}, thanks to Lemma \ref{lem:fundEst4A}. Therefore, it remains to treat only the bilinear and trilinear terms in \eqref{eq:fundEst4DF:pf:1}.

Let us begin with the proof of \eqref{eq:fundEst4DF:1}.  For the bilinear terms (i.e., the second and the third terms), we start with the inequality 
$\nrm{\phi_{1} \phi_{2}}_{L^{2}_{x}} \leq C \nrm{\phi_{1}}_{\dot{H}^{1}_{x}} \nrm{\phi_{2}}_{\dot{H}^{1/2}_{x}}$, 
which follows from Lemma \ref{lem:homSob}. Applying Leibniz's rule, the Correspondence Principle and Lemma \ref{lem:absP:Holder4Ls}, we obtain for $k \geq 0$
\begin{align*}
	& \nrm{s^{-1/2} \calO(A_{0}, \nb_{x} A)(t)}_{\calL^{5/4,p}_{s} \dot{\calH}^{k}_{x}} + \nrm{s^{-1/2} \calO(A, \nb_{0} A)(t)}_{\calL^{5/4,p}_{s} \dot{\calH}^{k}_{x}} \\
	& \quad \leq C \nrm{\nb_{x}  A_{0}}_{\calL^{0+3/8,\infty}_{s} \calH^{k}_{x}} \nrm{\nb_{x}  A}_{\calL^{1/4+1/8,p}_{s} \calH^{k+1}_{x}} 
	+  \nrm{\nb_{x}  A}_{\calL^{1/4+1/8,\infty}_{s} \calH^{k}_{x}} \nrm{\nb_{0}  A}_{\calL^{1/4+1/8,p}_{s} \calH^{k+1}_{x}}.
\end{align*}

Applying Lemma \ref{lem:fundEst4A} to $A$ and Lemma \ref{lem:fundEst4A0} to $A_{0}$, we see that the bilinear terms on the right-hand side of \eqref{eq:fundEst4DF:pf:1} are also okay.

Finally, for the trilinear term, we start with the inequality
$\nrm{\phi_{1} \phi_{2} \phi_{3}}_{L^{2}_{x}} \leq C \nrm{\phi_{1}}_{\dot{H}^{1}_{x}} \nrm{\phi_{2}}_{\dot{H}^{1}_{x}} \nrm{\phi_{3}}_{\dot{H}^{1}_{x}}$.
By Leibniz's rule, the Correspondence Principle and Lemma \ref{lem:absP:Holder4Ls}, we obtain for $k \geq 0$
\begin{equation*}
	\nrm{\calO(A, A, A_{0})}_{\calL^{5/4,p}_{s} \dot{\calH}^{k}_{x}} \leq \nrm{\nb_{x} A}_{\calL^{1/4+1/6,\infty}_{s} \calH^{k}_{x}} \nrm{\nb_{x} A}_{\calL^{1/4+1/6,p}_{s} \calH^{k}_{x}} \nrm{\nb_{x} A_{0}}_{\calL^{0+5/12,\infty}_{s} \calH^{k}_{x}}.
\end{equation*}

Applying Lemma \ref{lem:fundEst4A} to $A$ and Lemma \ref{lem:fundEst4A0} to $A_{0}$, we see that the last term on the right-hand side of \eqref{eq:fundEst4DF:pf:1} is acceptable. This proves \eqref{eq:fundEst4DF:1}.

Next, let us prove \eqref{eq:fundEst4DF:2}, which proceeds in an analogous way. For the bilinear terms, we begin with the obvious inequality
$\nrm{\phi_{1} \phi_{2}}_{L^{4}_{t,x}} \leq C \nrm{\phi_{1}}_{L^{\infty}_{t,x}} \nrm{\phi_{2}}_{L^{4}_{t,x}}$. Applying Leibniz's rule, the Correspondence Principle, Lemma \ref{lem:absP:Holder4Ls} and Lemma \ref{lem:absP:algEst}, we obtain for $k \geq 0$
\begin{align*}
	& \nrm{s^{-1/2} \calO(A_{0}, \nb_{x} A)}_{\calL^{5/4,p}_{s} \calL^{4}_{t} \dot{\calW}^{k,4}_{x}} + \nrm{s^{-1/2} \calO(A, \nb_{0} A)}_{\calL^{5/4,p}_{s} \calL^{4}_{t} \dot{\calW}^{k,4}_{x}} \\
	&\qquad \leq C \sup_{t \in I} \nrm{\nb_{x}  A_{0}(t)}_{\calL^{0+3/8,\infty}_{s} \calH^{k+1}_{x}} \nrm{\nb_{x}  A}_{\calL^{1/4+1/8,p}_{s} \calL^{4}_{t} \calW^{k,4}_{x}}  \\
	& \phantom{\qquad \leq } + C\nrm{\nb_{x}  A}_{\calL^{1/4+1/8,\infty}_{s} \calL^{\infty}_{t} \calH^{k+1}_{x}} \nrm{\nb_{0}  A}_{\calL^{1/4+1/8,p}_{s} \calL^{4}_{t} \calW^{k,4}_{x}}.
\end{align*}

Using Strichartz and the Correspondence Principle, we can estimate $\nrm{\nb_{t,x}  A}_{\calL^{1/4+1/8,p}_{s} \calL^{4}_{t} \calW^{k,4}_{x}} \leq C \nrm{A}_{\calL^{1/4+1/8,p}_{s} \widehat{\calS}^{k+2}}$. Then applying Lemma \ref{lem:fundEst4A} to $A$ and Lemma \ref{lem:fundEst4A0} to $A_{0}$, it easily follows that the bilinear terms on the right-hand side of \eqref{eq:fundEst4DF:pf:1} are acceptable.

For the trilinear term, we start with the inequality
$\nrm{\phi_{1} \phi_{2} \phi_{3}}_{L^{4}_{t,x}} \leq C \nrm{\phi_{1}}_{L^{4}_{t} L^{12}_{x}} \nrm{\phi_{2}}_{L^{\infty}_{t} L^{12}_{x}} \nrm{\phi_{3}}_{L^{\infty}_{t} L^{12}_{x}}$. 
By Leibniz's rule, the Correspondence Principle and Lemma \ref{lem:absP:Holder4Ls}, we obtain
\begin{equation*}
\begin{aligned}
	\nrm{\calO(A, A, A_{0})}_{\calL^{5/4,p}_{s} \calL^{4}_{t} \dot{\calW}^{k,4}_{x}} 
	\leq & C \nrm{A}_{\calL^{1/4+1/6,p}_{s} \calL^{4}_{t} \calW^{k,12}_{x}} \nrm{A}_{\calL^{1/4+1/6,\infty}_{s} \calL^{\infty}_{t} \calW^{k,12}_{x}} \\
	& \cdot \sup_{t \in I} \nrm{A_{0}(t)}_{\calL^{0+5/12,\infty}_{s} \calW^{k,12}_{x}}.
\end{aligned}
\end{equation*}

Using Strichartz and the Correspondence Principle, let us estimate the first factor $\nrm{A}_{\calL^{1/4+1/6,\infty}_{s} \calL^{4}_{t} \calW^{k,12}_{x}}$ by $C \nrm{A}_{\calL^{1/4+1/6,p}_{s} \widehat{\calS}^{k+2}}$. Next, using interpolation and the Correspondence Principle, we estimate the second factor $\nrm{A}_{\calL^{1/4+1/6,\infty}_{s} \calL^{\infty}_{t} \calW^{k,12}_{x}}$ by $\nrm{\nb_{x}  A}_{\calL^{1/4+1/6,\infty}_{s} \calL^{\infty}_{t} \calH^{k+1}_{x}}$. Finally, for the last factor, let us estimate $\nrm{A_{0}(t)}_{\calL^{0+5/12, \infty}_{s} \calW^{k,12}_{x}} \leq C \nrm{\nb_{x} A_{0}(t)}_{\calL^{0+5/12, \infty}_{s} \calH^{k+1}_{x}}$. At this point, we can simply apply Lemma \ref{lem:fundEst4A} to $A$ and Lemma \ref{lem:fundEst4A0} to $A_{0}$, and conclude that the trilinear term is acceptable as well. This proves \eqref{eq:fundEst4DF:2}.
\end{proof}

Finally, by essentially the same proof, we can prove an analogue of Lemma \ref{lem:fundEst4DF} for $\dlt \covD_{0} F_{ij} := \covD_{0} F_{ij} - \covD'_{0} F'_{ij}$ and $\dlt \covD_{i} F_{0j} := \covD_{i} F_{0j} - \covD'_{i} F'_{0j}$, whose statement we give below.
\begin{lemma} [Bounds for $\dlt \covD_{0} F_{ij}$ and $\dlt \covD_{i} F_{0j}$] \label{lem:fundEst4DF:Diff}
Suppose that the caloric-temporal gauge condition $A_{s} = 0$, $\Alow_{0}=0$ holds (for both $A$ and $A'$).
\begin{enumerate}
\item For any $2 \leq p \leq \infty$ and $k \geq 0$, we have
\begin{equation} \label{eq:fundEst4DF:Diff:1}
\begin{aligned}
	& \nrm{\dlt \covD_{0} F_{ij}(t)}_{\calL^{5/4,p}_{s} \dot{\calH}^{k}_{x}} + \nrm{\dlt \covD_{i} F_{0j}(t)}_{\calL^{5/4,p}_{s} \dot{\calH}^{k}_{x}} \\
	& \qquad \leq  C ( \nrm{\nb_{t,x} (\dlt F_{si})(t)}_{\calL^{5/4,2}_{s} \calH^{k+1}_{x}} + \nrm{\rd_{t,x} (\dlt \Alow_{i})(t)}_{H^{k+1}_{x}} 
	+ \nrm{\nb_{x} (\dlt F_{s0})(t)}_{\calL^{1,2}_{s} \calH^{k+1}_{x}}),
\end{aligned}
\end{equation}
where $C = C_{p,k}(\nrm{\nb_{t,x} F_{si}(t)}_{\calL^{5/4,2}_{s} \calH^{k+1}_{x}}, \nrm{\rd_{t,x} \Alow_{i}(t)}_{H^{k+1}_{x}} ,\nrm{\nb_{x} F_{s0}(t)}_{\calL^{1,2}_{s} \calH^{k+1}_{x}})$ is positive and non-decreasing in its arguments.

\item For any $2 \leq  p \leq \infty$ and $k \geq 0$, we have
\begin{equation} \label{eq:fundEst4DF:Diff:2}
\begin{aligned}
	& \nrm{\dlt \covD_{0} F_{ij}}_{\calL^{5/4,p}_{s} \calL^{4}_{t} \dot{\calW}^{k,4}_{x}} + \nrm{\dlt \covD_{i} F_{0j}}_{\calL^{5/4,p}_{s} \calL^{4}_{t} \dot{\calW}^{k,4}_{x}} \\
	& \qquad \leq  C_{p, k} (  \nrm{\dlt F_{si}}_{\calL^{5/4,2}_{s} \dot{\calS}^{k+5/2}} + \nrm{\dlt \Alow_{i}}_{\dot{S}^{k+5/2}} 
	+ T^{1/4} \sup_{t \in I} \nrm{\dlt F_{s0}(t)}_{\calL^{1,2}_{s} \dot{\calH}^{k+11/4}_{x}}) \\ 		
	& \phantom{\qquad \leq } + C (\nrm{\dlt F_{si}}_{\calL^{5/4,2}_{s} \widehat{\calS}^{k+2}} + \nrm{\dlt \Alow_{i}}_{\widehat{S}^{k+2}} 
	+ \sup_{t \in I} \nrm{\nb_{x} (\dlt F_{s0})(t)}_{\calL^{1,2}_{s} \calH^{k+1}_{x}}),
\end{aligned}
\end{equation}
where $C =C_{p, k}(\nrm{F_{si}}_{\calL^{5/4,2}_{s} \widehat{\calS}^{k+2}}, \nrm{\Alow_{i}}_{\widehat{S}^{k+2}}, \sup_{t \in I} \nrm{\nb_{x} F_{s0}(t)}_{\calL^{1,2}_{s} \calH^{k+1}_{x}})$ on the last line is positive and non-decreasing in its arguments.
\end{enumerate}
\end{lemma}

\subsection{Parabolic estimates for $F_{si}$} \label{subsec:pEst4HPYM:pEst4Fsi}
Recall that $F_{si}$ satisfies a covariant parabolic equation. Under the caloric gauge condition $A_{s} = 0$, expanding covariant derivatives and $F_{i\ell}$, we obtain a semi-linear heat equation for $F_{si}$, which looks schematically as follows.
\begin{equation*} 
	{}^{(F_{si})}\calN := (\rd_{s} - \lap) F_{si}= s^{-1/2} \calO( A, \nb_{x} F_{s}) + s^{-1/2} \calO(\nb_{x} A, F_{s}) + \calO(A, A, F_{s}).
\end{equation*}

Note that $\calF$ already controls some derivatives of $F_{si}$. Starting from this, the goal is to prove estimates for higher derivative of $F_{si}$. 

\begin{proposition} \label{prop:pEst4Fsi}
Suppose $0 < T \leq 1$, and that the caloric-temporal gauge condition holds.
\begin{enumerate}
\item For any $k \geq 0$, we have
\begin{equation} \label{eq:pEst4Fsi:1}
	\nrm{\nb_{t,x} F_{si}(t)}_{\calL^{5/4,\infty}_{s} \dot{\calH}^{k}_{x}(0,1]} + \nrm{\nb_{t,x} F_{si}(t)}_{\calL^{5/4,2}_{s} \dot{\calH}^{k+1}_{x}(0,1]} 
	\leq C_{k, \calF, \nrm{\rd_{t,x} \Alow(t)}_{H^{k}_{x}}} \cdot \calF.
\end{equation}

\item For $1 \leq k \leq 25$, we have
\begin{equation} \label{eq:pEst4Fsi:2}
	\nrm{F_{si}}_{\calL^{5/4,\infty}_{s} \dot{\calS}^{k}(0,1]} + \nrm{F_{si}}_{\calL^{5/4,2}_{s} \dot{\calS}^{k}(0,1]} 
	\leq C_{\calF, \calAlow} \cdot \calF.
\end{equation}
\end{enumerate}
\end{proposition}

Part (1) of the proposition states, heuristically, that in order to control $k+2$ derivatives of $F_{si}$ in the $\calL^{2}_{s}$ sense, we need $\calF$ and a control of $k+1$ derivatives of $\Alow_{i}$. This numerology is important for closing the bootstrap for the quantity $\calAlow$. On the other hand, in Part (2),  we obtain a uniform estimate in terms only of $\calF$ and $\calAlow$, thanks to the restriction of the range of $k$. We refer the reader to Remark \ref{rem:pEst4sptNrms} for more discussion.

\begin{proof} 
{\it Step 1: Proof of (1).}
Fix $t \in (-T, T)$. Let us start with the obvious inequalities 
\begin{equation} \label{eq:pEst4Fsi:pf:0}
\left\{
\begin{aligned}
	& \nrm{\rd_{t,x}(\phi_{1} \rd_{x} \phi_{2})}_{L^{2}_{x}} 
	\leq C \nrm{\rd_{t,x} \phi_{1}}_{\dot{H}^{1/2}_{x}} \nrm{\phi_{2}}_{\dot{H}^{1}_{x}} + C \nrm{\phi_{1}}_{L^{\infty}_{x}} \nrm{\rd_{t,x} \phi_{2}}_{\dot{H}^{1}_{x}}, \\
	& \nrm{\rd_{t,x}(\phi_{1} \phi_{2} \phi_{3})}_{L^{2}_{x}}  \leq C \sum_{\sgm}\nrm{\phi_{\sgm(1)}}_{\dot{H}^{4/3}_{x}} \nrm{\phi_{\sgm(2)}}_{\dot{H}^{4/3}_{x}} \nrm{\rd_{t,x} \phi_{\sgm(3)}}_{\dot{H}^{1/3}_{x}},
\end{aligned}
\right.
\end{equation}
where the sum $\sum_{\sgm}$ is over all permutations $\sgm$ of $\set{1,2,3}$. These can be proved by using Leibniz's rule, H\"older and Sobolev.

Using Leibniz's rule, the Correspondence Principle, Lemma \ref{lem:absP:Holder4Ls}, Gagliardo-Nirenberg (Lemma \ref{lem:absP:algEst}) and interpolation, the previous inequalities lead to the following inequalities for $k \geq 1$.
\begin{equation*}
%\left\{
\begin{aligned}
& \nrm{s^{-1/2} \nb_{t,x} \calO(\psi_{1}, \nb_{x} \psi_{2})}_{\calL^{5/4+1, 2}_{s} \dot{\calH}^{k-1}_{x}} 
+ \nrm{s^{-1/2} \nb_{t,x} \calO(\nb_{x} \psi_{1}, \psi_{2})}_{\calL^{5/4+1, 2}_{s} \dot{\calH}^{k-1}_{x}} \\
& \qquad \leq  C \nrm{\nb_{t,x} \psi_{1}}_{\calL^{1/4+1/4,\infty}_{s} \calH^{k}_{x}} \nrm{\nb_{t,x} \psi_{2}}_{\calL^{5/4,2}_{s} \calH^{k}_{x}}, \\
& \nrm{\nb_{t,x} \calO(\psi_{1}, \psi_{2}, \psi_{3})}_{\calL^{5/4+1, 2}_{s} \dot{\calH}^{k-1}_{x}} \\
& \qquad \leq  C \nrm{\nb_{t,x} \psi_{1}}_{\calL^{1/4+1/4, \infty}_{s} \calH^{k}_{x}} \nrm{\nb_{t,x} \psi_{2}}_{\calL^{5/4,2}_{s} \calH^{k}_{x}} \nrm{\nb_{t,x} \psi_{3}}_{\calL^{1/4+1/4, \infty}_{s} \calH^{k}}.
\end{aligned}
%\right.
\end{equation*}

Note the extra weight of $s^{1/4}$ for $\psi_{1}, \psi_{3}$. Put $\psi_{1} = A, \psi_{2} = F_{s}, \psi_{3} = A$, and apply Lemma \ref{lem:fundEst4A} (with $\ell > 0, p=\infty, q=2$ and $X=L^{2}_{x}$) for $\nrm{A}$. Then for $k \geq 1$, we have
\begin{equation} \label{eq:pEst4Fsi:pf:1}
	\begin{aligned}
\sup_{i} \nrm{\nb_{t,x} ({}^{(F_{si})} \calN)}_{\calL^{5/4+1, 2}_{s} \dot{\calH}^{k-1}_{x}} 
	\leq & C (\nrm{\nb_{t,x} F_{s}}_{\calL^{5/4, 2}_{s} \calH^{k}_{x}} + \nrm{\rd_{t,x} \Alow}_{H^{k}_{x}}) \nrm{\nb_{t,x} F_{s}}_{\calL^{5/4, 2}_{s} \calH^{k}_{x}} \\
	& + C (\nrm{\nb_{t,x} F_{s}}_{\calL^{5/4, 2}_{s} \calH^{k}_{x}} + \nrm{\rd_{t,x} \Alow}_{H^{k}_{x}})^{2} \nrm{\nb_{t,x} F_{s}}_{\calL^{5/4, 2}_{s} \calH^{k}_{x}}.
\end{aligned}
\end{equation}

Combining this with the obvious bound $\nrm{\nb_{t,x} F_{s}}_{\calL^{5/4,\infty}_{s} \calH^{1}_{x}} + \nrm{\nb_{t,x} F_{s}}_{\calL^{5/4,2}_{s} \calH^{2}_{x}} \leq \calF$, we obtain \eqref{eq:pEst4Fsi:1} from the second part of Theorem \ref{thm:absP:absPth}.

\vspace{0.1in}
{\it Step 2: Proof of (2).} We proceed in a similar fashion. The multilinear estimates are more complicated. On the other hand, as we are aiming to control derivatives of $F_{si}$ only up to order 25 whereas $\calAlow$ controls derivatives of $\Alow_{i}$ up to order 30, we can be relaxed on the number of derivatives falling on $\Alow_{i}$.

For $\eps > 0$, we claim that the following estimate for ${}^{(F_{si})} \calN$ holds for $1 \leq k \leq 24$.
\begin{equation} \label{eq:pEst4Fsi:pf:4}
	\sup_{i} \nrm{{}^{(F_{si})} \calN}_{\calL^{5/4+1,2}_{s} \dot{\calS}^{k}}
	\leq \eps \nrm{F_{s}}_{\calL^{5/4,2}_{s} \dot{\calS}^{k+2}} + \calB_{\eps, k, \calAlow}(\nrm{F_{s}}_{\calL^{5/4,2}_{s} \widehat{\calS}^{k+1}}) \nrm{F_{s}}_{\calL^{5/4,2}_{s} \widehat{\calS}^{k+1}}
\end{equation}
where $\calB_{\eps, k, \calA} (r)> 0$ is non-decreasing in $r > 0$. Then for $\eps > 0$ sufficiently small, the second part of Theorem \ref{thm:absP:absPth} can be applied. Combined with the obvious bound $\nrm{F_{s}}_{\calP^{5/4} \dot{\calS}^{2}} \leq \calF$, we obtain a bound for $\nrm{F_{s}}_{\calP^{5/4} \calS^{21}}$ which can be computed by \eqref{eq:absP:smth:3}. This leads to \eqref{eq:pEst4Fsi:2}, as desired.

Let us now prove the claim. We will begin by establishing the following multilinear estimates for $\dot{S}^{k}$:
\begin{equation} \label{eq:pEst4Fsi:pf:2}
\left\{
\begin{aligned}
	& \nrm{\phi_{1} \rd_{x} \phi_{2}}_{\dot{S}^{1}} 
	\leq T^{1/2} \nrm{\phi_{1}}_{\dot{S}^{3/2} } \nrm{\phi_{2}}_{\dot{S}^{5/2}} + \nrm{\phi_{1}}_{\dot{S}^{3/2} \cap L^{\infty}_{t,x}} \nrm{\phi_{2}}_{\dot{S}^{2}}, \\
	& \nrm{\phi_{1} \phi_{2} \phi_{3}}_{\dot{S}^{1}}
	\leq T^{1/2} \sum_{\sgm} \nrm{\phi_{\sgm(1)}}_{L^{\infty}_{t,x}} \nrm{\phi_{\sgm(2)}}_{\dot{S}^{3/2}} \nrm{\phi_{\sgm(3)}}_{\dot{S}^{3/2}} \\
	& \phantom{\nrm{\phi_{1} \phi_{2} \phi_{3}}_{\dot{S}^{1}}\leq}
	+ \sum_{\sgm} \nrm{\phi_{\sgm(1)}}_{\dot{S}^{1}} \nrm{\phi_{\sgm(2)}}_{\dot{S}^{1}} \nrm{\phi_{\sgm(3)}}_{\dot{S}^{2}}.
\end{aligned}
\right.
\end{equation}
where the sum $\sum_{\sgm}$ is over all permutations $\sgm$ of $\set{1,2,3}$.

For the first inequality of \eqref{eq:pEst4Fsi:pf:2}, it suffices to prove that $\nrm{\phi_{1} \rd_{x} \phi_{2}}_{L^{\infty}_{t} \dot{H}^{1}_{x}}$ and $T^{1/2} \nrm{\Box (\phi_{1} \rd_{x} \phi_{2})}_{L^{2}_{t,x}}$ can be controlled by the right-hand side. Using H\"older and Sobolev, we can easily bound the former by $\leq C \nrm{\phi_{1}}_{L^{\infty}_{t} \dot{H}^{3/2}_{x}} \nrm{\phi_{2}}_{L^{\infty}_{t} \dot{H}^{2}_{x}}$, which is acceptable. For the latter, using Leibniz's rule for $\Box$, let us further decompose
\begin{equation*}
	T^{1/2}\nrm{\Box (\phi_{1} \rd_{x} \phi_{2})}_{L^{2}_{t,x}} 
	\leq 2 T^{1/2}\nrm{\rd_{\mu} \phi_{1} \rd_{x} \rd^{\mu} \phi_{2}}_{L^{2}_{t,x}} 
	+ T^{1/2}\nrm{\Box \phi_{1} \rd_{x} \phi_{2}}_{L^{2}_{t,x}} + T^{1/2}\nrm{\phi_{1} \rd_{x} \Box\phi_{2}}_{L^{2}_{t,x}}.
\end{equation*}

Using H\"older and the $L^{4}_{t,x}$-Strichartz, we bound the first term by $\leq C T^{1/2}\nrm{\phi_{1}}_{\dot{S}^{3/2}} \nrm{\phi_{2}}_{\dot{S}^{5/2}}$, which is good. For the second term, let us use H\"older to put $\Box \phi_{1}$ in $L^{2}_{t} L^{3}_{x}$ and the other in $L^{\infty}_{t} L^{6}_{x}$. Then by Sobolev and the definition of $\dot{S}^{k}$, this is bounded by $\nrm{\phi_{1}}_{\dot{S}^{3/2}} \nrm{\phi_{2}}_{\dot{S}^{2}}$. Finally, for the third term, we use H\"older to estimate $\phi_{1}$ in $L^{\infty}_{t,x}$ and $\rd_{x} \Box \phi_{2}$ in $L^{2}_{t,x}$, which leads to a bound $\leq \nrm{\phi_{1}}_{L^{\infty}_{t,x}} \nrm{\phi_{2}}_{\dot{S}^{2}}$. This prove the first inequality of \eqref{eq:pEst4Fsi:pf:2}. 

The second inequality of \eqref{eq:pEst4Fsi:pf:2} follows by a similar consideration, first dividing $\nrm{\cdot}_{\dot{S}^{1}}$ into $\nrm{\cdot}_{L^{\infty}_{t} \dot{H}^{1}_{x}}$ and $\nrm{\Box (\cdot)}_{L^{2}_{t,x}}$, and then using Leibniz's rule for $\Box$ to further split the latter. We leave the details to the reader.

Let us prove \eqref{eq:pEst4Fsi:pf:4} by splitting ${}^{(F_{si})} \calN$ into its quadratic part $s^{-1/2} \calO(A, \nb_{x} F_{s}) + s^{-1/2} \calO(\nb_{x} A, F_{s})$ and its cubic part $\calO(A, A, F_{s})$. For the quadratic terms, we use the first inequality of \eqref{eq:pEst4Fsi:pf:2}, Leibniz's rule, the Correspondence Principle, Lemma \ref{lem:absP:Holder4Ls} and Lemma \ref{lem:absP:algEst}. Then for $k \geq 1$ we obtain
\begin{equation} \label{eq:pEst4Fsi:pf:3}
\begin{aligned}
& \nrm{s^{-1/2} \calO(\psi_{1}, \nb_{x} \psi_{2})}_{\calL^{5/4+1, 2}_{s} \dot{\calS}^{k}} 
\leq C T^{1/2} \sum_{p=0}^{k-1} \nrm{\psi_{1}}_{\calL^{\ell_{1}, p_{1}}_{s} \dot{\calS}^{3/2+p}} \nrm{\psi_{2}}_{\calL^{\ell_{2}, p_{2}}_{s} \dot{\calS}^{3/2+k-p}} \\
& \phantom{\nrm{s^{-1/2} \calO(\psi_{1}, \nb_{x} \psi_{2})}_{\calL^{5/4+1, 2}_{s} \dot{\calS}^{k}} \leq}
+ C \nrm{\psi_{1}}_{\calL^{\ell_{1}+1/8, p_{1}}_{s} \widehat{\calS}^{k+1}}  \nrm{\psi_{2}}_{\calL^{\ell_{2}+1/8, p_{2}}_{s} \widehat{\calS}^{k+1}}
\end{aligned}
\end{equation}
where $\ell_{1} + \ell_{2} = 3/2$ and $\frac{1}{p_{1}} + \frac{1}{p_{2}} = \frac{1}{2}$. Note that we have obtained an extra weight of $s^{1/8}$ for each factor in the last term.

Let $1 \leq k \leq 24$, and apply \eqref{eq:pEst4Fsi:pf:3} with $(\psi_{1}, \ell_{1}, p_{1}) = (A, 1/4, \infty)$, $(\psi_{2}, \ell_{2}, p_{2}) = (F_{s}, 2, 5/4)$ for $s^{-1/2} \calO(A, \nb_{x} F_{s})$ and vice versa for $s^{-1/2} \calO(\nb_{x} A, F_{s})$. We then apply Lemma \ref{lem:fundEst4A} with $X = \dot{S}^{1}$, $p=\infty$ and $q=2$ to control $\nrm{A}$ in terms of $\nrm{F_{s}}$ and $\nrm{\Alow}$ (here we use the extra weight of $s^{1/8}$). Next, we estimate $\nrm{\Alow}$ that arises by $\calAlow$, which is possible since we only consider $1 \leq k \leq 24$. As a result, we obtain the following inequality:
\begin{equation*}
\begin{aligned}
	\nrm{s^{-1} & \calO(A, \nb_{x} F_{s}) + s^{-1} \calO(\nb_{x} A, F_{s})}_{\calL^{5/4+1,2}_{s} \dot{\calS}^{k}} \\
	\leq & C T^{1/2} \sum_{p=0}^{k} (\nrm{F_{s}}_{\calL^{5/4,2}_{s} \dot{\calS}^{3/2+p}} + \calAlow) \nrm{F_{s}}_{\calL^{5/4,2}_{s} \dot{\calS}^{3/2+k-p}} 
	+ C (\nrm{F_{s}}_{\calL^{5/4,2}_{s} \widehat{\calS}^{k+1}} + \calAlow) \nrm{F_{s}}_{\calL^{5/4,2}_{s} \widehat{\calS}^{k+1}}.
\end{aligned}
\end{equation*}

The last term is acceptable. All summands of the first term on the right-hand side are also acceptable, except for the cases $p=0, k$. Let us first treat the case $p=0$. For $\eps > 0$, we apply Cauchy-Schwarz to estimate
\begin{align*}
	T^{1/2}  & (\nrm{F_{s}}_{\calL^{5/4,2}_{s} \dot{\calS}^{3/2}} + \calAlow) \nrm{F_{s}}_{\calL^{5/4,2}_{s} \dot{\calS}^{k+3/2}} \\
	& \leq (\eps/2) \nrm{F_{s}}_{\calL^{5/4,2}_{s} \dot{\calS}^{k+2}} + C_{\eps} T (\nrm{F_{s}}_{\calL^{5/4,2}_{s} \widehat{\calS}^{k+1}} + \calAlow)^{2} \nrm{F_{s}}_{\calL^{5/4,2}_{s} \widehat{\calS}^{k+1}}, 
\end{align*}
%\begin{align*}
%	T^{1/2}  & (\nrm{F_{s}}_{\calL^{5/4,2}_{s} \dot{\calS}^{k+3/2}} + \calAlow) \nrm{F_{s}}_{\calL^{5/4,2}_{s} \dot{\calS}^{3/2}} \\
%	& \leq (\eps/2) \nrm{F_{s}}_{\calL^{5/4,2}_{s} \dot{\calS}^{k+2}} + (C T^{1/2} \calAlow +  C_{\eps} T \nrm{F_{s}}_{\calL^{5/4,2}_{s} \widehat{\calS}^{k+1}}^{2}) \nrm{F_{s}}_{\calL^{5/4,2}_{s} \widehat{\calS}^{k+1}},
%\end{align*}

The case $p=k$ is similar. This proves \eqref{eq:pEst4Fsi:pf:4} for the quadratic terms $s^{-1/2} \calO(A, \nb_{x} F_{s}) + s^{-1/2} \calO(\nb_{x} A, F_{s})$.

Next, let us estimate the contribution of the cubic terms $\calO(A, A, F_{s})$. Starting from the second inequality of \eqref{eq:pEst4Fsi:pf:2} and applying Leibniz's rule, the Correspondence Principle, Lemma \ref{lem:absP:Holder4Ls} and Lemma \ref{lem:absP:algEst}, we obtain the following inequality:
\begin{equation*} 
\begin{aligned}
	\nrm{\calO(\psi_{1}, \psi_{2}, \psi_{3})}_{\calL^{5/4+1,2}_{s} \dot{\calS}^{k}}
	\leq & C T^{1/2} \prod_{j=1,2,3} \nrm{\psi_{j}}_{\calL^{\ell_{j}+1/12, p_{j}}_{s} \widehat{\calS}^{k+1}} 
	+ C\prod_{j=1,2,3} \nrm{\psi_{j}}_{\calL^{\ell_{j} + 1/6, p_{j}}_{s} \widehat{\calS}^{k+1}},
	\end{aligned}
\end{equation*}
for $\ell_{1} + \ell_{2} + \ell_{3} = 7/4$ and $\frac{1}{p_{1}} + \frac{1}{p_{2}} + \frac{1}{p_{3}} =\frac{1}{2}$. Note the extra weight of $s^{1/12}$ and $s^{1/6}$ for each factor in the first and second terms on the right-hand side, respectively. 

For $1 \leq k \leq 24$, let us put $(\psi_{1}, \ell_{1}, p_{1}) = (A, 1/4, \infty)$, $(\psi_{2}, \ell_{2}, p_{2}) = (A, 1/4, \infty)$ and $(\psi_{3}, \ell_{3}, p_{3}) = (F_{s}, 5/4, 2)$ in the last inequality, and furthermore apply Lemma \ref{lem:fundEst4A} with $X = \dot{S}^{1}$, $p=\infty$ and $q=2$ (which again uses the extra weights of powers of $s$) to control $\nrm{A}$ by $\nrm{F_{s}}$ and $\nrm{\Alow}$. Then estimating $\nrm{\Alow}$ by $\calAlow$ (which again is possible since $1 \leq k \leq 24$), we finally arrive at
\begin{equation*}
	\nrm{\calO(A, A, F_{s})}_{\calL^{5/4+1, 2}_{s} \dot{\calS}^{k}} 
	\leq C (1+T^{1/2}) (\nrm{F_{s}}_{\calL^{5/4,2}_{s} \widehat{\calS}^{k+1}} + \calAlow)^{2} \nrm{F_{s}}_{\calL^{5/4,2}_{s} \widehat{\calS}^{k+1}}
\end{equation*}
which is acceptable. This proves \eqref{eq:pEst4Fsi:pf:4}. \qedhere
\end{proof}

\begin{remark} \label{rem:pEst4sptNrms}
The fixed time parabolic estimate \eqref{eq:pEst4Fsi:1} will let us estimate $\calAlow$ in \S \ref{subsec:AlowWave} in terms of $\calF, \calAlow$, \emph{despite} the fact that $\calF$ controls a smaller number of derivatives (of $F_{si}$) than does $\calAlow$ (of $\Alow_{i}$). This is due to the smoothing property of the parabolic equation satisfied by $F_{si}$. It will come in handy in \S \ref{subsec:FsWave}, as it allows us to control only a small number of derivatives of $F_{si}$ to control $\calF$. 

Accordingly, the space-time estimate \eqref{eq:pEst4Fsi:2} (to be used in \S \ref{subsec:FsWave}) needs to be proved only for a finite range of $k$, which is taken to be smaller than the number of derivatives of $\Alow_{i}$ controlled by $\calAlow$. This allows us to estimate whatever $\nrm{\Alow_{i}}$ that arises by $\calAlow$; practically, we do not have to worry about the number of derivatives falling on $\Alow_{i}$. Moreover, we are also allowed to control (the appropriate space-time norm of) less and less derivatives for $F_{s0}$ and $w_{i}$ (indeed, see \eqref{eq:pEst4Fs0:high:2} and \eqref{eq:pEst4wi:3}, respectively), as long as we control enough derivatives to carry out the analysis in \S \ref{subsec:FsWave} in the end. Again, this lets us forget about the number of derivatives falling on $\Alow_{i}$ and $F_{si}$ (resp. $\Alow_{i}$, $F_{si}$ and $F_{s0}$) while estimating the space-time norms of $F_{s0}$ and $w_{i}$.
\end{remark}

By essentially the same proof, the following difference analogue of Proposition \ref{prop:pEst4Fsi} follows.
\begin{proposition} 
Suppose $0 < T \leq 1$, and that the caloric-temporal gauge condition holds.
\begin{enumerate}
\item Let $t \in (-T, T)$. Then for any $k \geq 0$, we have
\begin{equation} \label{eq:pEst4Fsi:Diff:1}
\begin{aligned}
	& \nrm{\nb_{t,x} (\dlt F_{si})(t)}_{\calL^{5/4,\infty}_{s} \dot{\calH}^{k}_{x}(0,1]} + \nrm{\nb_{t,x} (\dlt F_{si})(t)}_{\calL^{5/4,2}_{s} \dot{\calH}^{k+1}_{x}(0,1]}  \\
	& \qquad \leq C_{k, \calF, \nrm{\rd_{t,x} \Alow(t)}_{H^{k}_{x}}} \cdot (\dlt \calF + \nrm{\rd_{t,x} (\dlt \Alow)(t)}_{H^{k}_{x}}),
\end{aligned}
\end{equation}

\item For $1 \leq k \leq 25$, we have
\begin{equation} \label{eq:pEst4Fsi:Diff:2}
	\nrm{\dlt F_{si}}_{\calL^{5/4,\infty}_{s} \dot{\calS}^{k}(0,1]} +\nrm{\dlt F_{si}}_{\calL^{5/4,2}_{s} \dot{\calS}^{k}(0,1]} 
	\leq C_{\calF, \calAlow} \cdot (\dlt \calF + \dlt \calAlow).
\end{equation}
\end{enumerate}
\end{proposition}

\subsection{Estimates for $F_{si}$ via integration} \label{subsec:pEst4HPYM:lowEst4Fsi}
We also need some estimates for $F_{si}$ without any derivatives, which we state below. The idea of the proof is to simply integrate the parabolic equation $\rd_{s} F_{si} = \lap F_{si} + {}^{(F_{si})} \calN$ backwards from $s=1$.

\begin{proposition} \label{prop:lowEst4Fsi}
Suppose $0 < T \leq 1$, and that the caloric-temporal gauge condition holds.
\begin{enumerate}
\item Let $t \in (-T, T)$. Then we have
\begin{equation} \label{eq:lowEst4Fsi:1}
	\nrm{F_{si}(t)}_{\calL^{5/4,\infty}_{s} \calL^{2}_{x}(0,1]} + \nrm{F_{si}(t)}_{\calL^{5/4,2}_{s} \calL^{2}_{x}(0,1]}
	\leq C_{\calF, \calAlow} \cdot (\calF + \calAlow).
\end{equation}

\item We have
\begin{equation} \label{eq:lowEst4Fsi:2}
	\nrm{F_{si}}_{\calL^{5/4,\infty}_{s} \calL^{4}_{t, x}(0,1]} + \nrm{F_{si}}_{\calL^{5/4,2}_{s} \calL^{4}_{t, x}(0,1]}
	\leq C_{\calF, \calAlow} \cdot (\calF + \calAlow).
\end{equation}
\end{enumerate}
\end{proposition}

\begin{proof} 
In the proof, all norms will be taken on the interval $s \in (0, 1]$. Let us start with the equation
\begin{equation*}
	\rd_{s} F_{si} = \lap F_{si} + {}^{(F_{si})} \calN.
\end{equation*} 

Using the fundamental theorem of calculus, we obtain for $0 < s \leq 1$ the identity
\begin{equation} \label{eq:lowEst4Fsi:pf:1}
	F_{si}(s) = \Flow_{si} - \int_{s}^{1} s' \lap F_{si}(s') \, \frac{\ud s'}{s'} - \int_{s}^{1} s' ({}^{(F_{si})} \calN(s')) \, \frac{\ud s'}{s'}.
\end{equation}

To prove \eqref{eq:lowEst4Fsi:1} and \eqref{eq:lowEst4Fsi:2}, let us either fix $t \in (-T, T)$ and take the $\calL^{5/4,p}_{s} \calL^{2}_{x}$ norm of both sides or just take the $\calL^{5/4,p}_{s} \calL^{4}_{t,x}$ norm, respectively. We will estimate the contribution of each term on the right-hand side of \eqref{eq:lowEst4Fsi:pf:1} separately.

For the first term on the right-hand side of \eqref{eq:lowEst4Fsi:pf:1}, note the obvious estimates $\nrm{\Flow_{si}(t)}_{\calL^{5/4,p}_{s} \calL^{2}_{x}} \leq C_{p} \nrm{\Flow_{si}(t)}_{L^{2}_{x}}$ and $\nrm{\Flow_{si}}_{\calL^{5/4,p}_{s} \calL^{4}_{t,x}} \leq C_{p} \nrm{\Flow_{si}}_{L^{4}_{t,x}}$. Writing out $\Flow_{si} = \calO(\rd_{x}^{(2)} \Alow) + \calO(\Alow, \rd_{x} \Alow) + \calO(\Alow, \Alow, \Alow)$, we see that
\begin{equation*}
\sup_{t \in (-T, T)} \nrm{\Flow_{si}(t)}_{L^{2}_{x}} + \nrm{\Flow_{si}}_{L^{4}_{t,x}} \leq C \calAlow + C\calAlow^{2} + C\calAlow^{3},
\end{equation*} 
which is acceptable.

For the second term on the right-hand side of \eqref{eq:lowEst4Fsi:pf:1}, let us apply Lemma \ref{lem:SchurTest} with $p=2, \infty$ and $q =2$ to estimate
\begin{align*}
	\nrm{\int_{s}^{1} s' \lap F_{si}(t, s') \, \frac{\ud s'}{s'}}_{\calL^{5/4,p}_{s} \calL^{2}_{x}}
	\leq & \nrm{\int_{s}^{1} (s/s')^{5/4} (s')^{5/4} \nrm{\nb^{(2)}_{x} F_{si}(t, s')}_{\calL^{2}_{x}(s')} \, \frac{\ud s'}{s'}}_{\calL^{p}_{s}} \\
	\leq & C_{p} \nrm{F_{si}(t)}_{\calL^{5/4, 2}_{s} \dot{\calH}^{2}_{x}} \leq C_{p} \calF.
\end{align*}

Similarly, for $p = 2, \infty$, we can prove $\nrm{\int_{s}^{1} s' \lap F_{si}(s') \frac{\ud s'}{s'}}_{\calL^{5/4,p}_{s} \calL^{4}_{t,x}} \leq C_{p} \nrm{F_{si}}_{\calL^{5/4,2}_{s} \calL^{4}_{t} \dot{\calW}^{2,4}_{x}} \leq C_{p} \calF$. Therefore, the contribution of the second term is okay.

Finally, for the third term on the right-hand side of \eqref{eq:lowEst4Fsi:pf:1}, let us first proceed as in the previous case to reduce
\begin{equation*}
\nrm{\int_{s}^{1} s' ({}^{(F_{si})}\calN(t, s')) \frac{\ud s'}{s'}}_{\calL^{5/4,p}_{s} \calL^{2}_{x}} \leq C_{p} \nrm{{}^{(F_{si})} \calN(t)}_{\calL^{5/4+1,2}_{s} \calL^{2}_{x}}.
\end{equation*}

Recall that ${}^{(F_{si})} \calN = s^{-1/2} \calO(A, \nb_{x} F_{s}) + s^{-1/2} \calO(\nb_{x} A, F_{s}) + \calO(A, A, F_{s})$. Starting from the obvious inequalities
\begin{equation*}
	\nrm{\phi_{1} \rd_{x} \phi_{2}}_{L^{2}_{x}} \leq C \nrm{\phi_{1}}_{\dot{H}^{1}_{x}} \nrm{\phi_{2}}_{\dot{H}^{3/2}_{x}}, \quad \nrm{\phi_{1} \phi_{2} \phi_{3}}_{L^{2}_{x}} \leq C \prod_{j=1,2, 3} \nrm{\phi_{j}}_{\dot{H}^{1}_{x}},
\end{equation*}
and applying the Correspondence Principle, Lemma \ref{lem:absP:Holder4Ls} and interpolation, we obtain
\begin{align*}
	& \nrm{s^{-1/2} \calO(A, \nb_{x} F_{s}) + s^{-1/2} \calO(\nb_{x} A, F_{s})}_{\calL^{5/4+1,2}_{s} \calL^{2}_{x}}
	\leq C \nrm{\nb_{x} A}_{\calL^{1/4+1/4,\infty}_{s} \calH^{1}_{x}} \nrm{\nb_{x} F_{s}}_{\calL^{5/4,2}_{s} \calH^{1}_{x}} \\
	& \nrm{\calO(A, A, F_{s})}_{\calL^{5/4+1,2}_{s} \calL^{2}_{x}}
	\leq C \nrm{A}_{\calL^{1/4+1/4, \infty}_{s} \dot{\calH}^{1}_{x}}^{2} \nrm{F_{s}}_{\calL^{5/4,2}_{s} \dot{\calH}^{1}_{x}}.
\end{align*}

Note the extra weight of $s^{1/4}$ on each factor of $A$. This allows us to apply Lemma \ref{lem:fundEst4A} (with $q=2$) to estimate $\nrm{A}$ in terms of $\nrm{F_{s}}$ and $\nrm{\Alow}$. From the definition of $\calF$ and $\calAlow$, it then follows that $\nrm{{}^{(F_{si})} \calN(t)}_{\calL^{5/4+1, 2}_{s} \calL^{2}_{x}} \leq C (\calF+\calAlow)^{2} + C(\calF+\calAlow)^{3}$ uniformly in $t \in (-T, T)$, which finishes the proof of \eqref{eq:lowEst4Fsi:1}.

Finally, as in the previous case, we have
\begin{equation*}
\nrm{\int_{s}^{1} s' ({}^{(F_{si})}\calN(s')) \frac{\ud s'}{s'}}_{\calL^{5/4,p}_{s} \calL^{4}_{t,x}} 
\leq C_{p} \nrm{{}^{(F_{si})}\calN(s')}_{\calL^{5/4+1,2}_{s} \calL^{4}_{t,x}}.
\end{equation*}

Using the inequalities
\begin{equation*}
	\nrm{\phi_{1} \rd_{x} \phi_{2}}_{L^{4}_{t,x}} \leq C \nrm{\phi_{1}}_{L^{\infty}_{t,x}} \nrm{\phi_{2}}_{\dot{S}^{3/2}}, \quad 
	\nrm{\phi_{1} \phi_{2} \phi_{3}}_{L^{4}_{t,x}} \leq C T^{1/4} \prod_{j=1,2, 3} \nrm{\phi_{j}}_{L^{\infty}_{t} L^{12}_{x}},
\end{equation*}
and proceeding as before using the Correspondence Principle, Lemmas \ref{lem:absP:Holder4Ls}, \ref{lem:absP:algEst} and \ref{lem:fundEst4A}, it follows that $\nrm{{}^{(F_{si})} \calN}_{\calL^{5/4+1, 2}_{s} \calL^{4}_{t,x}} \leq C (\calF+\calAlow)^{2} + C(\calF+\calAlow)^{3}$. This concludes the proof of \eqref{eq:lowEst4Fsi:2}. \qedhere
\end{proof}

Again with essentially the same proof, the following difference analogue of Proposition \ref{prop:lowEst4Fsi} follows.
\begin{proposition} \label{prop:lowEst4Fsi:Diff}
Suppose $0 < T \leq 1$, and that the caloric-temporal gauge condition holds.
\begin{enumerate}
\item Let $t \in (-T, T)$. Then we have
\begin{equation} \label{eq:lowEst4Fsi:Diff:1}
	\nrm{\dlt F_{si}(t)}_{\calL^{5/4,\infty}_{s} \calL^{2}_{x}(0,1]} + \nrm{\dlt F_{si}(t)}_{\calL^{5/4,2}_{s} \calL^{2}_{x}(0,1]}
	\leq C_{\calF, \calAlow} \cdot (\dlt \calF + \dlt \calAlow).
\end{equation}

\item We have
\begin{equation} \label{eq:lowEst4Fsi:Diff:2}
	\nrm{\dlt F_{si}}_{\calL^{5/4,\infty}_{s} \calL^{4}_{t, x}(0,1]} + \nrm{\dlt F_{si}}_{\calL^{5/4,2}_{s} \calL^{4}_{t, x}(0,1]}
	\leq C_{\calF, \calAlow} \cdot (\dlt \calF + \dlt \calAlow).
\end{equation}
\end{enumerate}
\end{proposition}

\subsection{Parabolic estimates for $F_{s0}$} \label{subsec:pEst4Fs0}
In this subsection, we will study the parabolic equation \eqref{eq:covParabolic4w0} satisfied by $F_{s0} = -w_{0}$. Let us define
\begin{equation*} 
	{}^{(F_{s0})} \calN := (\rd_{s} - \lap) F_{s0} = {}^{(F_{s0})} \calN_{\forcing} + {}^{(F_{s0})} \calN_{\linear} 
\end{equation*}
where
\begin{align*}
	{}^{(F_{s0})}\calN_{\linear} &= 2 s^{-1/2} \LieBr{A^\ell}{\nb_\ell F_{s0}} + s^{-1/2} \LieBr{\nb^\ell A_\ell}{F_{s0}} + \LieBr{A^{\ell}}{\LieBr{A_{\ell}}{F_{s0}}}, \\
	{}^{(F_{s0})}\calN_{\forcing} &=2 s^{-1/2} \LieBr{\tensor{F}{_{0}^{\ell}}}{F_{s \ell}}.
\end{align*}

Our first proposition for $F_{s0}$ is an {\it a priori} parabolic estimate for $\calE(t)$, which requires a smallness assumption of some sort\footnote{In our case, as we normalized the $s$-interval to be $[0,1]$, we will require directly that $\calF+\calAlow$ is sufficiently small. On the other hand, we remark that this proposition can be proved just as well by taking the length of the $s$-interval to be sufficiently small.}. 
%We remark that combining this proposition with its difference analogue Proposition \ref{prop:pEst4Fs0:low:Diff} immediately implies Proposition \ref{prop:est4Fs0:low}.

\begin{proposition} [Estimate for $\calE$] \label{prop:pEst4Fs0:low}
Suppose that the caloric-temporal gauge condition holds, and furthermore that $\calF + \calAlow < \dlt_{E}$ where $\dlt_{E} > 0$ is a sufficiently small constant. Then 
\begin{equation} \label{eq:pEst4Fs0:low:1}
	\sup_{t \in (-T, T)}\calE(t) \leq C_{\calF, \calAlow} \cdot (\calF + \calAlow)^{2},
\end{equation}
where $C_{\calF, \calAlow} = C(\calF, \calAlow)$ can be chosen to be continuous and non-decreasing with respect to both arguments.
\end{proposition}

\begin{proof} 
Let us fix $t \in (-T, T)$. Define $E := \abs{\rd_{x}}^{-1/2} F_{s0}$, where $\abs{\rd_{x}}^{a} := (- \lap)^{a/2}$ is the fractional integration operator. From the parabolic equation for $F_{s0}$, we can derive the following parabolic equation for $E$:
\begin{equation*}
	(\rd_{s} - \lap) E =  s^{1/4} \abs{\nb_{x}}^{-1/2} ({}^{(F_{s0})}\calN),
\end{equation*}
where $\abs{\nb_{x}}^{a} := s^{a/2} \abs{\rd_{x}}^{a}$ is the p-normalization of $\abs{\rd_{x}}^{a}$. The idea is to work with the new variable $E$, and then translate to the corresponding estimates for $F_{s0}$ to obtain \eqref{eq:pEst4Fs0:low:1}.

We begin by making two claims. First, for every small $\eps, \eps' > 0$, by taking $\dlt_{E} >0$ sufficiently small, the following estimate holds for $p=1,2$ and $0 < \subr \leq 1$:
\begin{equation} \label{eq:pEst4Fs0:low:pf:1}
	\nrm{{}^{(F_{s0})} \calN}_{\calL^{2,p}_{s} \dot{\calH}^{-1/2}_{x}(0, \subr]} 
	\leq \eps \nrm{E}_{\calP^{3/4} \dot{\calH}^{2}_{x}(0, 1]} + C_{\calF, \calAlow} \cdot \big( \nrm{s^{1/4-\eps'}  E}_{\calL^{3/4,2}_{s} \dot{\calH}^{1}_{x}(0, \subr]} + (\calF + \calAlow)^{2} \big).
\end{equation}

Second, for $k = 1,2$, the following estimate holds.
\begin{equation} \label{eq:pEst4Fs0:low:pf:5}
	\nrm{{}^{(F_{s0})} \calN}_{\calL^{2,2}_{s} \dot{\calH}^{k-1/2}_{x}(0, 1]} 
	\leq \eps \nrm{E}_{\calP^{3/4} \dot{\calH}^{k+2}_{x}(0, 1]} + C_{\calF, \calAlow} \cdot \nrm{E}_{\calP^{3/4} \calH^{k+1}_{x}(0, 1]} + C_{\calF, \calAlow} \cdot (\calF + \calAlow)^{2}.
\end{equation}

Assuming these claims, we can quickly finish the proof. Note that $E = 0$ at $s=0$, as $F_{s0} = 0$ there, and that the left-hand side of \eqref{eq:pEst4Fs0:low:pf:1} is equal to $\nrm{s^{1/4} \abs{\nb_{x}}^{-1/2} ({}^{(F_{s0})} \calN)}_{\calL^{3/4+1,p}_{s} \calL^{2}_{x}(0, \subr]}$. Applying the first part of Theorem \ref{thm:absP:absPth}, we derive $\nrm{E}_{\calP^{3/4} \calH^{2}_{x}(0,1]} \leq C_{\calF, \calAlow} \cdot (\calF + \calAlow)^{2}$. Using the preceding estimate and \eqref{eq:pEst4Fs0:low:pf:1}, an application of the second part of Theorem \ref{thm:absP:absPth} then shows that $\nrm{E}_{\calP^{3/4} \calH^{4}_{x}(0,1]} \leq C_{\calF, \calAlow} \cdot (\calF + \calAlow)^{2}$. Finally, as $E = s^{1/4} \abs{\nb_{x}}^{-1/2} F_{s0}$, it is easy to see that $\calE(t) \leq \nrm{E}_{\calP^{3/4} \calH^{4}_{x}(0,1]}$, from which \eqref{eq:pEst4Fs0:low:1} follows.

To establish \eqref{eq:pEst4Fs0:low:pf:1} and \eqref{eq:pEst4Fs0:low:pf:5}, we split ${}^{(F_{s0})} \calN$ into ${}^{(F_{s0})} \calN_{\forcing}$ and ${}^{(F_{s0})} \calN_{\linear}$.

\pfstep{- Case 1: The contribution of ${}^{(F_{s0})} \calN_{\forcing}$}
In this case, we will work on the whole interval $(0, 1]$. Let us start with the product inequality
\begin{equation*}
	\nrm{\phi_{1} \phi_{2}}_{\dot{H}^{-1/2}_{x}} \leq C \nrm{\phi_{1}}_{\dot{H}^{1/2}_{x}} \nrm{\phi_{2}}_{\dot{H}^{1/2}_{x}},
\end{equation*}
which follows from Lemma \ref{lem:homSob}. Using Leibniz's rule, the Correspondence Principle and Lemma \ref{lem:absP:Holder4Ls}, we obtain for $0 \leq k \leq 2$
\begin{equation*} 
	\nrm{\calO(\psi_{1}, \psi_{2})}_{\calL^{2,p}_{s} \dot{\calH}^{k-1/2}_{x}}
	\leq C \sum_{j=0}^{k} \nrm{\psi_{1}}_{\calL^{3/4,r}_{s} \dot{\calH}^{j+1/2}_{x}} \nrm{\psi_{2}}_{\calL^{5/4,2}_{s} \dot{\calH}^{k-j+1/2}_{x}}.
\end{equation*}
where $\frac{1}{r} = \frac{1}{p} - \frac{1}{2}$. Let us put $\psi_{1} = F_{0\ell}$, $\psi_{2} = F_{s \ell}$. 

In order to estimate $\nrm{F_{0 \ell}}_{\calL^{3/4,r}_{s} \dot{\calH}^{1/2}_{x}}$ or $\nrm{F_{0 \ell}}_{\calL^{3/4,r}_{s} \dot{\calH}^{j+1/2}_{x}}$ with $j >0$, we apply \eqref{eq:fundEst4F0i:1} or (an interpolation of) \eqref{eq:fundEst4F0i:2} of Lemma \ref{lem:fundEst4F0i}, respectively. We then estimate $\nrm{F_{s\ell}}, \nrm{\Alow_{\ell}}$ which arise by $\calF, \calAlow$, respectively. (We remark that this is possible as $0 \leq k \leq 2$.) 

Next, to estimate $\nrm{F_{s\ell}}_{\calL^{5/4,2}_{s} \dot{\calH}^{1/2}_{x}}$, we first note, by interpolation, that it suffices to control $\nrm{F_{s\ell}}_{\calL^{5/4,2}_{s} \calL^{2}_{x}}$ and $\nrm{F_{s\ell}}_{\calL^{5/4,2}_{s} \dot{\calH}^{1}_{x}}$, to which we then apply Propositions \ref{prop:lowEst4Fsi} and \ref{prop:pEst4Fsi}, respectively. On the other hand, for $\nrm{F_{s\ell}}_{\calL^{5/4,2}_{s} \dot{\calH}^{k-j+1/2}_{x}}$ with $j < k$, we simply apply (after an interpolation) Proposition \ref{prop:pEst4Fsi}. Observe that all of $\nrm{\Alow}$ which arise can be estimated by $\calAlow$. As a result, for $1 \leq p \leq 2$ and $0 \leq k \leq 2$, we obtain 
\begin{equation*}
\begin{aligned}
	\nrm{{}^{(F_{s0})} \calN_{\forcing}}_{\calL^{2,p}_{s} \dot{\calH}^{k-1/2}_{x}(0,1]} 
	\leq &C_{\calF, \calAlow} \cdot (\calF + \calAlow) \sum_{j=0}^{k} (\nrm{F_{s0}}_{\calL^{1,2}_{s} \dot{\calH}^{1+j}_{x}(0,1]} + \nrm{F_{s0}}_{\calL^{1,2}_{s} \dot{\calH}^{3/2+j}_{x}(0,1]}) \\
	& + C_{\calF, \calAlow} \cdot (\calF + \calAlow)^{2}.
\end{aligned}
\end{equation*}

As $E = \abs{\rd_{x}}^{-1/2} F_{s0}$, note that 
\begin{equation*}
	\nrm{F_{s0}}_{\calL^{1,2}_{s} \dot{\calH}^{1+j}_{x}(0,1]} + \nrm{F_{s0}}_{\calL^{1,2}_{s} \dot{\calH}^{3/2+j}_{x}(0,1]} 
	= \nrm{E}_{\calL^{3/4,2}_{s} \dot{\calH}^{3/2+j}_{x}(0,1]} +\nrm{E}_{\calL^{3/4,2}_{s} \dot{\calH}^{2+j}_{x}(0,1]}.
\end{equation*}

Note furthermore that the right-hand side is bounded by $\nrm{E}_{\calP^{3/4} \calH^{2+j}_{x}(0,1]}$. Given $\eps > 0$, by taking $\dlt_{E} > 0$ sufficiently small (so that $\calF + \calAlow$ is sufficiently small), we obtain for $k=0$, $p=1,2$ 
\begin{equation*}
	\nrm{{}^{(F_{s0})} \calN_{\forcing}}_{\calL^{2,p}_{s} \dot{\calH}^{-1/2}_{x}(0,1]} 
	\leq  \eps \nrm{E}_{\calP^{3/4} \calH^{2}_{x}(0,1]} + C_{\calF, \calAlow} \cdot (\calF + \calAlow)^{2},
\end{equation*}
and for $k=1,2$ (taking $p=2$)
\begin{equation*}
	\nrm{{}^{(F_{s0})} \calN_{\forcing}}_{\calL^{2,2}_{s} \dot{\calH}^{k-1/2}_{x}(0,1]} 
	\leq  \eps \nrm{E}_{\calP^{3/4} \dot{\calH}^{k+2}_{x}(0,1]} + \eps \nrm{E}_{\calP^{3/4} \calH^{k+1}_{x}(0,1]} + C_{\calF, \calAlow} \cdot (\calF + \calAlow)^{2},
\end{equation*}
both of which are acceptable.

\pfstep{- Case 2: The contribution of ${}^{(F_{s0})} \calN_{\linear}$}
Let $\subr \in (0, 1]$; we will work on $(0, \subr]$ in this case. We will see that for this term, no smallness assumption is needed. 

Let us start with the inequalities
\begin{equation*}
%\left\{
\begin{aligned}
& \nrm{\phi_{1} \phi_{2}}_{\dot{H}^{-1/2}_{x}} \leq C \nrm{\phi_{1}}_{\dot{H}^{3/2}_{x} \cap L^{\infty}_{x}} \nrm{\phi_{2}}_{\dot{H}^{-1/2}_{x}}, \\
& \nrm{\phi_{1} \phi_{2}}_{\dot{H}^{-1/2}_{x}} \leq C \nrm{\phi_{1}}_{\dot{H}^{1/2}_{x}} \nrm{\phi_{2}}_{\dot{H}^{1/2}_{x}}, \\
& \nrm{\phi_{1} \phi_{2} \phi_{3}}_{\dot{H}^{-1/2}_{x}} \leq C \nrm{\phi_{1}}_{\dot{H}^{1}_{x}} \nrm{\phi_{2}}_{\dot{H}^{1/2}_{x}} \nrm{\phi_{3}}_{\dot{H}^{1}_{x}}.
\end{aligned}
%\right.
\end{equation*}

To prove the first inequality, note that it is equivalent to the product estimate $(\dot{H}^{3/2}_{x} \cap L^{\infty}_{x}) \cdot \dot{H}^{1/2}_{x} \subset \dot{H}^{1/2}_{x}$ by duality, which in turn follows from interpolation between $(\dot{H}^{3/2}_{x} \cap L^{\infty}_{x}) \cdot L^{2}_{x} \subset L^{2}_{x}$ and $(\dot{H}^{3/2}_{x} \cap L^{\infty}_{x}) \cdot \dot{H}^{1}_{x} \subset \dot{H}^{1}_{x}$. On the other hand, the second inequality was already used in the previous step. Finally, the third inequality is an easy consequences of the Hardy-Littlewood-Sobolev fractional integration $L^{3/2}_{x} \subset \dot{H}^{-1/2}_{x}$, H\"older and Sobolev.

Let $\eps' > 0$. Using the preceding inequalities, along with the Correspondence Principle and Lemma \ref{lem:absP:Holder4Ls}, we obtain the following inequalities for $p=1,2$ on $(0, \subr]$:
\begin{equation} \label{eq:pEst4Fs0:low:pf:4}
\left\{
\begin{aligned}
	& \nrm{s^{-1/2} \calO(\psi_{1}, \nb_{x} \psi_{2})}_{\calL^{2, p}_{s} \dot{\calH}^{-1/2}_{x}} + \nrm{s^{-1/2} \calO(\nb_{x} \psi_{1}, \psi_{2})}_{\calL^{2, p}_{s} \dot{\calH}^{-1/2}_{x}} \\
	& \quad \leq C \subr^{\eps'} \nrm{\psi_{1}}_{\calL^{1/4,\infty}_{s} (\dot{\calH}^{3/2}_{x} \cap \calL^{\infty}_{x})} \nrm{s^{1/4-\eps'} \, \psi_{2}}_{\calL^{1,2}_{s} \dot{\calH}^{1/2}_{x}}, \\
	& \nrm{\calO(\psi_{1}, \psi_{2}, \psi_{3})}_{\calL^{2, p}_{s} \dot{\calH}^{-1/2}_{x}}
	\leq C \subr^{\eps'} \nrm{\psi_{1}}_{\calL^{1/4+1/8, \infty}_{s} \dot{\calH}^{1}_{x}} \nrm{s^{1/4-\eps'} \, \psi_{2}}_{\calL^{1, 2}_{s} \dot{\calH}^{1/2}_{x}} \nrm{\psi_{3}}_{\calL^{1/4+1/8, \infty}_{s} \dot{\calH}^{1}_{x}}.
\end{aligned}
\right.
\end{equation}

We remark that the factors of $\subr^{\eps'}$, which can be estimated by $\leq 1$, arise due to an application of H\"older for $\calL^{\ell,p}_{s}$ (Lemma \ref{lem:absP:Holder4Ls}) in the case $p=1$. Taking $\psi_{1} = A$, $\psi_{2} = F_{s0}$, $\psi_{3}= A$ and using Lemmas \ref{lem:absP:algEst}, \ref{lem:fundEst4A} and the fact that $\nrm{s^{1/4-\eps'} F_{s0}}_{\calL^{1,2}_{s} \dot{\calH}^{1/2}_{x}} = \nrm{s^{1/4-\eps'} E}_{\calL^{3/4,2}_{s} \dot{\calH}^{1}_{x}}$, we see that
\begin{equation*} 
	\nrm{{}^{(F_{s0})} \calN_{\linear}}_{\calL^{2, p}_{s} \dot{\calH}^{-1/2}_{x} (0, \subr]} 
	\leq C_{\calF, \calAlow} \cdot \nrm{s^{1/4-\eps'} \, E}_{\calL^{3/4,2}_{s} \dot{\calH}^{1}_{x}(0, \subr]}.
\end{equation*}
for $p=1,2$. Combining this with Case 1, \eqref{eq:pEst4Fs0:low:pf:1} follows. 

Proceeding similarly, but this time applying Leibniz's rule to \eqref{eq:pEst4Fs0:low:pf:4}, choosing $p=2$ and $\subr =1$, we obtain for $k=1,2$
\begin{equation*}
	\nrm{{}^{(F_{s0})} \calN_{\linear}}_{\calL^{2,2}_{s} \dot{\calH}^{k-1/2}_{x}(0,1]} 
	\leq C_{\calF, \calAlow} \cdot \nrm{E}_{\calP^{3/4} \calH^{k+1}_{x} (0,1]},
\end{equation*}
(we estimated $s \leq 1$) from which, along with the previous case, \eqref{eq:pEst4Fs0:low:pf:5} follows. \qedhere
\end{proof}

Our next proposition for $F_{s0}$ states that once we have a control of $\calE(t)$, we can control higher derivatives of $F_{s0}$ \emph{without any smallness assumption}. 

\begin{proposition} [Parabolic estimates for $F_{s0}$] \label{prop:pEst4Fs0:high}
Suppose $0 < T \leq 1$, and that the caloric-temporal gauge condition holds.

\begin{enumerate}
\item Let $t \in (-T, T)$. Then for $m \geq 4$, we have
\begin{equation} \label{eq:pEst4Fs0:high:1}
\begin{aligned}
	& \nrm{F_{s0}(t)}_{\calL^{1,\infty}_{s} \dot{\calH}_{x}^{m-1}(0,1]} + \nrm{F_{s0}(t)}_{\calL^{1,2}_{s} \dot{\calH}_{x}^{m}(0,1]} \\
	& \qquad \leq C_{\calF, \nrm{\rd_{t,x} \Alow(t)}_{H^{m-2}_{x}}} \cdot \bb( \calE(t) +  (\calF + \nrm{\rd_{t,x} \Alow(t)}_{H^{m-2}_{x}})^{2} \bb).
\end{aligned}
\end{equation}

In particular, for $1 \leq m \leq 31$, we have
\begin{equation} \label{eq:pEst4Fs0:high:3}
	\nrm{F_{s0}(t)}_{\calL^{1,\infty}_{s} \dot{\calH}_{x}^{m-1}(0,1]} + \nrm{F_{s0}(t)}_{\calL^{1,2}_{s} \dot{\calH}_{x}^{m}(0,1]} 
	\leq C_{\calF, \calAlow} \cdot (\calE(t) + (\calF + \calAlow)^{2}).
\end{equation}

\item For $1 \leq m \leq 21$, we have
\begin{equation} \label{eq:pEst4Fs0:high:2}
	\nrm{F_{s0}}_{\calL^{1,\infty}_{s} \calL^{2}_{t} \dot{\calH}^{m-1}_{x}(0,1]} + \nrm{F_{s0}}_{\calL^{1,2}_{s} \calL^{2}_{t} \dot{\calH}^{m}_{x}(0,1]} 
	\leq C_{\calF, \calAlow} \cdot (\calE + (\calF + \calAlow)^{2}).
\end{equation}
\end{enumerate}
\end{proposition}

Part (1) of the preceding proposition tells us that in order to control $m$ derivatives of $F_{s0}$ \emph{uniformly} in $s$ (rather than in the $\calL^{2}_{s}$ sense), we need to control $m$ derivatives of $\Alow_{i}$. This fact will be used in an important way to close the estimates for $\calAlow$ in \S \ref{subsec:AlowWave}. On the other hand, as in Proposition \ref{prop:pEst4Fsi}, the range of $k$ in Part (2) was chosen so that we can estimate whatever derivative of $\Alow$ which arises by $\calAlow$. 

 \begin{proof} 
\pfstep{Step 1: Proof of (1)}
Fix $t \in (-T, T)$. We will be working on the whole interval $(0,1]$. 

Note that \eqref{eq:pEst4Fs0:high:3} follows immediately from \eqref{eq:pEst4Fs0:high:1} and the definition of $\calE(t)$, as $\nrm{\rd_{t,x} \Alow_{i}}_{L^{\infty}_{t} H^{29}_{x}} \leq \calAlow$. In order to prove \eqref{eq:pEst4Fs0:high:1}, we begin by claiming that the following estimate holds for $k \geq 2$:
\begin{equation} \label{eq:pEst4Fs0:high:pf:1:1}
\nrm{{}^{(F_{s0})} \calN}_{\calL^{1+1, 2}_{s} \dot{\calH}^{k}_{x}}
\leq C_{\calF, \nrm{\rd_{t,x} \Alow}_{H^{k}_{x}}} \cdot \nrm{\nb_{x} F_{s0}}_{\calL^{1,2}_{s} \calH^{k}_{x}} + C_{\calF, \nrm{\rd_{t,x} \Alow}_{H^{k}_{x}}} \cdot (\calF + \nrm{\rd_{t,x} \Alow(t)}_{H^{k}_{x}})^{2}.
\end{equation}

Assuming the claim, we may apply the second part of Theorem \ref{thm:absP:absPth}, along with the bound $\nrm{F_{s0}}_{\calP^{1} \calH^{3}_{x}} \leq \calE(t)$, to conclude \eqref{eq:pEst4Fs0:high:1}.

To prove \eqref{eq:pEst4Fs0:high:pf:1:1}, we estimate the contributions of ${}^{(F_{s0})} \calN_{\forcing}$ and ${}^{(F_{s0})} \calN_{\linear}$ separately.

\pfstep{ - Case 1.1: The contribution of ${}^{(F_{s0})} \calN_{\forcing}$}
%Note that
%\begin{equation*}
%{}^{(F_{s0})} \calN_{\forcing} + {}^{(F_{s0})} \calN_{A_{0}} = -2 \LieBr{\tensor{F}{_{0}^{\ell}}}{F_{s\ell}}.
%\end{equation*}
We start with the simple inequality
$\nrm{\phi_{1} \phi_{2}}_{\dot{H}^{2}_{x}} \leq C \nrm{\phi_{1}}_{\dot{H}^{3/2}_{x} \cap L^{\infty}_{x}} \nrm{\phi_{2}}_{\dot{H}^{2}_{x}} + C \nrm{\phi_{1}}_{\dot{H}^{2}_{x}} \nrm{\phi_{2}}_{\dot{H}^{3/2}_{x} \cap L^{\infty}_{x}}$.
Applying Leibniz's rule, the Correspondence Principle, Lemma \ref{lem:absP:Holder4Ls} and Lemma \ref{lem:absP:algEst}, we get
\begin{equation*}
	\nrm{\calO(\psi_{1}, \psi_{2})}_{\calL^{1+1,2}_{s} \dot{\calH}^{k}_{x}}
	\leq C \nrm{\nb_{x} \psi_{1}}_{\calL^{3/4,\infty}_{s} \calH^{k-1}_{x}} \nrm{\nb_{x} \psi_{2}}_{\calL^{5/4,2}_{s} \calH^{k-1}_{x}},
\end{equation*}
for $k \geq 2$.

Let us put $\psi_{1} = F_{0\ell}$, $\psi_{2} = F_{s\ell}$, and apply Lemma \ref{lem:fundEst4F0i} to control $\nrm{\nb_{x} F_{0\ell}}_{\calL^{3/4,\infty}_{s} \calH^{k-1}_{x}}$ in terms of $\nrm{F_{s}}$, $\nrm{\Alow}$ and $\nrm{F_{s0}}$. Then we apply Proposition \ref{prop:pEst4Fsi} to estimate $\nrm{F_{s}}$ in terms of $\calF$ and $\nrm{\Alow}$. At this point, one may check that all $\nrm{\Alow}$, $\nrm{F_{s0}}$ that have arisen may be estimated by $\nrm{\rd_{t,x} \Alow}_{H^{k}_{x}}$ and $\nrm{\nb_{x} F_{s0}}_{\calL^{1,2}_{s} \calH^{k}_{x}}$, respectively. As a result, for $k \geq 2$, we obtain
\begin{equation*}
\begin{aligned}
	\nrm{{}^{(F_{s0})} \calN_{\forcing}}_{\calL^{1+1,2}_{s} \dot{\calH}^{k}_{x}} 
	\leq & C_{\calF, \nrm{\rd_{t,x} \Alow}_{H^{k}_{x}}} \cdot \nrm{\nb_{x} F_{s0}}_{\calL^{1,2}_{s} \calH^{k}_{x}} 
	+ C_{\calF, \nrm{\rd_{t,x} \Alow}_{H^{k}_{x}}} \cdot (\calF + \nrm{\rd_{t,x} \Alow}_{H^{k}_{x}}) \, \calF,
\end{aligned}\end{equation*}
which is good enough for \eqref{eq:pEst4Fs0:high:pf:1:1}.

\pfstep{ - Case 1.2: The contribution of ${}^{(F_{s0})} \calN_{\linear}$}
Here, let us start from \eqref{eq:linCovHeat:pf:0} in the proof of Proposition \ref{prop:linCovHeat}. Applying Leibniz's rule, the Correspondence Principle, Lemma \ref{lem:absP:Holder4Ls} and Lemma \ref{lem:absP:algEst}, we obtain
\begin{equation*} 
\left\{
\begin{aligned}
& \nrm{s^{-1/2} \calO(\psi_{1}, \nb_{x} \psi_{2})}_{\calL^{1+1, 2}_{s} \dot{\calH}^{k}_{x}} 
+ \nrm{s^{-1/2} \calO(\nb_{x} \psi_{1}, \psi_{2})}_{\calL^{1+1, 2}_{s} \dot{\calH}^{k}_{x}} \\
& \qquad \leq  C \nrm{\nb_{x} \psi_{1}}_{\calL^{1/4+1/4,\infty}_{s} \calH^{k}_{x}} \nrm{\nb_{x} \psi_{2}}_{\calL^{1,2}_{s} \calH^{k}_{x}}, \\
& \nrm{\calO(\psi_{1}, \psi_{2}, \psi_{3})}_{\calL^{5/4+1, 2}_{s} \dot{\calH}^{k}_{x}}
\leq  C \nrm{\nb_{x} \psi_{1}}_{\calL^{1/4+1/4, \infty}_{s} \calH^{k}_{x}} \nrm{\nb_{x} \psi_{2}}_{\calL^{1, 2}_{s} \calH^{k}_{x}} \nrm{\nb_{x} \psi_{3}}_{\calL^{1/4+1/4, \infty}_{s} \calH^{k+1}_{x}}.
\end{aligned}
\right.
\end{equation*}
for $k \geq 1$.

Note the extra weight of $s^{1/4}$ on $\psi_{1}, \psi_{3}$. Let us put $\psi_{1} = A, \psi_{2} = F_{s0}, \psi_{3} = A$, and apply Lemma \ref{lem:fundEst4A} to control $\nrm{A}$ in terms of $\nrm{F_{s}}$ and $\nrm{\Alow}$. Then using Proposition \ref{prop:pEst4Fsi}, we can control $\nrm{F_{s}}$  by $\calF$ and $\nrm{\Alow}$. Observe that all of $\nrm{\Alow}$ which have arisen can be estimated by $\nrm{\rd_{t,x} \Alow}_{H^{k}_{x}}$. As a result, we obtain the estimate
\begin{equation*}
	\nrm{{}^{(F_{s0})} \calN_{\linear}}_{\calL^{1+1,2}_{s} \dot{\calH}^{k}_{x}} \leq C_{\calF, \nrm{\rd_{t,x} \Alow}_{H^{k}_{x}}} \cdot \nrm{\nb_{x} F_{s0}}_{\calL^{1,2}_{s} \calH^{k}_{x}},
\end{equation*}
for $k \geq 1$. Combining this with the previous case, \eqref{eq:pEst4Fs0:high:pf:1:1} follows.

\pfstep{Step 2: Proof of (2)}
Let $0 \leq k \leq 19$, where the number $k$ corresponds to the number of times the equation $(\rd_{s} - \lap) F_{s0} = {}^{(F_{s0})}\calN$ is differentiated. We remark that its range has been chosen to be small enough so that every norm of $F_{si}$ and $\Alow_{i}$ that arises in the argument below can be controlled by $C_{\calF, \calAlow} \cdot (\calF + \calAlow)$ (by Propositions \ref{prop:pEst4Fsi} and \ref{prop:lowEst4Fsi}) and $\calAlow$, respectively.

We claim that for $\eps' > 0$ small enough, $0 \leq k \leq 19$ an integer, $1 \leq p \leq 2$ and $0 < \subr \leq 1$, the following estimate holds:
\begin{equation} \label{eq:pEst4Fs0:high:pf:2:1}
\nrm{{}^{(F_{s0})} \calN}_{\calL^{1+1,p}_{s} \calL^{2}_{t} \dot{\calH}^{k}_{x}(0,\subr]}
\leq C_{\calF, \calAlow} \, \cdot \nrm{s^{1/2-\eps'}\nb_{x} F_{s0}}_{\calL^{1,2}_{s} \calL^{2}_{t} \calH^{k}_{x}(0, \subr]}
+ C_{\calF, \calAlow} \cdot (\calE + \calF + \calAlow) (\calF + \calAlow). 
\end{equation}

Assuming \eqref{eq:pEst4Fs0:high:pf:2:1}, and taking $k=0$, $p=1,2$, we can apply the first part of Theorem \ref{thm:absP:absPth} to obtain \eqref{eq:pEst4Fs0:high:2} in the cases $m=1,2$. Then taking $1 \leq k \leq 19$ and $p=2$, we can apply the second part of Theorem \ref{thm:absP:absPth}, along with the bound \eqref{eq:pEst4Fs0:high:2} in the case $m=2$ that was just established, to conclude the rest of \eqref{eq:pEst4Fs0:high:2}.

As before, in order to prove \eqref{eq:pEst4Fs0:high:pf:2:1}, we treat the contributions of ${}^{(F_{s0})} \calN_{\forcing}$ and ${}^{(F_{s0})} \calN_{\linear}$ separately.

\pfstep{ - Case 2.1: The contribution of ${}^{(F_{s0})} \calN_{\forcing}$}
We claim that the following estimate holds for $0 \leq k \leq 19$ and $1 \leq p \leq2$:
\begin{equation} \label{eq:pEst4Fs0:high:pf:2:2}
\nrm{{}^{(F_{s0})} \calN_{\forcing}}_{\calL^{1+1,p}_{s} \calL^{2}_{t} \dot{\calH}^{k}_{x}(0,1]}
\leq C_{\calF, \calAlow} \cdot (\calE + \calF + \calAlow) (\calF + \calAlow).
\end{equation}

Note in particular that the right-hand side does not involve $\nrm{\nb_{x} F_{s0}}_{\calL^{1,2}_{s} \calL^{2}_{t} \calH^{k}_{x}}$. This is because we can use \eqref{eq:pEst4Fs0:high:3} to estimate whatever factor of $\nrm{F_{s0}}$ that arises in this case.

In what follows, we work on the whole $s$-interval $(0, 1]$. Starting from H\"older's inequality $\nrm{\phi_{1} \phi_{2}}_{L^{2}_{t,x}} \leq \nrm{\phi_{1}}_{L^{4}_{t,x}} \nrm{\phi_{2}}_{L^{4}_{t,x}}$ and using Leibniz's rule, the Correspondence Principle and Lemma \ref{lem:absP:Holder4Ls}, we obtain
\begin{equation*}
	\nrm{\calO(\psi_{1}, \psi_{2})}_{\calL^{1+1,p}_{s} \calL^{2}_{t} \dot{\calH}^{k}_{x} }
	\leq C \nrm{\psi_{1}}_{\calL^{3/4, r}_{s} \calL^{4}_{t} \calW^{k,4}_{x}} \nrm{\psi_{2}}_{\calL^{5/4,2}_{s} \calL^{4}_{t} \calW^{k,4}_{x}}.
\end{equation*}
where $\frac{1}{r} = \frac{1}{p} - \frac{1}{2}$. Let us put $\psi_{1} = F_{0 \ell}$, $\psi_{2} = F_{s \ell}$ and use Lemma \ref{lem:fundEst4F0i} to control $\nrm{F_{0\ell}}$ in terms of $\nrm{F_{s0}}$, $\nrm{F_{s\ell}}$ and $\nrm{\Alow_{\ell}}$. Then thanks to the assumption $0 \leq k \leq 19$, we can use \eqref{eq:pEst4Fs0:high:3}, the second part of Proposition \ref{prop:pEst4Fsi} and the definition of $\calAlow$ to control $\nrm{F_{s0}}$, $\nrm{F_{s\ell}}$ and $\nrm{\Alow_{\ell}}$ have arisen by $C_{\calF, \calAlow} \cdot (\calE + (\calF + \calAlow)^{2})$, $C_{\calF, \calAlow} \cdot \calF$ and $\calAlow$, respectively. 

On the other hand, to control $\nrm{F_{s\ell}}_{\calL^{5/4,2}_{s} \calL^{4}_{t} \calW^{k,4}_{x}}$, we first use Strichartz to estimate
\begin{equation*}
	\nrm{F_{s\ell}}_{\calL^{5/4,2}_{s} \calL^{4}_{t} \calW^{k,4}_{x}} 
	\leq  C\nrm{F_{s\ell}}_{\calL^{5/4,2}_{s} \calL^{4}_{t,x}} + C \nrm{F_{s\ell}}_{\calL^{5/4,2}_{s} \widehat{\calS}^{k+1/2}}
\end{equation*}
and then use Propositions \ref{prop:lowEst4Fsi} and \ref{prop:pEst4Fsi} to estimate the first and the second terms by $C_{\calF, \calAlow} \cdot (\calF + \calAlow)$ and $C_{\calF, \calAlow} \cdot \calF$, respectively. As a result, we obtain \eqref{eq:pEst4Fs0:high:pf:2:2} for $1 \leq p \leq 2$.

\pfstep{ - Case 2.2: The contribution of ${}^{(F_{s0})} \calN_{\linear}$}
Let $0 < \subr \leq 1$; we will work on the interval $(0, \subr]$ in this case. Let us begin with the following estimates, which follow immediately from \eqref{eq:linCovHeat:pf:0} by square integrating in $t$ and using H\"older:  
\begin{equation*}
\begin{aligned}
& \nrm{\phi_{1} \rd_{x} \phi_{2}}_{L^{2}_{t,x}} + \nrm{\rd_{x} \phi_{1} \phi_{2}}_{L^{2}_{t,x}}
\leq C \nrm{\phi_{1}}_{L^{\infty}_{t} (\dot{H}^{3/2}_{x} \cap L^{\infty}_{x})} \nrm{\phi_{2}}_{L^{2}_{t} \dot{H}^{1}_{x}}, \\
& \nrm{\phi_{1} \phi_{2} \phi_{3}}_{L^{2}_{t,x}}
\leq C \nrm{\phi_{1}}_{L^{\infty}_{t} \dot{H}^{1}_{x}} \nrm{\phi_{2}}_{L^{2}_{t} \dot{H}^{1}_{x}} \nrm{\phi_{3}}_{L^{\infty}_{t} \dot{H}^{1}_{x}}.
\end{aligned}
\end{equation*}

Using Leibniz's rule, the Correspondence Principle and Lemma \ref{lem:absP:Holder4Ls}, we obtain the following inequalities for $\eps' > 0$ small, $0 \leq k \leq 19$ and $1 \leq p \leq q \leq \infty$:
\begin{equation*}
%\left\{
\begin{aligned}
& \nrm{s^{-1/2} \calO(\psi_{1}, \nb_{x} \psi_{2})}_{\calL^{1+1, p}_{s} \calL^{2}_{t} \dot{\calH}^{k}_{x}} 
+ \nrm{s^{-1/2} \calO(\nb_{x} \psi_{1}, \psi_{2})}_{\calL^{1+1, p}_{s} \calL^{2}_{t} \dot{\calH}^{k}_{x}} \\
& \qquad \leq  C \subr^{\eps'} \bb( \sum_{j=0}^{k} \nrm{\nb_{x}^{(j)} \psi_{1}}_{\calL^{1/4,\infty}_{s} \calL^{\infty}_{t} (\dot{\calH}^{3/2}_{x} \cap \calL^{\infty}_{x})} \bb)\nrm{s^{1/4-\eps'} \nb_{x} \psi_{2}}_{\calL^{1,q}_{s} \calL^{2}_{t} \calH^{k}_{x}}, \\
& \nrm{\calO(\psi_{1}, \psi_{2}, \psi_{3})}_{\calL^{5/4+1, 2}_{s} \dot{\calH}^{k}_{x}} \\
& \qquad \leq  C \subr^{\eps'} \nrm{\nb_{x} \psi_{1}}_{\calL^{1/4+1/8, \infty}_{s} \calL^{\infty}_{t} \calH^{k}_{x}} \nrm{s^{1/4-\eps'} \nb_{x} \psi_{2}}_{\calL^{1, q}_{s} \calL^{2}_{t} \calH^{k}_{x}} \nrm{\nb_{x} \psi_{3}}_{\calL^{1/4+1/8, \infty}_{s} \calL^{\infty}_{t} \calH^{k}_{x}}.
\end{aligned}
%\right.
\end{equation*}

The factors $\subr^{\eps'}$ have arisen from applications of H\"older for $\calL^{\ell, p}_{s}$ (Lemma \ref{lem:absP:Holder4Ls}); we estimate them by $\leq 1$. Let us put $\psi_{1} = A$, $\psi_{2} = F_{s0}$ and $\psi_{3} = A$, and apply Lemma \ref{lem:fundEst4A} to control $\nrm{A}$ in terms of $\nrm{F_{s}}$ and $\calAlow$ (the latter thanks to the range of $k$). Then we apply Proposition \ref{prop:pEst4Fsi} to control $\nrm{F_{s}}$ in terms of $\calF$ and $\calAlow$ (again using the restriction of the range of $k$). As a result, we arrive at
\begin{equation*}
	\nrm{{}^{(F_{s0})} \calN_{\linear}}_{\calL^{1+1,p}_{s} \calL^{2}_{t} \dot{\calH}^{k}_{x}(0, \subr]}
	\leq C_{\calF, \calAlow} \nrm{s^{1/4-\eps'} \nb_{x} F_{s0}}_{\calL^{1,q}_{s} \calL^{2}_{t} \calH^{k}_{x}(0, \subr]},
\end{equation*}
for $\eps' > 0$ small, $0 \leq k \leq 19$ and $1 \leq p \leq q \leq \infty$. Taking $q=2$ and combining with the previous case, we obtain \eqref{eq:pEst4Fs0:high:pf:2:1} \qedhere
\end{proof}

The difference analogues of Propositions \ref{prop:pEst4Fs0:low} and \ref{prop:pEst4Fs0:high} can be proved in a similar manner, using the non-difference versions which have been just established. We give their statements below, omitting the proof.

\begin{proposition} [Estimate for $\dlt \calE$] \label{prop:pEst4Fs0:low:Diff}
Suppose that the caloric-temporal gauge condition holds, and furthermore that $\calF + \calAlow < \dlt_{E}$ where $\dlt_{E} > 0$ is sufficiently small. Then 
\begin{equation} \label{eq:pEst4Fs0:low:Diff1}
	\sup_{t \in (-T, T)} \dlt \calE(t) \leq C_{\calF, \calAlow} \cdot (\calF+\calAlow) (\dlt \calF + \dlt \calAlow).
\end{equation}
\end{proposition}

\begin{proposition} [Parabolic estimates for $\dlt F_{s0}$] \label{prop:pEst4Fs0:high:Diff}
Suppose $0 < T \leq 1$, and that the caloric-temporal gauge condition holds. Then the following statements hold.
\begin{enumerate}
\item Let $t \in (-T, T)$. Then for $m \geq 4$, we have
\begin{equation} \label{eq:pEst4Fs0:high:Diff:1}
\begin{aligned}
	& \nrm{\dlt F_{s0}(t)}_{\calL^{1,\infty}_{s} \dot{\calH}_{x}^{m-1}(0,1]} + \nrm{\dlt F_{s0}(t)}_{\calL^{1,2}_{s} \dot{\calH}_{x}^{m}(0,1]} \\
	& \qquad \leq C_{\calF, \nrm{\rd_{t,x} \Alow(t)}_{H^{m-2}_{x}}} \cdot \dlt \calE(t) \\
	& \phantom{\qquad \leq } + C_{\calF, \nrm{\rd_{t,x} \Alow(t)}_{H^{m-2}_{x}}} \cdot  (\calE(t) + \calF + \nrm{\rd_{t,x} \Alow(t)}_{H^{m-2}_{x}}) (\dlt \calF + \nrm{\rd_{t,x} (\dlt \Alow)(t)}_{H^{m-2}_{x}}).
\end{aligned}
\end{equation}

In particular, for $1 \leq m \leq 31$, we have
\begin{equation} \label{eq:pEst4Fs0:high:Diff:3}
\begin{aligned}
	& \nrm{\dlt F_{s0}(t)}_{\calL^{1,\infty}_{s} \dot{\calH}_{x}^{m-1}(0,1]} + \nrm{\dlt F_{s0}(t)}_{\calL^{1,2}_{s} \dot{\calH}_{x}^{m}(0,1]} \\
	& \qquad \leq C_{\calF, \calAlow} \cdot \dlt \calE(t) + C_{\calF, \calAlow} \cdot  (\calE(t) + \calF + \calAlow) (\dlt \calF + \dlt \calAlow).
\end{aligned}
\end{equation}

\item For $1 \leq m \leq 21$, we have
\begin{equation} \label{eq:pEst4Fs0:high:Diff:2}
\begin{aligned}
	& \nrm{\dlt F_{s0}}_{\calL^{1,\infty}_{s} \calL^{2}_{t} \dot{\calH}_{x}^{m-1}(0,1]} + \nrm{\dlt F_{s0}}_{\calL^{1,2}_{s} \calL^{2}_{t} \dot{\calH}_{x}^{m}(0,1]} \\
	& \qquad \leq C_{\calF, \calAlow} \cdot \dlt \calE + C_{\calF, \calAlow} \cdot  (\calE + \calF + \calAlow) (\dlt \calF + \dlt \calAlow).
\end{aligned}
\end{equation}
\end{enumerate}
\end{proposition}

\subsection{Parabolic estimates for $w_{i}$} \label{subsec:pEst4HPYM:pEst4wi}
Here we study the parabolic equation \eqref{eq:covParabolic4w} satisfied by $w_{i}$. Let us define
\begin{equation*} 
	{}^{(w_{i})} \calN := (\rd_{s} - \lap) w_{i} = {}^{(w_{i})} \calN_{\forcing} + {}^{(w_{i})} \calN_{\linear} 
\end{equation*}
where
\begin{align*}
	{}^{(w_{i})}\calN_{\linear} &= 2 s^{-1/2} \LieBr{A^\ell}{\nb_\ell w_{i}} + s^{-1/2} \LieBr{\nb^\ell A_\ell}{w_{i}} + \LieBr{A^{\ell}}{\LieBr{A_{\ell}}{w_{i}}} + 2 \LieBr{\tensor{F}{_{i}^{\ell}}}{w_{\ell}}, \\
	{}^{(w_{i})}\calN_{\forcing} &= 2 \LieBr{F_{0 \ell}}{\covD^{\ell} F_{0 i} + \covD_{0} \tensor{F}{^{\ell}_{i}}}. 
\end{align*}

The following proposition proves parabolic estimates for $w_{i}$ that we will need in the sequel.
\begin{proposition} [Parabolic estimates for $w_{i}$] \label{prop:pEst4wi} 
Suppose $0 < T \leq 1$, and that the caloric-temporal gauge condition holds.
\begin{enumerate}
\item Let $t \in (-T, T)$. For $1 \leq m \leq 30$ we have
\begin{equation} \label{eq:pEst4wi:1}
\begin{aligned}
	\nrm{w_{i}(t)}_{\calL^{1,\infty}_{s} \dot{\calH}^{m-1}_{x}(0,1]} + \nrm{w_{i}(t)}_{\calL^{1,2}_{s} \dot{\calH}^{m}_{x}(0,1]} 
	\leq C_{\calE(t), \calF, \calAlow} \cdot  (\calE(t) + \calF + \calAlow)^{2}.
\end{aligned}
\end{equation}

In the case $m =31$, on the other hand, we have the following estimate.
\begin{equation} \label{eq:pEst4wi:2}
\begin{aligned}
	& \nrm{w_{i}(t)}_{\calL^{1,\infty}_{s} \dot{\calH}^{30}_{x}(0,1]} + \nrm{w_{i}(t)}_{\calL^{1,2}_{s} \dot{\calH}^{31}_{x}(0,1]} \\
	& \qquad \leq C_{\calE(t), \calF, \calAlow, \nrm{\rd_{0} \Alow(t)}_{\dot{H}^{30}_{x}}} \cdot  (\calE(t) + \calF + \calAlow + \nrm{\rd_{0} \Alow(t)}_{\dot{H}^{30}_{x}})^{2}.
\end{aligned}
\end{equation}
%\comment{Maybe it is more convenient to use only $\rd_{0} \, \div \Alow$.}

\item For $1 \leq m \leq 16$, we have
\begin{equation} \label{eq:pEst4wi:3}
	\nrm{w_{i}}_{\calL^{1,\infty}_{s} \calL^{2}_{t} \dot{\calH}^{m-1}_{x}(0,1]} + \nrm{w_{i}}_{\calL^{1,2}_{s} \calL^{2}_{t} \dot{\calH}^{m}_{x}(0,1]} 
	\leq C_{\calE, \calF, \calAlow} \cdot (\calE + \calF + \calAlow)^{2}.
\end{equation}

Furthermore, for $0 \leq k \leq 14$, we have the following estimate for ${}^{(w_{i})} \calN$.
\begin{equation} \label{eq:pEst4wi:4}
	\nrm{{}^{(w_{i})} \calN}_{\calL^{2,\infty}_{s} \calL^{2}_{t} \dot{\calH}^{k}_{x} (0,1]} + \nrm{{}^{(w_{i})} \calN}_{\calL^{2,2}_{s} \calL^{2}_{t} \dot{\calH}^{k}_{x} (0,1]}
	\leq C_{\calE, \calF, \calAlow} \cdot (\calE + \calF + \calAlow)^{2}.
\end{equation}
\end{enumerate}
\end{proposition}

\begin{remark} 
Note that Part (1) of Proposition \ref{prop:pEst4wi} does not require a smallness assumption, as opposed to Proposition \ref{prop:pEst4Fs0:low}. Moreover, in comparison with Proposition \ref{prop:pEst4Fs0:high}, we need $m$ derivatives of $\Alow$ (i.e., one more derivative) to estimate $m$ derivatives of $w$ uniformly in $s$. 
\end{remark}

\begin{proof} 
\pfstep{Step 1: Proof of (1), for $1 \leq m \leq 3$}
Fix $t \in (-T, T)$. Let us define $v_{i} := \abs{\rd_{x}}^{-1/2} w_{i}$. From the parabolic equation for $w_{i}$, we derive the following parabolic equation for $v_{i}$:
\begin{equation*}
	(\rd_{s} - \lap) v_{i} = s^{1/4} \abs{\nb_{x}}^{-1/2} ({}^{(w_{i})} \calN),
\end{equation*}
where the right-hand side is evaluated at $t$. Note that $\nrm{w_{i}}_{\calL^{1, p}_{s} \dot{\calH^{k}_{x}}} =\nrm{v_{i}}_{\calL^{3/4,p}_{s} \dot{\calH}^{k+1/2}_{x}}$. The idea, as in the proof of Proposition \ref{prop:pEst4Fs0:low}, is to derive estimates for $v_{i}$ and then to translate to the corresponding estimates for $w_{i}$ using the preceding observation.

We will make two claims: First, for $0 < \subr \leq 1$ and $1 \leq p \leq 2$, the following estimate holds.
\begin{equation} \label{eq:pEst4wi:pf:1}
\begin{aligned}
	\sup_{i} \nrm{{}^{(w_{i})} \calN}_{\calL^{2,p}_{s} \dot{\calH}^{-1/2}_{x} (0, \subr]} 
	\leq & C_{\calE(t), \calF, \calAlow} \cdot  \nrm{s^{1/4-\eps'} v}_{\calL^{3/4,2}_{s} \dot{\calH}^{1}_{x} (0, \subr]} \\
	& +C_{\calE(t), \calF, \calAlow} \cdot  (\calE(t) + \calF + \calAlow)^{2} .
\end{aligned}
\end{equation}

Second, for $k = 1,2$, the following estimate holds.
\begin{equation} \label{eq:pEst4wi:pf:2}
\begin{aligned}
	\sup_{i} \nrm{{}^{(w_{i})} \calN}_{\calL^{2,2}_{s} \dot{\calH}^{k-1/2}_{x} (0, 1]} 
	\leq & C_{\calE(t), \calF, \calAlow} \cdot  \nrm{v}_{\calP^{3/4} \calH^{k+1}_{x} (0, 1]} \\
	& +C_{\calE(t), \calF, \calAlow} \cdot  (\calE(t) + \calF + \calAlow)^{2} .
\end{aligned}
\end{equation}

Note that $\nrm{{}^{(w_{i})} \calN}_{\calL^{2,p}_{s} \dot{\calH}^{k}_{x}(0,\subr]} = \nrm{s^{1/4} \abs{\nb_{x}}^{-1/2} ({}^{(w_{i})} \calN)}_{\calL^{3/4+1,p}_{s} \dot{\calH}^{k+1/2}_{x}(0,\subr]}$. Assuming \eqref{eq:pEst4wi:pf:1} and using the preceding observation, we can apply the first part of Theorem \ref{thm:absP:absPth} to $v_{i}$ (note furthermore that $v_{i} = 0$ at $s=0$), from which we obtain a bound on $\nrm{v}_{\calP^{3/4} \calH^{2}_{x}}$. Next, assuming \eqref{eq:pEst4wi:pf:2} and applying the second part of Theorem \ref{thm:absP:absPth} to $v_{i}$, we can also control $\nrm{v}_{\calP^{3/4} \calH^{4}_{x}}$. Using the fact that $v_{i} = s^{1/4} \abs{\nb_{x}}^{-1/2} w_{i}$, \eqref{eq:pEst4wi:1} now follows.

We are therefore left with the task of establishing \eqref{eq:pEst4wi:pf:1} and \eqref{eq:pEst4wi:pf:2}. For this purpose, we divide ${}^{(w_{i})}  \calN = {}^{(w_{i})}  \calN_{\forcing} + {}^{(w_{i})}  \calN_{\linear}$, and treat each of them separately. 

\pfstep{- Case 1.1: Contribution of ${}^{(w_{i})} \calN_{\forcing}$}
In this case, we work on the whole interval $(0, 1]$. We start with the inequality
\begin{equation*}
	\nrm{\phi_{1} \phi_{2}}_{\dot{H}^{-1/2}_{x}} \leq \nrm{\phi_{1}}_{\dot{H}^{1}_{x}} \nrm{\phi_{2}}_{L^{2}_{x}},
\end{equation*}
which follows from Lemma \ref{lem:homSob}. Using Leibniz's rule, the Correspondence Principle and Lemma \ref{lem:absP:Holder4Ls}, we arrive at the following inequality for $k \geq 0$ and $\frac{1}{r} = \frac{1}{p} - \frac{1}{2}$:
\begin{equation*} 
	\nrm{\calO(\psi_{1}, \psi_{2})}_{\calL^{2,p}_{s} \dot{\calH}^{k-1/2}_{x}} \leq C \nrm{\nb_{x} \psi_{1}}_{\calL^{3/4,r}_{s} \calH^{k}_{x}} \nrm{\psi_{2}}_{\calL^{5/4,2}_{s} \calH^{k}_{x}}. 
\end{equation*}

Let us restrict to $0 \leq k \leq 2$ and put $\psi_{1} = F_{0 \ell}$, $\psi_{2} = \covD^{\ell} F_{0 i} + \covD_{0} \tensor{F}{^{\ell}_{i}}$. In order to estimate $\nrm{\nb_{x} F_{0\ell}}_{\calL^{3/4,r}_{s} \calH^{k}_{x}}$ and $\nrm{\covD^{\ell} F_{0 i} + \covD_{0} \tensor{F}{^{\ell}_{i}}}_{\calL^{5/4,2}_{s} \calH^{k}_{x}}$, we apply Lemmas \ref{lem:fundEst4F0i} (with $p=r$) and \ref{lem:fundEst4DF} (with $p=2$), respectively, from which we obtain an estimate of $\nrm{{}^{(w_{i})} \calN}_{\calL^{2,p}_{s} \dot{\calH}^{k-1/2}_{x}}$ in terms of $\nrm{F_{s0}}$, $\nrm{F_{s}}$ and $\nrm{\Alow}$. The latter two types of terms can be estimated by $\calF$ and $\calAlow$, respectively. Moreover, using Propositions \ref{prop:pEst4Fs0:high}, $\nrm{F_{s0}(t)}$ can be estimated by $\calE(t)$, $\calF$ and $\calAlow$. As a result, for $0 \leq k \leq 2$ and $1 \leq p \leq 2$, we obtain
\begin{equation*}
	\sup_{i} \nrm{{}^{(w_{i})} \calN(t)}_{\calL^{2,p}_{s} \dot{\calH}^{k-1/2}_{x}(0,1]}
	\leq C_{\calE(t), \calF, \calAlow} \cdot  (\calE(t) + \calF + \calAlow)^{2}.
\end{equation*} 
which is good enough for \eqref{eq:pEst4wi:pf:1} and \eqref{eq:pEst4wi:pf:2}.

\pfstep{- Case 1.2: Contribution of ${}^{(w_{i})} \calN_{\linear}$}
Note that ${}^{(w_{i})} \calN_{\linear}$ has the same schematic form as ${}^{(F_{s0})} \calN_{\linear}$. Therefore, the same proof as in Case 2 of the proof of Proposition \ref{prop:pEst4Fs0:low} gives us the estimates
\begin{equation*}
	\sup_{i} \nrm{{}^{(w_{i})} \calN(t)}_{\calL^{2,p}_{s} \dot{\calH}^{-1/2}_{x}(0, \subr]} 
	\leq C_{\calF, \calAlow} \cdot \nrm{s^{1/4-\eps'} \, v}_{\calL^{3/4,2}_{s} \dot{\calH}^{1}_{x}(0,\subr]},
\end{equation*}
for $p =1, 2$, $0 < \subr \leq 1$ and arbitrarily small $\eps' > 0$, and
\begin{equation*}
	\sup_{i} \nrm{{}^{(w_{i})} \calN(t)}_{\calL^{2,2}_{s} \dot{\calH}^{k-1/2}_{x}(0,1]} 
	\leq C_{\calF, \calAlow} \nrm{v}_{\calP^{3/4} \calH^{k+1}_{x}(0,1]},
\end{equation*}
for $k=1,2$. Combined with the previous case, we obtain \eqref{eq:pEst4wi:pf:1} and \eqref{eq:pEst4wi:pf:2}.

\pfstep{Step 2: Proof of (1), for $m \geq 4$}
By working with $v_{i}$ instead of $w_{i}$, we were able to prove the {\it a priori} estimate \eqref{eq:pEst4wi:1} for low $m$ by an application of Theorem \ref{thm:absP:absPth}. The drawback of this approach, as in the case of $F_{s0}$, is that the estimate that we derive is not good enough in terms of the necessary number of derivatives of $\Alow$. In order to prove \eqref{eq:pEst4wi:1} for higher $m$, and \eqref{eq:pEst4wi:2} as well, we revert back to the parabolic equation for $w_{i}$.

We claim that the following estimate holds for $k \geq 2$:
\begin{equation} \label{eq:pEst4wi:pf:3}
\begin{aligned}
	\sup_{i} \nrm{{}^{(w_{i})} \calN(t)}_{\calL^{1+1,2}_{s} \dot{\calH}^{k}_{x}(0,1]} 
	\leq & C_{\calF, \nrm{\rd_{t,x} \Alow}_{H^{k}_{x}}} \cdot \nrm{\nb_{x} w}_{\calL^{1,2}_{s} \calH^{k}_{x}(0,1]}\\
	& + C_{\calE(t), \calF, \nrm{\rd_{t,x} \Alow(t)}_{H^{k+1}_{x}}} (\calE(t) + \calF + \nrm{\rd_{t,x} \Alow(t)}_{H^{k+1}_{x}})^{2}.
\end{aligned}
\end{equation}

Assuming the claim, let us first finish the proof of (1). Note that for $0 \leq k \leq 29$, we have $\nrm{\rd_{t,x} \Alow}_{H^{k}_{x}} \leq \calAlow$. Therefore, every norm $\nrm{\rd_{t,x} \Alow}$ arising in \eqref{eq:pEst4wi:pf:3} for $2 \leq k \leq 28$ can be estimated by $\calAlow$. Using this, along with the estimate \eqref{eq:pEst4wi:1} for $1 \leq m \leq 3$ which has been established in Step 1, we can apply the second part of Theorem \ref{thm:absP:absPth} to conclude \eqref{eq:pEst4wi:1} for all $4 \leq m \leq 30$.

Note, on the other hand, that for $k= 30$ we only have $\nrm{\rd_{t,x} \Alow}_{H^{30}_{x}} \leq \calAlow + \nrm{\rd_{0} \Alow}_{\dot{H}^{30}_{x}}$. From \eqref{eq:pEst4wi:pf:3}, we therefore obtain the estimate
\begin{equation*}
\begin{aligned}
	\sup_{i} \nrm{{}^{(w_{i})} \calN(t)}_{\calL^{1+1,2}_{s} \dot{\calH}^{29}_{x}(0,1]} 
	\leq & C_{\calF, \calAlow} \cdot \nrm{\nb_{x} w}_{\calL^{1,2}_{s} \calH^{29}_{x}(0,1]}\\
	& + C_{\calE(t), \calF, \calAlow, \nrm{\rd_{0} \Alow(t)}_{\dot{H}^{30}_{x}}} (\calE(t) + \calF + \calAlow + \nrm{\rd_{0} \Alow(t)}_{\dot{H}^{30}_{x}})^{2}.
\end{aligned}
\end{equation*}

Combining this with the case $k=30$ of \eqref{eq:pEst4wi:1}, an application of the second part of Theorem \ref{thm:absP:absPth} gives \eqref{eq:pEst4wi:2}.

We are therefore only left to prove \eqref{eq:pEst4wi:pf:3}. As usual, we will treat ${}^{(w_{i})} \calN_{\forcing}$ and ${}^{(w_{i})} \calN_{\linear}$ separately, and work on the whole interval $(0, 1]$ in both cases. 

\pfstep{- Case 2.1: Contribution of ${}^{(w_{i})} \calN_{\forcing}$}
As in Case 1.1 in the proof of Proposition \ref{prop:pEst4Fs0:high}, we begin with the inequality $\nrm{\phi_{1} \phi_{2}}_{\dot{H}^{2}_{x}} \leq C \nrm{\phi_{1}}_{\dot{H}^{3/2}_{x} \cap L^{\infty}_{x}} \nrm{\phi_{2}}_{\dot{H}^{2}_{x}} + C \nrm{\phi_{1}}_{\dot{H}^{2}_{x}} \nrm{\phi_{2}}_{\dot{H}^{3/2}_{x} \cap L^{\infty}_{x}}$ and apply Leibniz's rule, the Correspondence Principle, Lemma \ref{lem:absP:Holder4Ls} and Lemma \ref{lem:absP:algEst}. As a result, for $k \geq 2$, we obtain
\begin{equation*}
	\nrm{\calO(\psi_{1}, \psi_{2})}_{\calL^{1+1,2}_{s} \dot{\calH}^{k}_{x}}
	\leq C \nrm{\nb_{x} \psi_{1}}_{\calL^{3/4,\infty}_{x} \calH^{k-1}_{x}} \nrm{\nb_{x} \psi_{2}}_{\calL^{5/4,2}_{x} \calH^{k-1}_{x}}
\end{equation*}

As in Case 1.1, we put $\psi_{1} = F_{0 \ell}$, $\psi_{2} = \covD^{\ell} F_{0i} + \covD_{0} \tensor{F}{^{\ell}_{i}}$, and apply Lemmas \ref{lem:fundEst4F0i} (with $p=\infty$) and \ref{lem:fundEst4DF} (with $p =2$), by which we obtain an estimate of $\nrm{{}^{(w_{i})} \calN_{\forcing}}_{\calL^{2,p}_{s} \dot{\calH}^{k}_{x}}$ in terms of $\nrm{F_{s0}}$, $\nrm{F_{si}}$ and $\nrm{\rd_{t,x} \Alow}$. Using Proposition \ref{prop:pEst4Fs0:high} and Proposition \ref{prop:pEst4Fsi} in order, we can estimate $\nrm{F_{s0}}$ and $\nrm{F_{si}}$ in terms of $\calE(t)$, $\calF$ and $\nrm{\rd_{t,x} \Alow}$. At this point, one may check that all $\nrm{\rd_{t,x} \Alow}$ that have arisen can be estimated by $\nrm{\rd_{t,x} \Alow(t)}_{H^{k+1}_{x}}$. As a result, we obtain the following estimate for $k \geq 2$:
\begin{equation*}
	\sup_{i} \nrm{{}^{(w_{i})} \calN(t)}_{\calL^{2,2}_{s} \dot{\calH}^{k}_{x}(0,1]}
	\leq C_{\calE(t), \calF, \nrm{\rd_{t,x} \Alow(t)}_{H^{k+1}_{x}}} (\calE(t) + \calF + \nrm{\rd_{t,x} \Alow(t)}_{H^{k+1}_{x}})^{2},
\end{equation*}
which is good.

\pfstep{- Case 2.2: Contribution of ${}^{(w_{i})} \calN_{\linear}$}
As ${}^{(w_{i})} \calN_{\linear}$ looks schematically the same as ${}^{(F_{s0})} \calN_{\linear}$, Step 1.2 of the proof of Proposition \ref{prop:pEst4Fs0:high} immediately gives
\begin{equation*}
	\sup_{i} \nrm{{}^{(w_{i})} \calN_{\linear}}_{\calL^{1+1,2}_{s} \dot{\calH}^{k}_{x}(0,1]} 
	\leq C_{\calF, \nrm{\rd_{t,x} \Alow}_{H^{k}_{x}}} \cdot \nrm{\nb_{x} w}_{\calL^{1,2}_{s} \calH^{k}_{x}(0,1]},
\end{equation*}
for $k \geq 1$. Combined with the previous case, this proves \eqref{eq:pEst4wi:pf:3}, as desired.

\pfstep{Step 3: Proof of (2)} 
Let $0 \leq k \leq 14$, where $k$ corresponds to the number of times the equation $(\rd_{s} - \lap) w_{i} = {}^{(w_{i})} \calN$ is differentiated. The range has been chosen so that Proposition \ref{prop:pEst4Fs0:high} can be applied to estimate every norm of $F_{s0}$ which arises in terms of $\calE$, $\nrm{F_{s}}$ and $\nrm{\calAlow}$, and furthermore so that all $\nrm{F_{s}}$ and $\nrm{\Alow}$ that arise can be estimated by $C_{\calF, \calAlow} \cdot \calF$ (by Proposition \ref{prop:pEst4Fsi}) and $\calAlow$, respectively.

We claim that for $\eps' > 0$ small enough, $0 \leq k \leq 14$ an integer, $1 \leq p \leq q \leq \infty$ and $0 < \subr \leq 1$, the following estimate holds:
\begin{equation} \label{eq:pEst4wi:pf:4}
	\sup_{i} \nrm{{}^{(w_{i})} \calN}_{\calL^{1+1, p}_{s} \calL^{2}_{t} \dot{\calH}^{k}_{x} (0, \subr]}
	\leq C_{\calF, \calAlow} \cdot \nrm{s^{1/4-\eps'} \nb_{x} w}_{\calL^{1,q}_{s} \calL^{2}_{t} \calH^{k}_{x}(0,\subr]} + C_{\calE, \calF, \calAlow} \cdot (\calE + \calF + \calAlow)^{2}.
\end{equation}

Assuming the claim, let us prove (2). Taking $k=0$, $p = 1,2$ and $q=2$, we may apply the first part of Theorem \ref{thm:absP:absPth} (along with the fact that $w = 0$ at $s=0$) to obtain \eqref{eq:pEst4wi:3} in the cases $m=1,2$. Combining this with \eqref{eq:pEst4wi:pf:4} in the cases $1 \leq k \leq 14$, $p=q=2$ and $\subr = 1$, we can apply the second part of Theorem \ref{thm:absP:absPth} to obtain the rest of \eqref{eq:pEst4wi:3}. Finally, considering \eqref{eq:pEst4wi:pf:4} with $0 \leq k \leq 14$ with $p=2, \infty$, $q=\infty$ and $\subr = 1$, and estimating $\nrm{\nb_{x} w}_{\calL^{1,\infty}_{s} \calH^{k}_{x}(0,1]}$ in the first term on the right-hand side by \eqref{eq:pEst4wi:3}, we obtain \eqref{eq:pEst4wi:4}, which finishes the proof of Part (2).

It therefore only remain to prove \eqref{eq:pEst4wi:pf:4}, for which we split ${}^{(w_{i})} \calN = {}^{(w_{i})} \calN_{\forcing} + {}^{(w_{i})} \calN_{\linear}$ as usual.

\pfstep{- Case 3.1: Contribution of ${}^{(w_{i})} \calN_{\forcing}$}
In this case, we work on the whole interval $(0, 1]$. 

Let us begin with the inequality $\nrm{\phi_{1} \phi_{2}}_{L^{2}_{t,x}} \leq \nrm{\phi_{1}}_{L^{4}_{t,x}} \nrm{\phi_{2}}_{L^{4}_{t,x}}$. Applying Leibniz's rule, the Correspondence Principle, Lemma \ref{lem:absP:Holder4Ls}, we obtain, for $k \geq 0$ and $1 \leq p \leq \infty$,
\begin{equation*}
	\nrm{\calO(\psi_{1}, \psi_{2})}_{\calL^{1+1,p}_{s} \calL^{2}_{t} \dot{\calH}^{k}_{x}}
	\leq C \nrm{\psi_{1}}_{\calL^{3/4, r_{1}}_{s} \calL^{4}_{t} \calW^{4, k}_{x}} \nrm{\psi_{2}}_{\calL^{5/4, r_{2}}_{s} \calL^{4}_{t} \calW^{4, k}_{x}}
\end{equation*}
where $\frac{1}{p} = \frac{1}{r_{1}} + \frac{1}{r_{2}}$. Since $p \geq 1$, we may choose $r_{1}, r_{2}$ so that $r_{1}, r_{2} \geq 2$. As before, let us take $\psi_{1} = F_{0\ell}$ and  $\psi_{2} = \covD^{\ell} F_{0i} + \covD_{0} \tensor{F}{^{\ell}_{i}}$ and apply Lemma \ref{lem:fundEst4F0i} (with $p=r_{1}$) and Lemma \ref{lem:fundEst4DF} (with $p=r_{2}$), respectively. Then we apply Proposition \ref{prop:pEst4Fs0:high} and Proposition \ref{prop:pEst4Fsi} in sequence, where we remark that both can be applied thanks to the restriction $0 \leq k \leq 14$. As a result, we obtain an estimate of $\nrm{{}^{(w_{i})} \calN_{\forcing}}_{\calL^{2,p}_{s} \calL^{2}_{t} \dot{\calH}^{k}_{x}}$ in terms of $\calE, \calF$ and $\nrm{\rd_{t,x} \Alow}$. One may then check that all terms that arise are at least quadratic in the latter three quantities, and furthermore that each $\nrm{\rd_{t,x} \Alow}$ which has arisen can be estimated by $\calAlow$, thanks again to the restriction $0 \leq k \leq 14$. In the end, we obtain, for $0 \leq k \leq 14$ and $1 \leq p \leq \infty$, the following estimate:
\begin{equation*}
	\sup_{i} \nrm{{}^{(w_{i})} \calN_{\forcing}}_{\calL^{1+1, p}_{s} \calL^{2}_{t} \dot{\calH}^{k}_{x} (0,1]} \leq C_{\calE, \calF, \calAlow} \cdot (\calE + \calF + \calAlow)^{2}.
\end{equation*}

\pfstep{- Case 3.2: Contribution of ${}^{(w_{i})} \calN_{\linear}$} As before, we utilize the fact that  ${}^{(w_{i})} \calN_{\linear}$ looks schematically the same as ${}^{(F_{s0})} \calN_{\linear}$. Consequently, Step 2.2 of the proof of Proposition \ref{prop:pEst4Fs0:high} implies
\begin{equation*}
	\sup_{i} \nrm{{}^{(w_{i})} \calN_{\linear}}_{\calL^{1+1,p}_{s} \calL^{2}_{t} \dot{\calH}^{k}_{x}(0,\subr]} 
	\leq C_{\calF, \calAlow} \cdot \nrm{s^{1/4-\eps'} \nb_{x} w}_{\calL^{1,r}_{s} \calL^{2}_{t} \calH^{k}_{x}(0,\subr]},
\end{equation*}
for $\eps' > 0$ small, $0 \leq k \leq 14$, $1 \leq p \leq r \leq \infty$ and $0 < \subr \leq 1$. Combined with the previous case, we obtain \eqref{eq:pEst4wi:pf:4}. \qedhere
\end{proof}

Again, by essentially the same proof, the following difference analogue of Proposition \ref{prop:pEst4wi} follows.
\begin{proposition} [Parabolic estimates for $\dlt w_{i}$] \label{prop:pEst4wi:Diff} 
Suppose $0 < T \leq 1$, and that the caloric-temporal gauge condition holds.
\begin{enumerate}
\item Let $t \in (-T, T)$. For $1 \leq m \leq 30$ we have
\begin{equation} \label{eq:pEst4wi:Diff:1}
\begin{aligned}
	\nrm{\dlt w_{i}(t)}_{\calL^{1,\infty}_{s} \dot{\calH}^{m-1}_{x}(0,1]} + \nrm{\dlt w_{i}(t)}_{\calL^{1,2}_{s} \dot{\calH}^{m}_{x}(0,1]} 
	\leq C_{\calE(t), \calF, \calAlow} \cdot  (\calE(t) + \calF + \calAlow) (\dlt \calE(t) + \dlt \calF + \dlt \calAlow).
\end{aligned}
\end{equation}

In the case $m =31$, on the other hand, we have the following estimate.
\begin{equation} \label{eq:pEst4wi:Diff:2}
\begin{aligned}
	\nrm{\dlt w_{i}(t)&}_{\calL^{1,\infty}_{s} \dot{\calH}^{30}_{x}(0,1]} + \nrm{\dlt w_{i}(t)}_{\calL^{1,2}_{s} \dot{\calH}^{31}_{x}(0,1]} \\
	\leq & C_{\calE(t), \calF, \calAlow, \nrm{\rd_{0} \Alow(t)}_{\dot{H}^{30}_{x}}} \cdot  (\calE(t) + \calF + \calAlow + \nrm{\rd_{0} \Alow(t)}_{\dot{H}^{30}_{x}}) \\
	& \times (\dlt \calE(t) + \dlt \calF + \dlt \calAlow + \nrm{\rd_{0} (\dlt \Alow)(t)}_{\dot{H}^{30}_{x}})
\end{aligned}
\end{equation}
%\comment{Maybe it is more convenient to use only $\rd_{0} \, \div \Alow$.}

\item For $1 \leq m \leq 16$, we have
\begin{equation} \label{eq:pEst4wi:Diff:3}
	\nrm{\dlt w_{i}}_{\calL^{1,\infty}_{s} \calL^{2}_{t} \dot{\calH}^{m-1}_{x}(0,1]} + \nrm{\dlt w_{i}}_{\calL^{1,2}_{s} \calL^{2}_{t} \dot{\calH}^{m}_{x}(0,1]} 
	\leq C_{\calE, \calF, \calAlow} \cdot (\calE + \calF + \calAlow)(\dlt \calE + \dlt \calF + \dlt \calAlow).
\end{equation}

Furthermore, for $0 \leq k \leq 14$, we have the following estimate for ${}^{(\dlt w_{i})} \calN := (\rd_{s} - \lap) (\dlt w_{i})$.
\begin{equation} \label{eq:pEst4wi:Diff:4}
	\nrm{{}^{(\dlt w_{i})} \calN}_{\calL^{2,\infty}_{s} \calL^{2}_{t} \dot{\calH}^{k}_{x} (0,1]} + \nrm{{}^{(\dlt w_{i})} \calN}_{\calL^{2,2}_{s} \calL^{2}_{t} \dot{\calH}^{k}_{x} (0,1]}
	\leq C_{\calE, \calF, \calAlow} \cdot (\calE + \calF + \calAlow) (\dlt \calE + \dlt \calF + \dlt \calAlow).
\end{equation}
\end{enumerate}
\end{proposition}

\section{Proofs of Propositions \ref{prop:est4a0} - \ref{prop:cont4FA}} \label{sec:pfOfProps}
In this section, we will sketch the proofs of Propositions \ref{prop:est4a0} - \ref{prop:cont4FA}. 
%We remark that this part may be read independently from other parts of the paper, as the techniques used here will not be used anywhere else.

\begin{proof}[Proof of Proposition \ref{prop:est4a0}]
We will give a proof of the non-difference estimate \eqref{eq:est4a0}, leaving the similar case of the difference estimate \eqref{eq:est4dltA0} to the reader.

In what follows, we work on the time interval $I = (-T, T)$. Recalling the definition of $\calA_{0}$, we need to estimate 
$\nrm{A_{0}(s=0)}_{L^{\infty}_{t} L^{3}_{x}}$, $\nrm{\rd_{x} A_{0}(s=0)}_{L^{\infty}_{t} L^{2}_{x}}$, $\nrm{A_{0}(s=0)}_{L^{1}_{t} L^{\infty}_{x}}$, $\nrm{\rd_{x} A_{0}(s=0)}_{L^{1}_{t} L^{3}_{x}}$ and $\nrm{\rd_{x}^{(2)} A_{0}(s=0)}_{L^{1}_{t} L^{2}_{x}}$ 
by the right-hand side of \eqref{eq:est4a0}.

Using $\rd_{s} A_{0} = F_{s0}$, the first two terms can be estimated simply by $C \calE$ as follows.
\begin{align*}
	& \nrm{A_{0}(s=0)}_{L^{\infty}_{t} L^{3}_{x}} + \nrm{\rd_{x} A_{0}(s=0)}_{L^{\infty}_{t} L^{2}_{x}}   \\
	& \quad  \leq  \sup_{t \in I} \int_{0}^{1} (s')^{1/2} (s')\nrm{F_{s0}(t, s')}_{\calL^{3}_{x}(s')} \frac{\ud s'}{s'} 
	+ \sup_{t \in I} \int_{0}^{1} (s')^{1/4} (s') \nrm{\nb_{x} F_{s0}(t, s')}_{ \calL^{2}_{x}(s')} \, \frac{\ud s'}{s'} \\
	& \phantom{\quad } \leq C \calE.
\end{align*}

For the next two terms, using H\"older in time, it suffices to estimate $\nrm{A_{0}(s=0)}_{L^{2}_{t} L^{\infty}_{x}}$, $\nrm{\rd_{x} A_{0}(s=0)}_{L^{2}_{t} L^{3}_{x}}$. Using \eqref{eq:pEst4Fs0:high:2} of Proposition \ref{prop:pEst4Fs0:high}, along with Gagliardo-Nirenberg, interpolation and Sobolev, these are estimated as follows.
\begin{align*}
	\nrm{A_{0}(s=0)}_{L^{2}_{t} L^{\infty}_{x}} + \nrm{\rd_{x} A_{0}(s=0)}_{L^{2}_{t} L^{3}_{x}} 
	\leq & \int_{0}^{1} (s')^{1/4} (s') \bb( \nrm{F_{s0}(s')}_{\calL^{2}_{t} \calL^{\infty}_{x}(s')} + \nrm{\nb_{x} F_{s0}(s')}_{\calL^{2}_{t} \calL^{3}_{x}(s')} \bb) \, \frac{\ud s'}{s'} \\
	\leq & C_{\calF, \calAlow} \cdot \calE + C_{\calF, \calAlow} \cdot (\calF+\calAlow)^{2}.
\end{align*}

Unfortunately, the same argument applied to the term $\nrm{\rd_{x}^{(2)} A_{0}(s=0)}_{L^{1}_{t} L^{2}_{x}}$ fails by a logarithm. In this case, we make use of the equations $\rd_{s} A_{0} = F_{s0}$ and the parabolic equation for $F_{s0}$. Indeed, let us begin by writing
\begin{align*}
	\lap A_{0}(s=0) 
	=& - \int_{0}^{1} \lap F_{s0}(s') \, \ud s' 
	= - \int_{0}^{1} \rd_{s} F_{s0}(s') \, \ud s' + \int_{s}^{1} {}^{(F_{s0})} \calN(s') \, \ud s' \\
	= & \Flow_{s0} + \int_{s}^{1} s' ({}^{(F_{s0})} \calN)(s') \, \frac{\ud s'}{s'},
\end{align*}
where on the last line, we used the fact that $F_{s0}(s=0) = -w_{0}(s=0) = 0$. Taking the $L^{2}_{t,x}$ norm of the above identity and applying triangle and Minkowski, we obtain
\begin{equation*}
	\nrm{\lap A_{0}(s=0)}_{L^{2}_{t,x}} \leq \nrm{\Flow_{s0}}_{L^{2}_{t,x}}  +  \int_{s}^{1} s' \nrm{{}^{(F_{s0})} \calN(s')}_{L^{2}_{t,x}} \, \frac{\ud s'}{s'}. 
\end{equation*}

The first term can be estimated using \eqref{eq:pEst4Fs0:high:2}, whereas the last term can be estimated by putting together  \eqref{eq:pEst4Fs0:high:pf:2:1} (in the proof of Proposition \ref{prop:pEst4Fs0:high}) and \eqref{eq:pEst4Fs0:high:2}. As a consequence, we obtain
\begin{equation*}
	\nrm{\lap A_{0}(s=0)}_{L^{2}_{t,x}} \leq C_{\calF, \calAlow} \cdot \calE + C_{\calF, \calAlow} \cdot (\calF+\calAlow)^{2}.
\end{equation*}

By a simple integration by parts\footnote{The fact that $A_{0}(t, s=0) \in H^{m}_{x}$ for any $m \geq 0$ can be used to show that the boundary terms vanish at the spatial infinity.}, it follows that $\nrm{\rd_{x}^{(2)} A_{0}(s=0)}_{L^{2}_{t, x}} \leq C \nrm{\lap A_{0}(s=0)}_{L^{2}_{t,x}}$. Then by H\"older in time, the desired $L^{1}_{t} L^{2}_{x}$-estimate follows. This completes the proof of \eqref{eq:est4a0}. \qedhere
\end{proof}

\begin{proof} [Proof of Proposition \ref{prop:est4ai}]
Again, we will only treat the non-difference case, as the difference case follows by essentially the same arguments.

The goal is to estimate $\sup_{i} \sup_{0 \leq s \leq 1} \nrm{A_{i}(s)}_{\SH^{1}}$ in terms of $\calF+\calAlow$. Note that, proceeding naively, one can easily prove the bound
\begin{equation} \label{eq:est4A:proof:0}
	\nrm{A_{i}(s)}_{\SH^{1}} \leq \int_{s}^{1} s' \nrm{F_{si}(s')}_{\SH^{1}} \, \frac{\ud s'}{s'} + \nrm{\Alow_{i}}_{\SH^{1}} \leq \abs{\log s}^{1/2} \calF + \calAlow .
\end{equation}

% In fact, such a bound sufficed for many purposes so far. However, In the end, it is essential that we remove this logarithm. 
 The essential reason for having a logarithm is that we have an \emph{absolute integral} of $\nrm{F_{si}(s')}_{\SH^{1}}$ in the inequality, whereas $\calF$ only controls its \emph{square integral}. The idea then is to somehow replace this absolute integral with a square integral, using the structure of the Yang-Mills system.

We start with the equation satisfied by $A_{i}$ under the condition $A_{s} = 0$. 
\begin{equation} \label{eq:est4A:proof:1}
	\rd_{s} A_{i} = \lap A_{i} - \rd^{\ell} \rd_{i} A_{\ell} + {}^{(A_{i})} \calN',
\end{equation}
where
\begin{equation*}
	{}^{(A_{i})} \calN' = \calO(A, \rd_{x} A) + \calO(A, A, A).
\end{equation*}

Fix $t \in (-T, T)$. Let us take $\rd_{t,x}$ of \eqref{eq:est4A:proof:1}, take the bi-invariant inner product\footnote{In fact, for the purpose of this argument, it is possible to use any inner product on $\LieAlg$ for which Leibniz's rule holds, so that integration by parts works.} with $\rd_{t,x} A_{i}$ and integrate over $\bbR^{3} \times [s, 1]$, for $0 < s \leq 1$. Summing up in $i$ and performing integration by parts, we obtain the following identity.
 \begin{align*} 
\frac 1 2 \sum_{i} \int \abs{\rd_{t,x} A_{i} (s)}^{2} \,\ud x  
 = & \frac 1 2 \sum_{i} \int \abs{\rd_{t,x} \Alow_{i}}^2 \, \ud x  - \sum_{i} \int_{s}^{1} \int s' (\rd_{t,x}({}^{(A_{i})} \calN'), \rd_{t,x} A_{i}) (s') \, \ud x \, \frac{\ud s'}{s'} \\
& + \sum_{i,\ell} \int_{s}^{1} \int s' \abs{\rd_{\ell} \rd_{t,x} A_{i}(s')}^2 \, \ud x \frac{\ud s'}{s'}  - \sum_{\ell} \int_{s}^{1} s' \abs{\rd_{t,x} \rd_{\ell} A_{\ell}(s')}^{2} \, \ud x  \, \frac{\ud s'}{s'}.
\end{align*}

Take the supremum over $0 \leq s \leq 1$, and apply Cauchy-Schwarz and H\"older to deal with the second term on the right-hand side. Then taking the supremum over $t \in (-T, T)$ and applying Minkowski, we easily arrive at the following inequality.
\begin{align*}
	\sup_{0 \leq s \leq 1} \nrm{\rd_{t,x} A(s)}_{L^{\infty}_{t} L^{2}_{x}} 
	\leq & C \nrm{\rd_{t,x} \Alow}_{L^{\infty}_{t} L^{2}_{x}} + C  \bb( \int_{0}^{1} s \nrm{\rd_{x} \rd_{t,x} A(s)}_{L^{\infty}_{t} L^{2}_{x}}^{2} \, \frac{\ud s}{s} \bb)^{1/2} \\
	& + C \sup_{i} \int_{0}^{1} s \nrm{\rd_{t,x}( {}^{(A_{i})} \calN')(s)}_{L^{\infty}_{t} L^{2}_{x}} \, \frac{\ud s}{s}.
\end{align*}

Similarly, taking $\Box$ of \eqref{eq:est4A:proof:1}, multiplying by $\Box A_{i}$, integrating over $(-T, T) \times \bbR^{3} \times [s, 1]$ and etc, we can also prove 
\begin{align*}
	 \sup_{0 \leq s \leq 1} \nrm{\Box A(s)}_{L^{2}_{t,x}} 
	\leq & C  \nrm{\Box \Alow}_{ L^{2}_{t,x}} + C  \bb( \int_{0}^{1} s \nrm{\rd_{x} \Box A(s)}_{L^{2}_{t,x}}^{2} \, \frac{\ud s}{s} \bb)^{1/2} \\
	& + C \sup_{i} \int_{0}^{1} s \nrm{\Box ({}^{(A_{i})} \calN')(s)}_{L^{2}_{t,x}} \, \frac{\ud s}{s}.
\end{align*}

Combining the last two inequalities and recalling the definition of the norm $\SH^{k}$, we get
\begin{equation*} 
	 \sup_{0 \leq s \leq 1} \nrm{ A(s)}_{\SH^{1}} 
	\leq C  \nrm{ \Alow}_{\SH^{1}} + C  \bb( \int_{0}^{1} s \nrm{A(s)}_{\SH^{2}}^{2} \, \frac{\ud s}{s} \bb)^{1/2} 
	+ C  \sup_{i} \int_{0}^{1} s \nrm{ {}^{(A_{i})} \calN'(s)}_{\SH^{1}} \, \frac{\ud s}{s}.
\end{equation*}

Applying Lemma \ref{lem:fundEst4A} (with $p=q=2$) to the second term on the right-hand side, we finally arrive at the following inequality.
\begin{equation} \label{eq:est4A:proof:2}
 \sup_{0 \leq s \leq 1} \nrm{ A(s)}_{\SH^{1}} 
	\leq C  \nrm{ \Alow}_{\SH^{1}} 
	+ C (\nrm{F_{s}}_{\calL^{5/4,2}_{s} \dot{\calS}^{2}} + \nrm{ \Alow}_{\SH^{2}} )
	+ C \sup_{i} \int_{0}^{1} s \nrm{ {}^{(A_{i})} \calN'(s)}_{\SH^{1}} \, \frac{\ud s}{s}.
\end{equation}

All terms on the right-hand side except the last term can be controlled by $C (\calF + \calAlow)$. Therefore, all that is left to show is that the last term on the right-hand side of \eqref{eq:est4A:proof:2} is okay. To this end, we claim
\begin{equation*} 
	\sup_{i} \int_{0}^{1} s \nrm{ {}^{(A_{i})} \calN'(s)}_{\SH^{1}} \, \frac{\ud s}{s} \leq C_{\calF+\calAlow} \cdot (\calF+\calAlow)^{2}.
\end{equation*}

Recalling the definition of the $\SH^{1}$ norm, we must bound the contribution of $\nrm{ \rd_{t,x} ({}^{(A_{i})} \calN')(s)}_{L^{\infty}_{t} L^{2}_{x}}$ and $T^{1/2} \nrm{ \Box({}^{(A_{i})} \calN')(s)}_{L^{2}_{t,x}}$. We will only treat the latter (which is slightly more complicated), leaving the former to the reader.

Using the product rule for $\Box$, we compute the schematic form of $\Box({}^{(A_{i})} \calN')$ as follows.
\begin{equation*}
\Box {}^{(A_{i})} \calN' = \calO(\rd^{\mu} A, \rd_{x} \rd_{\mu} A) + \calO(A, \rd^{\mu} A, \rd_{\mu}A) +  \calO(\Box A, \rd_{x} A) + \calO(A, \rd_{x} \Box A)  + \calO(A, A, \Box A). 
\end{equation*}

Let us treat each type in order. Terms of the first type are the most dangerous, in the sense that there is absolutely no extra $s$-weight to spare. Using Cauchy-Schwarz and Strichartz, we have
\begin{align*}
\int_{0}^{1} & s T^{1/2} \nrm{ \calO(\rd^{\mu} A(s), \rd_{x} \rd_{\mu} A(s))}_{L^{2}_{t,x}} \, \frac{\ud s}{s} \\
%\leq & T^{1/2}\bb( \int_{0}^{1} s^{1/2} \nrm{\rd^{\mu} A(s)}_{L^{4}_{t,x}}^{2} \, \frac{\ud s}{s} \bb)^{1/2} \bb( \int_{0}^{1} s^{3/2} \nrm{\rd_{x} \rd_{\mu} A(s)}_{L^{4}_{t,x}}^{2} \, \frac{\ud s}{s} \bb)^{1/2} \\
\leq & C T^{1/2} \bb( \int_{0}^{1} s^{1/2} \nrm{A(s)}_{\SH^{3/2}}^{2} \, \frac{\ud s}{s} \bb)^{1/2} \bb( \int_{0}^{1} s^{3/2} \nrm{A(s)}_{\SH^{5/2}}^{2} \, \frac{\ud s}{s} \bb)^{1/2}.
\end{align*}

Using Lemma \ref{lem:fundEst4A}, the last line can be estimated by $C(\calF + \calAlow)^{2}$, which is acceptable.

Terms of the second type can be treated similarly using H\"older, Strichartz and Lemma \ref{lem:fundEst4A}, being easier due to the presence of extra $s$-weights. We estimate these terms as follows.
\begin{align*}
\int_{0}^{1} & s T^{1/2} \nrm{ \calO(A, \rd^{\mu} A(s), \rd_{\mu} A(s))}_{L^{2}_{t,x}} \, \frac{\ud s}{s} \\
\leq & C T^{1/2} \int_{0}^{1} s^{1/4} \bb( s^{1/4} \nrm{A(s)}_{L^{\infty}_{t,x}} \bb) \bb(s^{1/4} \nrm{\rd^{\mu} A(s)}_{L^{4}_{t,x}}\bb) \bb( s^{1/4} \nrm{\rd_{\mu} A(s)}_{L^{4}_{t,x}} \bb) \, \frac{\ud s}{s} \\
%\leq & C T^{1/2} \int_{0}^{1} s^{1/4} \bb( s^{1/4} \nrm{A(s)}_{L^{\infty}_{t,x}} \bb) \bb(s^{1/4} \nrm{A(s)}_{\SH^{3/2}}\bb) \bb( s^{1/4} \nrm{ A(s)}_{\SH^{3/2}} \bb) \, \frac{\ud s}{s} \\
\leq & C T^{1/2} (\calF+\calAlow)^{3}.
\end{align*}

The remaining terms all involve the d'Alembertian $\Box$. For these terms, using H\"older, we always put the factor with $\Box$ in $L^{2}_{t,x}$ and estimate by the $\SH^{k}$ norm, whereas the other terms are put in $L^{\infty}_{t,x}$. We will always have some extra $s$-weight, and thus it is not difficult to show that
\begin{equation*}
\int_{0}^{1}  s T^{1/2} \nrm{ \calO(\Box A(s), \rd_{x} A(s))+\calO(A(s), \rd_{x} \Box A(s))}_{L^{2}_{t,x}} \, \frac{\ud s}{s} \leq  C (\calF+\calAlow)^{2},
\end{equation*}
\begin{equation*}
\int_{0}^{1}  s T^{1/2} \nrm{ \calO(A(s), A(s), \Box A(s))}_{L^{2}_{t,x}} \, \frac{\ud s}{s} \leq  C(\calF+\calAlow)^{3}.
\end{equation*}

As desired, we have therefore proved
\begin{equation*} 
	\sup_{i} \int_{0}^{1} s T^{1/2} \nrm{\Box( {}^{(A_{i})} \calN')(s)}_{L^{2}_{t,x}} \, \frac{\ud s}{s} \leq C_{\calF, \calAlow} \cdot (\calF+\calAlow)^{2}. \qedhere
\end{equation*}
\end{proof}

\begin{proof} [Proof of Proposition \ref{prop:est4Fs0:low}]
This is an immediate consequence of Propositions \ref{prop:pEst4Fs0:low} and \ref{prop:pEst4Fs0:low:Diff}. \qedhere
\end{proof}

\begin{proof} [Proof of Proposition \ref{prop:cont4FA}]
In fact, this proposition is a triviality in view of the simple definitions of the quantities $\calF, \calAlow, \dlt \calF, \dlt \calAlow$ and the fact that $A_{\bfa}$, $A'_{\bfa}$ are regular solutions to \eqref{eq:HPYM}. \qedhere
\end{proof}

\section{Hyperbolic estimates : Proofs of Theorems \ref{thm:AlowWave} and \ref{thm:FsWave}} \label{sec:wave}
The purpose of this section is to prove Theorems \ref{thm:AlowWave} and \ref{thm:FsWave}, which are based on analyzing the wave-type equations \eqref{eq:hyperbolic4Alow} and \eqref{eq:hyperbolic4F} for $\Alow_{i}$ and $F_{si}$, respectively. Note that the system of equations for $\Alow_{i}$ is nothing but the Yang-Mills equations with source in the temporal gauge. The standard way of solving this system (see \cite{Eardley:1982fb}) is by deriving a wave equation for $F_{\mu\nu}$; due to a technical point, however, we take a slightly different route, which is explained further in \S \ref{subsec:AlowWave}. The wave equation \eqref{eq:hyperbolic4F} for $F_{si}$, on the other hand, shares many similarities with that for $A_{i}$ in the Coulomb gauge. In particular, one can recover the null structure for the most dangerous bilinear interaction $\LieBr{A^{\ell}}{\rd_{\ell} F_{si}}$, which is perhaps the most essential structural feature of the caloric-temporal gauge which makes the whole proof work.

Throughout this section, we work with regular solutions $A_{\bfa}, A'_{\bfa}$ to \eqref{eq:HPYM} on $(-T, T) \times \bbR^{3} \times [0,1]$.

\subsection{Hyperbolic estimates for $\Alow_{i}$ : Proof of Theorem \ref{thm:AlowWave}} \label{subsec:AlowWave}
%The idea is to estimate $\Alow_{i}$ on $(-T', T')$ in terms of the size of the initial data $\calI$ and the bootstrapped quantities $\calF, \calAlow$, using the equation \eqref{eq:hyperbolicYM4Flow}. The upshot is that $\calI$ controls smooth norms of the initial data, and therefore one expects to be able to derive the required {\it a priori} estimates just by classical methods. On the other hand, we have the source terms $\wlow_{\nu}$, the control of which comes from analyzing the parabolic equations satisfies by $w_{\nu}$.

\subsubsection{Equations of motion for $\Alow_{i}$} 
Recall that at $s=1$, the connection coefficients $\Alow_{\mu} = A_{\mu}(s=1)$ satisfy the hyperbolic Yang-Mills equation with source, i.e.
\begin{equation} \label{eq:hyperbolicYM4Flow}
	\covDlow^\mu \Flow_{\nu \mu} = \wlow_{\nu} \hbox{ for $\nu = 0, 1, 2, 3$}.
\end{equation}
Furthermore, we have the temporal gauge condition $\Alow_0 = 0$.

Recall that $(\curl B)_{i} := \sum_{j,k} \eps_{ijk} \rd_{j} B_{k}$, where $\eps_{ijk}$ was the Levi-Civita symbol. In the proposition below, we record the equation of motion of $\Alow_{i}$, which are obtained simply by expanding \eqref{eq:hyperbolicYM4Flow} in terms of $\Alow_{i}$.

\begin{proposition} [Equations for $\Alow_{i}$] \label{prop:eqn4Alow}
The Yang-Mills equation with source \eqref{eq:hyperbolicYM4Flow} is equivalent to the following system of equations.
\begin{align}
	\rd_{0} (\rd^{\ell} \Alow_{\ell}) =& - \LieBr{\Alow^\ell}{\rd_{0} \Alow_\ell} + \wlow_0, \label{eq:transport4Alow} \\
\Box \Alow_i - \rd_i (\rd^{\ell} \Alow_{\ell}) =& - 2 \LieBr{\Alow^\ell}{\rd_\ell \Alow_i} + \LieBr{\Alow_i}{\rd^\ell \Alow_\ell} + \LieBr{\Alow^\ell}{\rd_i \Alow_\ell} - \LieBr{\Alow^\ell}{\LieBr{\Alow_\ell}{\Alow_i}} - \wlow_i. \label{eq:eqn4Alow}
\end{align}

Taking the curl (i.e., $\curl \cdot$) of \eqref{eq:eqn4Alow}, we obtain the following wave equation for $\curl \Alow$.
\begin{equation} \label{eq:wave4curlAlow}
\begin{aligned}
	\Box (\curl \Alow)_i = & - \curl (2 \LieBr{\Alow^\ell}{\rd_\ell \Alow_i} + \LieBr{\Alow_i}{\rd^\ell \Alow_\ell} + \LieBr{\Alow^\ell}{\rd_i \Alow_\ell}) \\
					& - \curl (\LieBr{\Alow^\ell}{\LieBr{\Alow_\ell}{\Alow_i}}) - (\curl \wlow)_i.
\end{aligned}\end{equation}
\end{proposition}

\begin{remark} 
The usual procedure of solving \eqref{eq:hyperbolicYM4Flow} in temporal gauge consists of first deriving the hyperbolic equation for $\Flow_{\nu \mu}$, using the Bianchi identity and \eqref{eq:hyperbolicYM4Flow}. Then one couples these equations with the transport equation
\begin{equation*}
	\Flow_{0i} = \rd_{0} \Alow_i,
\end{equation*}
(which follows just from the definition of $\Flow_{0i}$ and the temporal gauge condition $\Alow_{0} = 0$) and solves the system altogether. This is indeed the approach of Eardley-Moncrief \cite{Eardley:1982fb} and Klainerman-Machedon \cite{Klainerman:1995hz}. A drawback to this approach, however, is that it requires taking a $t$-derivative when deriving hyperbolic equations for $\Flow_{0i}$. In particular, one has to estimate $\rd_{0} \wlow_{i}$, which complicates matters in our setting.

The equations that we stated in Proposition \ref{prop:eqn4Alow} is the basis for a slightly different approach, which avoids taking $\rd_{0}$ at the expense of using a little bit of Hodge theory. We remark that such an approach had been taken by Tao \cite{Tao:2000vba}, but with greater complexity than here as \cite{Tao:2000vba} was concerned with lower regularity (but small data) solutions to \eqref{eq:hyperbolicYM}.
\end{remark}

\subsubsection{Proof of Theorem \ref{thm:AlowWave}}
In this part, we give a proof of Theorem \ref{thm:AlowWave}

\begin{proof} [Proof of Theorem \ref{thm:AlowWave}]
In the proof, we will work on the time interval $(-T, T)$, where $0 < T \leq 1$. We will give a rather detailed proof of \eqref{eq:AlowWave:1}. The difference analogue \eqref{eq:AlowWave:2} can be proved in an analogous manner, whose details we leave to the reader.

Let us begin with a few product estimates.
\begin{align} 
	\nrm{\calO(\Alow, \rd_{0} \Alow)}_{L^{\infty}_{t} \dot{H}^{m}_{x}} &\leq C \calAlow^{2}, &\hbox{ for $0 \leq m \leq 29$,} \label{eq:AlowWave:pf:1} \\
	\nrm{\calO(\Alow, \rd_{x} \Alow)}_{L^{\infty}_{t} \dot{H}^{m}_{x}} &\leq C \calAlow^{2}, &\hbox{ for $0 \leq m \leq 30$,} \label{eq:AlowWave:pf:2}\\
	\nrm{\calO(\Alow, \Alow, \Alow)}_{L^{\infty}_{t} \dot{H}^{m}_{x}} &\leq C \calAlow^{3}, &\hbox{ for $0 \leq m \leq 31$.} \label{eq:AlowWave:pf:3}
\end{align}

Each of these can be proved by Leibniz's rule, H\"older and Sobolev, as well as the fact that $\nrm{\rd_{x} \Alow}_{L^{\infty}_{t} H^{30}_{x}} + \nrm{\rd_{0} \Alow}_{L^{\infty}_{t} H^{29}_{x}} \leq \calAlow$. Using the same techniques, we can also prove the following weaker version of \eqref{eq:AlowWave:pf:1} in the case $m=30$:
\begin{equation} \label{eq:AlowWave:pf:4}
	\nrm{\calO(\Alow, \rd_{0} \Alow)}_{L^{\infty}_{t} \dot{H}^{30}_{x}} \leq C \calAlow^{2} + C \nrm{\Alow}_{L^{\infty}_{t,x}} \nrm{\rd_{x}^{(30)} \rd_{0} \Alow}_{L^{\infty}_{t} L^{2}_{x}}.
\end{equation}

Next, observe that $\nrm{\wlow_{0}}_{L^{\infty}_{t} \dot{H}^{m}_{x}} \leq \sup_{t \in (-T, T)} \nrm{F_{s0}(t)}_{\calL^{1, \infty}_{s} \dot{\calH}^{m}_{x}}$, where the latter can be controlled by \eqref{eq:pEst4Fs0:high:3} for $0 \leq m \leq 30$. Combining this with \eqref{eq:transport4Alow}, \eqref{eq:AlowWave:pf:1}, we obtain the following estimate for $0 \leq m \leq 29$:
\begin{equation} \label{eq:AlowWave:pf:5}
	\nrm{\rd_{0}(\rd^{\ell} \Alow_{\ell})}_{L^{\infty}_{t} \dot{H}^{m}_{x}} 
	\leq C_{\calF, \calAlow} \cdot \calE + C_{\calF, \calAlow} \cdot (\calF+\calAlow)^{2}.
\end{equation}

In the case $m=30$, replacing the use of \eqref{eq:AlowWave:pf:1} by \eqref{eq:AlowWave:pf:4}, we have
\begin{equation*}
	\nrm{\rd_{0}(\rd^{\ell} \Alow_{\ell})}_{L^{\infty}_{t} \dot{H}^{30}_{x}} 
	\leq C \calAlow \nrm{\rd_{x}^{(30)} \rd_{0} \Alow}_{L^{\infty}_{t} L^{2}_{x}} 
	+ C_{\calF, \calAlow} \cdot \calE + C_{\calF, \calAlow} \cdot (\calF+\calAlow)^{2}.
\end{equation*}

Recall the simple div-curl identity $\sum_{i,j} \nrm{\rd_{i} B_{j}}^{2} = \frac{1}{2}\nrm{\curl B}_{L^{2}_{x}}^{2} + \nrm{\rd^{\ell} B_{\ell}}^{2}$ with $B = \Alow(t)$. Using furthermore \eqref{eq:AlowWave:pf:5} with $m=29$ and the fact that $\calAlow$ controls $\nrm{\rd_{x}^{(29)} \rd_{0} (\curl \Alow)}_{L^{\infty}_{t} L^{2}_{x}}$, we obtain the following useful control on $\nrm{\rd_{x}^{(30)} \rd_{0} \Alow}_{L^{\infty}_{t} L^{2}_{x}}$:
\begin{equation} \label{eq:AlowWave:pf:6}
%\left\{
%\begin{aligned} 
	\nrm{\rd_{x}^{(30)} \rd_{0} \Alow}_{L^{\infty}_{t} L^{2}_{x}} 
%	\leq & C \nrm{\rd_{x}^{(29)} \rd_{0} (\curl \Alow)}_{L^{\infty}_{t} L^{2}_{x}} + C \nrm{\rd_{x}^{(29)} \rd_{0} \rd^{\ell} \Alow_{\ell}}_{L^{\infty}_{t} L^{2}_{x}} \\
	\leq C \calAlow + C_{\calF, \calAlow} \cdot \calE + C_{\calF, \calAlow} \cdot (\calF+\calAlow)^{2} .
%\end{aligned}
%\right.
\end{equation}

Therefore, \eqref{eq:AlowWave:pf:5} holds in the case $m=30$ as well, i.e.
\begin{equation*}
	\nrm{\rd_{0}(\rd^{\ell} \Alow_{\ell})}_{L^{\infty}_{t} \dot{H}^{30}_{x}} 
	\leq C_{\calF, \calAlow} \cdot \calE + C_{\calF, \calAlow} \cdot (\calF+\calAlow)^{2}.
\end{equation*}

Integrating \eqref{eq:AlowWave:pf:5} with respect to $t$ from $t=0$, we obtain for $0 \leq m \leq 30$
\begin{equation} \label{eq:AlowWave:pf:7}
	\nrm{\rd^{\ell} \Alow_{\ell}}_{L^{\infty}_{t} \dot{H}^{m}_{x}} \leq \calI + T\bb( C_{\calF, \calAlow} \cdot \calE + C_{\calF, \calAlow} \cdot (\calF+\calAlow)^{2} \bb).
\end{equation}

Next, observe that $\nrm{\wlow_{i}}_{L^{\infty}_{t} \dot{H}^{m}_{x}} \leq \sup_{t \in (-T, T)}\nrm{w_{i}(t)}_{\calL^{1,\infty}_{s} \dot{\calH}^{m}_{x}}$. Combining this observation with \eqref{eq:pEst4wi:1} and \eqref{eq:pEst4wi:2} from Proposition \ref{prop:pEst4wi}, as well as \eqref{eq:AlowWave:pf:6} to control $\nrm{\rd_{0} \Alow}_{\dot{H}^{30}_{x}}$, we have the following estimates for $0 \leq m \leq 30$:
\begin{equation} \label{eq:AlowWave:pf:8}
	\nrm{\wlow_{i}}_{L^{\infty}_{t} \dot{H}^{m}_{x}} \leq C_{\calE, \calF, \calAlow} \cdot (\calE + \calF+\calAlow)^{2}.
\end{equation}

We are now ready to finish the proof. Let $i=1,2,3$ and $1 \leq m \leq 30$. By the energy inequality and H\"older, we have
\begin{equation*}
	\nrm{\Alow_{i}}_{\dot{S}^{m}} \leq C \nrm{\rd_{t,x} \Alow(t=0)}_{\dot{H}^{m-1}_{x}} + C T \nrm{\Box \Alow_{i}}_{L^{\infty}_{t} \dot{H}^{m-1}_{x}}.
\end{equation*}
The first term is controlled by $C \calI$. To control the second term, apply \eqref{eq:eqn4Alow}, \eqref{eq:AlowWave:pf:2}, \eqref{eq:AlowWave:pf:3}, \eqref{eq:AlowWave:pf:8}. Furthermore, use \eqref{eq:AlowWave:pf:7} to control the contribution of $\rd_{i} \rd^{\ell} \Alow_{\ell}$. As a result, we obtain
\begin{equation} \label{eq:AlowWave:pf:9}
	\nrm{\Alow_{i}}_{\dot{S}^{m}}
	\leq C \calI + T \bb( C_{\calF, \calAlow} \cdot \calE + C_{\calE, \calF, \calAlow} \cdot (\calE + \calF + \calAlow)^{2} \bb),
\end{equation}
for $i=1,2,3$ and $1 \leq m \leq 30$.

Similarly, by the energy inequality and H\"older, we have
\begin{equation*}
	\nrm{(\rd \times \Alow)_{i}}_{\dot{S}^{30}} \leq C \nrm{\rd_{t,x} (\curl \Alow)(t=0)}_{\dot{H}^{29}_{x}} + C T \nrm{\Box (\curl \Alow)_{i}}_{L^{\infty}_{t} \dot{H}^{30}_{x}}.
\end{equation*}
The first term is again controlled by $C \calI$. To control the second term, we apply \eqref{eq:wave4curlAlow}, \eqref{eq:AlowWave:pf:2}, \eqref{eq:AlowWave:pf:3}, \eqref{eq:AlowWave:pf:8}; note that this time we do not need an estimate for $\rd^{\ell} \Alow_{\ell}$. We conclude
\begin{equation} \label{eq:AlowWave:pf:10}
	\nrm{(\curl \Alow)_{i}}_{\dot{S}^{30}} \leq C \calI + T \, C_{\calE, \calF, \calAlow} \cdot (\calE + \calF + \calAlow)^{2}. 
\end{equation}

Finally, using the div-curl identity, \eqref{eq:AlowWave:pf:7} and \eqref{eq:AlowWave:pf:10}, we have 
\begin{equation*}
\nrm{\Alow}_{\dot{H}^{31}_{x}} \leq C \calI + T\bb( C_{\calF, \calAlow} \cdot \calE + C_{\calF, \calAlow} \cdot (\calF+\calAlow)^{2} \bb).
\end{equation*}

This concludes the proof. \qedhere
\end{proof}

\subsection{Hyperbolic estimates for $F_{si}$ : Proof of Theorem \ref{thm:FsWave}} \label{subsec:FsWave}
Let us recall the hyperbolic equation \eqref{eq:hyperbolic4F} satisfied by $F_{si}$:
\begin{equation*}
 \covD^\mu \covD_\mu F_{s i} = 2 \LieBr{\tensor{F}{_s^\mu}}{F_{i \mu}} - \covD^\ell \covD_\ell w_i + \covD_i \covD^\ell w_\ell - {}^{(w_{i})} \calN.
\end{equation*}
Note that we have rewritten $2 \LieBr{\tensor{F}{_i^\ell}}{w_\ell}  + 2 \LieBr{F^{\mu \ell}}{\covD_\mu F_{i \ell} + \covD_\ell F_{i \mu}} = {}^{(w_{i})} \calN$ for convenience.

\subsubsection{Semi-linear wave equation for $F_{si}$}
Let us begin by rewriting the wave equation for $F_{si}$ in a form more suitable for our analysis. Writing out the covariant derivatives in \eqref{eq:hyperbolic4F}, we obtain the following semi-linear wave equation for $F_{si}$.
%\begin{align}
%\Box F_{si} = \, & 2 \LieBr{A_0}{\rd_t F_{si}}  - 2\LieBr{(\Acf)^\ell}{\rd_\ell F_{si}} - 2 \LieBr{(\Adf)^\ell}{\rd_\ell F_{si}} \label{eq:hyperbolic4Fsi:1}\\
%		     &+ \LieBr{\rd_t A_0}{F_{si}} - 2\LieBr{F_{i 0}}{F_{s0}} - \LieBr{\rd^\ell A_\ell}{F_{si}} + 2\LieBr{\tensor{F}{_i^\ell}}{F_{s \ell}} \label{eq:hyperbolic4Fsi:2}\\
%		     & + \LieBr{A_{0}}{\LieBr{A_{0}}{F_{si}}} - \LieBr{A^{\ell}}{\LieBr{A_{\ell}}{F_{si}}} \label{eq:hyperbolic4Fsi:3}\\
%		     &- 2 \LieBr{\tensor{F}{_0^\ell}}{\rd_0 F_{i \ell} + \rd_\ell F_{i 0}} - 2\LieBr{\tensor{F}{_{0}^{\ell}}}{ \LieBr{A_{0}}{F_{i\ell}} + \LieBr{A_{\ell}} {F_{i0}} }\label{eq:hyperbolic4Fsi:3} \\
%		     &+ \rd^\ell \rd_\ell w_i - \rd_i \rd^\ell w_\ell \label{eq:hyperbolic4Fsi:4}\\
%		     &+ 2 \LieBr{A^\ell}{\rd_\ell w_i} - 2\LieBr{A_i}{\rd^\ell w_\ell} + \LieBr{\rd^\ell A_\ell}{w_i} + \LieBr{\rd_i A^\ell}{w_\ell} + 2 \LieBr{\tensor{F}{_i^\ell}}{w_\ell} \label{eq:hyperbolic4Fsi:5}\\
%		     &+ \LieBr{A^{\ell}}{\LieBr{A_{\ell}}{w_{i}}} - \LieBr{A_{i}}{\LieBr{A^{\ell}}{w_{\ell}}}. \label{eq:hyperbolic4Fsi:6}
%\end{align}
\begin{equation*}
	\Box F_{si} = {}^{(F_{si})} \calM_{\quadratic} + {}^{(F_{si})} \calM_{\cubic} + {}^{(F_{si})} \calM_{w},
\end{equation*}
where
\begin{align*}
	{}^{(F_{si})} \calM_{\quadratic} :=  &   - 2\LieBr{A^\ell}{\rd_\ell F_{si}}  + 2 \LieBr{A_0}{\rd_{0} F_{si}}\\
		     &+ \LieBr{\rd_{0} A_0}{F_{si}}  - \LieBr{\rd^\ell A_\ell}{F_{si}} - 2\LieBr{\tensor{F}{_i^\ell}}{F_{s \ell}} + 2\LieBr{F_{i 0}}{F_{s0}}, \\
	{}^{(F_{si})} \calM_{\cubic} := & \LieBr{A_{0}}{\LieBr{A_{0}}{F_{si}}} - \LieBr{A^{\ell}}{\LieBr{A_{\ell}}{F_{si}}} \\
	{}^{(F_{si})} \calM_{w} := & - \covD^{\ell} \covD_{\ell} w_{i} + \covD_{i} \covD^{\ell} w_{\ell} - {}^{(w_{i})} \calN.
\end{align*}

The semi-linear equation for the difference $\dlt F_{si} := F_{si} - F'_{si}$ is then given by
\begin{equation*}
	\Box \dlt F_{si} = {}^{(\dlt F_{si})} \calM_{\quadratic} + {}^{(\dlt F_{si})} \calM_{\cubic} + {}^{(\dlt F_{si})} \calM_{w},
\end{equation*}
where ${}^{(\dlt F_{si})} \calM_{\quadratic} := {}^{(F_{si})} \calM_{\quadratic} - {}^{(F'_{si})} \calM_{\quadratic}$, ${}^{(\dlt F_{si})} \calM_{\cubic} := {}^{(F_{si})} \calM_{\cubic} - {}^{(F'_{si})} \calM_{\cubic}$ and ${}^{(\dlt F_{si})} \calM_{w} := {}^{(F_{si})} \calM_{w} - {}^{(F'_{si})} \calM_{w}$.  

\subsubsection{Estimates for quadratic terms}
We begin the proof of Theorem \ref{thm:FsWave} by estimating the contribution of quadratic terms.
\begin{lemma} [Estimates for quadratic terms] \label{lem:FsWave:quadratic}
Assume $0 < T \leq 1$. For $1 \leq m \leq 10$ and $p=2, \infty$, the following estimates hold.
\begin{equation} \label{eq:FsWave:quadratic}
	\sup_{i} \nrm{{}^{(F_{si})} \calM_{\quadratic}}_{\calL^{2,p}_{s} \calL^{2}_{t} \dot{\calH}^{m-1}_{x}(0,1]} 
	\leq C_{\calE, \calF, \calAlow} \cdot (\calE + \calF + \calAlow)^{2},
\end{equation}
\begin{equation} \label{eq:FsWave:quadratic:Diff}
	\sup_{i} \nrm{{}^{(\dlt F_{si})} \calM_{\quadratic}}_{\calL^{2,p}_{s} \calL^{2}_{t} \dot{\calH}^{m-1}_{x}(0,1]} 
	\leq C_{\calE, \calF, \calAlow} \cdot (\calE + \calF + \calAlow)(\dlt \calE + \dlt \calF + \dlt \calAlow).
\end{equation}
\end{lemma}
\begin{proof} 
We will give a rather detailed proof of \eqref{eq:FsWave:quadratic}. The other estimate \eqref{eq:FsWave:quadratic:Diff} may be proved by first using Leibniz's rule for $\dlt$ to compute ${}^{(\dlt F_{si})} \calM_{\quadratic}$, and then proceeding in an analogous fashion. We will omit the proof of the latter.

Let $1 \leq m \leq 10$ and $p = 2$ or $\infty$. We will work on the whole $s$-interval $(0, 1]$. Let us begin with an observation that in order to prove \eqref{eq:FsWave:quadratic}, it suffices to prove that each of the following can be bounded by $C_{\calE, \calF, \calAlow} \cdot (\calE + \calF + \calAlow)^{2}$:
\begin{equation*}
\left\{
\begin{aligned}
&\nrm{s^{-1/2} \LieBr{(A^{\cf})^{\ell}}{\nb_{\ell} \nb_{x}^{(m-1)} F_{si}}}_{\calL^{2,p}_{s} \calL^{2}_{t,x}}, \quad
\nrm{s^{-1/2} \LieBr{(A^{\df})^{\ell}}{\nb_{\ell} \nb_{x}^{(m-1)} F_{si}}}_{\calL^{2,p}_{s} \calL^{2}_{t,x}}, \\
&\sum_{j=1}^{m-1} \nrm{s^{-1/2} \LieBr{\nb_{x}^{(j)} A^{\ell}}{\nb_{\ell} \nb_{x}^{(m-1-j)} F_{si}}}_{\calL^{2,p}_{s} \calL^{2}_{t,x}}, \\
& \nrm{s^{-1/2} \LieBr{A_0}{\nb_{0} F_{si}}}_{\calL^{2,p}_{s} \calL^{2}_{t} \dot{\calH}^{m-1}_{x}}, \quad
\nrm{s^{-1/2} \LieBr{\nb_{0} A_0}{F_{si}}}_{\calL^{2,p}_{s} \calL^{2}_{t} \dot{\calH}^{m-1}_{x}}, \\
& \nrm{\LieBr{F_{i 0}}{F_{s0}}}_{\calL^{2,p}_{s} \calL^{2}_{t} \dot{\calH}^{m-1}_{x}}, \quad
\nrm{\LieBr{\tensor{F}{_i^\ell}}{F_{s \ell}}}_{\calL^{2,p}_{s} \calL^{2}_{t} \dot{\calH}^{m-1}_{x}}.
\end{aligned}
\right.
\end{equation*}

Here, $A^{\cf}$ and $A^{\df}$, called the \emph{curl-free} and the \emph{divergence-free} parts of $A$, respectively, constitute the \emph{Hodge decomposition} of $A$, i.e., $A_{i} = A^{\cf}_{i} + A^{\df}_{i}$. They are defined by the formulae
\begin{equation*}
	A^{\cf} := - (-\lap)^{-1} \rd_{i} \rd^{\ell} A_{\ell}, \quad A^{\df} := (-\lap)^{-1} (\rd \times (\rd \times A))_{i}.
\end{equation*}

Let us treat each of them in order.

\pfstep{- Case 1 : Proof of $\nrm{s^{-1/2} \LieBr{(A^{\cf})^{\ell}}{\nb_{\ell} \nb_{x}^{(m-1)} F_{si}}}_{\calL^{2,p}_{s} \calL^{2}_{t,x}} \leq C_{\calE, \calF, \calAlow} \cdot (\calE + \calF + \calAlow)^{2}$} 

We claim that the following estimate for $A^{\cf}$ holds.
\begin{equation} \label{eq:FsWave:quadratic:est4Acf}
	\nrm{A^{\cf} (s)}_{\calL^{1/4, \infty}_{s} \calL^{2}_{t} \calL^{\infty}_{x}} \leq C \calAlow + C_{\calF, \calAlow} \cdot (\calF + \calAlow)^{2}.
\end{equation}	

Note, on the other hand, that $\nrm{\nb_{\ell} \nb_{x}^{(m-1)} F_{si}}_{\calL^{5/4,p}_{s} \calL^{\infty}_{t} \calL^{2}_{x}} \leq \calF$ for $1 \leq m \leq 10$. Assuming the claim, the desired estimate then follows immediately by H\"older.

The key to our proof of \eqref{eq:FsWave:quadratic:est4Acf} is the covariant Coulomb condition satisfied by $F_{si}$
\begin{equation*}
	\covD^{\ell} F_{s\ell} = 0,
\end{equation*}
which was proved in Appendix \ref{sec:HPYM}. Writing out the covariant derivative $\covD^{\ell} = \rd^{\ell} + A^{\ell}$ and using the relation $F_{s\ell} = \rd_{s} A_{\ell}$, we arrive at the following \emph{improved transport equation} for $\rd^{\ell}A_{\ell}$.
\begin{equation} \label{eq:FsWave:quadratic:est4Acf:0}
	\rd_{s} \big( \rd^{\ell} A_{\ell} (s) \big) = - \LieBr{A^{\ell}(s)}{F_{s\ell}(s)}.
\end{equation}

Observe furthermore that $\nrm{A^{\cf}_{\ell}(s)}_{\calL^{1/4, \infty}_{s} \calL^{2}_{t} \calL^{\infty}_{x}} = \sup_{0 < s \leq 1} \nrm{A^{\cf}(s)}_{L^{2}_{t} L^{\infty}_{x}}$. Our goal, therefore, is to estimate the latter by using \eqref{eq:FsWave:quadratic:est4Acf:0}.
		
Using the fundamental theorem of calculus and Minkowski, we obtain, for $1 \leq r \leq \infty$, the inequality
\begin{equation}\label{eq:FsWave:quadratic:est4Acf:1}
	\sup_{0 < s \leq 1} \nrm{\rd^{\ell} A_{\ell}(s)}_{L^{2}_{t} L^{r}_{x}} \leq \nrm{\rd^{\ell} \Alow_{\ell}}_{L^{2}_{t} L^{r}_{x}} + \int_{0}^{1} \nrm {\LieBr{A^{\ell}(s)}{F_{s\ell}(s)}}_{L^{2}_{t} L^{r}_{x}} \, \ud s.
\end{equation}

Let us recall that $(A^{\cf})_{i} = (- \lap)^{-1} \rd_{i} \rd^{\ell} A_{\ell}$ by Hodge theory. It then follows that $\rd_{i} (A^{\cf})_{j} = R_{i} R_{j} (\rd^{\ell} A_{\ell})$, where $R_{i}, R_{j}$ are Riesz transforms. By elementary harmonic analysis \cite{MR0290095}, for $1 < r < \infty$, we have the inequality
\begin{equation*}
	\nrm{\rd_{x} A^{\cf}}_{L^{2}_{t} L^{r}_{x}} \leq C_{r} \nrm{\rd^{\ell} A_{\ell}}_{L^{2}_{t} L^{r}_{x}}.
\end{equation*}
		
On the other hand, using Sobolev and Gagliardo-Nirenberg, we have
\begin{equation*}
	\nrm{A^{\cf}}_{L^{2}_{t} L^{\infty}_{x}} 
	\leq C \nrm{\rd_{x} A^{\cf}}_{L^{2}_{t,x}}^{1/3} \nrm{\rd_{x} A^{\cf}}_{L^{2}_{t} L^{4}_{x}}^{2/3}.
\end{equation*}

As a result of these two inequalities, it suffices to bound the $L^{2}_{t,x}$ and $L^{2}_{t} L^{4}_{x}$ norms of $\rd^{\ell} A_{\ell}(s)$ using \eqref{eq:FsWave:quadratic:est4Acf:1}. For the first term on the right-hand side of \eqref{eq:FsWave:quadratic:est4Acf:1}, we obviously have 
\begin{equation*}
	\nrm{\rd^{\ell} \Alow_{\ell}}_{L^{2}_{t,x}} + \nrm{\rd^{\ell} \Alow_{\ell}}_{L^{2}_{t} L^{4}_{x}}
		\leq C T^{1/2} (\nrm{\rd^{\ell} \Alow_{\ell}}_{L^{\infty}_{t} L^{2}_{x}} + \nrm{\rd^{\ell} \Alow_{\ell}}_{L^{\infty}_{t} L^{4}_{x}})
		\leq C \calAlow.
\end{equation*}
by H\"older in time. Next, note that the second term on the right-hand side of \eqref{eq:FsWave:quadratic:est4Acf:1} is equal to $\nrm{\LieBr{A^{\ell}}{F_{s\ell}}}_{\calL^{\ell_{r},1}_{s} \calL^{2}_{t} \calL^{r}_{x}}$, where $\ell_{r} = \frac{5}{4} + \frac{3}{2r}$. In the case $r=2$, we estimate this, using Lemma \ref{lem:fundEst4A} and Proposition \ref{prop:lowEst4Fsi}, as follows.
\begin{equation*}
	\nrm{\LieBr{A^{\ell}}{F_{s\ell}}}_{\calL^{2,1}_{s} \calL^{2}_{t,x}}
	\leq C T^{1/2} \nrm{s^{1/4}}_{\calL^{2}_{s}} \nrm{A}_{\calL^{1/4, \infty}_{s} \calL^{\infty}_{t,x}} \nrm{F_{s}}_{\calL^{5/4,2}_{s} \calL^{\infty}_{t} \calL^{2}_{x}}
	\leq C_{\calF, \calAlow} \cdot (\calF + \calAlow)^{2}.
\end{equation*}

In the other case $r=4$, we proceed similarly, again using Lemma \ref{lem:fundEst4A} and Proposition \ref{prop:lowEst4Fsi}.
\begin{equation*}
	\nrm{\LieBr{A^{\ell}}{F_{s\ell}}}_{\calL^{13/8,1}_{s} \calL^{2}_{t} \calL^{4}_{x}}
	\leq C \nrm{s^{1/8}}_{\calL^{2}_{s}} \nrm{A}_{\calL^{1/4, \infty}_{s} \calL^{4}_{t} \calL^{\infty}_{x}} \nrm{F_{s}}_{\calL^{5/4,2}_{s} \calL^{4}_{t,x}}
	\leq C_{\calF, \calAlow} \cdot (\calF + \calAlow)^{2}. 
\end{equation*}

Combining these estimates, we obtain \eqref{eq:FsWave:quadratic:est4Acf}.
	
\pfstep{- Case 2 : Proof of $\nrm{s^{-1/2} \LieBr{(A^{\df})^{\ell}}{\nb_{\ell} \nb_{x}^{(m-1)} F_{si}}}_{\calL^{2,p}_{s} \calL^{2}_{t,x}} \leq C_{\calE, \calF, \calAlow} \cdot (\calE + \calF + \calAlow)^{2}$} 

	In this case, we \emph{cannot} estimate $A^{\df}$ in $L^{2}_{t} L^{\infty}_{x}$. Here, we need to look more closely into the exact form of the nonlinearity, and recover a \emph{null form}, \`a la Klainerman \cite{Klainerman:tc}, Christodoulou \cite{MR820070} and Klainerman-Machedon \cite{Klainerman:ei}. We remark that this is the only place where we utilize the null form estimate.
	
	For $B = B_{i}$ $(i=1,2,3)$, $\phi$ smooth and $B_{i}, \phi \in \dot{S}^{1}$, we claim that the following estimate holds for $0 < s \leq 1$ :
	\begin{equation} \label{eq:FsWave:quadratic:Adf:0}
		\nrm{\LieBr{(B^{\df})^{\ell}}{\rd_{\ell} \phi }}_{L^{2}_{t,x}} \leq C(\sup_{k} \nrm{B_{k}}_{\SH^{1}}) \nrm{\phi}_{\SH^{1}}.
	\end{equation}

	Assuming the claim, by the Correspondence Principle, we then obtain the estimate
	\begin{equation*}
		\nrm{s^{-1/2} \LieBr{(\calT^{\df})^{\ell}}{\nb_{\ell} \psi }}_{\calL^{2,p}_{s} \calL^{2}_{t,x}} \leq C(\sup_{k} \nrm{\calT_{k}}_{\calL^{1/4,\infty}_{s} \dot{\calS}^{1}}) \nrm{\psi}_{\calL^{5/4,p}_{s} \dot{\calS}^{1}},
	\end{equation*}
	for smooth $\calT = \calT_{i}(s)$ $(i=1,2,3)$ $\psi$ such that the right-hand side is finite. Let us take $\calT = A$, $\psi = \nb_{x}^{(m-1)} F_{si}$. By Proposition \ref{prop:est4ai}, we have $\nrm{A}_{\calL^{1/4,\infty}_{s} \dot{\calS}^{1}} \leq C_{\calF, \calAlow} \cdot (\calF + \calAlow)$, whereas by definition $\nrm{\nb_{x}^{(m-1)} F_{si}}_{\calL^{5/4,p}_{s} \dot{\calS}^{1}} \leq C \calF$ for $1 \leq m \leq 10$. The desired estimate therefore follows.

	Now, it is only left to prove \eqref{eq:FsWave:quadratic:Adf:0}. The procedure that we are about to describe is standard, due to Klainerman-Machedon \cite{Klainerman:1994jb}, \cite{Klainerman:1995hz}. We reproduce the argument here for the sake of completeness. 
	
	Let us first assume that $B_{i}$ is Schwartz in $x$ for every $t, s$. Then simple Hodge theory tells us that $B^{\df}_{i} = (\curl V)_{i}$, where
	\begin{equation*}
		V_{i}(x) := (-\lap)^{-1} (\curl B)_{i}(x) = \frac{1}{4 \pi} \int \bb( B(y) \times \frac{(x-y)}{\abs{x-y}^{3}}\bb)_{i} \ud y,
	\end{equation*}
	where we suppressed the variables $t, s$. Substituting $(B^{\df})^{\ell} = (\curl V)^{\ell}$ on the left-hand side of \eqref{eq:FsWave:quadratic:Adf:0}, we have
\begin{align*}
	\nrm{\sum_{j,k,\ell} \eps_{\ell j k} \LieBr{\rd_{j} V_{k}(s)}{\rd_{\ell} \psi(s)} }_{L^{2}_{t,x}}
	\leq & \frac 1 2 \sum_{j, k, \ell} \nrm{Q_{j\ell} (V_{k}(s), \psi)(s)}_{L^{2}_{t,x}} \\
	\leq & C ( \sup_{k} \nrm{V_{k}(s)}_{\SH^{2}} ) \nrm{\psi(s)}_{\SH^{1}} ,
\end{align*}
	where we remind the reader that $Q_{ij} (\phi, \psi) = \rd_{i}\phi \rd_{j} \psi - \rd_{j} \phi  \rd_{i} \psi$, and on the last line we used \eqref{eq:prelim:est4SH:nullform} of Proposition \ref{prop:prelim:est4SH} (null form estimate). Since $\rd_{j} V_{k} = (-\lap)^{-1} \rd_{j} (\curl B)_{i}$, and $\nrm{\cdot}_{\SH^{1}}$ is an $L^{2}_{x}$-type norm, we see that
	\begin{equation*}
		 \sup_{k} \nrm{V_{k}(s)}_{\SH^{2}} = \sup_{j,k} \nrm{\rd_{j} V_{k}(s)}_{\SH^{1}} \leq C \sup_{k} \nrm{B_{k}(s)}_{\SH^{1}},
	\end{equation*}
	from which \eqref{eq:FsWave:quadratic:Adf:0} follows, under the additional assumption that $B_{i}$ are Schwartz in $x$. Then, using the quantitative estimate \eqref{eq:FsWave:quadratic:Adf:0}, it is not difficult to drop the Schwartz assumption by approximation.  

\pfstep{- Case 3 : Proof of $\sum_{j=1}^{m-1} \nrm{s^{-1/2} \LieBr{\nb_{x}^{(j)} A^{\ell}}{\nb_{\ell} \nb_{x}^{(m-1-j)} F_{si}}}_{\calL^{2,p}_{s} \calL^{2}_{t,x}} \leq C_{\calE, \calF, \calAlow} \cdot (\calE + \calF + \calAlow)^{2}$} 

By the H\"older inequality $L^{4}_{t,x} \cdot L^{4}_{t,x} \subset L^{2}_{t,x}$, the Correspondence Principle and H\"older for $\calL^{\ell,p}_{s}$ (Lemma \ref{lem:absP:Holder4Ls}), we immediately obtain the estimate
\begin{align*}
\sum_{j=1}^{m-1} \nrm{s^{-1/2} \LieBr{\nb_{x}^{(j)} A^{\ell}}{\nb_{\ell} \nb_{x}^{(m-1-j)} F_{si}}}_{\calL^{2,p}_{s} \calL^{2}_{t,x}}
\leq C \sum_{j=1}^{m-1} \nrm{A}_{\calL^{1/4,\infty}_{s} \calL^{4}_{t} \dot{\calW}^{j,4}_{x}} \nrm{F_{si}}_{\calL^{5/4,p}_{s} \calL^{4}_{t} \dot{\calW}^{m-j,4}_{x}}.
\end{align*}

Let us apply Lemma \ref{lem:fundEst4A} to $\nrm{A}_{\calL^{1/4,\infty}_{s} \calL^{4}_{t} \dot{\calW}^{j,4}_{x}}$; as $1 \leq j \leq m-1 \leq 9$, this can be estimated by $C (\calF + \calAlow)$. On the other hand, as $1 \leq m-j \leq m-1 \leq 9$, $\nrm{F_{si}}_{\calL^{5/4,p}_{s} \calL^{4}_{t} \dot{\calW}^{m-j,4}_{x}}$ can be controlled by $C\calF$ via Strichartz. The desired estimate follows.

\pfstep{- Case 4 : Proof of $ \nrm{s^{-1/2} \LieBr{A_0}{\nb_{0} F_{si}}}_{\calL^{2,p}_{s} \calL^{2}_{t} \dot{\calH}^{m-1}_{x}} \leq C_{\calE, \calF, \calAlow} \cdot (\calE + \calF + \calAlow)^{2}$} 

By Leibniz's rule, the H\"older inequality $L^{2}_{t} L^{\infty}_{x} \cdot L^{\infty}_{t} L^{2}_{x} \subset L^{2}_{t,x}$, the Correspondence Principle and H\"older for $\calL^{\ell,p}_{s}$ (Lemma \ref{lem:absP:Holder4Ls}), we have
\begin{equation*}
\nrm{s^{-1/2} \LieBr{A_0}{\nb_{0} F_{si}}}_{\calL^{2,p}_{s} \calL^{2}_{t} \dot{\calH}^{m-1}_{x}}
\leq C \sum_{j=0}^{m-1} \nrm{\nb_{x}^{(j)} A_{0}}_{\calL^{0+1/4,\infty}_{s} \calL^{2}_{t} \calL^{\infty}_{x}} \nrm{\nb_{0} F_{s}}_{\calL^{5/4,p}_{s} \calL^{\infty}_{t} \dot{\calH}^{m-1-j}_{x}}.
\end{equation*}

Thanks to the extra weight of $s^{1/4}$ and the fact that $0 \leq j \leq m-1 \leq 9$, we can easily prove $\nrm{\nb_{x}^{(j)} A_{0}}_{\calL^{1/4,\infty}_{s} \calL^{2}_{t} \calL^{\infty}_{x}} \leq C_{\calE, \calF, \calAlow} \cdot (\calE + \calF + \calAlow)^{2}$ via Lemma \ref{lem:fundEst4A0}, Gagliardo-Nirenberg (Lemma \ref{lem:absP:algEst}) and Proposition \ref{prop:pEst4Fs0:high}. On the other hand, as $0 \leq m-1-j \leq 9$, we have $\nrm{\nb_{0} F_{s}}_{\calL^{5/4,p}_{s} \calL^{\infty}_{t} \dot{\calH}^{m-1-j}_{x}} \leq C \calF$. The desired estimate then follows.

\pfstep{- Case 5 : Proof of $ \nrm{s^{-1/2} \LieBr{\nb_{0} A_0}{F_{si}}}_{\calL^{2,p}_{s} \calL^{2}_{t} \dot{\calH}^{m-1}_{x}} \leq C_{\calE, \calF, \calAlow} \cdot (\calE + \calF + \calAlow)^{2}$} 

We claim that the following estimate for $\nb_{0} A_{0}$ holds for $0 \leq j \leq 9$.
\begin{equation} \label{eq:FsWave:quadratic:D0A0:0}
	\nrm{\nb_{x}^{(j)} \nb_{0} A_{0}}_{\calL^{0, \infty}_{s} \calL^{2}_{t} \calL^{\infty}_{x}} \leq C_{\calE, \calF, \calAlow} \cdot (\calE + \calF + \calAlow)^{2}.
\end{equation}

Assuming the claim, let us prove the desired estimate. As in the previous case, we have
\begin{equation*}
\nrm{s^{-1/2} \LieBr{\nb_{0} A_0}{F_{si}}}_{\calL^{2,p}_{s} \calL^{2}_{t} \dot{\calH}^{m-1}_{x}}
\leq C \sum_{j=0}^{m-1} \nrm{\nb_{x}^{(j)} \nb_{0} A_{0}}_{\calL^{1/4,\infty}_{s} \calL^{2}_{t} \calL^{\infty}_{x}} \nrm{F_{s}}_{\calL^{5/4,p}_{s} \calL^{\infty}_{t} \dot{\calH}^{m-1-j}_{x}}.
\end{equation*}

The factor $\nrm{\nb_{x}^{(j)} \nb_{0} A_{0}}_{\calL^{1/4,\infty}_{s} \calL^{2}_{t} \calL^{\infty}_{x}}$ can be controlled by \eqref{eq:FsWave:quadratic:D0A0:0}. For the other factor, we divide into two cases: For $1 \leq j \leq m-1 \leq 9$, we have $\nrm{F_{s}}_{\calL^{5/4,p}_{s} \calL^{\infty}_{t} \dot{\calH}^{m-1-j}_{x}} \leq C \calF$, whereas for $j=0$ we use Proposition \ref{prop:lowEst4Fsi}.  The desired estimate then follows.

To prove the claim, we begin with the formula $\rd_{0} A_{0} = - \int_{s}^{1} \rd_{0} F_{s0}(s') \, \ud s'$. Proceeding as in the proofs of the Lemmas \ref{lem:fundEst4A} and \ref{lem:fundEst4A0}, we obtain the estimate
\begin{equation*}
	\nrm{\nb_{x}^{(j)} \nb_{0} A_{0}}_{\calL^{0, \infty}_{s} \calL^{2}_{t} \calL^{\infty}_{x}} \leq C \nrm{\nb_{x}^{(j)} \nb_{0} F_{s0}}_{\calL^{1,2}_{s} \calL^{2}_{t} \calL^{\infty}_{x}}. 
\end{equation*}

In order to estimate the right-hand side, recall the identity $\rd_{0} F_{s0} = \rd^{\ell} w_{\ell} + \LieBr{A^{\ell}}{w_{\ell}} + \LieBr{A_{0}}{F_{s0}}$ from Appendix \ref{sec:HPYM}. It therefore suffices to prove
\begin{equation*}
	\nrm{\nb_{x}^{(j)} \nb^{\ell} w_{\ell}}_{\calL^{1,2}_{s} \calL^{2}_{t} \calL^{\infty}_{x}} 
	+ \nrm{\nb_{x}^{(j)} \LieBr{A^{\ell}}{w_{\ell}}}_{\calL^{3/2,2}_{s} \calL^{2}_{t} \calL^{\infty}_{x}}
	+ \nrm{\nb_{x}^{(j)} \LieBr{A_{0}}{F_{s0}}}_{\calL^{3/2,2}_{s} \calL^{2}_{t} \calL^{\infty}_{x}}
	\leq C_{\calE, \calF, \calAlow} \cdot (\calE + \calF + \calAlow)^{2},
\end{equation*}
for $0 \leq j \leq 9$.

By Gagliardo-Nirenberg (Lemma \ref{lem:absP:algEst}) and Proposition \ref{prop:pEst4wi}, we have $	\nrm{\nb_{x}^{(j)} \nb^{\ell} w_{\ell}}_{\calL^{1,2}_{s} \calL^{2}_{t} \calL^{\infty}_{x}} \leq C_{\calE, \calF, \calAlow} \cdot (\calE + \calF + \calAlow)^{2}$ for $0 \leq j \leq 9$. 

Next, by Leibniz's rule, H\"older, the Correspondence Principle and Lemma \ref{lem:absP:Holder4Ls}, we obtain
\begin{equation*}
	\nrm{\nb_{x}^{(j)} \LieBr{A^{\ell}}{w_{\ell}}}_{\calL^{3/2,2}_{s} \calL^{2}_{t} \calL^{\infty}_{x}}
	\leq C \sum_{j' = 0}^{j} \nrm{\nb_{x}^{(j')} A}_{\calL^{1/4+1/4,\infty}_{s} \calL^{\infty}_{t,x}} \nrm{\nb_{x}^{(j-j')} w}_{\calL^{1,2}_{s} \calL^{2}_{t} \calL^{\infty}_{x}}.
\end{equation*}

Note the extra weight of $s^{1/4}$ on the first factor. As $0 \leq j' \leq 9$, by Lemma \ref{lem:fundEst4A}, Gagliardo-Nirenberg (Lemma \ref{lem:absP:algEst}) and Proposition \ref{prop:pEst4Fsi}, we have $\nrm{\nb_{x}^{(j')} A}_{\calL^{1/2,\infty}_{s} \calL^{\infty}_{t,x}} \leq C_{\calF, \calAlow} \cdot (\calF + \calAlow).$ On the other hand, $\nrm{\nb_{x}^{(j-j')} w}_{\calL^{1,2}_{s} \calL^{2}_{t} \calL^{\infty}_{x}} \leq C_{\calE, \calF, \calAlow} \cdot (\calE + \calF + \calAlow)^{2}$ by Gagliardo-Nirenberg (Lemma \ref{lem:absP:algEst}) and Proposition \ref{prop:pEst4wi}. 

Finally, we can show $\nrm{\nb_{x}^{(j)} \LieBr{A_{0}}{F_{s0}}}_{\calL^{3/2,2}_{s} \calL^{2}_{t} \calL^{\infty}_{x}} \leq C_{\calE, \calF, \calAlow} \cdot (\calE + \calF + \calAlow)^{2}$ by proceeding similarly, with applications of Propositions \ref{prop:pEst4Fsi} and \ref{prop:pEst4wi} replaced by Proposition \ref{prop:pEst4Fs0:high}. We leave the details to the reader.

\pfstep{- Case 6 : Proof of $\nrm{\LieBr{F_{i 0}}{F_{s0}}}_{\calL^{2,p}_{s} \calL^{2}_{t} \dot{\calH}^{m-1}_{x}} \leq C_{\calE, \calF, \calAlow} \cdot (\calE + \calF + \calAlow)^{2}$}

By Leibniz's rule, H\"older, the Correspondence Principle and Lemma \ref{lem:absP:Holder4Ls}, we have
\begin{equation*}
\nrm{\LieBr{F_{i0}}{F_{s0}}}_{\calL^{2,p}_{s} \calL^{2}_{t} \dot{\calH}^{m-1}_{x}}
\leq C \sum_{j=0}^{m-1} \nrm{\nb_{x}^{(j)} F_{i0}}_{\calL^{3/4,2}_{s} \calL^{\infty}_{t,x}} \nrm{F_{s0}}_{\calL^{5/4,\infty}_{s} \calL^{2}_{t} \dot{\calH}^{m-1-j}_{x}}.
\end{equation*}

Using Gagliardo-Nirenberg (Lemma \ref{lem:absP:algEst}) and Lemma \ref{lem:fundEst4F0i}, combined with Propositions \ref{prop:pEst4Fs0:high}, \ref{prop:pEst4Fsi}, we can prove the following estimate for the first factor (for $0 \leq j \leq 9$):
\begin{equation*}
	\nrm{\nb_{x}^{(j)} F_{i0}}_{\calL^{3/4,2}_{s} \calL^{\infty}_{t,x}} 
	\leq C_{\calE, \calF, \calAlow} \cdot (\calE + \calF + \calAlow).
\end{equation*}

For the second factor, we simply apply Proposition \ref{prop:pEst4Fs0:high} to conclude $\nrm{F_{s0}}_{\calL^{5/4,\infty}_{s} \calL^{2}_{t} \dot{\calH}^{m-1-j}_{x}} \leq C_{\calF, \calAlow} \cdot \calE + C_{\calF, \calAlow} \cdot (\calF + \calAlow)^{2}$ for $0 \leq m-1-j \leq 9$, which is good.

\pfstep{- Case 7 : Proof of $\nrm{\LieBr{\tensor{F}{_{i}^{\ell}}}{F_{s\ell}}}_{\calL^{2,p}_{s} \calL^{2}_{t} \dot{\calH}^{m-1}_{x}} \leq C_{\calE, \calF, \calAlow} \cdot (\calE + \calF + \calAlow)^{2}$}

In this case, we simply expands out $F_{i\ell} = \rd_{i} A_{\ell} - \rd_{\ell} A_{i} + \LieBr{A_{i}}{A_{\ell}}$. Note that the first two terms give additional terms of the form already handled in Step 3, whereas the last term will give us cubic terms which can simply be estimated by using H\"older and Sobolev. For more details, we refer to the proof of Lemma \ref{lem:FsWave:cubic} below. \qedhere
\end{proof}

\subsubsection{Estimates for cubic terms}
The contribution of cubic terms are much easier to handle compared to quadratic terms. Indeed, we have the following lemma.

\begin{lemma} [Estimates for cubic terms] \label{lem:FsWave:cubic}
Assume $0 < T \leq 1$. For $1 \leq m \leq 10$ and $p=2, \infty$, the following estimates hold.
\begin{equation} \label{eq:FsWave:cubic}
	\sup_{i} \nrm{{}^{(F_{si})} \calM_{\cubic}}_{\calL^{2,p}_{s} \calL^{2}_{t} \dot{\calH}^{m-1}_{x} (0,1]} 
	\leq C_{\calE, \calF, \calAlow} \cdot (\calE + \calF + \calAlow)^{3},
\end{equation}
\begin{equation} \label{eq:FsWave:cubic:Diff}
	\sup_{i} \nrm{{}^{(\dlt F_{si})} \calM_{\cubic}}_{\calL^{2,p}_{s} \calL^{2}_{t} \dot{\calH}^{m-1}_{x} (0,1]} 
	\leq C_{\calE, \calF, \calAlow} \cdot (\calE + \calF + \calAlow)^{2} (\dlt \calE + \dlt \calF + \dlt \calAlow).
\end{equation}
\end{lemma}

\begin{proof} 
As before, we give a proof of \eqref{eq:FsWave:cubic}, leaving the similar proof of the difference version \eqref{eq:FsWave:cubic:Diff} to the reader.

Let $1 \leq m \leq 10$ and $p = 2$ or $\infty$. As before, we work on the whole interval $(0,1]$. We begin with the obvious inequality
\begin{equation*}
	\nrm{\phi_{1} \phi_{2} \phi_{3}}_{L^{2}_{t,x}} \leq C T^{1/2} \prod_{i=1,2,3} \nrm{\phi_{i}}_{L^{\infty}_{t} \dot{H}^{1}_{x}},
\end{equation*}
which follows from H\"older and Sobolev. By Leibniz's rule, the Correspondence Principle and H\"older for $\calL^{\ell,p}_{s}$ (Lemma \ref{lem:absP:Holder4Ls}), we obtain
\begin{equation*}
\begin{aligned}
	\sup_{i} \nrm{{}^{(F_{si})} \calM_{\cubic}}_{\calL^{2,p}_{s} \calL^{2}_{t} \dot{\calH}^{m-1}_{x}} 
	\leq & C T^{1/2} \nrm{\nb_{x} A}_{\calL^{1/4,\infty}_{s} \calL^{\infty}_{t} \calH_{x}^{m-1}}^{2} \nrm{\nb_{x} F_{s}}_{\calL^{5/4,p}_{s} \calL^{\infty}_{t} \calH_{x}^{m-1}} \\
	& + C T^{1/2} \nrm{\nb_{x} A_{0}}_{\calL^{0+1/4,\infty}_{s} \calL^{\infty}_{t} \calH_{x}^{m-1}}^{2} \nrm{\nb_{x} F_{s}}_{\calL^{5/4,p}_{s} \calL^{\infty}_{t} \calH_{x}^{m-1}}.
\end{aligned}\end{equation*}

Note the obvious bound $\nrm{\nb_{x} F_{s}}_{\calL^{5/4,p}_{s} \calL^{\infty}_{t} \calH_{x}^{m-1}} \leq C\calF$. Applying Lemma \ref{lem:fundEst4A0} to $\nrm{A_{0}}$ (using the extra weight of $s^{1/4}$) and Proposition \ref{prop:pEst4Fs0:high}, we also obtain $\nrm{\nb_{x} A_{0}}_{\calL^{0+1/4,\infty}_{s} \calL^{\infty}_{t} \calH_{x}^{m-1}} \leq C_{\calF, \calAlow} \cdot \calE + C_{\calF, \calAlow} \cdot (\calF + \calAlow)^{2}$. Finally, we split $\nrm{\nb_{x} A}_{\calL^{1/4,\infty}_{s} \calL^{\infty}_{t} \calH_{x}^{m-1}}$ into $\nrm{A}_{\calL^{1/4,\infty}_{s} \calL^{\infty}_{t} \dot{\calH}_{x}^{1}}$ and $\nrm{\nb_{x}^{(2)} A}_{\calL^{1/4,\infty}_{s} \calL^{\infty}_{t} \calH_{x}^{m-2}}$ (where the latter term does not exist in the case $m = 1$). For the former we apply Proposition \ref{prop:est4ai}, whereas for the latter we apply Lemma \ref{lem:fundEst4A}. We then conclude $\nrm{\nb_{x} A}_{\calL^{1/4,\infty}_{s} \calL^{\infty}_{t} \calH_{x}^{m-1}} \leq C_{\calF, \calAlow} \cdot (\calF + \calAlow)$. Combining all these estimates, \eqref{eq:FsWave:cubic} follows. \qedhere
\end{proof}

\subsubsection{Estimates for terms involving $w_{i}$}
Finally,  the contribution of ${}^{(F_{si})} \calM_{w}$ is estimated by the following lemma.

\begin{lemma}[Estimates for terms involving $w_{i}$] \label{lem:FsWave:w}
Assume $0 < T \leq 1$. For $1 \leq m \leq 10$ and $p=2, \infty$, the following estimates hold.
\begin{equation} \label{eq:FsWave:w}
	\sup_{i} \nrm{{}^{(F_{si})} \calM_{w}}_{\calL^{2,p}_{s} \calL^{2}_{t} \dot{\calH}^{m-1}_{x}(0,1]} 
	\leq C_{\calE, \calF, \calAlow} \cdot (\calE + \calF + \calAlow)^{2},
\end{equation}
\begin{equation} \label{eq:FsWave:w:Diff}
	\sup_{i} \nrm{{}^{(\dlt F_{si})} \calM_{w}}_{\calL^{2,p}_{s} \calL^{2}_{t} \dot{\calH}^{m-1}_{x}(0,1]} 
	\leq C_{\calE, \calF, \calAlow} \cdot (\calE + \calF + \calAlow)(\dlt \calE + \dlt \calF + \dlt \calAlow).
\end{equation}
\end{lemma}

\begin{proof} 
As before, we will only give a proof of \eqref{eq:FsWave:w}, leaving the similar proof of \eqref{eq:FsWave:w:Diff} to the reader.

Let $1 \leq m \leq 10$ and $p = 2$ or $\infty$. We work on the whole interval $(0,1]$. Note that, schematically,
\begin{align*}
{}^{(F_{si})}\calM_{w} 
= & \rd^\ell \rd_\ell w_i - \rd_i \rd^\ell w_\ell + {}^{(w_{i})} \calN + \calO(A, \rd_{x} w) + \calO(\rd_{x} A, w) + \calO(A, A, w).
\end{align*}

By Leibniz's rule, the Correspondence Principle (from the H\"older inequality $L^{\infty}_{t,x} \cdot L^{2}_{t,x} \subset L^{2}_{t,x}$) and Lemma \ref{lem:absP:Holder4Ls}, we obtain the following estimate.
%\begin{aligned}
%	 \nrm{{}^{(F_{si})}\calM_{w}}_{\calL^{2,p}_{s} \calL^{2}_{t} \dot{\calH}^{m-1}_{x}} 
%	\leq &C \nrm{w}_{\calL^{1,p}_{s} \calL^{2}_{t} \dot{\calH}^{m+1}_{x}} +  C \nrm{{}^{(w_{i})} \calN}_{\calL^{2,p}_{s} \calL^{2}_{t} \dot{\calH}^{m+1}_{x}} \\
%	& + C \nrm{s^{-1/2} \calO(A, \nb_{x} w)}_{\calL^{2,p}_{s} \calL^{2}_{t} \dot{\calH}^{m-1}_{x}} 
%	+C \nrm{s^{-1/2} \calO(\nb_{x} A, w)}_{\calL^{2,p}_{s} \calL^{2}_{t} \dot{\calH}^{m-1}_{x}} \\
%	& + C \nrm{\calO(A, A, \nb_{x} w)}_{\calL^{2,p}_{s} \calL^{2}_{t} \dot{\calH}^{m-1}_{x}}
%\end{aligned}
\begin{equation} \label{eq:FsWave:w:pf:1}
\begin{aligned}
	& \nrm{{}^{(F_{si})}\calM_{w}}_{\calL^{2,p}_{s} \calL^{2}_{t} \dot{\calH}^{m-1}_{x}} \\
	& \qquad \leq C \nrm{w}_{\calL^{1,p}_{s} \calL^{2}_{t} \dot{\calH}^{m+1}_{x}} +  C \nrm{{}^{(w_{i})} \calN}_{\calL^{2,p}_{s} \calL^{2}_{t} \dot{\calH}^{m+1}_{x}} 
	+C \sum_{j=0}^{m} \nrm{\nb_{x}^{(j)} A}_{\calL^{1/4,\infty}_{s} \calL^{\infty}_{t,x}} \nrm{ w}_{\calL^{1,2}_{s} \calL^{2}_{t} \dot{\calH}^{m-j}_{x}} \\
	&\phantom{\qquad \leq}+ C \sum_{j, j' \geq 0, j+j' \leq m-1} \nrm{\nb_{x}^{(j)} A}_{\calL^{1/4,\infty}_{s} \calL^{\infty}_{t,x}}\nrm{\nb_{x}^{(j')} A}_{\calL^{1/4,\infty}_{s} \calL^{\infty}_{t,x}} \nrm{w}_{\calL^{1,2}_{s} \calL^{2}_{t} \dot{\calH}^{m-1-j-j'}_{x}}
\end{aligned}
\end{equation}

By Lemma \ref{lem:fundEst4A}, combined with Proposition \ref{prop:pEst4Fsi}, the following estimate holds for $0 \leq j \leq 10$.
\begin{equation} \label{eq:FsWave:w:pf:2}
	\nrm{\nb_{x}^{(j)} A}_{\calL^{1/4,\infty}_{s} \calL^{\infty}_{t,x}} \leq C_{\calF, \calAlow} (\calF + \calAlow).
\end{equation}

Now \eqref{eq:FsWave:w} follows from \eqref{eq:FsWave:w:pf:1}, \eqref{eq:FsWave:w:pf:2} and \eqref{eq:pEst4wi:3}, \eqref{eq:pEst4wi:4} of Proposition \ref{prop:pEst4wi}, thanks to the restriction $1 \leq m \leq 10$. \qedhere
\end{proof}

\subsubsection{Completion of the proof}
We are now prepared to give a proof of Theorem \ref{thm:FsWave}.
%We have finally developed enough preparations to give a quick proof of Theorem \ref{thm:FsWave}.

\begin{proof} [Proof of Theorem \ref{thm:FsWave}]
Let us begin with \eqref{eq:FsWave:1}. Recalling the definition of $\calF$, it suffices to show
\begin{equation*}
\nrm{F_{si}}_{\calL^{5/4,p}_{s} \dot{\calS}^{m}(0,1]} \leq C \calI + T^{1/2} C_{\calE, \calF, \calAlow} \cdot (\calE + \calF + \calAlow)^{2},
\end{equation*}
for $i=1,2,3$, $p= 2, \infty$ and $1 \leq m \leq 10$. Starting with the energy inequality and applying the Correspondence Principle, we obtain
\begin{equation*}
	\nrm{F_{si}}_{\calL^{5/4, p}_{s} \dot{\calS}^{m}} 
	\leq C \nrm{\nb_{t,x} F_{si}(t=0)}_{\calL^{5/4,p}_{s} \dot{\calH}^{m-1}_{x}} + C T^{1/2} \nrm{\Box F_{si}}_{\calL^{2, p}_{s} \calL^{2}_{t} \dot{\calH}^{m-1}_{x}}.
\end{equation*}

The first term on the right-hand side is estimated by $C\calI$. For the second term, as $\Box F_{si} = {}^{(F_{si})} \calM_{\quadratic} + {}^{(F_{si})} \calM_{\cubic} + {}^{(F_{si})} \calM_{w}$, we may apply Lemmas \ref{lem:FsWave:quadratic} -- \ref{lem:FsWave:w} (estimates \eqref{eq:FsWave:quadratic}, \eqref{eq:FsWave:cubic} and \eqref{eq:FsWave:w}, in particular) to conclude
\begin{equation*}
	\nrm{\Box F_{si}}_{\calL^{5/4+1, p}_{s} \calL^{1}_{t} \dot{\calH}^{m-1}_{x}} \leq T^{1/2} C_{\calE, \calF, \calAlow} (\calE + \calF + \calAlow)^{2},
\end{equation*}
which is good.

The proof of \eqref{eq:FsWave:2} is basically identical, this time controlling the initial data term by $C \dlt \calI$ and using \eqref{eq:FsWave:quadratic:Diff}, \eqref{eq:FsWave:cubic:Diff}, \eqref{eq:FsWave:w:Diff} in place of \eqref{eq:FsWave:quadratic}, \eqref{eq:FsWave:cubic}, \eqref{eq:FsWave:w}. \qedhere
\end{proof}

\begin{remark} 
Recall that in \cite{Klainerman:1995hz}, one has to recover two types of null forms, namely $Q_{ij}(\abs{\rd_{x}}^{-1} A, A)$ and $\abs{\rd_{x}}^{-1} Q_{ij}(A, A)$, in order to prove $H^{1}$ local well-posedness in the Coulomb gauge. An amusing observation is that we did not need to uncover the second type of null forms in our proof. 
%This raises the possibility that the caloric-type gauges for the Yang-Mills equation might offer even more advantageous structure than the Coulomb gauge, which was indeed the case for the wave map and Schr\"dingier equations (See \cite{?}).
\end{remark}

\appendix
\section{Derivation of covariant equations of \eqref{eq:HPYM}} \label{sec:HPYM}
Let $I \subset \bbR$ be an open interval, $s_{0} > 0$, and $A_{\bfa}$ a connection 1-form on $I \times \bbR^{3} \times [0,s_{0}]$ with coordinates $(t,x^{1}, x^{2}, x^{3}, s)$. Recall that a bold-faced latin index $\bfa$ runs over all the indices corresponding to $t, x^{1}, x^{2}, x^{3}, s$. As in the introduction, we may define the \emph{covariant derivative} $\covD_{\bfa}$ associated to $A_{\bfa}$. The commutator of the covariant derivatives in turn defines the \emph{curvature 2-form} $F_{\bfa \bfb}$. Note that the \emph{Bianchi identity} holds automatically:
\begin{equation} \label{eq:Bianchi4Fab:A}
	\covD_\bfa F_{\bfb \bfc} + \covD_\bfb F_{\bfc \bfa} + \covD_\bfc F_{\bfa \bfb} =0.
\end{equation}

In this appendix, we will consider a solution $A_{\bfa}$ to the \emph{hyperbolic-parabolic Yang-Mills system}, which we restate here for the reader's convenience: 
\begin{equation*} \tag{HPYM}
\left\{
\begin{aligned}
	F_{s \mu} &= \covD^{\ell} F_{\ell \mu} \hspace{.25in} \hbox{ on } \hspace{.1in} I \times \bbR^{3} \times [0,s_{0}], \\
	\covD^{\mu} F_{\mu \nu} &= 0 \hspace{.5in} \hbox{ along } I \times \bbR^{3} \times \set{0},
\end{aligned}
\right.
\end{equation*}

Let us also recall the definition of the Yang-Mills tension field $w_{\nu}$:
\begin{equation*} 
	w_{\nu} := \covD^{\mu} F_{\nu \mu}.
\end{equation*}

We will use the convention of using an underline to signify the variable being evaluated at $s=s_{0}$. For example, $\Alow_{\mu} = A_{\mu}(s=s_{0})$, $\covDlow_{\mu} B = \rd_{\mu} B + \LieBr{A_{\mu}(s=s_{0})}{B}$, $\wlow_{\mu} = w_{\mu}(s=s_{0})$ and etc. In the main body of the paper, $s_{0}$ is always set to be $1$ by scaling.

The main result of this appendix is the following theorem.
\begin{theorem}[Covariant equations of the hyperbolic-parabolic Yang-Mills system] \label{thm:A:covEqns}
Let $A_{\bfa}$ be a smooth solution to the hyperbolic-parabolic Yang-Mills system \eqref{eq:HPYM}. Then the following covariant equations hold.
\begin{align} 
	\covD^\ell F_{s\ell} = & 0,\label{eq:covCoulomb:A} \\
	\covD_0 F_{s0} = & - \covD_0 w_0 = - \covD^\ell w_\ell, \label{eq:D0Fs0:A} \\
 	\covD_s F_{\bfa \bfb} =& \covD^\ell \covD_\ell F_{\bfa \bfb} - 2\LieBr{\tensor{F}{_\bfa^\ell}}{F_{\bfb \ell}}, \label{eq:covParabolic4Fab:A} \\
	\covD_s w_\nu 
	= & \covD^\ell \covD_\ell w_\nu + 2 \LieBr{\tensor{F}{_\nu^\ell}}{w_\ell} + 2 \LieBr{F^{\mu \ell}}{\covD_{\mu} F_{\nu \ell} + \covD_{\ell} F_{\nu \mu}}, 
	\label{eq:covParabolic4w:A} \\
	 \covD^\mu \covD_\mu F_{s \nu} 
 	= & 2 \LieBr{\tensor{F}{_s^\mu}}{F_{\nu \mu}} - 2 \LieBr{F^{\mu \ell}}{\covD_\mu F_{\nu \ell} + \covD_\ell F_{\nu \mu}} 
		- \covD^\ell \covD_\ell w_\nu + \covD_\nu \covD^\ell w_\ell - 2 \LieBr{\tensor{F}{_\nu^\ell}}{w_\ell}. \label{eq:hyperbolic4F:A} \\
	\covDlow^{\mu}\Flow_{\nu \mu} = & \wlow_{\nu} \label{eq:hyperbolic4Alow:A}
\end{align}

Moreover, we have $w_{\nu}(s=0) = 0$.
\end{theorem}
\begin{remark} 
An inspection of the proof shows that \eqref{eq:covCoulomb:A}--\eqref{eq:hyperbolic4Alow:A} hold under the weaker hypothesis that $A_{\bfa}$ satisfies only \eqref{eq:dYMHF}. The last statement of the theorem is equivalent to \eqref{eq:hyperbolicYM} along $\set{s=0}$.
\end{remark}
\begin{proof} 
Let us begin with \eqref{eq:covCoulomb:A}. This is a consequence of the following simple computation:
\begin{align*}
\covD^\ell F_{s \ell} =  \covD^\ell \covD^k F_{k \ell} 
			= & \frac 1 2 \covD^\ell \covD^k F_{k \ell} + \frac 1 2 \covD^k \covD^\ell F_{k \ell} + \frac 1 2 \LieBr{F^{\ell k}}{F_{k \ell}} \\
			= & \frac 1 2 (\covD^\ell \covD^k + \covD^k \covD^\ell) F_{k \ell}  = 0,
\end{align*}
where we have used anti-symmetry of $F_{k \ell}$ for the last equality.

Next, in order to derive \eqref{eq:D0Fs0:A}, we first compute
\begin{equation*}
	\covD^{\mu} w_{\mu} = \covD^{\mu} \covD^{\nu} F_{\nu \mu} = 0,
\end{equation*}
by a computation similar to the preceding one. This give the second equality of \eqref{eq:D0Fs0:A}. In order to prove the first equality, we compute 
\begin{equation} \label{eq:divFsmu:A}
	\covD^{\mu }F_{s \mu} = \covD^{\mu} \covD^{\ell} F_{\ell \mu} 
	= \covD^{\ell} \covD^{\mu} F_{\ell \mu} + \LieBr{F^{\mu \ell}}{F_{\ell \mu}} 
	= \covD^{\ell} w_{\ell}
\end{equation}
and note that $\covD^{\mu} F_{s \mu} = \covD^{0} F_{s 0} = - \covD_{0} F_{s0}$ by \eqref{eq:covCoulomb:A}.

Next, let us derive \eqref{eq:covParabolic4Fab:A}. We begin by noting that the equation $F_{s \bfa} = \covD^{\ell} F_{\ell \bfa}$ holds for $\bfa = t,x,s$; for the last case, we use \eqref{eq:covCoulomb:A}. Using this and the Bianchi identity \eqref{eq:Bianchi4Fab:A}, we compute
\begin{align*}
	\covD_s F_{\bfa \bfb} = & \covD_{\bfa} F_{s \bfb} - \covD_{\bfb} F_{s \bfa} 
					 = \covD_{\bfa} \covD^\ell F_{\ell \bfb} - \covD_{\bfb} \covD^\ell F_{\ell \bfa} \\
					 = & \covD^\ell (\covD_{\bfa} F_{\ell \bfb} - \covD_{\bfb} F_{\ell \bfa}) + \LieBr{\tensor{F}{_{\bfa}^{\ell}}}{F_{\ell \bfb}} - \LieBr{\tensor{F}{_{\bfb}^{\ell}}}{F_{\ell \bfa}} \\
					 = & \covD^\ell \covD_\ell F_{\bfa \bfb} - 2\LieBr{\tensor{F}{_\bfa^\ell}}{F_{\bfb \ell}}. 
\end{align*}

In order to prove \eqref{eq:covParabolic4w:A}, we will use \eqref{eq:covParabolic4Fab:A}. We compute as follows.
\begin{align*}
	\covD_s w_\nu = & \covD_s \covD^\mu F_{\nu \mu} \\
%				= & \covD^\mu \covD_s F_{\nu \mu} + \LieBr{\tensor{F}{_s^\mu}}{F_{\nu \mu}} \\
				= & \covD^\mu \bb( \covD^\ell \covD_\ell F_{\nu \mu} - 2 \LieBr{\tensor{F}{_\nu^\ell}}{F_{\mu \ell}} \bb) + \LieBr{\tensor{F}{_s^\mu}}{F_{\nu \mu}} \\
				= & \covD^\ell \covD_\ell \bb( \covD^\mu F_{\nu \mu} \bb) + \LieBr{F^{\mu \ell}}{\covD_{\ell} F_{\nu \mu}} + \covD^\ell \LieBr{\tensor{F}{^\mu_\ell}}{F_{\nu \mu}} - 2 \covD^\mu \LieBr{\tensor{F}{_\nu^\ell}}{F_{\mu \ell}} + \LieBr{\tensor{F}{_s^\mu}}{F_{\nu \mu}} \\
%				= & \covD^\ell \covD_\ell w_\nu + 2 \LieBr{F^{\mu \ell}}{\covD_\ell F_{\nu \mu}} - \LieBr{\tensor{F}{_s^\mu}}{F_{\nu \mu}} +2 \LieBr{\tensor{F}{_\nu^\ell}}{w_\ell} - 2 \LieBr{\covD^\mu \tensor{F}{_\nu^\ell}}{F_{\mu \ell}} + \LieBr{\tensor{F}{_s^\mu}}{F_{\nu \mu}} \\
				= & \covD^\ell \covD_\ell w_\nu +2 \LieBr{\tensor{F}{_\nu^\ell}}{w_\ell} +2 \LieBr{F^{\mu \ell}}{\covD_{\mu} F_{\nu \ell} + \covD_{\ell} F_{\nu \mu}}.
\end{align*}

Note that \eqref{eq:hyperbolic4Alow:A} is exactly the definition of $w_{\nu}$ at $s=s_{}$. We are therefore only left to prove \eqref{eq:hyperbolic4F:A}. 

Here, the idea is to start with the Bianchi identity $0 = \covD_{\mu} F_{s \nu} + \covD_{s} F_{\nu \mu} + \covD_{\nu} F_{\mu s}$ and to take $\covD^{\mu}$ of both sides. The first term on the right-hand side gives the desired term $\covD^{\mu} \covD_{\mu} F_{s \nu}$. For the second term, we compute 
\begin{align*}
	\covD^{\mu} \covD_{s} F_{\nu \mu} 
	= \covD_{s} \covD^{\mu} F_{\nu \mu} + \LieBr{\tensor{F}{^{\mu}_{s}}}{F_{\nu \mu}}  
	= \covD_{s} w_{\nu} - \LieBr{\tensor{F}{_{s}^{\mu}}}{F_{\nu \mu}},
\end{align*}
and for the third term, we compute, using \eqref{eq:divFsmu:A},
\begin{align*}
	\covD^{\mu} \covD_{\nu} F_{\mu s} 
	= \covD_{\nu} \covD^{\mu} F_{\mu s} + \LieBr{\tensor{F}{^{\mu}_{\nu}}}{F_{\mu s}} 
	= - \covD_{\nu} \covD^{\ell} w_{\ell} - \LieBr{\tensor{F}{_{s}^{\mu}}}{F_{\nu \mu}} 
\end{align*}

Combining these with \eqref{eq:covParabolic4w:A}, we obtain \eqref{eq:hyperbolic4F:A}. \qedhere
\end{proof}

\section{Estimates for gauge transforms} \label{sec:gt}
In this section, we prove the two gauge transform lemmas stated in the main body of the paper, namely Lemmas \ref{lem:est4gt2temporal} and \ref{lem:est4gt2caloric}. Let us start with a general proposition concerning the solution to an ODE, which arises from the equations satisfied by the gauge transforms (and their differences).

\begin{proposition} [ODE estimates] \label{prop:est4gt}
Let $\Omg \subset \bbR$ be an interval, $\omg_{0} \in \Omg$, and $A = A(\omg, x)$, $F=F(\omg, x)$ a pair of $n \times n$ matrix valued functions on $\Omg \times \bbR^{3}$. Consider the following ODE for an $n \times n$ matrix-valued function $U = U(\omg, x)$:
\begin{equation} \label{eq:est4gt:hyp}
\left\{
\begin{aligned}
	\rd_{\omg} U =& A U + F, \\
	U(\omg_{0}) =& U_{0}.
\end{aligned}
\right.
\end{equation}

The following statements hold.

\begin{enumerate}
\item Suppose that $U_{0} \in C^{0}_{x} \cap \dot{W}^{1,3}_{x} \cap \dot{H}^{2}_{x}$, $F \in L^{1}_{\omg}(C^{0}_{x} \cap \dot{W}^{1,3}_{x} \cap \dot{H}^{2}_{x})(\Omg)$, and that there exists $D > 0$ such that
\begin{equation*}
	\nrm{A}_{L^{1}_{\omg} L^{\infty}_{x}(\Omg)} + \nrm{\rd_{x} A}_{L^{1}_{\omg} L^{3}_{x}(\Omg)} \leq D.
\end{equation*}

Suppose furthermore that
\begin{equation*}
	\hbox{either } \nrm{\rd_{x}^{(2)} A}_{L^{1}_{\omg} L^{2}_{x}(\Omg)} \leq D \hbox{ or } \sup_{\omg \in \Omg} \nrm{\int_{\omg_{0}}^{\omg} \lap A (\omg') \, \ud \omg'}_{L^{2}_{x}} \leq D.
\end{equation*}

Then there exists a unique solution $U$ to \eqref{eq:est4gt:hyp} which is continuous in $x$ and obeys the following estimates.
\begin{align}
	\nrm{U}_{L^{\infty}_{\omg} L^{\infty}_{x} (\Omg)} 
	\leq & e^{CD} \bb( \nrm{U_{0}}_{L^{\infty}_{x}} + \nrm{F}_{L^{1}_{\omg} L^{\infty}_{x}(\Omg)}\bb) , \label{eq:est4gt:1:1}\\
	\nrm{\rd_{x} U}_{L^{\infty}_{\omg} L^{3}_{x} (\Omg)} 
	\leq & e^{CD} \bb( \nrm{\rd_{x} U_{0}}_{L^{3}_{x}} + \nrm{\rd_{x} F}_{L^{1}_{\omg} L^{3}_{x}(\Omg)} 
		+ D(\nrm{U_{0}}_{L^{\infty}_{x}} + \nrm{F}_{L^{1}_{\omg} L^{\infty}_{x}(\Omg)})\bb), \label{eq:est4gt:1:2} \\
	\nrm{\rd_{x}^{(2)} U}_{L^{\infty}_{\omg} L^{2}_{x} (\Omg)} 
	\leq & e^{CD} \bb( \nrm{\rd_{x}^{(2)} U_{0}}_{L^{2}_{x}} + \nrm{\rd_{x}^{(2)} F}_{L^{1}_{\omg} L^{2}_{x}(\Omg)}
		+ D(\nrm{U_{0}}_{L^{\infty}_{x}} + \nrm{F}_{L^{1}_{\omg} L^{\infty}_{x}(\Omg)}) \bb). \label{eq:est4gt:1:3}
\end{align}

\item In addition to the hypotheses of part (1), suppose furthermore that $F \in L^{\infty}_{\omg} (L^{3}_{x} \cap \dot{H}^{1}_{x}) (\Omg)$ and
	\begin{equation*}
		\nrm{A}_{L^{\infty}_{\omg} L^{3}_{x}(\Omg)} + \nrm{\rd_{x} A}_{L^{\infty}_{\omg} L^{2}_{x}(\Omg)} \leq D.
	\end{equation*}
	
	Then the unique solution $U$ in (1) furthermore obeys
\begin{align}
	\nrm{\rd_{\omg} U}_{L^{\infty}_{\omg} L^{3}_{x} (\Omg)} 
	\leq & \nrm{F}_{L^{\infty}_{\omg} L^{3}_{x}(\Omg)}
		+ CD e^{CD} (\nrm{U_{0}}_{L^{\infty}_{x}} + \nrm{F}_{L^{1}_{\omg} L^{\infty}_{x}(\Omg)}), \label{eq:est4gt:2:1} \\
	\nrm{\rd_{x} \rd_{\omg} U}_{L^{\infty}_{\omg} L^{2}_{x} (\Omg)} 
	\leq & \nrm{\rd_{x} F}_{L^{\infty}_{\omg} L^{2}_{x}(\Omg)} \notag \\
		&+ CD e^{CD} (\nrm{U_{0}}_{L^{\infty}_{x}} + \nrm{\rd_{x} U_{0}}_{L^{3}_{x}} 
					+ \nrm{F}_{L^{1}_{\omg} L^{\infty}_{x}(\Omg)} + \nrm{\rd_{x} F}_{L^{1}_{\omg} L^{3}_{x}(\Omg)}). \label{eq:est4gt:2:2}
\end{align}

\item Finally, the same conclusions still hold if we replace the ODE by $\rd_{\omg} U = U A + F$.
\end{enumerate}
\end{proposition}

\begin{proof} 
With the understanding that all norms are taken over $\Omg$, we will omit writing $\Omg$.

Mollifying $U_{0}$ and $A, F$ for every $\omg$, we may assume that $A$ and $F$ depends smoothly on $x$. By the standard ODE theory, there then exists a unique solution $U$ to \eqref{eq:est4gt:hyp} which is also smooth in $x$, and whose derivatives obey the equations obtained by differentiating \eqref{eq:est4gt:hyp}. We will establish the proposition for such smooth objects; by a standard approximation procedure, the original proposition will then follow.

\pfstep{Proof of (1)}
We begin by integrating \eqref{eq:est4gt:hyp} to recast it in the following equivalent form.
\begin{equation} \label{eq:est4gt:hyp:int}
	U(\omg) = U_{0} + \int_{\omg_{0}}^{\omg} F(\omg') \, \ud \omg' + \int_{\omg_{0}}^{\omg} A(\omg') U(\omg') \, \ud \omg'.
\end{equation}

Taking the supremum in $x$ of both sides of \eqref{eq:est4gt:hyp:int}, the first estimate \eqref{eq:est4gt:1:1} follows as a simple consequence of Gronwall's inequality. The second estimate \eqref{eq:est4gt:1:2} can be proved similarly (i.e., by Gronwall and the first estimate), by differentiating the equation \eqref{eq:est4gt:hyp:int} and taking the $L^{3}_{x}$ norm. Differentiating once more and taking the $L^{2}_{x}$ norm, the third estimate \eqref{eq:est4gt:1:3} easily follows (again by Gronwall and the previous estimates) \emph{if one assumes } $\nrm{\rd_{x}^{(2)} A}_{L^{1}_{\omg} L^{2}_{x}} \leq D$. The case of the alternative hypothesis $\sup_{\omg \in \Omg} \nrm{\int_{\omg_{0}}^{\omg} \lap A (\omg') \, \ud \omg'}_{L^{2}_{x}} \leq D$, on the other hand, is not as apparent, and requires a separate argument.

Here we follow the arguments in \cite[Proof of Theorem 7]{Klainerman:1995hz}. First, in order to control $\nrm{\rd_{x}^{(2)} U}_{L^{2}_{x}}$, observe that it suffices to control $\nrm{\lap U}_{L^{\infty}_{\omg} L^{2}_{x}}$, as a simple integration by parts\footnote{In order to deal with the boundary term, one may multiply $U$ by a smooth cut-off $\chi(x/R)$, and then take $R \to \infty$. Using the fact that $\nrm{U(\omg)}_{L^{\infty}_{x}} + \nrm{\rd_{x} U(\omg)}_{L^{3}_{x}} < \infty$, it can be shown that $\nrm{\rd_{i} \rd_{j} (U (\omg) \chi(x/R))}_{L^{2}_{x}} \to \nrm{\rd_{i} \rd_{j} U(\omg)}_{L^{2}_{x}}$.} shows that $\nrm{\rd_{x}^{(2)} U(\omg)}_{L^{2}_{x}} \leq C \nrm{\lap U (\omg)}_{L^{2}_{x}}$. Taking $\lap$ of \eqref{eq:est4gt:hyp:int}, we obtain
\begin{align*}
	\lap U(\omg) 
	= &\lap U_{0} + \int_{\omg_{0}}^{\omg} \lap F(\omg') \, \ud \omg' + \int_{\omg_{0}}^{\omg} \lap A(\omg') U(\omg') \, \ud \omg' \\
	& + 2 \int_{\omg_{0}}^{\omg} \rd^{\ell} A(\omg') \rd_{\ell} U(\omg') \, \ud \omg' + \int_{\omg_{0}}^{\omg} A(\omg') \lap U(\omg') \, \ud \omg'.
\end{align*}

Let us take the $L^{2}_{x}$ norm of both sides. The contribution of the first two terms are fine by hypotheses. In order to deal with the third term, we use the trick of using \eqref{eq:est4gt:hyp:int} and Fubini to rewrite it as
\begin{align*}
	&\nrm{\int_{\omg_{0}}^{\omg} \lap A(\omg') \bb( U_{0} + \int_{\omg_{0}}^{\omg'} F(\omg'') + A(\omg'') U(\omg'') \, \ud \omg'' \bb)}_{L^{2}_{x}} \\
	&\quad \leq \nrm{\int_{\omg_{0}}^{\omg} \lap A(\omg') \, \ud \omg'}_{L^{2}_{x}} \nrm{U_{0}}_{L^{\infty}_{x}}
	+ \nrm{\int_{\omg_{0}}^{\omg} \bb( \int_{\omg''}^{\omg} \lap A(\omg') \, \ud \omg' \bb) 
			\bb( F(\omg'') + A(\omg'') U(\omg'') \bb) \, \ud \omg''}_{L^{2}_{x}} \\
	& \quad \leq D \nrm{U_{0}}_{L^{\infty}_{x}} + D^{2} (\nrm{F}_{L^{1}_{\omg}L^{\infty}_{x}}+\nrm{U}_{L^{\infty}_{\omg} L^{\infty}_{x}})
	\leq C D e^{CD} (\nrm{U_{0}}_{L^{\infty}_{x}} + \nrm{F}_{L^{1}_{\omg} L^{\infty}_{x}} ),
\end{align*}
where we used \eqref{eq:est4gt:1:1} in the last inequality. For the last two terms, we first use H\"older and Sobolev to estimate
\begin{align*}
	&\nrm{2 \int_{\omg_{0}}^{\omg} \rd^{\ell} A(\omg') \rd_{\ell} U(\omg') \, \ud \omg' + \int_{\omg_{0}}^{\omg} A(\omg') \lap U(\omg') \, \ud \omg' }_{L^{2}_{x}} \\
	&\quad \leq C \int_{\omg_{0}}^{\omg} (\nrm{\rd_{x} A(\omg')}_{L^{3}_{x}} + \nrm{A(\omg')}_{L^{\infty}_{x}}) \nrm{\rd_{x}^{(2)} U(\omg')}_{L^{2}_{x}} \, \ud \omg',
\end{align*}
at which point we can use Gronwall's inequality to conclude \eqref{eq:est4gt:1:3}.

\pfstep{Proof of (2)}
Taking the $L^{3}_{x}$ norm of both sides of the equation $\rd_{\omg} U = A U + F$, the first estimate \eqref{eq:est4gt:2:1} follows immediately by triangle, H\"older, \eqref{eq:est4gt:1:1} and the hypothesis $\nrm{A}_{L^{\infty}_{\omg} L^{3}_{x}} \leq D$. To prove the second estimate \eqref{eq:est4gt:2:2}, we take the $L^{2}_{x}$ norm of both sides of the differentiated equation $\rd_{x} \rd_{\omg} U = (\rd_{x} A) U + A \rd_{x} U + \rd_{x} F$. Then we use triangle, H\"older, \eqref{eq:est4gt:1:1}, \eqref{eq:est4gt:1:2} and the hypotheses.
\end{proof}

Now we are ready to prove Lemmas \ref{lem:est4gt2temporal} and \ref{lem:est4gt2caloric}.
\begin{proof} [Proof of Lemma \ref{lem:est4gt2temporal}]
Recalling the definition of $\calA_{0}$, \eqref{eq:est4gt2temporal:V} follow immediately from Proposition \ref{prop:est4gt}, with $\omg = t$, $\Omg = (-T, T)$, $\omg_{0} = 0$, $A = A_{0}$, $F=0$ and $D = \calA_{0}(-T, T)$. For the difference analogue \eqref{eq:est4gt2temporal:dltV}, note that $\dlt V$ satisfies the ODE
\begin{equation*}
\rd_{t} (\dlt V) = (\dlt V) A_{0}  + V \dlt A_{0},
\end{equation*}
to which Proposition \ref{prop:est4gt} can also be applied, with $F =  V  \dlt A_{0}$. The appropriate bounds for $F$ can be proved easily by using $\dlt \calA_{0}(-T, T)$ and the previously established estimates for $V$. This proves \eqref{eq:est4gt2temporal:dltV}.

Finally, the corresponding estimates for $V^{-1}$ and $\dlt V^{-1}$ follows in the same manner once one observes that they satisfy the equations
\begin{equation*}
\rd_{t} V^{-1} = - A_{0} V^{-1}, \qquad \rd_{t} \dlt V^{-1} = - A_{0} \dlt V^{-1} - (\dlt A_{0}) V^{-1},
\end{equation*}
respectively. \qedhere
\end{proof}

\begin{proof} [Proof of Lemma \ref{lem:est4gt2caloric}]
By \eqref{eq:YMHF4A:smth4A} of Proposition \ref{prop:YMHF4A:smth4deT}, the coefficient matrix $A_{s} = \rd^{\ell} A_{\ell}$ satisfies 
\begin{equation} \label{eq:est4gt2caloric:pf:1}
	\nrm{s^{k/2} \rd_{x}^{(k)} A_{s}}_{L^{\infty}_{s} L^{2}_{x}} \leq C_{k, \nrm{\Aini}_{\dot{H}^{1}_{x}}} \nrm{\Aini}_{\dot{H}^{1}_{x}}
\end{equation}
for all integers $k \geq 0$. By interpolation and Gagliardo-Nirenberg, it then easily follows that $\nrm{A_{s}}_{L^{1}_{s} L^{\infty}_{x}} + \nrm{\rd_{x} A_{s}}_{L^{1}_{s} L^{3}_{x}}$ is bounded by the right-hand side of \eqref{eq:est4gt2caloric:pf:1}. On the other hand, by Corollary \ref{cor:YMHF4A:bnd4lapAs}, $\sup_{0 \leq s \leq 1} \nrm{\int_{s'}^{1} \lap A_{s}}_{L^{2}_{x}}$ is also bounded by the right-hand side of \eqref{eq:est4gt2caloric:pf:1}. Applying Proposition \ref{prop:est4gt} with $\omg = s$,  $\Omg = [0,1]$, $\omg_{0} = s_{0}$, $A = A_{s}$, $F = 0$ and $U_{0} = \mathrm{Id}$, we then obtain \eqref{eq:est4gt2caloric:U:1} and \eqref{eq:est4gt2caloric:U:2} for $m=2$.

Next, assuming $s_{0} = 1$, let us prove \eqref{eq:est4gt2caloric:U:2} for integers $m \geq 3$. We will proceed by induction on $m$. The estimate for $\nrm{\rd_{x}^{(2)} U}_{L^{\infty}_{s} L^{2}_{x}}$ that we just proved will serve as the base step. 

Fix $m \geq 3$. Suppose, for the purpose of induction, that we already have the estimates 
\begin{equation*}
	\nrm{s^{(k-2)/2} \rd_{x}^{(k)} U}_{L^{\infty}_{s} L^{2}_{x}} \leq C_{k, \nrm{\Aini}_{\dot{H}^{1}_{x}}} \nrm{\Aini}_{\dot{H}^{1}_{x}}
\end{equation*}
for $k = 2, 3, \ldots, (m-1)$. Let us differentiate the ODE $m$-times and integrate from $s$ to $1$. Note that $\rd_{x}^{(m)} U(1) = \rd_{x}^{(m)} \mathrm{Id} = 0$. Taking the $L^{2}_{x}$ norm and multiplying by $s^{(m-2)/2}$, we obtain
\begin{equation} \label{eq:est4gt2caloric:pf:2}
	s^{(m-2)/2} \nrm{\rd_{x}^{(m)} U(s)}_{L^{2}_{x}} \leq\int_{s}^{1} (s/s')^{(m-2)/2} (s')^{m/2} \nrm{\rd_{x}^{(m)} (U(s') A_{s} (s') )}_{L^{2}_{x}} \, \frac{ \ud s'}{s'}
\end{equation}

We remark that here we have written the line element in the scale-invariant form $(\ud s'/s')$, as it is more convenient for keeping track of $s'$-weights. 
Using Leibniz's rule, H\"older and the fact that $s/s' \leq 1$, we have
\begin{align*}
(s/s')^{(m-2)/2} & (s')^{m/2} \nrm{\rd_{x}^{(m)} (U(s') A_{s} (s') )}_{L^{2}_{x}} \\
\leq & C (s/s')^{(m-2)/2} \nrm{U (s')}_{L^{\infty}_{x}} \bb( (s')^{m/2}\nrm{\rd_{x}^{(m)} A_{s}(s')}_{L^{2}_{x}} \bb)  \\
	& + C (s/s')^{(m-2)/2} \nrm{\rd_{x} U (s')}_{L^{3}_{x}} \bb( (s')^{m/2} \nrm{\rd_{x}^{(m-1)} A_{s}(s')}_{L^{6}_{x}}  \bb)  \\
	&+ C (s')^{1/4} \sum_{k=0}^{m-2} \bb( (s')^{(m-k-2)/2}\nrm{\rd_{x}^{(m-k)} U(s')}_{L^{2}_{x}} \bb) \bb( (s')^{k/2 + 3/4} \nrm{\rd_{x}^{(k)} A_{s}(s')}_{L^{\infty}_{x}} \bb) .
\end{align*}

To estimate the first two terms on the right hand side, we use Sobolev and \eqref{eq:est4gt2caloric:pf:1} for $A_{s}$, and \eqref{eq:est4gt2caloric:U:1} for $U$. In order to estimate the last term, we use Gagliardo-Nirenberg and \eqref{eq:est4gt2caloric:pf:1} to estimate $A_{s}$, and the induction hypothesis for $U$. As a consequence, the right-hand side of \eqref{eq:est4gt2caloric:pf:2} is estimated by
\begin{equation*}
	C_{m, \nrm{\Aini}_{\dot{H}^{1}}}  \nrm{\Aini}_{\dot{H}^{1}} \bb( \int_{s}^{1} (s/s')^{(m-2)/2} \, \frac{\ud s'}{s'} + \int_{s}^{1} (s')^{1/4} \, \frac{\ud s'}{s'} \bb)
	+ C_{m} \int_{s}^{1} (s')^{1/4} \nrm{\rd_{x}^{(m)} U(s')}_{L^{2}_{x}} \, \frac{\ud s'}{s'}.
\end{equation*}

Note that the first term is uniformly bounded in $s$, since $m \geq 3$. The second term, on the other hand, can be treated via Gronwall's inequality as before, since $(s')^{1/4} \, \frac{\ud s'}{s'} = (s')^{-3/4} \, \ud s'$ is integrable on $[0, 1]$. As a result, we obtain $\nrm{s^{(m-2)/2} \rd_{x}^{(m)} U(s)}_{L^{2}_{x}} \leq C_{m, \nrm{\Aini}_{\dot{H}^{1}_{x}}} \cdot \nrm{\Aini}_{\dot{H}^{1}_{x}}$, as desired.

Next, as in the proof of Lemma \ref{lem:est4gt2temporal}, the corresponding estimates for $\dlt U$, $U^{-1}$, $\dlt U^{-1}$ can be proved in the same manner, based on the observation that these variables satisfy
\begin{equation*}
\rd_{t} (\dlt U) = (\dlt U) A_{s} + U \dlt A_{s} , \quad
\rd_{t} U^{-1} = -A_{s} U^{-1}, \quad 
\rd_{t} \dlt U^{-1} = -  A_{s} \dlt U^{-1} - (\dlt A_{s}) U^{-1},
\end{equation*}
respectively. We leave the precise details to the reader. \qedhere
\end{proof}
% ----------------------------------------------------------------

\bibliographystyle{amsplain}
%\bibliography{SJ-library}
% ----------------------------------------------------------------

\end{document}